\documentclass[10pt, a4paper, onecolumn]{IEEEtran}

\usepackage{graphicx}
\usepackage{subfig}
\usepackage{amsmath}
\usepackage{amsfonts}
\usepackage{amssymb}
\usepackage{color}
\usepackage{rotating}
\usepackage[pdftex]{hyperref}

\newtheorem{lemma}{Lemma}
\newtheorem{theorem}{Theorem}
\newtheorem{corollary}{Corollary}
\newtheorem{definition}{Definition}
\newtheorem{example}{Example}
\newtheorem{proposition}{Proposition}
\newtheorem{claim}{Claim}

\DeclareMathOperator{\esssup}{ess\,sup}
\title{Intermittent Kalman Filtering:\\Eigenvalue Cycles and Nonuniform Sampling}
\author{Se Yong Park (separk@eecs.berkeley.edu), Anant Sahai (sahai@eecs.berkeley.edu)}
\begin{document}
\maketitle
\abstract
We consider Kalman filtering problems when the observations are intermittently erased or lost.
It was known that the estimates are mean-square unstable when the erasure probability is larger than a certain critical value, and stable otherwise. But the characterization of the critical erasure probability has been open for years. We introduce a new concept of \textit{eigenvalue cycles} which captures periodicity of systems, and characterize the critical erasure probability based on this.  It is also proved that eigenvalue cycles can be easily broken if the original physical system is considered to be continuous-time --- randomly-dithered nonuniform sampling of the observations makes the critical erasure probability almost surely $\frac{1}{|\lambda_{max}|^2}$.

\section{Introduction}
Unlike classical control systems where the controller and the plant are closely located or connected by dedicated wired links, in post-modern systems the controllers and plants can be located far apart and thus control has to happen over communication channels.
In other words, there is an observer which can only observe the plant but cannot control it. There is a separate actuator which can only control the plant but cannot observe it. The observer and actuator are connected by a communication channel. Therefore, to control the plant the observer has to send information about its observation to the actuator through the communication channel. Understanding the tradeoff between control performance and communication reliability or finding the optimal controller structures become the fundamental questions to build such post-modern control systems.

Not only practically, but also philosophically, control-over-communication-channel problems are important. When we are controlling systems, there is a corresponding life cycle of information. In other words, the uncertainty or new information is generated and disturbs the plant. This information is propagated to the controller as the controller observes the plant. Finally, when the controller controls the system by removing the uncertainty, the information is dissipated. It is conceptually very important to understand and quantify these information flows which naturally occur as we control systems. In control-over-communication-channel systems, all the information for control has to flow through the communication channel. Therefore, by relating the communication channels with the control performance, we can measure how much information has to flow to achieve a certain control performance.

Theoretical study of control-over-communication-channel problems was pioneered by Baillieul~\cite{Baillieul_feedback,Baillieul_feedback2} and Tatikonda {\em et al.}~\cite{Tatikonda_Control}. They restricted the communication channels to noiseless rate-limited channels, and asked what the minimum rate of the channel is to stabilize the plant. They found that the rate of the channel has to be at least the sum of the logarithms of the unstable eigenvalues, and indeed it is sufficient. This fact is known as the data-rate theorem. Later, Nair~\cite{nair2003exponential} relaxed the bounded disturbance assumption that they had to Gaussian disturbances, and proved that the same data-rate theorem holds.

However, an important question was whether we can reduce noisy communication channels to noiseless channels with the same Shannon capacity, i.e. whether the classical notion of Shannon capacity is still appropriate when the channel is used for control. In \cite{Sahai_Anytime}, Sahai {\em et al.} found the answer for this question is no. Intuitively, since the system keep evolving in time, not only the rate but also the delay of communication is important. Since Shannon capacity ignores the delay issue, it is insufficient to understand information flows for control. Thus, they proposed a new notion of \textit{anytime capacity} which captures the delay of communication. The stabilizability condition for noisy communication channels with feedback\footnote{By introducing a feedback, they reduced the problem to the one with nested information structure~\cite{Witsenhausen_Separation} which is known to be much easier to solve in decentralized control theory.} was characterized by anytime capacity.

Since then, researchers have accumulated lots of literature~\cite{hespanha2007survey, Schenato_Foundations, nair2007feedback, martins2008feedback, gupta2007optimal, yuksel2006minimum, yuksel2011control} which consider various generalized and related problems. However, still most of the problems are wide open, and \textit{intermittent Kalman filtering} problem which we will study in this paper had been one of them.
In \cite{Sinopoli_Kalman}, Sinopoli {\em et al.} considered `control over real erasure channels' which can be thought as a special case of \cite{Sahai_Anytime}, but with a structural constraint on controller design.

Figure~\ref{fig:system} shows the system diagram for control-over-real-erasure-channels. The observer makes the observation about the plant, and then uncodedly transmits its observation through the real erasure channel. The real erasure channel drops the transmitted signal with a certain probability but otherwise noiselessly transmits the signal. Finally, based on the received signals from the channel, the controller generates its control inputs to stabilize the system.

The situation that this problem is modeling is that of control over a so-called \textit{packet drop channel}. A memoryless observer samples the output of an unstable continuous-time system, quantizes this sample to a sufficient number of bits, binds the resulting bits into a single packet, and transmits the packet to the controller through a communication system. Due to network congestion or wireless fading, the transmitted packet may be lost\footnote{Such losses need not come from network effects --- they could also occur because of sensor occlusion or otherwise at the sampling time itself. That is why the issue of intermittent observations needs to be studied on its own.} with a certain probability and this packet erasure process is further simplified to be i.i.d. The problem is designed to focus attention on the delay/reliability effect of losing packets and so the number of bits per packet (capacity) is unconstrained. The main problem is finding what is the maximum tolerable erasure probability keeping the system stable.

The \textit{linearity} and \textit{memorylessness} of the observer is at the heart of what Sinopoli {\em et al.} are trying to model. Otherwise, the earlier results of \cite{Sahai_Thesis} immediately reveal that the critical erasure probability for the stabilizability only depends on the magnitude of the largest eigenvalue of the plant. However, to achieve the minimal erasure probability shown in \cite{Sahai_Thesis}, the observer and controller design has to be quite complicated and not realistic in practice. Therefore, it is practically and theoretically important to understand how much the control performance degradates with linear observer and controller constraints.

In this paper, we will see that the degradation of stabilizability due to linear constraints fundamentally comes only from the periodicity of the system.
Nonuniform sampling is proposed as a simple way to force the system to behave aperiodically.
Therefore, by using linear controllers in a junction with nonuniform sampling, we can expect a significant performance gain and indeed recover the optimal stabilizability condition over all possible controller designs.


Furthermore, by the estimation-control separation principle~\cite{KumarVaraiya}, the closed-loop control system can be reduced to an equivalent open-loop estimation problem~\cite{Schenato_Foundations}. Figure~\ref{fig:system2} shows the resulting open-loop estimation system so-called \textit{intermittent Kalman filtering}~\cite{Sinopoli_Kalman}. As before, the sensor uncodedly transmits its observation to the real erasure channel. Then, the estimator tries to estimate the state based on its received signals. We refer to  \cite{Schenato_Foundations} for a literature review and practical applications of the problem.

This paper is organized as follows: First, we formally state the problem in Section~\ref{sec:statement}. Then, we introduce some definitions in Section~\ref{sec:def}. In Section~\ref{sec:connecting}, we consider the intermittent observability as a connection of the stability and the observability. From this, we distinguish our approach to the previous approaches. In Section~\ref{sec:intui}, we introduce the intuition for the characterization of the intermittent observability using representative examples. In Section~\ref{sec:interob}, we formally define the eigenvalue cycle and characterize the intermittent observability. In Section~\ref{sec:nonuniform}, we discuss that nonuniform sampling can break the eigenvalue cycle and significantly improve the performance of the intermittent Kalman filtering. Finally, Section~\ref{sec:proof} gives the proof of the main results.

\begin{figure*}[t]
\begin{center}
\includegraphics[width=3in]{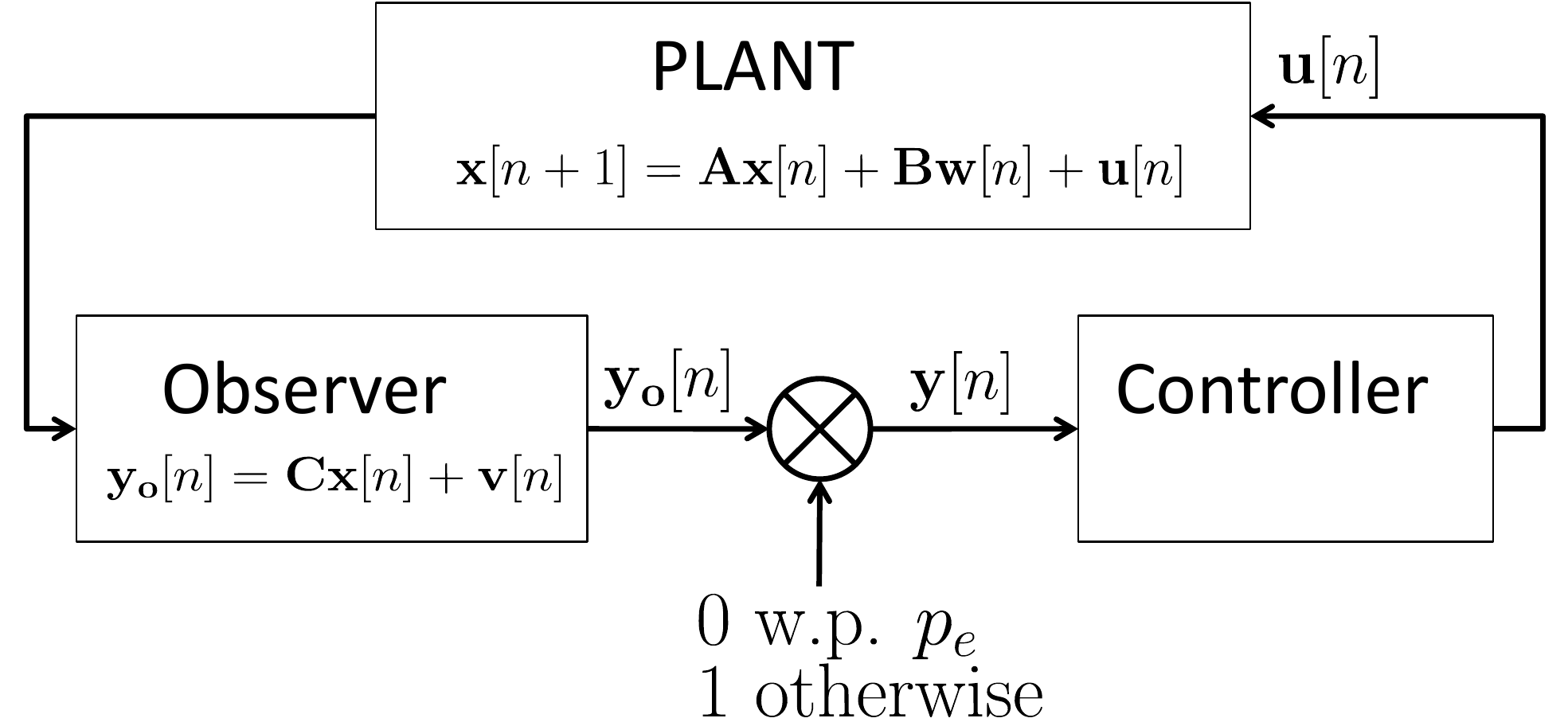}
\caption{Closed-loop system for `control over real erasure channels'. Here, the observer just bypasses its observation to the channel without any coding.}
\label{fig:system}
\end{center}
\end{figure*}

\begin{figure*}[t]
\begin{center}
\includegraphics[width=3.2in]{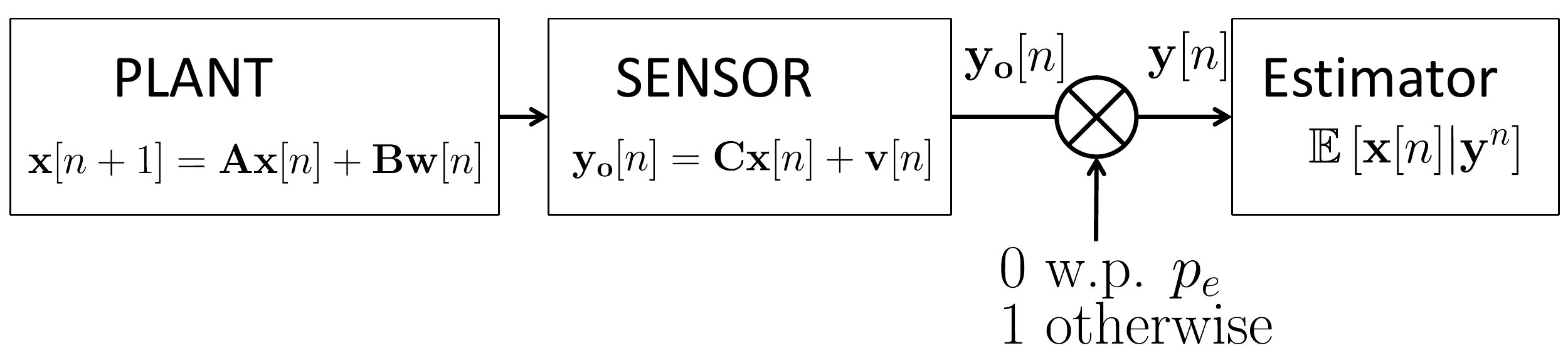}
\caption{System diagram for `intermittent Kalman filtering'. This open-loop estimation system is equivalent to the closed-loop control system of Figure~\ref{fig:system}. Like Figure~\ref{fig:system}, the sensor bypasses its observation to the channel without any coding.}
\label{fig:system2}
\end{center}
\end{figure*}

\section{Problem Statement}
\label{sec:statement}
Formally, the intermittent Kalman filtering problem is formulated as follows in discrete time:
\begin{align}
&\mathbf{x}[n+1]=\mathbf{A}\mathbf{x}[n]+\mathbf{B}\mathbf{w}[n]\label{eqn:dis:system} \\
&\mathbf{y}[n]=\beta[n]\left(\mathbf{C}\mathbf{x}[n]+\mathbf{v}[n]\right)\label{eqn:dis:system2}.
\end{align}

Here $n$ is the non-negative integer-valued time index and the system
variables can take on complex values --- i.e. $\mathbf{x}[n]
\in \mathbb{C}^m, \mathbf{w}[n] \in \mathbb{C}^g, \mathbf{y}[n] \in
\mathbb{C}^l, \mathbf{v}[n] \in \mathbb{C}^l$. $\mathbf{A} \in \mathbb{C}^{m \times m}$, $
\mathbf{B} \in \mathbb{C}^{m \times g}$ and $\mathbf{C} \in \mathbb{C}^{l \times m}$. The underlying randomness comes from the initial
state $\mathbf{x}[0]$, the persistent driving disturbances
$\mathbf{w}[n]$, the observation noises $\mathbf{v}[n]$ and
the Bernoulli packet-drops $\beta[n]$. $\beta[n] = 0$ with probability
$p_e$. $\mathbf{x}[0]$, $\mathbf{w}[n]$ and $\mathbf{v}[n]$ are jointly Gaussian.

The objective is to find the best causal estimator $\mathbf{\widehat{x}}[n]$ of $\mathbf{x}[n]$ that minimizes the mean
square error (MMSE) $\mathbb{E}[(\mathbf{x}[n]-\mathbf{\widehat{x}}[n])^\dag
(\mathbf{x}[n]-\mathbf{\widehat{x}}[n])]$, i.e. $\mathbf{\widehat{x}}[n]=\mathbb{E}[\mathbf{x}[n]| \mathbf{y}^n]$. We assume that the statistics of all random variables are known to the estimator. If $\mathbf{x}[0]$, $\mathbf{w}[n]$ and $\mathbf{v}[n]$ do not have zero mean, the estimator can properly shift its estimation. Thus, without loss
of generality, $\mathbf{x}[0], \mathbf{w}[n]$ and
$\mathbf{v}[n]$ are assumed to be zero mean. $\mathbf{x}[0], \mathbf{w}[n]$ and
$\mathbf{v}[n]$ are independent and have uniformly bounded second moments so that there exists a positive $\sigma^2$ such that
\begin{align}
&\mathbb{E}[\mathbf{x}[0]\mathbf{x}[0]^\dag] \preceq \sigma^2 \mathbf{I} \label{eqn:dis:systemconst1} \\
&\mathbb{E}[\mathbf{w}[n]\mathbf{w}[n]^\dag] \preceq \sigma^2 \mathbf{I} \nonumber \\
&\mathbb{E}[\mathbf{v}[n]\mathbf{v}[n]^\dag] \preceq \sigma^2 \mathbf{I}. \nonumber
\end{align}

To prevent degeneracy, we also assume that there exists a positive $\sigma'^2$ such that
\footnote{The second condition on $\mathbf{v}[n]$ may seem redundant, and $\mathbf{v}[n]=0$ is enough since at each time the new disturbance $\mathbf{w}[n]$ is added. However, when $\mathbf{v}[n]=0$, we can make the following counterexample when the estimation error of the state is bounded even if the system matrices $(\mathbf{A},\mathbf{C})$ are not observable: $\mathbf{A}=\begin{bmatrix} 2 & 1 \\ 0 & 2 \end{bmatrix}, \mathbf{B}=\begin{bmatrix} 0 \\ 1 \end{bmatrix}, \mathbf{C}=\begin{bmatrix} 0 & 1 \end{bmatrix}$. Thus, this assumption is usually kept in the analysis of Kalman filtering including \cite[p.100]{KumarVaraiya}.}
\begin{align}
&\mathbb{E}[\mathbf{w}[n]\mathbf{w}[n]^\dag] \succeq \sigma'^2 \mathbf{I} \label{eqn:dis:systemconst2} \\
&\mathbb{E}[\mathbf{v}[n]\mathbf{v}[n]^\dag] \succeq \sigma'^2 \mathbf{I}. \nonumber
\end{align}

Under these assumptions we call \eqref{eqn:dis:system} and \eqref{eqn:dis:system2} an \textit{intermittent system}.

\begin{definition}
The linear system equations \eqref{eqn:dis:system} and \eqref{eqn:dis:system2} with the second moment conditions \eqref{eqn:dis:systemconst1} and \eqref{eqn:dis:systemconst2} are called an \textbf{intermittent system} $(\mathbf{A},\mathbf{B},\mathbf{C})$, or an \textbf{intermittent system} $(\mathbf{A},\mathbf{B},\mathbf{C})$ \textbf{with erasure probability} $p_e$ when we only want to specify the erasure probability, or an \textbf{intermittent system} $(\mathbf{A},\mathbf{B},\mathbf{C}, \sigma, \sigma')$ \textbf{with erasure probability} $p_e$ when we specify the upper and lower bounds on disturbances as well.
\end{definition}
We say that the intermittent system is \textit{intermittent observable} if the MMSE is uniformly bounded for all time.
\begin{definition}
An intermittent system $(\mathbf{A},\mathbf{B},\mathbf{C},\sigma,\sigma')$ with erasure probability $p_e$ is called \textbf{intermittent observable} if there exists a casual estimator $\mathbf{\widehat{x}}[n]$ of $\mathbf{x}[n]$ such that
\begin{align}
\sup_{n \in \mathbb{Z}^+}
\mathbb{E}[(\mathbf{x}[n]-\mathbf{\widehat{x}}[n])^\dag
(\mathbf{x}[n]-\mathbf{\widehat{x}}[n])] < \infty. \nonumber
\end{align}
\end{definition}

Before we discuss truly intermittent cases, let's consider two extreme cases, when $p_e=1$ and $p_e=0$, to get some insight into the problem. When $p_e=1$, the estimator does not have any observations. As a result, the system can be intermittent observable if and only if the system itself is stable. On the other hand, when $p_e=0$, the estimator has all the observations without any erasures. Intermittent observability reduces to observability. Thus, intermittent observability can be understood as a new concept which interpolates two core concepts of linear system theory: stability and observability.

Moreover, in intermittent systems, we can see the monotonicity of performance with the erasure probability $p_e$. A process with higher erasure probability can be simulated from a process with lower erasure probability by randomly dropping the observations. Therefore, it is obvious that the average estimation error is an increasing function on $p_e$. Especially, if we consider an unstable but observable system, when $p_e=1$ the estimation error goes to infinity, and when $p_e=0$ the estimation error is bounded. Therefore, between $1$ and $0$ there must be a threshold on $p_e$ when the estimation error first becomes infinity.

\begin{theorem}[Theorem 2. of \cite{Sinopoli_Kalman}]
Given an intermittent system $(\mathbf{A},\mathbf{B},\mathbf{C},\sigma,\sigma')$ with erasure probability $p_e$, let $(\mathbf{A},\mathbf{B})$ be controllable, $\sigma < \infty$, and $\sigma' > 0$.\footnote{See Definiton~\ref{def:con} for controllability} Then, there
exists a threshold $p_e^{\star}$, such that for $p_e < p_e^{\star}$
the intermittent system $(\mathbf{A},\mathbf{B},\mathbf{C},\sigma,\sigma')$ with erasure probability $p_e$ is intermittent
observable and for $p_e \geq p_e^{\star}$ the intermittent system
$(\mathbf{A},\mathbf{B},\mathbf{C},\sigma,\sigma')$ with erasure probability $p_e$ is not intermittent observable.
\end{theorem}

Therefore, the characterization of intermittent observability reduces to the characterization of the critical erasure probability $p_e^\star$. For characterizing the critical erasure probability, we can consider it as a generalization of either stability or observability.

In \cite{Sinopoli_Kalman}, Sinopoli \textit{et al.} thought of intermittent observability as a generalization of stability.
Based on Lyapunov stability, they could find a lower bound on the critical erasure probability in a LMI (linear matrix inequality) form. However, this bound is not tight in general and does not give any insight into the solution. A more intuitive bound can be found in \cite{Elia_Remote}.

\begin{theorem}[Corollary 8.4. of \cite{Elia_Remote}]
Given an intermittent system $(\mathbf{A},\mathbf{B},\mathbf{C},\sigma,\sigma')$ with erasure probability $p_e$, let $(\mathbf{A},\mathbf{B})$ be controllable, $\sigma < \infty$, $\sigma' > 0$, and $(\mathbf{A},\mathbf{C})$ be observable. Then,
\begin{align}
\frac{1}{\prod_{i}|\lambda_i|^2} \leq p_e^{\star} \leq
\frac{1}{|\lambda_{max}|^2} ,\nonumber
\end{align}
where $\lambda_i$ are the unstable eigenvalues of $\mathbf{A}$ and $\lambda_{max}$ is the one with the largest magnitude.
\end{theorem}

Therefore, the critical erasure probability characterization boils down to understanding where the gap between $\frac{1}{\prod_{i}|\lambda_i|^2}$ and $\frac{1}{|\lambda_{max}|^2}$ comes from.

In \cite{Yilin_Characterization}, Mo and Sinopoli found two interesting cases that give further insight into this question. The first is when $\mathbf{A}$ is diagonalizable and all eigenvalues of $\mathbf{A}$ have distinct magnitudes --- then the critical erasure probability is $\frac{1}{|\lambda_{max}|^2}$ just it would be in the formulation of \cite{Sahai_Thesis}. The second case is when $\mathbf{A}=\begin{bmatrix} 2 & 0 \\ 0 & -2 \end{bmatrix}$ and $\mathbf{C}=\begin{bmatrix}1 & 1  \end{bmatrix}$ --- the critical erasure probability is $\frac{1}{\prod_{i}|\lambda_i|^2}=\frac{1}{2^4}$. This second case showed that the gap is real and requiring packets to be about a scalar observation can have serious consequences.

To extend these cases and solve the general problem, we will apply insights from observability and introduce the new concept of an \textit{eigenvalue cycle}. As a corollary, we show that in the absence of eigenvalue cycles the critical value becomes $\frac{1}{|\lambda_{max}|^2}$. Furthermore, we show that simply by introducing nonuniform sampling to the sensor, eigenvalue cycles can be broken and the critical erasure probability becomes effectively $\frac{1}{|\lambda_{max}|^2}$.

These results can be surprising if we remember that computing random Lyapunov exponents are difficult problems in general~\cite{tsitsiklis1997lyapunov}. However, the intermittent Kalman filtering problem turns out to have a special structure which makes the problem tractable. Precisely speaking, as we will see in Section~\ref{sec:separability}, the subspaces of the vector state can be separated asymptotically.
To justify such separation, we use ideas from information theory (for example, decoding functions~\cite{nazer2007computation} or successive decoding~\cite{Cover}). Therefore, the whole system can be divided into parallel sub-systems in effect. As we will see in Section~\ref{sec:powerproperty}, each sub-system can be solved using ideas from large deviation theory~\cite{Dembo}.

\section{Definitions and Notations}
\label{sec:def}
Before we start the formal discussion of the problem, we first have to introduce mathematical definitions and notations.

We will use controllability and observability notions from linear system theory.
\begin{definition}
For a $m \times m$ matrix $\mathbf{A}$ and a $m \times p$ matrix $\mathbf{B}$, $(\mathbf{A},\mathbf{B})$ is called controllable if
\begin{align}
\mathbf{\mathcal{C}}=\begin{bmatrix}
\mathbf{B} & \mathbf{A}\mathbf{B} & \cdots & \mathbf{A^{m-1}}\mathbf{B}
\end{bmatrix} \nonumber
\end{align}
is full rank, or equivalently $\begin{bmatrix} \lambda \mathbf{I} - \mathbf{A} & \mathbf{B} \end{bmatrix}$ is full rank for all $\lambda \in \mathbb{C}$. Moreover, we call an eigenvalue $\lambda$ of $\mathbf{A}$ uncontrollable if $\begin{bmatrix} \lambda \mathbf{I} - \mathbf{A} & \mathbf{B} \end{bmatrix}$ is rank deficient.
\label{def:con}
\end{definition}
\begin{definition}
For a $m \times m$ matrix $\mathbf{A}$ and a $l \times m$ matrix $\mathbf{C}$, $(\mathbf{A},\mathbf{C})$ is called observable if
\begin{align}
\mathbf{\mathcal{O}}=\begin{bmatrix}
\mathbf{C} \\
\mathbf{C}\mathbf{A} \\
\vdots \\
\mathbf{C}\mathbf{A^{m-1}}
\end{bmatrix} \nonumber
\end{align}
is full rank, or equivalently $\begin{bmatrix} \lambda \mathbf{I} -\mathbf{A} \\ \mathbf{C} \end{bmatrix}$ is full rank for all $\lambda \in \mathbb{C}$. Moreover, we call an eigenvalue $\lambda$ of $\mathbf{A}$ unobservable if $\begin{bmatrix} \lambda \mathbf{I} -\mathbf{A} \\ \mathbf{C} \end{bmatrix}$ is rank deficient.
\end{definition}

We will use Bernoulli processes and geometric random variables from probability theory.
\begin{definition}
An one-sided discrete-time random process $a[n]~(n \geq 0)$ is called a Bernoulli random process with probability $p$ if $a[n]$ are i.i.d.~random variables with the following probability mass function (p.m.f.):
\begin{align}
\left\{
\begin{array}{l}
\mathbb{P}(a[n]=1)=p\\
\mathbb{P}(a[n]=0)=1-p
\end{array}\right.  \nonumber
\end{align}
We also call $a[n]$ as a Bernoulli random variable with erasure probability $1-p$. A two-sided Bernoulli random process is defined in the same way except that $n$ comes from the integers.
\end{definition}
\begin{definition}
A random variable $X \in \mathbb{Z^+}$ is called a geometric random variable with probability $p$ if it has a probability mass function
$\mathbb{P}\{ X=x \} = p(1-p)^{x}$ for $x \geq 0$. We also call $X$ as a geometric random variable with erasure probability $1-p$.
\end{definition}

Then, we have the following relationship between Bernoulli random processes and geometric random variables. Let
\begin{align}
X:=\min\{n \in \mathbb{Z}^+: a[n]=1 \mbox{ where $a[n]$ is a Bernoulli random variable with probability $p$} \}. \nonumber
\end{align}
Then, $X$ is a geometric random variable with probability $p$.

We will also use the following basic notions about matrices.
\begin{definition}
Given a matrix $\mathbf{A} \in \mathbb{C}^{m \times m}$, $|\mathbf{A}|_{max}$ is the elementwise max norm of $\mathbf{A}$ i.e. $|\mathbf{A}|_{max}=\max_{1 \leq i,j \leq m}|a_{ij}|$.
\end{definition}
\begin{definition}
Given a matrix $\mathbf{A} \in \mathbb{C}^{m \times m}$, $\dim \mathbf{A}$ denotes $m$. Given a column vector $\mathbf{x_1} \in \mathbb{C}^{m \times 1}$ and a row vector $\mathbf{x_2} \in \mathbb{C}^{1 \times m}$, $\dim \mathbf{x_1}$ and $\dim \mathbf{x_2}$ denote $m$.
\end{definition}
\begin{definition}
Given ${n_i} \times {n_i}$ matrices $\mathbf{\mathbf{A_i}}$ for $i \in \{1,2,\cdots, m\}$, $diag\{ \mathbf{A_1}, \mathbf{A_2}, \cdots, \mathbf{A_m} \}$ is a $\left(\sum^{m}_{i=1} n_i \right) \times \left(\sum^{m}_{i=1} n_i \right)$ matrix in the form of
$
\begin{bmatrix}
\mathbf{A_1} & 0 & \cdots & 0 \\
0 & \mathbf{A_2} & \cdots & 0 \\
\vdots & \vdots & \ddots & \vdots \\
0 & 0 & \cdots & \mathbf{A_m} \\
\end{bmatrix}
$.
\end{definition}

We also define modulo operation on numbers.
\begin{definition}
A sequence, $a_1,a_2,\cdots, a_n$, is called congruent mod $p$ if $a_i \equiv a_j (mod\ p)$ for all $i,j$.
\end{definition}

\begin{definition}
A sequence, $a_1,a_2,\cdots, a_n$, is called pairwise incongruent mod $p$ if $a_i \not\equiv a_j  (mod\ p)$ for all $i \neq j$.
\end{definition}

Since we will only focus on the scalings behavior, we will use the following definition which can be used as big $O$ and big $\Omega$ notations in complexity theory.
\begin{definition}
Consider two real functions $a(t)$ and $b(t)$ whose common domain is $T \in \mathbb{R}$. We say
$a(t) \lesssim b(t)$ for $t$ on $T$ if there exists a positive $c$ such that $a(t) \leq c b(t)$ for all $t \in T$.
\label{def:lesssim}
\end{definition}
We omit the argument and the domain of the above definition, when they are obvious from the context and do not cause confusion.

We will also use an abbreviated notation for a sequence of random variables.
\begin{definition}
Given a discrete time random variable $a[0], \cdots, a[n]$, we denote $a[n_1],\cdots,a[n_2]$ as $a_{n_1}^{n_2}$, and $a[0],\cdots,a[n]$ as $a^n$. Likewise given a continuous time random variable $b(t)$, we define $\mathbf{b}(t_1:t_2)$ to be $\mathbf{b}(t)$ for $t_1 \leq t \leq t_2$.
\end{definition}
\section{Intermittent Observability as an Extension of Stability}
\label{sec:connecting}

As we mentioned before, the characterization of the critical erasure probability can be considered from two different directions --- an extension of stability or an extension of observability. In \cite{Sinopoli_Kalman}, Sinopoli \textit{et al.} took the first approach, and attempted to characterize the critical erasure probability by the Lyapunov stability condition. Let's review a property of Schur complements and Lyapunov stability theorem.

\begin{lemma}[Schur complements]
Let $\mathbf{X}=\begin{bmatrix} \mathbf{A} & \mathbf{B} \\ \mathbf{B}^\dag & \mathbf{C} \end{bmatrix}$ be a symmetric matrix and $\mathbf{C}$ be invertible. Then,  $\mathbf{X} \succ 0$ if and only if $\mathbf{C} \succ 0$ and $\mathbf{A} - \mathbf{B}\mathbf{C}^{-1} \mathbf{B}^\dag \succ 0$.
\label{lem:schur}
\end{lemma}
\begin{proof}
See \cite[p. 650]{Boyd}.
\end{proof}

\begin{theorem}[Lyapunov Stability Theorem]
Given a linear system \eqref{eqn:dis:system}, the following three conditions are equivalent.\\
(i) The system is stable.\\
(ii)
$\exists \mathbf{M},\mathbf{N} \succ \mathbf{0} $ such that
\begin{align}
\mathbf{M} - \mathbf{A} \mathbf{M} \mathbf{A}^\dag = \mathbf{N}. \nonumber
\end{align}\\
(iii)
$\exists \mathbf{M} \succ \mathbf{0}$ such that
\begin{align}
\begin{bmatrix}
\mathbf{M} & \mathbf{A}\mathbf{M} \\
\mathbf{M} \mathbf{A}^\dag  & \mathbf{M}
\end{bmatrix} \succ \mathbf{0}. \nonumber
\end{align}\label{thm:lyapunov}
\end{theorem}
\begin{proof}
The equivalence between (i) and (ii) can be easily found in linear system theory books including \cite[p.30]{KumarVaraiya} and \cite[Theorem 5.D5]{Chen}. The equivalence between (ii) and (iii) comes from Schur complements in Lemma~\ref{lem:schur} by simply choosing $\mathbf{A}=\mathbf{M}$, $\mathbf{B}=\mathbf{A}\mathbf{M}$ and $\mathbf{C}=\mathbf{M}$.
\end{proof}

Before we consider intermittent observability, let's first characterize the standard observability condition using Lyapunov stability.
The fundamental theorem of observability tells that if $(\mathbf{A},\mathbf{C})$ is observable, the eigenvalues of the closed loop system $\mathbf{A}+\mathbf{K}\mathbf{C}$ can be placed anywhere by a proper selection of $\mathbf{K}$.
Based on this, we can characterize observability in terms of Lyapunov stability.

\begin{theorem}
Given a linear system \eqref{eqn:dis:system} and \eqref{eqn:dis:system2} with $p_e=0$, the following four conditions are equivalent.\\
(i) All the unstable modes of $\mathbf{A}$ are observable.\\
(ii) $\exists \mathbf{K}$ such that $\mathbf{A}+\mathbf{K}\mathbf{C}$ is stable.\\
(iii) $\exists \mathbf{K}$ and $\mathbf{M},\mathbf{N} \succ \mathbf{0}$ such that\\
\begin{align}
\mathbf{M} - (\mathbf{A}+\mathbf{K}\mathbf{C}) \mathbf{M} (\mathbf{A}+\mathbf{K}\mathbf{C})^\dag = \mathbf{N}. \nonumber
\end{align}
(iv) $\exists \mathbf{K}$ and $\mathbf{M} \succ \mathbf{0}$ such that\\
\begin{align}
\begin{bmatrix} \mathbf{M} & (\mathbf{A}+ \mathbf{K}\mathbf{C})\mathbf{M} \\ \mathbf{M} (\mathbf{A}+ \mathbf{K}\mathbf{C})^\dag  & \mathbf{M} \end{bmatrix} \succ 0.  \nonumber
\end{align}\label{thm:lyaob}
\end{theorem}
\begin{proof}
The equivalence of (i) and (ii) is the fundamental theorem of observability~\cite[Theorem 8.M3]{Chen}. The equivalence of (ii), (iii) and (iv) follows from Theorem~\ref{thm:lyapunov}.
\end{proof}

Unfortunately, this observability characterization based on Lyapunov stability cannot be generalized for intermittent observability.
The main reason is that in intermittent Kalman filtering the optimal estimator does not converge to a linear time-invariant one.
In conventional Kalman filtering for linear time-invariant systems, it is well-known that the optimal Kalman filter converges to the linear time-invariant estimator which is known as the \textit{Wiener filter}~\cite{wiener1964extrapolation}. In fact, we can directly plug in the Wiener filter gain for the matrix $\mathbf{K}$ of Theorem~\ref{thm:lyaob}. However, when observations are erased, the optimal estimator also depends on the erasure pattern and since the erasure pattern is random and time-varying, the whole system becomes random and time-varying. Therefore, the optimal estimator is also time-varying and does not converge.

In \cite{Sinopoli_Kalman}, Sinopoli \textit{et al.} wrote the optimal time-varying linear estimator in a recursive equation form. The strictly causal estimator $\mathbf{\widehat{x}}[n] = \mathbb{E}[\mathbf{x}[n]|\mathbf{y}^{n-1}]$, is given as follows:
\begin{align}
\mathbf{\widehat{x}}[n+1]=\mathbf{A}\mathbf{\widehat{x}}[n]-\mathbf{K_n}(\mathbf{y}[n]-\mathbf{C}\mathbf{\widehat{x}}[n]) \label{eqn:lyainter}
\end{align}
Here, $\mathbf{K_n}$ depends not only on $n$ but also the history of the $\beta[n]$, and does not converge to a constant matrix in probability.
Therefore, in the intermittent Kalman filtering problem it is not possible to find a stability-optimal time-invariant gain $\mathbf{K}$ in Theorem~$\ref{thm:lyaob}$.

However, we can still force the estimator to be linear time-invariant, and thereby find a sufficient condition for intermittent observability using Lyapunov stability ideas. This is the idea that Sinopoli \textit{et al.} used to find a lower bound on the critical erasure probability in \cite{Sinopoli_Kalman}. By restricting the filtering gain to be a linear time-invariant matrix $\mathbf{K}$, we get the following sub-optimal estimator which looks similar to \eqref{eqn:lyainter}.
\begin{align}
\mathbf{\widehat{x}}[n+1]&=\mathbf{A}\mathbf{\widehat{x}}[n]-\beta[n]\mathbf{K}(\mathbf{y}[n]-\mathbf{C}\mathbf{\widehat{x}}[n])
\label{eqn:connect:sub1}
\end{align}
with $\mathbf{\widehat{x}}[0]=\mathbf{0}$. By analyzing this sub-optimal estimator, Sinopoli \textit{et al.} found the following sufficient condition for intermittent observability. Here, we further prove that their condition is both necessary and sufficient for the sub-optimal estimators of \eqref{eqn:connect:sub1} to have an expected estimation error uniformly bounded over time.\footnote{This fact is implicitly shown in Elia's paper~\cite{Elia_Remote}.}
\begin{theorem}[Extension of Theorem~5 of \cite{Sinopoli_Kalman}]
Given an intermittent system $(\mathbf{A},\mathbf{B},\mathbf{C},\sigma,\sigma')$ with erasure probability $p_e$, let $(\mathbf{A},\mathbf{B})$ be controllable, $\sigma < \infty$, and $\sigma' > 0$.
Then, the following three conditions are equivalent.\\
(i) The system is intermittently observable by the suboptimal estimator of \eqref{eqn:connect:sub1} with some $\mathbf{K}$.\\
(ii) $\exists \mathbf{K}$ and $\mathbf{M}, \mathbf{N} \succ \mathbf{0}$ such that
\begin{align}
\mathbf{M} - p_e \mathbf{A} \mathbf{M} \mathbf{A}^\dag - (1- p_e)(\mathbf{A}+\mathbf{K}\mathbf{C})\mathbf{M}(\mathbf{A}+\mathbf{K}\mathbf{C})^\dag = \mathbf{N}. \nonumber
\end{align}
(iii) $\exists \mathbf{K}$ and $\mathbf{M} \succ \mathbf{0}$ such that
\begin{align}
\begin{bmatrix}
\mathbf{M} & \sqrt{1-p_e}( \mathbf{M}\mathbf{A} + \mathbf{K}\mathbf{C})  & \sqrt{p_e} \mathbf{M} \mathbf{A} \\
\sqrt{1-p_e}( \mathbf{M}\mathbf{A} + \mathbf{K}\mathbf{C})^\dag & \mathbf{M} & 0 \\
\sqrt{p_e}( \mathbf{M}\mathbf{A} )^\dag & 0 & \mathbf{M}
\end{bmatrix}\succ \mathbf{0}. \nonumber
\end{align}
\label{thm:lyainter}
\end{theorem}
\begin{proof}
By \eqref{eqn:dis:system}, \eqref{eqn:dis:system2} and \eqref{eqn:connect:sub1}, we can see that the estimation error follows the following dynamics:
\begin{align}
\mathbf{x}[n+1]-\mathbf{\widehat{x}}[n+1]
&=\mathbf{A}\mathbf{x}[n]+\mathbf{B}\mathbf{w}[n]-(\mathbf{A}\mathbf{\widehat{x}}[n]-\beta[n] \mathbf{K}(\mathbf{y}[n]-\mathbf{C}\mathbf{\widehat{x}}[n])) \nonumber \\
&=\mathbf{A}\mathbf{x}[n]+\mathbf{B}\mathbf{w}[n]-(\mathbf{A}\mathbf{\widehat{x}}[n]-\beta[n] \mathbf{K}(\mathbf{C}\mathbf{x}[n]+\mathbf{v}[n]-\mathbf{C}\mathbf{\widehat{x}}[n])) \nonumber \\
&=(\mathbf{A}+\beta[n]\mathbf{K}\mathbf{C})(\mathbf{x}[n]-\mathbf{\widehat{x}}[n])+\mathbf{B}\mathbf{w}[n]
+\beta[n]\mathbf{K}\mathbf{v}[n]. \label{eqn:connect:sub4}
\end{align}
Denote $(\mathbf{x}[n]-\mathbf{\widehat{x}}[n])$ as $(\mathbf{e}[n]$ and $\mathbf{B}\mathbf{w}[n]
+\beta[n]\mathbf{K}\mathbf{v}[n])$ as $\mathbf{w'}[n]$. Then, $\mathbf{w'}[n]$ also has a uniformly bounded variance over time, and \eqref{eqn:connect:sub4} can be written as
\begin{align}
\mathbf{e}[n+1]=(\mathbf{A}+\beta[n]\mathbf{K}\mathbf{C})\mathbf{e}[n]+\mathbf{w'}[n]. \nonumber
\end{align}
Since $\mathbf{e}[n]$ is independent from $\mathbf{w'}[n], \beta[n]$ by causality, the covariance matrix of $\mathbf{e}[n]$ follows the following dynamics:
\begin{align}
\mathbb{E}[ \mathbf{e}[0]\mathbf{e}^\dag[0]] &= \mathbb{E}[\mathbf{x}[0]\mathbf{x}^\dag[0]], \nonumber \\
\mathbb{E}[ \mathbf{e}[n+1]\mathbf{e}^\dag [n+1] ]&= \mathbb{E}[(\mathbf{A}+\beta[n]\mathbf{K}\mathbf{C})\mathbf{e}[n]\mathbf{e}^\dag[n](\mathbf{A}+\beta[n]\mathbf{K}\mathbf{C})^\dag]+\mathbb{E}[\mathbf{w'}[n]\mathbf{w'}^\dag[n]]
\nonumber \\
&=p_e \mathbf{A} \mathbb{E}[\mathbf{e}[n] \mathbf{e}^\dag[n]] \mathbf{A}^\dag
+(1-p_e) (\mathbf{A}+\mathbf{K}\mathbf{C})\mathbb{E}[\mathbf{e}[n]\mathbf{e}^\dag[n]](\mathbf{A}+\mathbf{K}\mathbf{C})^\dag
+\mathbb{E}[\mathbf{w'}[n]\mathbf{w'}^\dag[n]]. \label{eqn:connect:sub2}
\end{align}
Now, we will prove the theorem in three steps.

(1) Condition (i) implies condition (ii).\\
First of all, by linearity we can prove that the estimation error $\mathbb{E}[\mathbf{e}[n]\mathbf{e}^\dag[n]]$ is an increasing function of the variance of the underlying random variables.

Thus, if the system is intermittently observable by $\mathbf{K}$, the same system with $\mathbf{x}[0]=0$, $\mathbf{v}[n]=0$, $\mathbb{E}[\mathbf{w}[n]\mathbf{w}^\dag[n]]= \sigma'^2 \mathbf{I}$ is also intermittently observable. So set $\mathbf{x}[0]=0$, $\mathbf{v}[n]=0$, $\mathbb{E}[\mathbf{w}[n]\mathbf{w}^\dag[n]]= \sigma'^2 \mathbf{I}$ without loss of generality. With these parameters, we have $\mathbb{E}[\mathbf{e}[0]\mathbf{e}^\dag[0]]=0$ and $\mathbb{E}[\mathbf{w'}[n]\mathbf{w'}^\dag[n]]=\sigma'^2\mathbf{B} \mathbf{B}^\dag$. By the recursive equation in \eqref{eqn:connect:sub2}, we can show that for $n \geq 1$, the covariance matrix of $\mathbf{e}[n]$ can be written as
\begin{align}
\mathbb{E}[\mathbf{e}[n]\mathbf{e}^\dag[n]]=\sigma'^2 \mathbf{B}\mathbf{B}^\dag + \sum^n_{k=1} \sum_{l \in \{-1,1 \}^k} \mathbf{\Delta}_l \mathbf{\Delta}_l^\dag.
\end{align}
where
\begin{align}
\mathbf{\Delta}_l := (\sqrt{p_e}\mathbf{A})^{\frac{1+l_1}{2}}(\sqrt{1-p_e}(\mathbf{A}+\mathbf{K}\mathbf{C}))^\frac{1-l_1}{2} \cdots (\sqrt{p_e}\mathbf{A})^{\frac{1+l_k}{2}}(\sqrt{1-p_e}(\mathbf{A}+\mathbf{K}\mathbf{C}))^\frac{1-l_k}{2} \sigma' \mathbf{B}. \nonumber
\end{align}
Here, $l_i=1$ means the $i$th observation was erased and $l_i=-1$ means that the $i$th observation was not erased.

Here, we can notice that $\mathbb{E}[\mathbf{e}[n]\mathbf{e}^\dag[n]]$ are positive semidefinite matrices and increasing in $n$. Furthermore, since the system is intermittently observable by condtion (i), $\mathbb{E}[\mathbf{e}[n]\mathbf{e}^\dag[n]]$ has to  be uniformly bounded over time. Therefore,
\begin{align}
\mathbf{\bar{M}} := \lim_{n \rightarrow \infty}\mathbb{E}[\mathbf{e}[n]\mathbf{e}^\dag[n]] = \sigma'^2 \mathbf{B}\mathbf{B}^\dag + \sum^{\infty}_{k=1} \sum_{l \in \{-1,1 \}^k} \mathbf{\Delta}_l \mathbf{\Delta}_l^\dag \label{eqn:barm}
\end{align}
must exist even though it involves an infinite sum. Let's define $\mathbf{M}$ and $\mathbf{N}$ as follows:
\begin{align}
\mathbf{M}&:= \sigma'^2 \mathbf{B}\mathbf{B}^\dag + \sum_{k=1}^{m-1} \sum_{l=\{-1,1\}^k} (k+1) \mathbf{\Delta}_l \mathbf{\Delta}_l^\dag + \sum_{k'=m}^{\infty} \sum_{l' = \{-1,1\}^{k'}} m \mathbf{\Delta}_l \mathbf{\Delta}_l^\dag \label{eqn:defM}\\
\mathbf{N}&:= \sigma'^2 \mathbf{B}\mathbf{B}^\dag + \sum_{k=1}^{m-1} \sum_{l \in \{-1,1\}^{k}} \mathbf{\Delta}_l \mathbf{\Delta}_l^\dag
\label{eqn:defN}
\end{align}
where $m$ is the dimension of $\mathbf{A}$ as we defined in Section~\ref{sec:statement}. By the definitions of $\mathbf{\bar{M}}$ and $\mathbf{M}$, we can easily see that $m \mathbf{\bar{M}}  \succeq \mathbf{M}$. Therefore, $\mathbf{M}$ also exists even though it involves an infinite sum. Furthermore, by the definitions of $\mathbf{M}$ and $\mathbf{N}$, we can easily see that
\begin{align}
\mathbf{M} \succeq \sigma'^2(\mathbf{B}\mathbf{B}^\dag+p_e \mathbf{A}\mathbf{B}\mathbf{B}^\dag \mathbf{A}^{\dag} + \cdots + p_e^m \mathbf{A}^{m}\mathbf{B}\mathbf{B}^\dag \mathbf{A}^{\dag m} ) \\
\mathbf{N} \succeq \sigma'^2(\mathbf{B}\mathbf{B}^\dag+p_e \mathbf{A}\mathbf{B}\mathbf{B}^\dag \mathbf{A}^{\dag} + \cdots + p_e^m \mathbf{A}^{m}\mathbf{B}\mathbf{B}^\dag \mathbf{A}^{\dag m} )
\end{align}
since the terms in L.H.S. are just subsets of the terms in $\mathbf{M}$ and $\mathbf{N}$.

Thus, we can see that $\mathbf{M} \succ 0$, $\mathbf{N} \succ 0$ since $\begin{bmatrix}
\mathbf{B} & \mathbf{A}\mathbf{B} & \cdots & \mathbf{A}^{m-1}\mathbf{B} \end{bmatrix}$ is full rank by the controllability of $(\mathbf{A},\mathbf{B})$ and all terms $\mathbf{B}\mathbf{B}^\dag, \cdots, p_e^m \mathbf{A}^{m}\mathbf{B}\mathbf{B}^\dag \mathbf{A}^{\dag m}$ are positive semidefinite. Finally, by the definitions and simple matrix algebra, we can verify that $\mathbf{M}$ and $\mathbf{N}$ satisfy the following relationship:
\begin{align}
\mathbf{M} = p_e \mathbf{A}\mathbf{M}\mathbf{A}^\dag + (1-p_e)  (\mathbf{A}+\mathbf{K}\mathbf{C}) \mathbf{M} (\mathbf{A}+\mathbf{K}\mathbf{C})^\dag + \mathbf{N}. \label{eqn:fixedpoint}
\end{align}
Therefore, $\mathbf{M}$ and $\mathbf{N}$ satisfy condition (ii).\footnote{

Consider a fixed point equation, $f(x)=xf(x)+g(x)$. There exist multiple $f(x)$ and $g(x)$ that satisfy this equation. For example, $(f(x), g(x))=(1+x+x^2+ \cdots, 1)$, $(f(x), g(x))=(1+2x+2x^2+ \cdots, 1+x)$, $\cdots$, $(f(x), g(x))=(1+2x+ \cdots + (k-1)x^{k-1} + kx^k + kx^{k+1} \cdots, 1+x+\cdots+x^k)$ all satisfy the equation. Likewise, there are multiple matrices that satisfy the fixed point equation of \eqref{eqn:fixedpoint}. For example, we can easily check that $\mathbf{\bar{M}}$ of \eqref{eqn:barm} and $\mathbf{\bar{N}}:=\sigma'^2 \mathbf{B}\mathbf{B}^\dag$ satisfy \eqref{eqn:fixedpoint}, i.e. $\mathbf{\bar{M}}=p_e \mathbf{A}\mathbf{\bar{M}}\mathbf{A}^\dag+(1-p_e)(\mathbf{A}+\mathbf{K}\mathbf{C})(\mathbf{A}+\mathbf{K}\mathbf{C})^\dag + \mathbf{\bar{N}}$. However, unlike $\mathbf{N}$, $\mathbf{\bar{N}}$ does not have to be positive definite. Thus, the choice of $\mathbf{\bar{M}}$, $\mathbf{\bar{N}}$ is not enough to prove the theorem. Here, we choose $\mathbf{M}$, $\mathbf{N}$ as shown in \eqref{eqn:defM}, \eqref{eqn:defN} as another solution for \eqref{eqn:fixedpoint}. In fact, the choice of coefficient in $\mathbf{M}, \mathbf{N}$ was inspired by the solutions of $f(x)=xf(x)+g(x)$ shown above.}

(2) Condition (ii) implies condition (i).\\
Since $\mathbf{M}$ and $\mathbf{N}$ of condition (ii) are positive definite, we can find $a$ such that
$a^2 \mathbf{M} \succ \mathbb{E}[\mathbf{x}[0]\mathbf{x}^\dag[0]]$ and $a^2 \mathbf{N} \succ \mathbb{E}[\mathbf{w'}[n]\mathbf{w'}^\dag[n]]$ for all $n \in \mathbb{Z}^+$. And we can easily see that even if we replace $\mathbf{K}$, $\mathbf{M}$, $\mathbf{N}$ with $\mathbf{K}$, $a^2 \mathbf{M}$, $a^2 \mathbf{N}$, condition (ii) still holds.

We will prove that $a^2 \mathbf{M} \succ \mathbb{E}[\mathbf{e}[n]\mathbf{e}^\dag[n]]$ for all $n \in \mathbb{Z}^+$ by induction. Since $a^2\mathbf{M} \succ \mathbb{E}[\mathbf{x}[0]\mathbf{x}^\dag[0]]=\mathbb{E}[\mathbf{e}[0]\mathbf{e}^\dag[0]]$, the claim is true for $n=0$. Assume the claim is true for $n$. Then, from the definition of $a$ and \eqref{eqn:connect:sub2},
\begin{align}
\mathbb{E}[\mathbf{e}[n+1]\mathbf{e}^\dag[n+1]] \prec p_e \mathbf{A} (a^2 \mathbf{M}) \mathbf{A}^\dag
+(1-p_e) (\mathbf{A}+\mathbf{K}\mathbf{C})(a^2 \mathbf{M})(\mathbf{A}+\mathbf{K}\mathbf{C})^\dag
+a^2 \mathbf{N}=a^2 \mathbf{M} \nonumber
\end{align}
where the last equality comes from condition (ii). Therefore, the estimation error is uniformly upper bounded by $a^2 \mathbf{M}$ when we use the $\mathbf{K}$ of condition (ii) as a gain matrix, and so condition (ii) implies condition (i).

(3) Condition (ii) is equivalent to condition (iii).\\
Since $\mathbf{M}^{-1} \succ \mathbf{0}$, by Schur complements in Lemma~\ref{lem:schur}, condition (ii) is equivalent to
\begin{align}
\begin{bmatrix}
\mathbf{M}- p_e \mathbf{A} \mathbf{M} \mathbf{A}^\dag  & \sqrt{1-p_e} (\mathbf{A}+\mathbf{K}\mathbf{C}) \\
\sqrt{1-p_e}(\mathbf{A}+\mathbf{K}\mathbf{C})^\dag & \mathbf{M}^{-1}
\end{bmatrix}
\succ \mathbf{0} . \nonumber
\end{align}
Since
\begin{align}
&\begin{bmatrix}
\mathbf{M}- p_e \mathbf{A} \mathbf{M} \mathbf{A}^\dag  & \sqrt{1-p_e} (\mathbf{A}+\mathbf{K}\mathbf{C}) \\
\sqrt{1-p_e}(\mathbf{A}+\mathbf{K}\mathbf{C})^\dag & \mathbf{M}^{-1}
\end{bmatrix} \nonumber\\
&=
\begin{bmatrix}
\mathbf{M} & \sqrt{1-p_e} (\mathbf{A}+\mathbf{K}\mathbf{C}) \\
\sqrt{1-p_e}(\mathbf{A}+\mathbf{K}\mathbf{C})^\dag & \mathbf{M}^{-1}
\end{bmatrix}
-
\begin{bmatrix}
\sqrt{p_e}\mathbf{A} \\
0
\end{bmatrix}
\mathbf{M}
\begin{bmatrix}
\sqrt{p_e}\mathbf{A}^\dag & 0
\end{bmatrix} \nonumber
\end{align}
and $\mathbf{M}^{-1} \succ \mathbf{0}$, we can apply Schur complement again. Thus, condition (ii) is equivalent to
\begin{align}
\begin{bmatrix}
\mathbf{M} & \sqrt{1-p_e}(\mathbf{A}+\mathbf{K}\mathbf{C}) & \sqrt{p_e} \mathbf{A} \\
\sqrt{1-p_e}(\mathbf{A}+\mathbf{K}\mathbf{C})^\dag & \mathbf{M}^{-1} & 0 \\
\sqrt{p_e} \mathbf{A}^\dag & 0 & \mathbf{M}^{-1}
\end{bmatrix} \succ \mathbf{0}. \nonumber
\end{align}
Since $\mathbf{M}^{-1} \succ 0$, this condition is again equivalent to
\begin{align}
&\begin{bmatrix}
\mathbf{M}^{-1} & 0 & 0 \\
0 & \mathbf{I} & 0 \\
0 & 0 & \mathbf{I}
\end{bmatrix}
\begin{bmatrix}
\mathbf{M} & \sqrt{1-p_e}(\mathbf{A}+\mathbf{K}\mathbf{C}) & \sqrt{p_e} \mathbf{A} \\
\sqrt{1-p_e}(\mathbf{A}+\mathbf{K}\mathbf{C})^\dag & \mathbf{M}^{-1} & 0 \\
\sqrt{p_e} \mathbf{A}^\dag & 0 & \mathbf{M}^{-1}
\end{bmatrix}
\begin{bmatrix}
\mathbf{M}^{-1} & 0 & 0 \\
0 & \mathbf{I} & 0 \\
0 & 0 & \mathbf{I}
\end{bmatrix} \nonumber \\
&=
\begin{bmatrix}
\mathbf{M}^{-1} & \sqrt{1-p_e}(\mathbf{M}^{-1}\mathbf{A}+ \mathbf{M}^{-1}\mathbf{K}\mathbf{C}) & \sqrt{p_e}\mathbf{M}^{-1}\mathbf{A} \\
\sqrt{1-p_e}(\mathbf{M}^{-1}\mathbf{A}+ \mathbf{M}^{-1}\mathbf{K}\mathbf{C})^\dag & \mathbf{M}^{-1} & 0 \\
\sqrt{p_e} (\mathbf{M}^{-1}\mathbf{A})^\dag & 0 & \mathbf{M}^{-1}
\end{bmatrix}
\succ \mathbf{0}. \nonumber
\end{align}
Since $\mathbf{M}^{-1} \succ \mathbf{0}$ and $\mathbf{K}$ is an arbitrary matrix, by replacing $\mathbf{M}^{-1}$ by $\mathbf{M}$ and $\mathbf{M}^{-1}\mathbf{K}$ by $\mathbf{K}$ we get condition (iii).
\end{proof}

As we can expect, the conditions of this theorem reduce to those of stability and those of observability when $p_e=1$ and $p_e=0$ respectively. One can easily observe that condition (ii) of Theorem~\ref{thm:lyainter} reduces to condition (ii) of Theorem~\ref{thm:lyapunov} when $p_e=1$ and condition (iii) of Theorem~\ref{thm:lyaob} when $p_e=0$. Likewise, condition (iii) of Theorem~\ref{thm:lyainter} reduces to condition (iii) of Theorem~\ref{thm:lyapunov} and condition (iv) of Theorem~\ref{thm:lyaob} respectively.

Even though condition (ii) and (iii) of Theorem~\ref{thm:lyainter} are equivalent, condition (iii) is preferred since it is given in a LMI (linear matrix inequality) form and convex optimization techniques~\cite{Boyd} are applicable. In fact, in \cite{Sinopoli_Kalman} Sinopoli~\textit{et al.} related condition (iii) with quasi-convex problems.

Since we imposed an additional linear time-invariant constraint on the estimator, Theorem~\ref{thm:lyainter} gives a lower bound on $p_e^\star$. However, we can conclude that this lower bound is loose in general.\footnote{Numerical computation of the lower bound of Theorem~\ref{thm:lyainter} is shown in Figure 4 of \cite{Sinopoli_Kalman}.
For a system with $\mathbf{A}=\begin{bmatrix} 1.25 & 0 \\ 1 & 1.1 \end{bmatrix}$ and $\mathbf{C}=\begin{bmatrix}1 &1 \end{bmatrix}$. The numerical simulation shows the lower bound is approximately $\frac{1}{(1.25 \times 1.1)^2}=0.528\cdots$, while the exact characterization of Theorem~\ref{thm:mainsingle}
tells the critical erasure probability is $\frac{1}{1.25^2}=0.64$.} Moreover, even for stability, the characterization that the magnitudes of all eigenvalues are less than $1$ is much more intuitive than the LMI condition based on Lyapunov stability. Therefore, researchers including \cite{Elia_Remote} and \cite{Yilin_Characterization} were looking for a tight and intuitive characterization of the critical erasure probability.

\section{Intermittent observability as an extension of observability: Main Intuition}
\label{sec:intui}
To reach this goal, we borrow insights from a characterization of observability. $(\mathbf{A},\mathbf{C})$ is observable if and only if for all $s \in \mathbb{C}$
\begin{align}
\begin{bmatrix}
s \mathbf{I} - \mathbf{A} \\
\mathbf{C}
\end{bmatrix} \mbox{ is full rank.}\nonumber
\end{align}
Moreover, by a similarity transform~\cite{Chen} we can assume that $\mathbf{A}$ is in Jordan form\footnote{Throughout the paper, we will use the Jordan form that induces an upper triangular matrix.} without loss of generality. With this additional assumption, the observability condition can be further simplified.
\begin{theorem}[\cite{Chen}]
Consider a linear system with system matrices $(\mathbf{A},\mathbf{C})$ where $\mathbf{A}$ is given in a Jordan form. For an eigenvalue $\lambda$ of $\mathbf{A}$, denote $\mathbf{C}_\lambda$ as a matrix whose columns are consist of the columns of $\mathbf{C}$ which correspond to the first elements of the Jordan blocks in $\mathbf{A}$ associated with $\lambda$. Then, the states associated with $\lambda$ are observable if and only if the rank of $\mathbf{C}_\lambda$ is equal to the number of Jordan blocks associated with $\lambda$. The whole system is observable if and only if all states associated with all eigenvalues are observable.
\label{thm:jordanob}
\end{theorem}

For example, let
\begin{align}
&\mathbf{A}=
\begin{bmatrix}
2 & 1 & 0 & 0 \\
0 & 2 & 0 & 0 \\
0 & 0 & 2 & 0 \\
0 & 0 & 0 & 3
\end{bmatrix}\nonumber \\
&\mathbf{C}=
\begin{bmatrix}
\mathbf{c_1} & \mathbf{c_2} & \mathbf{c_3} & \mathbf{c_4}
\end{bmatrix}.\nonumber
\end{align}
Then, $\mathbf{C}_2 = \begin{bmatrix} \mathbf{c_1} & \mathbf{c_3} \end{bmatrix}$ and $\mathbf{C}_3= \begin{bmatrix} \mathbf{c_4} \end{bmatrix}$. The eigenvalue $2$ is observable if and only if $\mathbf{C}_2$ is full rank, and the eigenvalue $3$ is observable if and only if $\mathbf{C}_3$ is full rank. The whole system with $(\mathbf{A},\mathbf{C})$ is observable if and only if both eigenvalues are observable.

This characterization reminds us of a \textit{divide-and-conquer} approach. First, divide the observability problem into smaller problems according to the identical eigenvalues. Then, check whether the smaller sub-problem for each eigenvalue is observable. Finally, the whole system is observable if and only if all the sub-problems are observable.

This suggests applying a divide-and-conquer approach for the characterization of intermittent observability. However, before we apply a divide-and-conquer approach, we first have to answer the following three questions:\\
(a) What are the minimal irreducible sub-problems?\\
(b) How can we solve each sub-problem?\\
(c) How can we combine the answers of the sub-problems?

We will make an exact characterization of intermittent observability by resolving these questions. The concept of eigenvalue cycles appears naturally as the answer of the question (a).

Before we answer these questions, let's first start from the simplest case, scalar plants. For simplicity, we will only give hand-waving arguments in this section, and the rigorous justification will be shown in later sections. The basic idea for the characterization of the intermittent observability is to consider the dynamics reverse in time. For example, consider the following scalar system: for $n \in \mathbb{Z}^+$,
\begin{align}
\left\{
\begin{array}{l}
x[n+1]=2 x[n] + w[n] \\
y[n]=\beta[n] x[n]
\end{array}\right. . \label{eqn:intui:1}
\end{align}
Here, $x[0]=0$, $w[n]$ are i.i.d.~zero-mean unit-variance Gaussian, and $\beta[n]$ is an independent Bernoulli process with probability $1-p_e$. Then, we will show that the critical erasure probability $p_e^\star=\frac{1}{2^2}$.

First, we extend the one-sided random process \eqref{eqn:intui:1} to the two-sided process. Let $w[n]=0$ for $n \in \mathbb{Z}^{--}$ where $\mathbb{Z}^{--}$ implies negative integers, and $\beta[n]$ be a two-sided Bernoulli process with probability $1-p_e$. Then, we can see that the new two-sided process is equivalent to the original process except that $x[n]=0, y[n]=0$ for $n \in \mathbb{Z}^{--}$.

Let $n-S$ be the most recent non-erased observation at time $n$, i.e. $S:=\min \{k \geq 0: \beta[n-k]=1 \}$. Since $\beta[n]$ is a two-sided Bernoulli process, the stopping time $S$ is a geometric random variable, i.e. $\mathbb{P}\{S=s \}=(1-p_e){p_e}^s$.

(1) Sufficiency: We first prove that $p_e < \frac{1}{2^2}$ is sufficient for the intermittent observability of the example. For this, we analyze the performance of a suboptimal estimator $\widehat{x}[n]=2^S y[n-S]=2^S x[n-S]$. Then, the estimation error is upper bounded by
\begin{align}
\mathbb{E}[(x[n]-\widehat{x}[n])^2]&=\mathbb{E}[\mathbb{E}[ (x[n]-\widehat{x}[n])^2 |S]]\\
&=\mathbb{E}[\mathbb{E}[(2^S x[n-S]+2^{S-1}w[n-S]+\cdots+w[n-1]-2^S x[n-S])^2|S]]\\
&\leq \mathbb{E}[2^{2(S-1)}+2^{2(S-2)}+\cdots+1]\\
&=\mathbb{E}[\frac{2^{2S}-1}{2^2-1}]\\
&=\frac{1}{2^2-1}\left(\left( \sum^{\infty}_{i=0}(1-p_e)(p_e 2^2)^{i} \right) - 1 \right).
\end{align}
Therefore, the estimation error is uniformly bounded if $p_e < \frac{1}{2^2}$.

(2) Necessity: For necessity, we use the fact that the disturbance $w[n-S]$ is independent of the non-erased observations present up to the time $n$. Therefore, the estimation error is lower bounded by
\begin{align}
\mathbb{E}[(x[n]-\mathbb{E}[x[n]|y^n])^2] & \geq \mathbb{E}[\mathbb{E}[(2^{S-1} w[n-S])^2 | S]] \\
&=\mathbb{E}[2^{2(S-1)}\cdot \mathbf{1}(n-S \geq 0)] \\
&=\frac{1}{2^2}\left( \sum^{n}_{i=0} (1-p_e)(p_e 2^2)^{i}\right)
\end{align}
Therefore, if $p_e \geq \frac{1}{2^2}$ the estimation error must diverge to $\infty$.

(3) Remarks: From the above proof, we can notice that the intermittent observability is decided by whether $p_e 2^2$ is less than $1$. Here, $2$ is the largest eigenvalue of the system, and $p_e$ is the probability mass function (p.m.f.) tail of $S$ which can be defined as $\exp \limsup_{s \rightarrow \infty} \frac{1}{s} \ln \mathbb{P}\{S=s\}$. Thus, we can think of two potential differences between scalar and vector systems: (i) The maximum eigenvalue (ii) The p.m.f. tail.

It turns out the latter is true, and the p.m.f. tail is the difference between scalar and vector systems. The following example shows why and how the p.m.f tail changes in vector systems.

\subsection{Power Property}
\label{sec:powerproperty}
The power property answers question (b) of the previous section, ``How can we solve each sub-problem?".
Consider the example of \cite{Yilin_Characterization}.
\begin{align}
\left\{
\begin{array}{l}
\mathbf{x}[n+1]=\begin{bmatrix}
2 & 0 \\
0 & -2
\end{bmatrix}
\mathbf{x}[n] + \mathbf{w}[n] \\
y[n]=\beta[n] \begin{bmatrix}
1 & 1
\end{bmatrix}
\mathbf{x}[n]
\end{array}
\right. \nonumber
\end{align}
Like above, we put $\mathbf{x}[0]=\mathbf{0}$, $\mathbf{w}[n]$ is $2$-dimensional i.i.d. Gaussian vector with mean $\mathbf{0}$ and variance $\mathbf{I}$, and $\beta[n]$ is an independent Bernoulli process with probability $1-p_e$. We also extend the one-sided process to the two-sided process in the same way.

We can see the states are $2$-dimensional, while the observations are $1$-dimensional. Therefore, unlike scalar systems at least two observations are required to estimate the states. Moreover, if we write the observability Gramian matrix, we immediately notice cyclic behavior:
\begin{align}
&\mathbf{C}=\begin{bmatrix}1 & 1 \end{bmatrix} \nonumber \\
&\mathbf{C}\mathbf{A^{-1}}=\begin{bmatrix} \frac{1}{2} & -\frac{1}{2} \end{bmatrix} \nonumber \\
&\mathbf{C}\mathbf{A^{-2}}=\begin{bmatrix} \frac{1}{4} & \frac{1}{4} \end{bmatrix} \nonumber \\
&\mathbf{C}\mathbf{A^{-3}}=\begin{bmatrix} \frac{1}{8} & -\frac{1}{8} \end{bmatrix} \nonumber \\
&\vdots \nonumber
\end{align}
Notice that
$\mathbf{C},\mathbf{C}\mathbf{A^{-2}},\mathbf{C}\mathbf{A^{-4}},\cdots$
are linearly dependent and
$\mathbf{CA^{-1}},\mathbf{C}\mathbf{A^{-3}},\mathbf{C}\mathbf{A^{-5}},\cdots$
are linearly dependent. Therefore, as observed in \cite{Yilin_Characterization}, we need both even and odd time observations to estimate the states. In this example, we will show that $p_e^\star = \frac{1}{2^4}$.

(1) Sufficiency: Let $p_e < \frac{1}{2^4}$. From \eqref{eqn:dis:system} and \eqref{eqn:dis:system2}, we can see that when $\beta[n-k]=1$ the following equations hold:
\begin{align}
\mathbf{x}[n]&=\mathbf{A}^k \mathbf{x}[n-k] + \mathbf{A}^{k-1}\mathbf{w}[n-k]+ \cdots + \mathbf{w}[n-1] \label{eqn:intui:2}\\
\mathbf{y}[n-k]&=\mathbf{C}\mathbf{x}[n-k] + \mathbf{v}[n-k] \nonumber \\
&=\mathbf{C}\mathbf{A}^{-k} \mathbf{x}[n] - \underbrace{(\mathbf{C} \mathbf{A}^{-1} \mathbf{w}[n-k] + \cdots + \mathbf{C} \mathbf{A}^{-k} \mathbf{w}[n-1] - \mathbf{v}[n-k])}_{:=\mathbf{v'}[n-k]} \label{eqn:intui:3}
\end{align}
Here, we can see the variance of $\mathbf{v'}[n-k]$ is bounded as $\mathbb{E}[|\mathbf{v'}[n-k]|^2]=
\mathbb{E}[(
\begin{bmatrix}\frac{1}{2} & -\frac{1}{2} \end{bmatrix}\mathbf{w}[n-1]+\cdots+\begin{bmatrix}\frac{1}{2^k} & \frac{1}{(-2)^k} \end{bmatrix}\mathbf{w}[n-k])^2] \leq 2 \frac{\frac{1}{4}}{1-\frac{1}{4}}=\frac{2}{3}$.

Now, the stopping time $S$ until we have enough observations to estimate the states becomes the first time until we get both even and odd time observations, i.e. $S:=\inf \{k : 0 \leq k_1 < k_2 \leq k,
\beta[n-k_1]=1, \beta[n-k_2]=1, k_1 \neq k_2 (\bmod 2) \}$. Here, the p.m.f. of $S$ gets thicker than that of scalar cases. We can actually prove that the p.m.f. tail of $S$ is $\exp \limsup_{s\rightarrow \infty} \frac{1}{s} \ln \mathbb{P}\{S=s \}=p_e^{\frac{1}{2}}$, which we will rigorously justify in Lemma~\ref{lem:app:geo}. Thus, we can find $\delta, c>0$ such that $p_e < \frac{1}{2^4}-\delta$ and $\mathbb{P}\{S=s \} \leq c\left( \frac{1}{2^4}-\delta \right)^{\frac{s}{2}}$ for all $s \in \mathbb{Z}^+$.

Now, we will analyze the performance of a suboptimal estimator which only uses two observations. Let $\mathbf{\widehat{x}}[n]:=\begin{bmatrix} \mathbf{C}\mathbf{A}^{-k_1} \\ \mathbf{C}\mathbf{A}^{-k_2} \end{bmatrix}^{-1} \begin{bmatrix} \mathbf{y}[n-k_1] \\ \mathbf{y}[n-k_2] \end{bmatrix}$. Here, we can see the matrix inverse exists since $k_1$ and $k_2$ are even and odd time observations. Let $\mathcal{F}_{\beta}$ be the $\sigma$-field generated by $\beta[n]$. Then, $k_1, k_2, S$ are deterministic variables conditioned on $\mathcal{F}_{\beta}$. The estimation error is upper bounded by
\begin{align}
\mathbb{E}[|\mathbf{x}[n]-\mathbf{\widehat{x}}[n]|_2^2]&=
\mathbb{E}[\mathbb{E}[|\mathbf{x}[n]-\mathbf{\widehat{x}}[n]|_2^2| \mathcal{F}_{\beta}]]=
\mathbb{E}[\mathbb{E}[\left|\begin{bmatrix}\mathbf{C}\mathbf{A}^{-k_1} \\ \mathbf{C}\mathbf{A}^{-k_2}\end{bmatrix}^{-1}
\begin{bmatrix}
\mathbf{v'}[n-k_1]\\
\mathbf{v'}[n-k_2]
\end{bmatrix}
\right|_2^2
|\mathcal{F}_{\beta}]] \nonumber \\
&\leq
\mathbb{E}[\mathbb{E}[ 8 \cdot \left|\begin{bmatrix}\mathbf{C}\mathbf{A}^{-k_1} \\ \mathbf{C}\mathbf{A}^{-k_2}\end{bmatrix}^{-1}
 \right|_{max}^2 \cdot \left|\begin{bmatrix}
\mathbf{v'}[n-k_1]\\
\mathbf{v'}[n-k_2]
\end{bmatrix}
\right|_{max}^2
|\mathcal{F}_{\beta}]] \nonumber \\
&=
8 \cdot \mathbb{E}[ \left|
\begin{bmatrix}
2^{-k_1}& (-2)^{-k_1}\\
2^{-k_2}& (-2)^{-k_2}
\end{bmatrix}^{-1}
 \right|_{max}^2 \cdot
 \mathbb{E}[ \left|\begin{bmatrix}
\mathbf{v'}[n-k_1]\\
\mathbf{v'}[n-k_2]
\end{bmatrix}
\right|_{max}^2
|\mathcal{F}_{\beta}]] \nonumber \\
&=
8 \cdot \mathbb{E}[ \left|
\frac{1}{2 \cdot 2^{-k_1} \cdot (-2)^{-k_2}}
\begin{bmatrix}
(-2)^{-k_2} & -(-2)^{-k_1} \\
-2^{-k_2} & 2^{-k_1} \\
\end{bmatrix}
 \right|_{max}^2 \cdot
 \mathbb{E}[ \left|\begin{bmatrix}
\mathbf{v'}[n-k_1]\\
\mathbf{v'}[n-k_2]
\end{bmatrix}
\right|_{max}^2
|\mathcal{F}_{\beta}]] \nonumber \\
&=
8 \cdot \mathbb{E}[
\frac{1}{2^2} \left(\frac{2^{-k_1}}{2^{-k_1} \cdot 2^{-k_2}}\right)^2 \cdot
 \mathbb{E}[ \left|\begin{bmatrix}
\mathbf{v'}[n-k_1]\\
\mathbf{v'}[n-k_2]
\end{bmatrix}
\right|_{max}^2
|\mathcal{F}_{\beta}]] \nonumber \\
&\leq
2 \cdot \mathbb{E}[
 2^{2 k_2} \cdot
 \mathbb{E}[ |\mathbf{v'}[n-k_1]|^2+|\mathbf{v'}[n-k_2]|^2
|\mathcal{F}_{\beta}]]\nonumber \\
&\leq \frac{8}{3} \mathbb{E}[2^{2 S}]
\leq
\frac{8}{3} \sum^{\infty}_{s=0} 2^{2s} c \left( \frac{1}{2^4}-\delta \right)^{\frac{s}{2}}= \frac{8}{3} \sum^{\infty}_{s=0} c ( 1- 2^4 \delta )^{\frac{s}{2}} < \infty \nonumber
\end{align}
Therefore, the estimation error is uniformly bounded for $p_e < \frac{1}{2^4}$.

(2) Necessity: We will show that the system is not intermittent observable when $p_e \geq \frac{1}{2^4}$.
Denote the stopping time $S'$ to be $\inf\{k \geq 0 : \beta[n-k]=1, \mbox{$k$ is even} \}$. Then, $\mathbb{P}\{S'=0 \}=1-p_e$, $ \mathbb{P}\{S'=1 \}=0$, $\mathbb{P}\{S'=2 \}=(1-p_e)p_e$, $\cdots$. Thus, the p.m.f. tail of $S'$, $\exp \limsup_{s \rightarrow \infty}\frac{1}{s} \ln \mathbb{P}\{S'=s \}$, is $p_e^{\frac{1}{2}}$.

The state disturbance $\mathbf{w}[n-S']$ can be decomposed into two orthogonal components, $\mathbf{w}[n-S']=\begin{bmatrix} 1 \\ 1\end{bmatrix} w_1[n-S']+\begin{bmatrix} 1 \\ -1\end{bmatrix}w_2[n-S']$ where $w_1[n-S']$ and $w_2[n-S']$ are independent Gaussian random variables with zero mean and variance $\frac{1}{2}$. From the system equations \eqref{eqn:intui:2}, \eqref{eqn:intui:3} and the definition of $S'$, we can see that all the observations between time $n-S'$ and $n$ are orthogonal to $w_2[n-S']$. Thus, the estimator does not  know anything about $w_2[n-S']$ at time $n$, and thus we can lower bound the estimation error as follows.
\begin{align}
\mathbb{E}[(\mathbf{x}[n]-\mathbb{E}[\mathbf{x}[n]|\mathbf{y}^n])^2]
&\geq \mathbb{E}[\mathbb{E}[|\mathbf{A}^{S'-1}\begin{bmatrix} 1 \\ -1 \end{bmatrix} w_2[n-S']|_2^2]|S']  \nonumber \\
&\geq \mathbb{E}[2^{2(S'-1)}\mathbb{E}[(w_2[n-S'])^2|S']]= \frac{1}{2^3} \mathbb{E}[2^{2S'} \cdot \mathbf{1}(S' \geq n)] \nonumber \\
&=\frac{1}{2^3} \sum^{\lfloor \frac{n}{2} \rfloor}_{i=0} (1-p_e) ( \sqrt{p_e} 2^2)^{2i} \nonumber
\end{align}
Thus, if $p_e \geq \frac{1}{2^4}$ the estimation error diverges to $\infty$.


(3) Remarks: Compared to the scalar case, the p.m.f. tails of both $S$ and $S'$ in this vector system thicken to $\sqrt{p_e}$. This results in decreasing the critical erasure to $\frac{1}{2^4}$. The cyclic behavior of the observability Gramian matrix, $\mathbf{C}$, $\mathbf{C}\mathbf{A}^{-1}$, $\cdots$, causes the thickening of the p.m.f. tails. Thus, to capture this cyclic behavior of the observability Gramian matrix, we tentatively define an eigenvalue cycle as follows\footnote{We will formally define eigenvalue cycles later in Section~\ref{sec:interob}.}: We say that the eigenvalues of $\mathbf{A}$, $\lambda_1$ and $\lambda_2$ belong to the same \textbf{eigenvalue cycle} if $\frac{\lambda_1}{\lambda_2}$ is a root of unity, i.e. $\left(\frac{\lambda_1}{\lambda_2}\right)^{n}=1$ for some $n\in \mathbb{Z}$. Moreover, we say that $\mathbf{A}$ has \textbf{no eigenvalue cycles} if all the ratios between the eigenvalues of $\mathbf{A}$ are $1$ or not roots of unity, which implies $\mathbf{A}$ has no nontrivial eigenvalue cycles.

To generalize this example and find the p.m.f. tail for arbitrary eigenvalue cycles, we use the idea of large deviations~\cite{Dembo} which is equivalent to a union bound for simple cases. The idea goes as follows.

First, consider test channels that are erasure-type channels which would make the observability gramian rank-deficient. For this example, these would be the channel that erases every odd-time observations, the channel that erases every even-time observations and the channel that erases all observations.\footnote{In the actual characterization shown in Section~\ref{sec:interob}, we will see that the set $S'$ in \eqref{eqn:def:lprime:thm} is a proxy for these test channels. This minimum distance to the test channels will be denoted as $l_i$ in \eqref{eqn:def:lprime:thm}.}

Next, measure the distance from the true channel to the test channels. In our case, the true channel is the channel without any restriction and the distance measure between the true and test channel is the hamming distance. For the test channels considered above, the distance to the odd-time erasure channel is $1$ since we are restricting every one out of two indexes to be erasure. Likewise, the distance to the even-time erasure channel is $1$ and the distance to the all erasure channel is $2$.

Then, the large deviation principle intuitively says that the performance is decided by the minimum-distance test channel. For the example, the odd-time or even-time erasure channel whose distances are $1$ will govern the performance.

So the effect of the eigenvalue cycle is to thicken the tail of the stopping time until you get enough observations to estimate the states. Analytically, the effect is equivalent to taking a proper power to the $p_e$ and hence the name ``power property''.

\subsection{Max Combining}
\label{sec:maxcombining}
This property answers the question (c) i.e. how we go from a single eigenvalue cycle to
multiple eigenvalue cycles. Consider the following example with two eigenvalue cycles:
\begin{align}
\left\{
\begin{array}{l}
\begin{bmatrix}
x_1[n+1] \\
x_2[n+1] \\
x_3[n+1]
\end{bmatrix}=\begin{bmatrix}
3 & 0 & 0 \\
0 & 2 & 0 \\
0 & 0 & -2 \\
\end{bmatrix}
\begin{bmatrix}
x_1[n] \\
x_2[n] \\
x_3[n]
\end{bmatrix} + \mathbf{w}[n] \\
y[n]=\beta[n] \begin{bmatrix}
1 & 1 & 1
\end{bmatrix}
\mathbf{x}[n]
\end{array}
\right.
\end{align}
As before, we put $\mathbf{x}[0]=\mathbf{0}$, $\mathbf{w}[n]$ be i.i.d. Gaussian with mean $\mathbf{0}$ and variance $\mathbf{I}$, and $\beta[n]$ be an independent Bernoulli process with probability $1-p_e$. We also extend the one-sided process to the two-sided process. Here, we can see there are two eigenvalue cycles. One eigenvalue cycle is $\{ 2,-2\}$ and the other one is $\{3\}$, which can be thought as two subsystems of the original system.

Then, from the previous arguments, we can see that the p.m.f. tails for these two systems are different. The p.m.f. tail for the eigenvalue cycle $\{3\}$ is $p_e$, while the p.m.f. tail for the eigenvalue cycle $\{2,-2 \}$ is thickened to $p_e^{\frac{1}{2}}$. Therefore, the question is whether the thickened tail in the eigenvalue cycle $\{2,-2 \}$ affects $\{ 3 \}$. The answer turns out to be ``No", and we can consider two subsystems separately. Thus, in this example, the system is intermittent observable if and only if both subsystems are intermittent observable, i.e. $p_e^\star = \frac{1}{\max\{3^2, 2^{2 \cdot 2}\}}$. The main idea to justify this is so-called \textit{successive decoding} developed in information theory~\cite{Cover}.

(1) Sufficiency: We will prove that $p_e < \frac{1}{\max\{ 3^2, 2^{2 \cdot 2} \}}$ is sufficient for the intermittent observability using a successive decoding idea. The idea is simple. We first estimate the state $x_1[n]$. Then, since we have an estimate for $x_1[n]$, we can subtract the estimate from the system and reduce the dimension of the system. The remaining estimation error is considered as noise.

Let $S$ be the stopping time until we receive three observations in the reverse process, i.e. $S:=\inf \{k : 0 \leq k_1 < k_2 < k_3 \leq k, \beta[n-k_1]=1, \beta[n-k_2]=1, \beta[n-k_3]=1 \}$. Here, we can prove that the p.m.f. tail of $S$ is the same as the scalar case. Therefore, $\exp\limsup_{s\rightarrow \infty} \ln \mathbb{P}\{S=s \}=p_e$, which we will justify in Lemma~\ref{lem:app:geo}. Since we have the three observations at time $n-k_1$, $n-k_2$ and $n-k_3$, by the pigeon-hole principle at least two among them have to be congruent mod $2$. Assume that $k_1$ and $k_2$ are both even. Then, by \eqref{eqn:intui:3} we have
\begin{align}
y[n-k_1]&=\begin{bmatrix} 1 & 1 & 1 \end{bmatrix} \begin{bmatrix} 3 & 0 & 0 \\ 0 & 2 & 0 \\ 0 & 0 & -2 \end{bmatrix}^{-k_1} \begin{bmatrix} x_1[n] \\ x_2[n] \\ x_3[n] \end{bmatrix} + v'[n-k_1]\\
&=\begin{bmatrix} 1 & 1 & 1 \end{bmatrix} \begin{bmatrix} 3^{-k_1} & 0 & 0 \\ 0 & 2^{-k_1} & 0 \\ 0 & 0 & 2^{-k_1} \end{bmatrix} \begin{bmatrix} x_1[n] \\ x_2[n] \\ x_3[n] \end{bmatrix} + v'[n-k_1] \\
&=\begin{bmatrix} 1 & 1 \end{bmatrix}  \begin{bmatrix} 3 & 0 \\ 0 & 2 \end{bmatrix}^{-k_1} \begin{bmatrix} x_1[n] \\ x_2[n]+x_3[n] \end{bmatrix}+ v'[n-k_1]
\end{align}
Like the above section, we can also prove that
$\mathbb{E}[|v'[n-k]|^2] \leq 2 \frac{\frac{1}{4}}{1-\frac{1}{4}}+\frac{\frac{1}{9}}{1-\frac{1}{9}}=\frac{19}{24}$.
Here, we can notice that instantaneously at time $n-k_1$ and $n-k_2$ the system equation behaves like the following system with no eigenvalue cycles:
\begin{align}
\left\{
\begin{array}{l}
\begin{bmatrix}
x_1[n+1] \\
x_2[n+1]+x_3[n+1]
\end{bmatrix}=\begin{bmatrix}
3 & 0\\
0 & 2\\
\end{bmatrix}
\begin{bmatrix}
x_1[n] \\
x_2[n]+x_3[n]
\end{bmatrix} +
\begin{bmatrix}
w_1[n] \\
w_2[n]+w_3[n]
\end{bmatrix}
\\
y[n]=\beta[n] \begin{bmatrix}
1 & 1 \end{bmatrix}
\begin{bmatrix}
x_1[n] \\
x_2[n]+x_3[n]
\end{bmatrix}
\end{array}
\right. \nonumber
\end{align}
Consider the suboptimal estimator $
\mathbf{\widehat{x}}[n]=
\begin{bmatrix}
\widehat{x}_1[n]\\
\widehat{x}_2[n]+\widehat{x}_3[n]
\end{bmatrix}=
\begin{bmatrix}
3^{-k_1} & 2^{-k_1} \\
3^{-k_2} & 2^{-k_2} \\
\end{bmatrix}^{-1}
\begin{bmatrix}
y[n-k_1] \\
y[n-k_2]
\end{bmatrix}
$. Let $\mathcal{F}_{\beta}$ be the $\sigma$-field generated by $\beta[n]$, and $F$ be the event that $k_1$ and $k_2$ are even. The estimation error is upper bounded by
\begin{align}
&\mathbb{E}[\left|\begin{bmatrix} x_1[n] \\ x_2[n]+x_3[n] \end{bmatrix} - \mathbf{\widehat{x}}[n] \right|_2^2 | \mathcal{F}_{\beta} \cap F ]
=\mathbb{E}[\left|
\begin{bmatrix}
3^{-k_1} & 2^{-k_1} \\
3^{-k_2} & 2^{-k_2} \\
\end{bmatrix}^{-1}
\begin{bmatrix}
v'[n-k_1] \\
v'[n-k_2]
\end{bmatrix}
\right|_2^2 | \mathcal{F}_{\beta} \cap F ] \\
&\leq
8 \cdot \left|
\begin{bmatrix}
3^{-k_1} & 2^{-k_1} \\
3^{-k_2} & 2^{-k_2} \\
\end{bmatrix}^{-1}\right|_{max}^2
\cdot \mathbb{E}[\left|\begin{bmatrix}
v'[n-k_1] \\
v'[n-k_2]
\end{bmatrix}
\right|_{max}^2 | \mathcal{F}_{\beta} \cap F ] \\
\\
&=
8 \cdot \frac{19}{12}\cdot \left| \frac{1}{3^{-k_1}2^{-k_2}-2^{-k_1}3^{-k_2}}
\begin{bmatrix}
2^{-k_2} & -2^{-k_1} \\
-3^{-k_2} & 3^{-k_1} \\
\end{bmatrix}\right|_{max}^2 \\
&=8 \cdot  \frac{19}{12} \cdot
\left(
\frac{2^{-k_1}}
{3^{-k_1} 2^{-k_2}\left(1- \left( \frac{2}{3} \right)^{k_2-k_1} \right)}
\right)^2 \\
&\leq 8 \cdot \frac{19}{12} \cdot 3^2  \cdot (3^{k_1} \cdot 2^{k_2-k_1})^2  \leq 57 \cdot 3^{2 k_2}\leq 57 \cdot 3^{2 S}
\end{align}
Likewise, we can prove the same bound holds even if $k_1$ and $k_2$ are not even. Therefore, the estimation error is bounded by $57 \cdot 3^{2 S}$. Like the previous section, we can prove that if $p_e < \frac{1}{3^2}$ then $\mathbb{E}[ 3^{2 S}]<\infty$. Thus, the expectation of the estimation error for $x_1[n]$ is uniformly bounded over time.

Once we estimate $x_3[n]$, we can subtract the estimation $\widehat{x}_3[n]$ from the observation, i.e. $y'[n]:=y[n]-\beta[n]\widehat{x}_1[n]$. Then, the new system with the observation $y'[n]$ behaves like the following system:
\begin{align}
\left\{
\begin{array}{l}
\begin{bmatrix}
x_2[n+1] \\
x_3[n+1]
\end{bmatrix}=\begin{bmatrix}
 2 & 0 \\
 0 & -2 \\
\end{bmatrix}
\begin{bmatrix}
x_2[n] \\
x_3[n]
\end{bmatrix} + \mathbf{w}[n] \\
y'[n]=\beta[n]\left( \begin{bmatrix}
1 & 1
\end{bmatrix}
\begin{bmatrix}
x_2[n] \\
x_3[n]
\end{bmatrix}
+(x_1[n]-\widehat{x}_1[n])
\right)
\end{array}
\right.
\end{align}
Since the expectation of the estimation error for $x_1[n]$ is uniformly bounded, it can be considered as a part of the observation noise.\footnote{Precisely speaking, the estimation error for $x_1[n]$ is a random variable which depends on the channel erasure process. Therefore, the rigorous proof of Section~\ref{sec:dis:suff} has more steps to justify the argument.} In the same way as the previous section, we can prove that the estimation error for $x_2[n], x_3[n]$ is uniformly bounded if $p_e < \frac{1}{2^{2 \cdot 2}}$. Notice that the minimum number of required information to estimate the state by observability gramian matrix inversion is $3$, the number of the states. However, here we used more number of observation to apply successive decoding idea.

(2) Necessity: To prove that the example is not intermittent observable if $p_e \geq \frac{1}{\max\{ 3^2, 2^{2 \cdot 2} \}}$, we will use a genie argument. If the states $x_2[n],x_3[n]$ is given to the estimator as a side-information, the remaining system with $x_1[n]$ is a scalar system with the eigenvalue $3$. We know that if $p_e \geq \frac{1}{3^2}$, $x_1[n]$ is not intermittent observable. We can also give $x_1[n]$ as a side-information to conclude that $p_e \geq \frac{1}{2^{2\cdot 2}}$ is a necessary condition for the intermittent observability.

(3) Remarks: In summary, we can solve problems with multiple eigenvalue cycles one by one without worrying about the existence of the other eigenvalue cycles. In other words, at each step we estimate the eigenvalue cycle associated with the largest eigenvalue. After the estimation, the eigenvalue cycle can be subtracted from the system except uniformly bounded estimation error. Then, we can simply repeat the steps for the remaining system. This procedure of successively solving and subtracting the unknowns is called successive decoding in information theory, and used as a decoding procedure for the multiple-access channel~\cite{Cover}.

Therefore, we can conclude that the intermittent observability for a multiple eigenvalue-cycle system is bottlenecked by the hardest-to-estimate eigenvalue cycle, which manifests as the max operation in the critical erasure probability calculation.

\subsection{Separability of Eigenvalue Cycles}
\label{sec:separability}
The remaining question is what are the minimal irreducible sub-problems, whose answer can be expected to be eigenvalue cycles from the discussion up to now. In other words, we will understand general systems with multiple eigenvalue cycles by dividing into sub-systems with a single eigenvalue cycle. In the max-combining property, we already saw an example with multiple eigenvalue cycles. In the example, we first reduce the problem with multiple eigenvalue cycles to the problem with no eigenvalue cycles by sub-sampling plants. For example, in Section~\ref{sec:maxcombining} we already saw that by sub-sampling (by $2$), the system with an eigenvalue cycle (period $2$) becomes a system with no eigenvalue cycles.

Thus, the question reduces to the fact that for systems with no eigenvalue cycles the critical erasure probability is $\frac{1}{|\lambda_{max}|^2}$, which will be shown in Corollary~\ref{thm:nocycle}. To intuitively understand why this is true, we will consider three cases depending on the structure of $\mathbf{A}$.

The first case is when $\mathbf{A}$ is a diagonal matrix, and the magnitudes of its eigenvalues are distinct. In fact, this case is already proved in \cite{Yilin_Characterization}. Let's consider a descriptive example when $\mathbf{A}=\begin{bmatrix} 3 & 0 \\ 0 & 2\end{bmatrix}$, $\mathbf{C}=\begin{bmatrix} 1 & 1 \end{bmatrix}$. Then, the observability gramian of the system becomes $\begin{bmatrix} \mathbf{C}\mathbf{A}^{n_1} \\ \mathbf{C}\mathbf{A}^{n_2} \end{bmatrix} = \begin{bmatrix} 3^{n_1} & 2^{n_1} \\ 3^{n_2} & 2^{n_2} \end{bmatrix}$. To prove that the critical erasure probability is given as $\frac{1}{|\lambda_{max}|^2} = \frac{1}{3^2}$, it is enough to prove that the determinant of the observability gramian is large enough for almost all distinct $n_1$ and $n_2$. To justify this, we can use the fact that the ratio of the elements, $(\frac{3}{2})^n$,  is an exponentially increasing function.

The second case is when $\mathbf{A}$ is a diagonal matrix, and the eigenvalues are distinct but have the same magnitude. Let's consider the system with $\mathbf{A}=\begin{bmatrix} e^{j} & 0 \\ 0 &
  e^{j\sqrt{2}} \end{bmatrix}$ and $\mathbf{C}=\begin{bmatrix} 1 & 1 \end{bmatrix}$. The observability gramian is given as $\begin{bmatrix} \mathbf{C}\mathbf{A}^{n_1} \\ \mathbf{C}\mathbf{A}^{n_2} \end{bmatrix} = \begin{bmatrix} e^{j n_1} & e^{j\sqrt{2}n_1} \\ e^{j n_2} & e^{j\sqrt{2}n_2} \end{bmatrix}$, and like above it is enough to show that the determinant of this observability gramian is large enough for almost all distinct $n_1$, $n_2$. Here, the arguments from \cite{Yilin_Characterization} cannot work. For this, we instead used Weyl's criterion~\cite{Kuipers} which tells each element $(e^{jn}, e^{j \sqrt{2}n})$ behaves like a random variable $(e^{j \theta_1}, e^{j \theta_2})$ where $\theta_1$ and $\theta_2$ are independent random variables uniformly distributed on $[0, 2\pi]$. In fact, the effect of the hypothetical random variables $(e^{j \theta_1}, e^{j \theta_2})$ is quite similar to the actually randomly-dithered nonuniform sampling discussed in Section~\ref{sec:nonuniform}.

The last case is when $\mathbf{A}$ is a Jordan block matrix. Let's consider the system with $\mathbf{A}=\begin{bmatrix} 2 & 1 \\ 0 &2 \end{bmatrix}$ and $\mathbf{C}=\begin{bmatrix} 1 & 0 \end{bmatrix}$. The observability gramian is given as $\begin{bmatrix} \mathbf{C}\mathbf{A}^{n_1} \\ \mathbf{C}\mathbf{A}^{n_2} \end{bmatrix} = \begin{bmatrix} 2^{n_1} & n_1 2^{n_1} \\
 2^{n_2} & n_2 2^{n_2}  \end{bmatrix}$, and we have to show that the determinant of this observability gramian is large enough for almost all distinct $n_1$, $n_2$. Unlike the above cases, this example has polynomial terms in $n_1$, $n_2$. Exploiting this fact, we can reduce the problem to the fact that a polynomial function on $n$ becomes zero only on a measure zero set.

By combining the insights from these three examples, we can prove that for a general matrix $A$ with no eigenvalue cycles, the critical erasure probability is given as $\frac{1}{|\lambda_{max}|^2}$.

\section{Intermittent Observability Characterization}
\label{sec:interob}
Based on the intuition of the previous section, the intermittent observability condition can be characterized. We begin with the formal definition of a cycle.
\begin{definition}
A multiset (a set that allows repetitions of its elements)
$\{a_1,a_2,\cdots, a_l \}$ is called a cycle with length $l$ and
period $p$ if $\left(\frac{a_i}{a_j}\right)^p=1$ for all $i,j \in \{
1,2,\cdots, l \}$ and some $p \in \mathbb{N}$. Following convention, $p$ is denoted\footnote{We use $\frac{0}{0}=1$, $\frac{1}{0}= \infty$, $1^{\infty}=\infty$ and $\frac{1}{\infty}=0$.} as
\begin{align}
p:=\min \left\{n \in \mathbb{N} :
  \left(\frac{a_i}{a_j}\right)^n =1, \forall i,j \in
  \left\{1,2,\cdots, l \right\} \right\}.
\end{align}
\end{definition}

For example, $\{a\}$ is a cycle with length $1$ and period $1$ by
itself. $\{ e^{j\omega}, e^{j (\omega+ \frac{2 \pi }{6})} \}$ is a
cycle with length $2$ and period $6$. $\{e^{j}, e^{j\sqrt{2}} \}$
and $\{1,2 \}$ are not cycles. One trivially necessary condition for
$a_1,a_2$ to belong to the same cycle is $|a_1|=|a_2|$. It can be also shown
that cycles are closed under overlapping unions, meaning that if $\{
a_1, a_2 \}$ and $\{a_2 ,a_3 \}$ are cycles, $\{a_1, a_2, a_3 \}$ is
also a cycle.

Now, we can define an eigenvalue cycle. It is well-known in linear system theory~\cite{Chen} that by properly changing coordinates, any linear system equations~\eqref{eqn:dis:system} can be written in an equivalent form with a Jordan matrix $\mathbf{A}$. Moreover, even though the MMSE value can be changed by the coordinate change, the condition for the boundedness (stabilizability) remains the same. Rigorously, for any system matrix $\mathbf{A}$, there exists an invertible matrix $\mathbf{U}$ and an upper-triangular Jordan matrix $\mathbf{A'}$ such that $\mathbf{A}=\mathbf{U}\mathbf{A'}\mathbf{U}^{-1}$.  We also define $\mathbf{B'}:=\mathbf{U}\mathbf{B}$ and $\mathbf{C'}:=\mathbf{C}\mathbf{U}$. Then, the matrix $\mathbf{A'}$ and $\mathbf{C'}$ can be written as the following form:
\begin{align}
&\mathbf{A'}=diag\{ \mathbf{A_{1,1}}, \mathbf{A_{1,2}}, \cdots, \mathbf{A_{\mu,\nu_\mu}}\} \nonumber \\
&\mathbf{C'}=\begin{bmatrix} \mathbf{C_{1,1}} & \mathbf{C_{1,2}} & \cdots & \mathbf{C_{\mu,\nu_\mu}} \end{bmatrix} \nonumber \\
&\mbox{where} \nonumber \\
&\quad \mbox{$\mathbf{A_{i,j}}$ is a Jordan block with an eigenvalue $\lambda_{i,j}$} \nonumber \\
&\quad \{ \lambda_{i,1},\cdots, \lambda_{i,\nu_i} \} \mbox{ is a cycle with length $\nu_i$ and period $p_i$}\nonumber \\
&\quad \mbox{For $i \neq i'$, $\{\lambda_{i,j},\lambda_{i',j'} \}$ is not a cycle} \nonumber \\
&\quad \mbox{$\mathbf{C_{i,j}}$ is a $l \times \dim \mathbf{A_{i,j}}$ complex matrix}.\label{eqn:ac:jordan:thm}
\end{align}
Since cycles are closed under overlapping unions, the eigenvalues of
$\mathbf{A}$ can be uniquely partitioned into maximal cycles, $\{\lambda_{i,1},\cdots,\lambda_{i,\nu_i} \}$. We call
these cycles \textit{eigenvalue cycles} and we say $\mathbf{A}$ has no
eigenvalue cycle if all of its eigenvalue cycles are period $1$.

Define
\begin{align}
&\mathbf{A_i}=diag\{ \lambda_{i,1},\cdots, \lambda_{i,\nu_i} \}\nonumber \\
&\mathbf{C_i}=\begin{bmatrix} \left(\mathbf{C_{i,1}}\right)_1 & \cdots & \left(\mathbf{C_{i,\nu_i}}\right)_1 \end{bmatrix} \nonumber\\
&\mbox{where $\left(\mathbf{C_{i,j}}\right)_1$ is the first column of $\mathbf{C_{i,j}}$.} \label{eqn:ac2:jordan:thm}
\end{align}
In other words, we are dividing the original problem to sub-problems according to eigenvalue cycles.

Let $l_i$ be the minimum cardinality among the sets $S' \subseteq \{ 0,1,\cdots,p_i-1 \}$ whose resulting $S:=\{ 0,1,\cdots, p_i-1 \} \setminus S'=\{s_1,s_2,\cdots,s_{|S|} \}$ makes
\begin{align}
\begin{bmatrix}
\mathbf{C_i}\mathbf{A_i}^{s_1}\\
\mathbf{C_i}\mathbf{A_i}^{s_2}\\
\vdots \\
\mathbf{C_i}\mathbf{A_i}^{s_{|S|}}
\end{bmatrix} \label{eqn:def:lprime:thm}
\end{align}
be rank deficient, i.e. the rank is strictly less than $\nu_i$. Here, $p_i$ and $l_i$ will be used for the power property. $l_i$ represents how many observations have to be erased out of $p_i$ time steps to make the observability Gramian matrix rank deficient. This corresponds to the critical error event in large deviation theory.

Now, we can apply the max-combination property to characterize intermittent observability. Here is the main theorem of the paper.
\begin{theorem}
Given an intermittent system $(\mathbf{A},\mathbf{B},\mathbf{C}, \sigma, \sigma')$ with probability of erasure $p_e$, let $\sigma < \infty$, $\sigma' > 0$, and $(\mathbf{A},\mathbf{B})$ be controllable. Then, the intermittent system is intermittent observable if and only if
\begin{align}
p_e < \frac{1}{\underset{1 \leq i \leq \mu}{\max} |\lambda_{i,1}|^{2 \frac{p_i}{l_i}}}
. \nonumber
\end{align}
or equivalently $\underset{1 \leq i \leq \mu}{\max} p_e^{\frac{l_i}{p_i}} |\lambda_{i,1}|^2 < 1$.
\label{thm:mainsingle}
\end{theorem}
\begin{proof}
See Section~\ref{sec:dis:suff} for sufficiency, and Section~\ref{sec:dis:nece} for necessity.
\end{proof}

Here, we can notice that there is no assumption about stability or observability of the system. Let's first do a validity test of the theorem by trying stable modes and unobservable modes.
If $|\lambda_{i,1}|<1$, $\frac{1}{|\lambda_{i,1}|^{2 \frac{p_i}{l_i}}}>1$. Therefore, the stable modes do not contribute to the characterization of the critical erasure probability.
If $(\mathbf{A_i},\mathbf{C_i})$ are unobservable, $l_i=0$. So, $\frac{1}{|\lambda_{i,1}|^{2 \frac{p_i}{0}}}=0$ if $|\lambda_{i,1}| \geq 1$ and $\frac{1}{|\lambda_{i,1}|^{2 \frac{p_i}{0}}}=\infty$ if $|\lambda_{i,1}|< 1$.
Therefore, if the unobservable modes are stable they do not affect the intermittent observability of the system and if they are not the system is not intermittent observable even if $p_e=0$.\\

Even though in general $l_i$ does not admit a closed form, it is computable for special cases.
\begin{corollary}
Given an intermittent system $(\mathbf{A},\mathbf{B},\mathbf{C}, \sigma, \sigma')$ with probability of erasure $p_e$, let $\sigma < \infty$, $\sigma' > 0$, and $(\mathbf{A},\mathbf{B})$ be controllable.
We further assume that $(\mathbf{A},\mathbf{C})$ is observable and $\mathbf{A}$ has no eigenvalue cycles (i.e. $\left(\frac{\lambda_i}{\lambda_j}\right)^n \neq 1 $ for all $\lambda_i \neq \lambda_j$ and $n \in \mathbb{N}$). Then, the intermittent system is intermittent observable if and only if $p_e < \frac{1}{|\lambda_{max}|^2}$ where $\lambda_{max}$ is the largest magnitude eigenvalue of $\mathbf{A}$.
\label{thm:nocycle}
\end{corollary}
\begin{proof}
Since $\mathbf{A}$ has no eigenvalue cycles, $p_i$ equal to $1$ for all $i$ and $\mathbf{A_i}$ are scalars. Moreover, by the observability condition and Theorem~\ref{thm:jordanob}, $\mathbf{C_i}$ is full-rank.
Thus, $l_i=1$ for all $i$ and by Theorem~\ref{thm:mainsingle} the critical erasure probability is $\frac{1}{\max_{i} |\lambda_{i,1}|^2}=\frac{1}{|\lambda_{max}|^2}$.
\end{proof}

For a more precise understanding of the critical erasure
probability, we will focus on the case of a row vector $\mathbf{C}$
--- i.e. single-output systems. Heuristically, a row vector
$\mathbf{C}$ is the worst among $\mathbf{C}$ matrices since a vector
observation is clearly better than a scalar observation.

Furthermore, we will also restrict the periods of the all eigenvalue
cycles of $\mathbf{A}$ to be primes\footnote{For convenience, we
  include $1$ as a prime number here.}. The technical reason for this
restriction is that prime periods give us a useful invariance property
of the sub-eigenvalue cycles. Let $\{
\lambda_1,\lambda_2,\cdots,\lambda_l \}$ be an eigenvalue cycle with
prime period $p$. Then, all subsets of $\{
\lambda_1,\lambda_2,\cdots, \lambda_l \}$ with distinct elements are eigenvalue cycles with the same period $p$. This invariance property
need not hold for eigenvalue cycles with composite periods as we will see by example later.
\begin{corollary}
Given an intermittent system $(\mathbf{A},\mathbf{B},\mathbf{C}, \sigma, \sigma')$ with probability of erasure $p_e$, let $\sigma < \infty$, $\sigma' > 0$, and $(\mathbf{A},\mathbf{B})$ be controllable.
We further assume that $(\mathbf{A},\mathbf{C})$ is observable, $\mathbf{C}$ is a row vector, and $\mathbf{A}$ has only prime-period eigenvalue cycles of length $\nu_i$. Then, the intermittent system is intermittent observable if and only if $p_e <  \frac{1}{ \underset{1 \leq i \leq \mu}{\max}
  |\lambda_{i,1}|^{ \frac{2 p_i}{p_i-\nu_i+1}}}$.
\label{thm:cycle}
\end{corollary}
\begin{proof}
First, we introduce the following fact regarding Vandermonde matrix determinants~\cite{Evans_Generalized}: Let $p$ be a prime, $a_1,\cdots, a_n$ be pairwise incongruent in mod $p$ and $b_1,\cdots,b_n$ be pairwise incongruent in mod $p$. Then,
\begin{align}
\begin{bmatrix}
e^{j 2 \pi \frac{a_1 b_1}{p}} & e^{j 2 \pi \frac{a_1 b_2}{p}} & \cdots & e^{j 2 \pi \frac{a_1 b_n}{p}} \\
e^{j 2 \pi \frac{a_2 b_1}{p}} & e^{j 2 \pi \frac{a_2 b_2}{p}} & \cdots & e^{j 2 \pi \frac{a_2 b_n}{p}} \\
\vdots & \vdots & \ddots & \vdots \\
e^{j 2 \pi \frac{a_n b_1}{p}} & e^{j 2 \pi \frac{a_n b_2}{p}} & \cdots & e^{j 2 \pi \frac{a_n b_n}{p}} \\
\end{bmatrix}\nonumber
\end{align}
is full rank.

Furthermore, since $(\mathbf{A},\mathbf{C})$ is observable and $\mathbf{C}$ is a row vector, by Theorem~\ref{thm:jordanob}, $\lambda_{i,j}$ are distinct and $(\mathbf{C_{i,j}})_1$ are not zeros. Therefore, let $\lambda_{i,j}=|\lambda_i| e^{j 2 \pi \frac{q_{i,j}}{p_i}}$ where $q_{i,1}, \cdots, q_{i,\nu_i}$ are incongruent in mod $p_i$ and $p_i$ are primes.

Now, we will evaluate the critical erasure probability shown in Theorem~\ref{thm:mainsingle}. For this system, \eqref{eqn:def:lprime:thm} can be written as
\begin{align}
\begin{bmatrix}
\mathbf{C_i}\mathbf{A_i}^{s_1}\\
\vdots \\
\mathbf{C_i}\mathbf{A_i}^{s_{|S|}}\\
\end{bmatrix}
&=
\begin{bmatrix}
\lambda_{i,1}^{s_1} & \cdots & \lambda_{i,\nu_i}^{s_1} \\
\vdots & \ddots & \vdots \\
\lambda_{i,1}^{s_{|S|}} & \cdots & \lambda_{i,\nu_i}^{s_{|S|}}
\end{bmatrix}
\begin{bmatrix}
(\mathbf{C_{i,1}})_{1} & \cdots & 0 \\
\vdots & \ddots & \vdots \\
0 & \cdots & (\mathbf{C_{i,\nu_i}})_{1}
\end{bmatrix} \nonumber \\
&=
\begin{bmatrix}
|\lambda_i|^{s_1} & \cdots & 0 \\
\vdots & \ddots & \vdots \\
0 & \cdots & |\lambda_i|^{s_{|S|}}
\end{bmatrix}
\begin{bmatrix}
e^{j 2 \pi \frac{q_{i,1}}{p_i} s_1} & \cdots & e^{j 2 \pi \frac{q_{i,\nu_i}}{p_i} s_1} \\
\vdots & \ddots & \vdots \\
e^{j 2 \pi \frac{q_{i,1}}{p_i} s_{|S|}} & \cdots & e^{j 2 \pi \frac{q_{i,\nu_i}}{p_i} s_{|S|}}
\end{bmatrix}
\begin{bmatrix}
(\mathbf{C_{i,1}})_{1} & \cdots & 0 \\
\vdots & \ddots & \vdots \\
0 & \cdots & (\mathbf{C_{i,\nu_i}})_{1}
\end{bmatrix} \nonumber
\end{align}
Since $\lambda_i$ and $\mathbf{C_{i,j}}_1$ are non-zeros, the rank of $\begin{bmatrix} \mathbf{C_i}\mathbf{A_i}^{s_1}\\
\vdots \\
\mathbf{C_i}\mathbf{A_i}^{s_{|S|}} \end{bmatrix}$ is equal to the rank of $\begin{bmatrix}
e^{j 2 \pi \frac{q_{i,1}}{p_i} s_1} & \cdots & e^{j 2 \pi \frac{q_{i,\nu_i}}{p_i} s_1} \\
\vdots & \ddots & \vdots \\
e^{j 2 \pi \frac{q_{i,1}}{p_i} s_{|S|}} & \cdots & e^{j 2 \pi \frac{q_{i,\nu_i}}{p_i} s_{|S|}}
\end{bmatrix}$.

Furthermore, since $q_{i,1}, \cdots, q_{i,\nu_1}$ are incongruent in mod $p_i$ and $s_1, \cdots, s_{|S|}$ are also incongruent in mod $p_i$, by the property of the Vandermonde matrix discussed above, the rank of the observability gramian is greater or equal to $\nu_i$ if and only if $|S| \geq \nu_i$.

Therefore, $l_i$ of \eqref{eqn:def:lprime:thm} is $p_i-\nu_i+1$, and the corollary follows from Theorem~\ref{thm:mainsingle}.
\end{proof}

 One may wonder why we could not get a simple answer in Theorem~\ref{thm:mainsingle} unlike Corollary~\ref{thm:cycle}. To understand this, consider two potential extensions of Corollary~\ref{thm:cycle}:

(1) Eigenvalue cycles with periods that are composite numbers:
Consider $\mathbf{A}=\begin{bmatrix} 2 & 0 & 0 \\ 0 & 2 e^{j \frac{2
      \pi}{16}} & 0 \\ 0 & 0 & 2 e^{j\frac{2
      \pi}{16}9} \end{bmatrix}$ and $\mathbf{C}=\begin{bmatrix} 1 & 1 & 1 \end{bmatrix}$. The eigenvalue cycle
has length $3$ and period $16$. If we naively apply the formula of
Corollary~\ref{thm:cycle} then we would get a critical value
$\frac{1}{2^{2 \cdot \frac{16}{16-3+1}}}=\frac{1}{2^{\frac{16}{7}}}$. However, if we consider the sub-eigenvalue
cycle $\{ 2e^{j \frac{2 \pi}{16}}, 2e^{j \frac{2 \pi}{16}9}\}$, the
length is $2$ and the period is $2$. The formula of
Corollary~\ref{thm:cycle} gives $\frac{1}{2^{2 \cdot \frac{2}{2-2+1}}}=\frac{1}{2^4}$ as a critical value,
which gives a tighter condition than the previous one. In fact, the latter value is the correct critical erasure probability. Because the period invariant property does not
hold for a composite number cycle, the longest cycle does not necessarily give the
right critical probability.

(2) A general matrix $\mathbf{C}$, multiple-output systems:
If we have a vector observation, an eigenvalue cycle can be
divided into smaller cycles. As an extreme case, when $\mathbf{C}$
is an identity matrix every eigenvalue cycle is divided into trivial
cycles with length $1$ and the critical erasure probability becomes
$\frac{1}{|\lambda_{max}|^2}$ as observed in
\cite{Sinopoli_Kalman}. Consider now $\mathbf{A}=\begin{bmatrix} 2
  &
  0 & 0 & 0 \\ 0 & 2e^{j\frac{2 \pi }{5}} & 0 & 0 \\ 0 & 0 &
  2e^{j\frac{2 \pi }{5}2} & 0 \\ 0 & 0 & 0 & 2e^{j\frac{2 \pi
    }{5}3} \end{bmatrix}$ and
$\mathbf{C}=\begin{bmatrix} 1 & 2 & 3 & 4 \\ 0 & 0 & 0 &
  \delta \end{bmatrix}$. The eigenvalue cycle $\{ 2, 2e^{j\frac{2\pi}{5}},
2e^{j\frac{2\pi}{5}2}, 2e^{j\frac{2\pi}{5}3} \}$ of $\mathbf{A}$ has
length $4$ and period $5$.
However, if $\delta \neq 0$, by elementary row operations $\mathbf{C}$
can be converted to $\begin{bmatrix} 1 & 2 & 3 & 0 \\ 0 & 0 & 0 & 1 \end{bmatrix}$.
Thus, the eigenvalue cycle is divided into two sub-cycles, $\{ 2, 2e^{\frac{2 \pi}{5}}, 2 e^{\frac{2 \pi}{5}2} \}$ and $\{ 2 e^{\frac{2\pi}{5}3} \}$. The longer cycle with length $3$ would dominate and
the critical erasure probability would be $\frac{1}{2^{2 \cdot \frac{5}{5-3+1}}}=\frac{1}{2^{\frac{10}{3}}}$. Meanwhile, if $\delta = 0$, the second row of $\mathbf{C}$ would be ignorable. Thus, the eigenvalue cycle would not be divided and the critical erasure probability would be
$\frac{1}{2^{2 \cdot \frac{5}{5-4+1}}}=\frac{1}{2^{\frac{10}{2}}}$.

In this example, we can see that the critical erasure
probability depends on whether $\delta$ is equal to $0$ or not, which is related to the rank of $\mathbf{C}$. Thus, it is inevitable to have a rank condition of some sort in the characterization of the critical erasure probability.

\subsection{Extension to Intermittent Kalman Filtering with Parallel Channels}
The concept of eigenvalue cycles and the divide-and-conquer approach can be also applied to extensions and variations of the intermittent Kalman filtering.

Let's consider intermittent Kalman filtering with parallel erasure channels as introduced in \cite{Garone_LQG}.
\begin{align}
&\mathbf{x}[n+1]=\mathbf{A}\mathbf{x}[n]+\mathbf{B}\mathbf{w}[n]\nonumber \\
&\mathbf{y_1}[n]=\beta_1[n](\mathbf{C_1}\mathbf{x}[n]+\mathbf{v_1}[n]) \nonumber \\
&\vdots  \nonumber \\
&\mathbf{y_d}[n]=\beta_d[n](\mathbf{C_d}\mathbf{x}[n]+\mathbf{v_d}[n])\nonumber
\end{align}
Here $n$ is the non-negative integer-valued time index, and $\mathbf{x}[n] \in \mathbb{C}^{m}$, $\mathbf{w}[n] \in \mathbb{C}^{g}$, $\mathbf{y_i}[n] \in \mathbb{C}^{l_i}$, $\mathbf{v_i}[n] \in \mathbb{C}^{l_i}$, $\mathbf{A} \in \mathbb{C}^{m \times m}$, $\mathbf{B} \in \mathbb{C}^{m \times g}$, $\mathbf{C_i} \in \mathbb{C}^{l_i \times m}$. The underlying randomness comes from $\mathbf{x}[0]$, $\mathbf{w}[n]$, $\mathbf{v_i}[n]$ and $\beta_i[n]$. $\mathbf{x}[0]$, $\mathbf{w}[n]$ and $\mathbf{v_i}[n]$ are independent Gaussian vectors with zero mean, and there exist positive $\sigma^2$ and $\sigma'^2$ such that
\begin{align}
&\mathbb{E}[\mathbf{x}[0]\mathbf{x}[0]^\dag] \preceq \sigma^2 \mathbf{I} \nonumber \\
&\mathbb{E}[\mathbf{w}[n]\mathbf{w}[n]^\dag] \preceq \sigma^2 \mathbf{I} \nonumber \\
&\mathbb{E}[\mathbf{v_i}[n]\mathbf{v_i}[n]^\dag] \preceq \sigma^2 \mathbf{I} \nonumber\\
&\mathbb{E}[\mathbf{w}[n]\mathbf{w}[n]^\dag] \succeq \sigma'^2 \mathbf{I} \nonumber \\
&\mathbb{E}[\mathbf{v_i}[n]\mathbf{v_i}[n]^\dag] \succeq \sigma'^2 \mathbf{I}. \nonumber
\end{align}
$\beta_i[n]$ are independent Bernoulli random processes with erasure probabilities $p_{e,i}$.

We call this system as an intermittent system $(\mathbf{A},\mathbf{B},\mathbf{C_i})$ with erasure probabilities $p_{e,i}$.

Since the observations go through independent parallel erasure channels, we can expect diversity gain~\cite{Tse}, i.e. even though the observations from some channels are lost, we can still estimate the state based on other successfully transmitted observations. At the first glance, this extension may seem much harder than the original problem since we have to characterize the whole region $(p_{e,1},\cdots,p_{e,d})$ rather than a single critical erasure value. However, a simple extension of Theorem~\ref{thm:mainsingle} turns out to be enough to characterize this critical erasure probability region. As in Section~\ref{sec:interob}, let $\mathbf{A}=\mathbf{U}\mathbf{A'}\mathbf{U}^{-1}$ where $U$ is an invertible matrix and $\mathbf{A'}$ is an upper-triangular Jordan matrix. We also define $\mathbf{B'}:=\mathbf{U}\mathbf{B}$ and $\mathbf{C_i'}:=\mathbf{C_i}\mathbf{U}$.

Then, we can make the following generalized definitions of \eqref{eqn:ac:jordan:thm}, \eqref{eqn:ac2:jordan:thm}, \eqref{eqn:def:lprime:thm} for  $\mathbf{A'}$ and $\mathbf{C_i'}$.
\begin{align}
&\mathbf{A'}=diag\{ \mathbf{A_{1,1}}, \mathbf{A_{1,2}},\cdots, \mathbf{A_{\mu,\nu_{\mu}}} \} \nonumber \\
&\mathbf{C_i'}=\begin{bmatrix} \mathbf{C_{1,1,i}} & \mathbf{C_{1,2,i}} & \cdots & \mathbf{C_{\mu,\nu_{\mu},i}} \end{bmatrix} \nonumber \\
&\mbox{where}\nonumber \\
&\quad \mbox{$\mathbf{A_{i,j}}$ is a Jordan block matrix with an eigenvalue $\lambda_{i,j}$} \nonumber \\
&\quad \{ \lambda_{i,1},\cdots, \lambda_{i,\nu_i} \} \mbox{ is a cycle with length $\nu_i$ and period $p_i$}\nonumber \\
&\quad \mbox{For $i \neq i'$, $\{\lambda_{i,j},\lambda_{i',j'} \}$ is not a cycle} \nonumber \\
&\quad \mbox{$\mathbf{C_{i,j,k}}$ is a $l_k \times \dim \mathbf{A_{i,j}}$ matrix}.\nonumber
\end{align}
Denote \begin{align}
&\mathbf{A_i}=diag\{ \lambda_{i,1},\cdots, \lambda_{i,\nu_i} \}\nonumber \\
&\mathbf{C_{i,j}}=\begin{bmatrix} (\mathbf{C_{i,1,j}})_1,\cdots, (\mathbf{C_{i,\nu_i,j}})_1 \end{bmatrix} \nonumber\\
&\mbox{where $(\mathbf{C_{i,j,k}})_1$ is the first column of $\mathbf{C_{i,j,k}}$}.\nonumber
\end{align}
Let $(l_{i,1},l_{i,2},\cdots,l_{i,d})$ be the cardinality vector of the sets $S_1',S_2',\cdots,S_d'$ such that $S_j := \{0,1,\cdots, p_i-1 \} \setminus S_j' = \{ s_{j,1}, s_{j,2}, \cdots, s_{j,|S_j|} \}$ and
\begin{align}
\begin{bmatrix}
\mathbf{C_{i,1}} \mathbf{A_i}^{s_{1,1}} \\
\vdots \\
\mathbf{C_{i,1}} \mathbf{A_i}^{s_{1,|S_1|}} \\
\mathbf{C_{i,2}} \mathbf{A_i}^{s_{2,1}} \\
\vdots \\
\mathbf{C_{i,d}} \mathbf{A_i}^{s_{d,|S_d|}} \\
\end{bmatrix}\nonumber
\end{align}
is rank deficient, i.e. has rank strictly less than $\nu_i$. Denote $L_i$ as a set of all such vectors.

Then, the intermittent observability with parallel channels is characterized as follows.
\begin{proposition}
Given an intermittent system $(\mathbf{A},\mathbf{B},\mathbf{C_i},\sigma,\sigma')$ with probabilities of erasures $(p_{e,1}, \cdots, p_{e,d} )$, let $\sigma < \infty$, $\sigma' > 0$, and $(\mathbf{A},\mathbf{B})$ be controllable. Then, the intermittent system is intermittent observable if and only if
\begin{align}
\max_{1 \leq i \leq \mu} \max_{(l_{i,1},l_{i,2},\cdots, l_{i,d}) \in L_i} \left( \prod_{1 \leq j \leq d} p_{e,j}^{\frac{l_{i,j}}{p_i}} \right) |\lambda_{i,1}|^{2} < 1. \nonumber
\end{align} \label{thm:multi}
\end{proposition}
We omit the proof of the proposition, since it is similar to that of Theorem~\ref{thm:mainsingle}.

Compared to Theorem~\ref{thm:mainsingle}, the max-combination and separability principle remain the same, but the test channels in the power property become more complicated.
Here, $(S_1',\cdots,S_d')$ represents the test channels such that when they are erased, the observability Gramian becomes rank-deficient. $(l_{i,1},\cdots,l_{i,d})$ represents the distance vector to these test channels.


\section{Intermittent Kalman Filtering with Nonuniform Sampling}
\label{sec:nonuniform}
In the previous section, we proved that eigenvalue cycles are the
only factor that prevents us from having the critical erasure probability be
$\frac{1}{|\lambda_{max}|^2}$. Based on this understanding, we can
look for a simple way to avoid this troublesome phenomenon. Here, we
propose nonuniform sampling as a simple way of breaking the eigenvalue
cycles and achieving the critical value
$\frac{1}{|\lambda_{max}|^2}$.

As an intuitive example, consider $\mathbf{A}=\begin{bmatrix} 1 & 0 \\
  0 & -1 \end{bmatrix}$. Then, $\mathbf{A}=\begin{bmatrix}1 & 0 \\ 0 &
  -1 \end{bmatrix}, \mathbf{A}^2=\begin{bmatrix}1 & 0 \\ 0 &
  1 \end{bmatrix}, \mathbf{A}^3=\begin{bmatrix}1 & 0 \\ 0 &
  -1 \end{bmatrix}, \mathbf{A}^4=\begin{bmatrix}1 & 0 \\ 0 &
  1 \end{bmatrix}, \cdots$. What the eigenvalue cycle is capturing is
that half of $\mathbf{A},\mathbf{A}^2,\mathbf{A}^3,\cdots$ are
identical. Therefore, the question is how we can make every matrix in
$\mathbf{A},\mathbf{A}^2,\mathbf{A}^3,\cdots$ distinct. To simplify
the question, consider the sequence of $-1,1,-1,1,\cdots$  which
corresponds to $(2,2)$ elements of
$\mathbf{A},\mathbf{A}^2,\mathbf{A}^3,\cdots$.

Rewrite this sequence $-1,1,-1,1,\cdots$ as $(e^{j \pi})^1,(e^{j
  \pi})^2,(e^{j \pi})^3,(e^{j \pi})^4,\cdots$ and introduce a jitter
$t_i$ to each sampling time. The resulting sequence becomes $(e^{j  \pi})^{1+t_1},(e^{j \pi})^{2+t_2},(e^{j \pi})^{3+t_3},(e^{j  \pi})^{4+t_4},\cdots$ and if $t_i$s are uniformly distributed i.i.d.~random variables on $[0,T]$ each element in the sequence is distinct almost surely as long as $T>0$.

Operationally, this idea can be implemented as follows: at
design-time, the sensor and the estimator agree on the nonuniform
sampling pattern which is a realization of i.i.d.~random variables
whose distribution is uniform on $[0,T]~(T>0)$. Whenever the sensor
samples the system, it jitters its sampling time according to this
nonuniform pattern. Knowing the sampling time jitter, the sampled
continuous-time system looks like a discrete {\em time-varying} system
to the estimator. The joint Gaussianity between the observation and the
state is preserved, and furthermore, Kalman filters are optimal even for
time-varying systems!  This intermittent Kalman
filtering problem with nonuniform samples has the critical erasure
probability $\frac{1}{|\lambda_{max}|^2}$ almost surely. Therefore, an
eigenvalue cycle is breakable by nonuniform sampling.

One may be bothered by the probabilistic argument on the nonuniform
sampling pattern. However, this probabilistic proof is an indirect
argument for the existence of an appropriate deterministic nonuniform
sampling pattern, which is similar to how the existence of
capacity achieving codes is proved in information
theory~\cite{Shannon_mathematical}.

To write the scheme formally, consider a continuous-time dynamic system:
\begin{align}
&d\mathbf{x_c}(t)=\mathbf{A_c} \mathbf{x_c}(t)dt + \mathbf{B_c} d \mathbf{W_c}(t) \label{eqn:contistate}\\
&\mathbf{y_c}(t)=\mathbf{C_c} \mathbf{x_c}(t) + \mathbf{D_c} \frac{d \mathbf{V_c}(t)}{dt}. \label{eqn:contiob}
\end{align}
Here $t$ is the non-negative real-valued time index. $\mathbf{W_c}(t)$ and $\mathbf{V_c}(t)$ are independent $g$ and $l$-dimension standard Wiener processes respectively, i.e. for $a,b \geq 0$, $\mathbf{W_c}(a+b)-\mathbf{W_c}(b)$ is distributed as $\mathcal{N}(\mathbf{0},a\mathbf{I})$ and $\mathbf{V_c}(a+b)-\mathbf{V_c}(b)$ is also distributed as $\mathcal{N}(0,a\mathbf{I})$. $\mathbf{A_c} \in \mathbb{C}^{m \times m}$, $\mathbf{B_c} \in \mathbb{C}^{m \times g}$, $\mathbf{C_c} \in \mathbb{C}^{l \times m}$, and $\mathbf{D_c} \in \mathbb{C}^{l \times l}$ where $\mathbf{D_c}$ is invertible. Thus, $\mathbf{x}[n] \in \mathbb{C}^{m}$ and $\mathbf{y}[n] \in \mathbb{C}^{l}$. For a convenience, we assume $\mathbf{x}[0]=0$ but the results of this paper hold for any $\mathbf{x}[0]$ with finite variance. Throughout this paper, we use the Ito's integral~\cite[p.80]{Gardiner} for stochastic calculus.

The process of \eqref{eqn:contistate} is known as Ornstein-Uhlenbeck process~\cite[p.109]{Gardiner} whose solution is $\mathbf{x_c}(t)=e^{\mathbf{A_c}t}\mathbf{x_c}(0)+\int^t_0 e^{\mathbf{A_c}(t-t')} \mathbf{B_c} d \mathbf{W_c}(t')$. Therefore, for $t_1 \leq t_2$ we have
\begin{align}
\mathbf{x_c}(t_2)&=e^{\mathbf{A_c}t_2} \mathbf{x_c}(0)+ \int^{t_2}_{0} e^{\mathbf{A_c}(t_2 -t')}\mathbf{B_c} d \mathbf{W_c}(t') \label{eqn:non:0} \\
&=e^{\mathbf{A_c}(t_2-t_1)} \left( e^{\mathbf{A_c}t_1} \mathbf{x_c}(0) + \int^{t_2}_{0} e^{\mathbf{A_c}(t_1-t')} \mathbf{B_c} d \mathbf{W_c}(t') \right) \nonumber  \\
&=e^{\mathbf{A_c}(t_2-t_1)} \left( e^{\mathbf{A_c}t_1} \mathbf{x_c}(0) + \int^{t_1}_{0} e^{\mathbf{A_c}(t_1-t')} \mathbf{B_c} d \mathbf{W_c}(t') + \int^{t_2}_{t_1} e^{\mathbf{A_c}(t_1 - t')} \mathbf{B_c}d \mathbf{W_c}(t')
\right) \nonumber \\
&=e^{\mathbf{A_c}(t_2-t_1)} \left( \mathbf{x_c}(t_1) + \int^{t_2}_{t_1} e^{\mathbf{A_c}(t_1 - t')} \mathbf{B_c}d \mathbf{W_c}(t')
\right)\nonumber
\end{align}
which can be rewritten as
\begin{align}
\mathbf{x_c}(t_1)=e^{\mathbf{A_c}(t_1-t_2)} \mathbf{x_c}(t_2) - \int^{t_2}_{t_1} e^{\mathbf{A_c}(t_1-t')} \mathbf{B_c} d\mathbf{W_c}(t').\label{eqn:non:1}
\end{align}
The point of doing this is to understand the values of the states during sampling intervals in terms of the states at the end of the interval.

Let's say we want to sample the system with a sampling interval $I~(I>0)$.
Conventional samplers uses integration filters to sample, i.e. in the uniform sampling case, the $n$th
sample $\mathbf{y}[n]$ corresponds to the integration of $\mathbf{y_c}(t)$ for $(n-1)I \leq t < nI$:
\begin{align}
\mathbf{y}[n]&=\int^{nI}_{(n-1)I} \mathbf{y_c}(t) dt. \nonumber
\end{align}

Nonuniform sampling can be thought of in two ways with respect to sampler's integration filters: (1)
The starting times of the integrations are uniform, but the sampling intervals are non-uniform. (2)
The sampling intervals are uniform, but the starting times of the integrations are non-uniform. Since the
analysis and performance is similar in both cases, we will focus on
the latter case. To take the $n$th sample of the system, the non-uniform sampler takes the integration of $\mathbf{y_c}(t)$ for $(n-1)I-t_n \leq t < nI - t_n$:
\begin{align}
\mathbf{y_o}[n]&=\int^{nI-t_n}_{(n-1)I-t_n} \mathbf{y_c}(t) dt \nonumber \\
&=\int^{nI-t_n}_{(n-1)I-t_n} \mathbf{C_c} \mathbf{x_c}(t) dt + \int^{nI-t_n}_{(n-1)I-t_n} \mathbf{D_c}d\mathbf{V_c}(t) \label{eqn:non:2}\\
&=\int^{nI-t_n}_{(n-1)I-t_n} \mathbf{C_c} \left( e^{\mathbf{A_c}(t-(nI-t_n))} \mathbf{x_c}(nI-t_n)
-\int^{nI-t_n}_{t} e^{\mathbf{A_c}(t-t')} \mathbf{B_c} d\mathbf{W_c}(t')
 \right) dt + \int^{nI-t_n}_{(n-1)I-t_n} \mathbf{D_c}d\mathbf{V_c}(t) \label{eqn:non:3} \\
&=\left( \int^{nI-t_n}_{(n-1)I-t_n} \mathbf{C_c}e^{\mathbf{A_c}(t-(nI-t_n))} dt \right)\mathbf{x_c}(nI-t_n) \nonumber  \\
&-\int^{nI-t_n}_{(n-1)I-t_n} \int^{nI-t_n}_{t} \mathbf{C_c} e^{\mathbf{A_c}(t-t')} \mathbf{B_c} d\mathbf{W_c}(t')dt
+ \int^{nI-t_n}_{(n-1)I-t_n} \mathbf{D_c}d\mathbf{V_c}(t) \nonumber \\
&=\underbrace{\left( \int^{I}_{0} \mathbf{C_c} e^{\mathbf{A_c}(t-I)} dt \right)}_{:=\mathbf{C}} \mathbf{x_c}(nI-t_n) \nonumber \\
&\underbrace{-\int^{nI-t_n}_{(n-1)I-t_n} \int^{nI-t_n}_{t} \mathbf{C_c} e^{\mathbf{A_c}(t-t')} \mathbf{B_c} d\mathbf{W_c}(t')dt
+ \int^{nI-t_n}_{(n-1)I-t_n} \mathbf{D_c}d\mathbf{V_c}(t)}_{:=\mathbf{v}[n]} \label{eqn:non:4}
\end{align}
Here \eqref{eqn:non:2} follows from \eqref{eqn:contiob}, and \eqref{eqn:non:3} follows from \eqref{eqn:non:1}. Since $\mathbf{y_o}[n]$ is transmitted over the erasure channel, the intermittent system $(\mathbf{A_c},\mathbf{B_c},\mathbf{C})$ with nonuniform samples and erasure probability $p_e$ has the following system equation:
\begin{align}
&d\mathbf{x_c}(t)=\mathbf{A_c}\mathbf{x_c}(t)dt+\mathbf{B_c}d \mathbf{W_c}(t) \label{eqn:conti:xsample}\\
&\mathbf{y}[n]=\beta[n](\mathbf{C}\mathbf{x_c}(nI-t_n)+\mathbf{v}[n])\label{eqn:conti:ysample}
\end{align}
where $\mathbf{y}[n] \in \mathbb{C}^{l}$ and $\beta[n]$ is an independent Bernoulli random process with erasure probability $p_e$. The variance of $\mathbf{v}[n]$ is uniformly bounded since the integration interval is bounded, but $\mathbf{v}[n]$ can be correlated since the integration intervals could overlap. Since $\mathbf{C}$ is a function of $\mathbf{C_c}$, the observability of $(\mathbf{A_c},\mathbf{C_c})$ does not necessarily imply the observability of $(\mathbf{A_c},\mathbf{C})$ while the observability of $(\mathbf{A_c},\mathbf{C})$ always implies the observability of $(\mathbf{A_c},\mathbf{C_c})$.

\begin{figure*}[t]
\begin{center}
\includegraphics[width=3.7in]{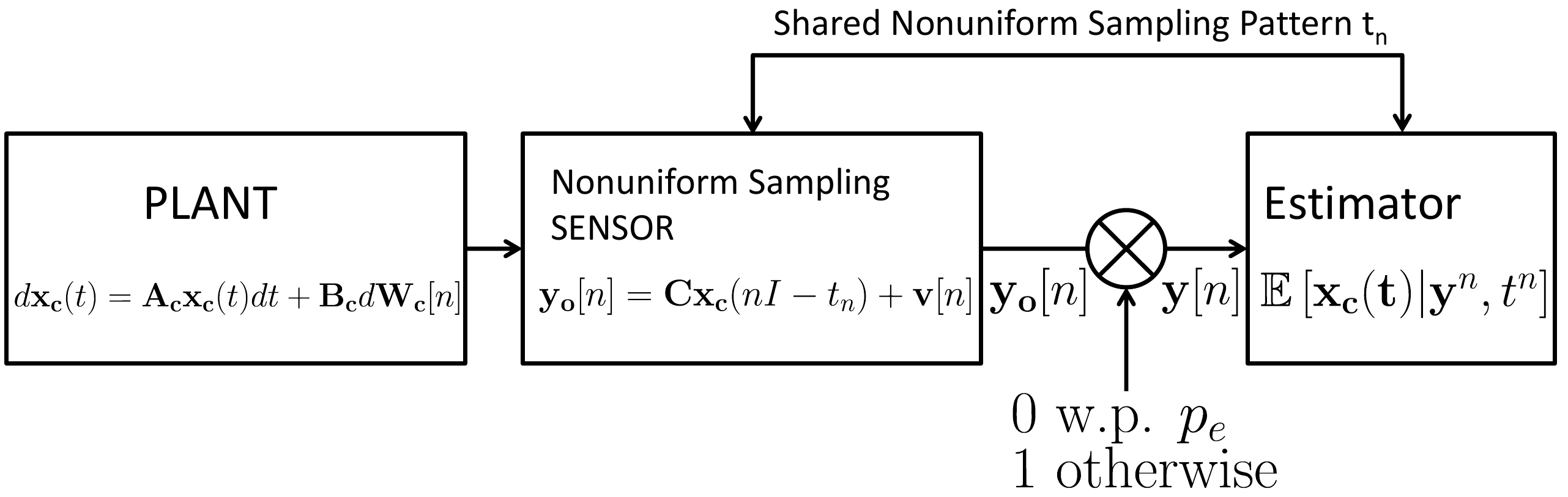}
\caption{System diagram for `intermittent Kalman filtering with nonuniform sampling'.
The sensor samples the plant according to the nonuniform sampling pattern $t_n$, and sends the observation through the real erasure channel without any coding. The estimator tries to estimate the state based on its received signals and the nonuniform sampling pattern $t_n$.}
\label{fig:system3}
\end{center}
\end{figure*}

Figure~\ref{fig:system3} shows the system diagram for intermittent Kalman filtering with nonuniform sampling. The nonuniform sampler samples the plant according to the nonuniform sampling pattern $t_n$ and generates the observation $y_o[n]$. The observation is transmitted through the real erasure channel without any coding. Then, the estimator tries to estimate the state $x_c(t)$ based on its received signals $y^n$ and the nonuniform sampling pattern $t^n$.

As before, the intermittent system $(\mathbf{A_c},\mathbf{B_c},\mathbf{C})$ with nonuniform samples is called intermittent observable if there exists a causal estimator $\mathbf{\widehat{x}}(t)$ of $\mathbf{x}(t)$ based on $\mathbf{y}[\lfloor \frac{t}{I} \rfloor ], \cdots, \mathbf{y}[0]$ such that
\begin{align}
\sup_{t \in \mathbb{R}^+} \mathbb{E}[(\mathbf{x}(t)-\mathbf{\widehat{x}}(t))^\dag(\mathbf{x}(t)-\mathbf{\widehat{x}}(t))] < \infty.
\end{align}
Intermittent observability with nonuniform samples is characterized by the following theorem.
\begin{theorem}
Let $t_n$ be i.i.d.~random variables uniformly distributed on $[0,T]~(T>0)$, and $(\mathbf{A_c},\mathbf{B_c})$ be controllable.
When $(\mathbf{A_c},\mathbf{C})$ has unobservable and unstable eigenvalues --- i.e. $\exists \lambda \in \mathbb{C}^+$ such that $\begin{bmatrix} \lambda \mathbf{I} - \mathbf{A_c} \\ \mathbf{C} \end{bmatrix}$ is rank deficient ---, the intermittent system $(\mathbf{A_c},\mathbf{B_c},\mathbf{C})$ with nonuniform samples is not intermittent observable for all $p_e$. Otherwise, the intermittent system $(\mathbf{A_c},\mathbf{B_c},\mathbf{C})$ with nonuniform samples is intermittent observable if and only if $p_e < \frac{1}{|e^{2 \lambda_{max}I}|}$. Here $\lambda_{max}$ is the eigenvalue of $\mathbf{A_c}$ with the largest real part.
\label{thm:nonuniform}
\end{theorem}
\begin{proof}
See Section~\ref{sec:cont:suf} for sufficiency, and Section~\ref{sec:cont:nec} for necessity.
\end{proof}

Since $\exp\left(\mbox{(eigenvalue of $\mathbf{A_c}$)}I\right)$ corresponds to the eigenvalue of the
sampled discrete time system, the critical value of
Theorem~\ref{thm:nonuniform} is equivalent to that of
Corollary~\ref{thm:nocycle}. The nonuniform sampling allows us to no
longer care if eigenvalue cycles could exist for the original
continuous-time system under uniform sampling.


Nonuniform sampling is the right way of breaking eigenvalue cycles
from a practical point of view. So the critical erasure probability
of $\frac{1}{|\lambda_{max}|^2}$ can be achieved not only by using the
computationally challenging estimation-before-packetization strategy
of \cite{Sahai_Thesis}, but also by the simple memoryless approach of
dithered sampling before packetization. And so, even if the sensors
were themselves distributed, the critical erasure probability with
nonuniform sampling is still \textit{critical value optimal} in a sense that
they can achieve the same critical erasure probability as sensors with causal or noncausal information about the erasure pattern and with unbounded complexity.

\subsection{Extensions of Intermittent Kalman Filtering with Nonuniform Sampling}
In this section, we discuss variations and extensions of intermittent Kalman filtering with nonuniform samples. Since the proofs of the results shown in this section are similar to that of Theorem~\ref{thm:nonuniform}, we only present the results without proofs.

\subsubsection{General Distribution on $t_n$}
First, we relax the condition on the distribution of $t_n$ of Theorem~\ref{thm:nonuniform}. There, we assume that $t_n$ are identically and uniformly distributed. However, they do not have to be identical or uniform.
\begin{proposition}
Assume that $t_0,t_1,\cdots$ are independent and there exist $a,c>0$ such that $\mathbb{P}\{ |t_n| \geq a \} =0 $ and $\mathbb{P}\{ t_n \in B \} \leq c |B|_{\mathcal{L}}$ for all $n \in \mathbb{Z}^+$ and $B \in \mathcal{B}$, where $\mathcal{B}$ is Borel $\sigma$-algebra and $|\cdot|_{\mathcal{L}}$ is Lebesgue measure. Then, Theorem~\ref{thm:nonuniform} still holds, i.e. if $(\mathbf{A_c},\mathbf{C})$ has no unobservable and unstable eigenvalues, the intermittent system with nonuniform samples is intermittent observable if and only if $p_e < \frac{1}{|e^{2 \lambda_{max} I}|}$.
\end{proposition}

For the proof of the proposition, we can repeat the proof steps of Theorem~\ref{thm:nonuniform} using an improper distribution $\mu$ such that $\mu(A)=c|A \cap [-a,a]|_{\mathcal{L}}$.

\subsubsection{Deterministic Sequences for $t_n$}
The randomness assumption on $t_n$ can be also removed. As we mentioned earlier, the probabilistic proof is an indirect proof for the existence of deterministic nonuniform sampling patterns. In fact, any nonuniform sequence satisfying Weyl's criteria ---which gives the sufficient and necessary condition for a sequence equidistributed on the interval --- can be used to break eigenvalue cycles.
\begin{proposition}
Let a sequence $t_n \in [0,T]$ satisfy Weyl's criteria, i.e. for all $h \in \mathbb{Z}\setminus \{ 0 \}$, $\underset{N \rightarrow \infty}{\lim} | \frac{1}{N} \underset{1 \leq n \leq N}{\sum} e^{j 2\pi h \cdot \frac{t}{T}} | =0$. Then, Theorem~\ref{thm:nonuniform} still holds, i.e. if $(\mathbf{A_c},\mathbf{C})$ has no unobservable and unstable eigenvalues, the intermittent system with nonuniform samples is intermittent observable if and only if $p_e < \frac{1}{|e^{2 \lambda_{max} I}|}$.
\end{proposition}

For example, a sequence like $t_n = \sqrt{2}n - \lfloor \sqrt{2}n \rfloor$ can be used to break eigenvalue cycles. The proof is by merging the proof of Theorem~\ref{thm:mainsingle} and Theorem~\ref{thm:nonuniform}.

\subsubsection{Nonuniform-length integration interval}
In Theorem~\ref{thm:nonuniform}, we introduce nonuniform sampling by changing the starting time of the length of the integration. Another way of introducing nonuniform sampling is changing the integration interval. To take the $n$th sample of the system, the sensor integrates $\mathbf{y_c}(t)$ from $(n-1)I-t_n$ to $nI$. Parallel to \eqref{eqn:non:4}, we have the following equation.
\begin{align}
\mathbf{y_o}[n] &= \int^{nI}_{(n-1)I-t_n} \mathbf{y_c}(t) dt  \nonumber \\
&=\left( \int^{nI}_{(n-1)I-t_n} \mathbf{C_c} e^{\mathbf{A_c}(t-nI)} \right) \mathbf{x_c}(nI)  \nonumber \\
&- \int^{nI}_{(n-1)I-t_n} \int^{nI-t_n}_{t} \mathbf{C_c}e^{\mathbf{A_c}(t-t')}\mathbf{B_c}d\mathbf{W_c}(t')dt + \int^{nI}_{(n-1)I-t_n} \mathbf{D_c} d\mathbf{V_c}(t)  \nonumber \\
&=\underbrace{\left( \int^{n+t_n}_{0} \mathbf{C_c} e^{\mathbf{A_c}(t-nI-t_n)} \right)}_{:=\mathbf{C_n}} \mathbf{x_c}(nI)  \nonumber \\
&\underbrace{-\int^{nI}_{(n-1)I-t_n} \int^{nI-t_n}_{t} \mathbf{C_c}e^{\mathbf{A_c}(t-t')}\mathbf{B_c}d\mathbf{W_c}(t')dt + \int^{nI}_{(n-1)I-t_n} \mathbf{D_c} d\mathbf{V_c}(t)}_{:=\mathbf{v}[n]}  \nonumber
\end{align}
$\mathbf{y_o}[n]$ is transmitted over the erasure channel, and the intermittent system $(\mathbf{A_c},\mathbf{B_c},\mathbf{C_n})$ with nonuniform samples and erasure probability $p_e$ has the following system equations which correspond to \eqref{eqn:conti:xsample} and \eqref{eqn:conti:ysample}.
\begin{align}
&d\mathbf{x_c}(t)=\mathbf{A_c}\mathbf{x_c}(t)dt+\mathbf{B_c} d\mathbf{W_c}(t) \nonumber \\
&\mathbf{y}[n]=\beta[n](\mathbf{C_n}\mathbf{x_c}(nI) + \mathbf{v}[n]) \nonumber
\end{align}
Then, the intermittent observability condition for $(\mathbf{A_c},\mathbf{B_c},\mathbf{C_n})$ is similar to Theorem~\ref{thm:nonuniform}.
\begin{proposition}
Let $t_n$ be i.i.d.~random variables uniformly distributed on $[0,T]\ (T>0)$, and $(\mathbf{A_c},\mathbf{B_c})$ be controllable. If $(\mathbf{A_c},\mathbf{C_c})$ has unobservable and unstable eigenvalues, the intermittent system $(\mathbf{A_c},\mathbf{B_c},\mathbf{C_n})$ with nonuniform samples is not intermittent observable for all $p_e$. Otherwise, the intermittent system $(\mathbf{A_c},\mathbf{B_c},\mathbf{C_n})$ with nonuniform samples is intermittent observable if and only if $p_e < \frac{1}{|e^{2 \lambda_{max} I}|}$ where $\lambda_{max}$ is the eigenvalue of $\mathbf{A_c}$ with the largest real part.
\label{prop:4}
\end{proposition}

Compared to Theorem~\ref{thm:nonuniform}, we can see that the observability condition of $(\mathbf{A_c},\mathbf{C})$ is relaxed to the observability condition of $(\mathbf{A_c},\mathbf{C_c})$. This is due to the following fact:
$\int^{nI-t_n}_{(n-1)I-t_n}e^{j \frac{2 \pi}{I}t}dt=0$ for all $t_n$ and $\int^{nI}_{(n-1)I-t_n}e^{j\frac{2 \pi}{I}t}dt \neq 0$ for some $t_n$. Even if $(\mathbf{A_c},\mathbf{C_c})$ is observable, $(\mathbf{A_c},\mathbf{C})$ can be unobservable for all $t_n$ while $(\mathbf{A_c},\mathbf{C_n})$ is observable for almost all $t_n$.

\subsubsection{Nonuniform Time-varying Filtering}
In some cases, it is impossible to change the sampling time. In this case, we can use nonuniform time-varying filtering to break eigenvalue cycles. Consider the following discrete-time system:
\begin{align}
&\mathbf{x}[n+1]=\mathbf{A}\mathbf{x}[n]+\mathbf{B}\mathbf{w}[n] \nonumber \\
&\mathbf{y_o}[n]=\mathbf{C}\mathbf{x}[n]+\mathbf{v}[n] \nonumber
\end{align}
Here $\mathbf{y_o}[n]$ are the observations at the sensor, and the sensor cannot change the sampling intervals. Instead, the sensor introduces nonuniform filtering to the observations as follows:
\begin{align}
&\mathbf{y_o'}[n]=\alpha[n]\mathbf{y_o}[n]+\alpha'[n]\mathbf{y_o}[n-1] \nonumber
\end{align}

This is just like introducing an FIR (finite impulse response) filter at the sensor except that the impulse response of the filter keeps changing over time.

The output of the nonuniform time-varying filter, $\mathbf{y_o'}[n]$, is transmitted over the erasure channel. Therefore, the intermittent system $(\mathbf{A},\mathbf{B},\mathbf{C})$ with erasure probability $p_e$ and nonuniform time-varying filtering has the following system equations:
\begin{align}
&\mathbf{x}[n+1]=\mathbf{A}\mathbf{x}[n]+\mathbf{B}\mathbf{w}[n] \nonumber \\
&\mathbf{y}[n]=\beta[n](\mathbf{y_o'}[n]) \nonumber \\
&\ \quad=\beta[n](\alpha[n]\mathbf{C}\mathbf{x}[n]+\alpha'[n]\mathbf{C}\mathbf{x}[n-1]+\alpha[n]\mathbf{v}[n]+\alpha'[n]\mathbf{v}[n-1] ) \nonumber
\end{align}
The intermittent observability with nonuniform filtering is given as the following proposition.
\begin{proposition}
Let $\alpha[n]$ and $\alpha'[n]$ be i.i.d.~random variables uniformly distributed on $[0,T]\ (T>0)$, and $(\mathbf{A},\mathbf{B})$ be controllable. If $(\mathbf{A},\mathbf{C})$ has unobservable and unstable eigenvalues, the intermittent system $(\mathbf{A},\mathbf{B},\mathbf{C})$ with nonuniform filtering is not intermittent observable for all $p_e$. Otherwise, the intermittent system $(\mathbf{A},\mathbf{B},\mathbf{C})$ with nonuniform filtering is intermittent observable if and only if $p_e < \frac{1}{|\lambda_{max}|^2}$ where $\lambda_{max}$ is the largest magnitude eigenvalue of $\mathbf{A}$.
\label{prop:5}
\end{proposition}

\begin{figure}[top]
\centering
\subfloat[]{\label{fig:1a}\includegraphics[width = 1.5in]{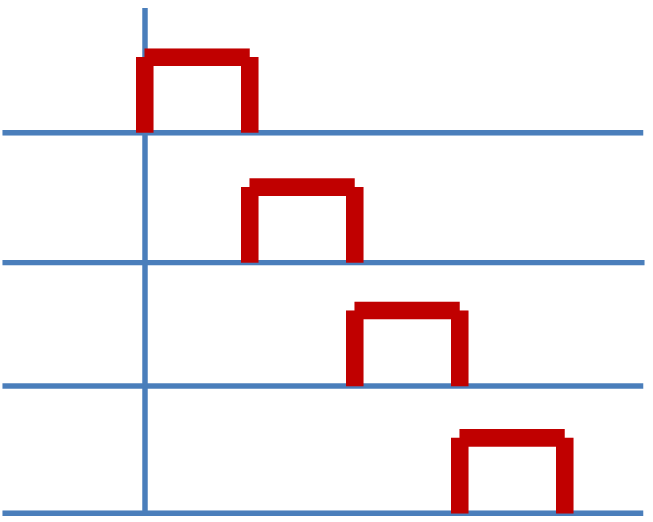}} \quad
\subfloat[]{\label{fig:1b}\includegraphics[width = 1.5in]{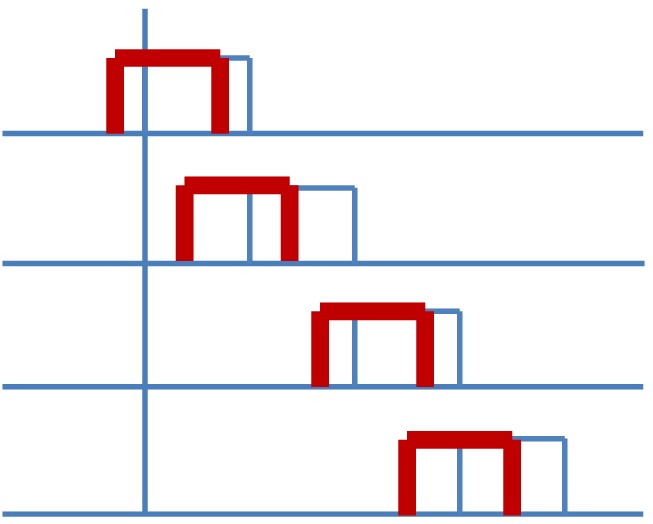}} \quad
\subfloat[]{\label{fig:1c}\includegraphics[width = 1.5in]{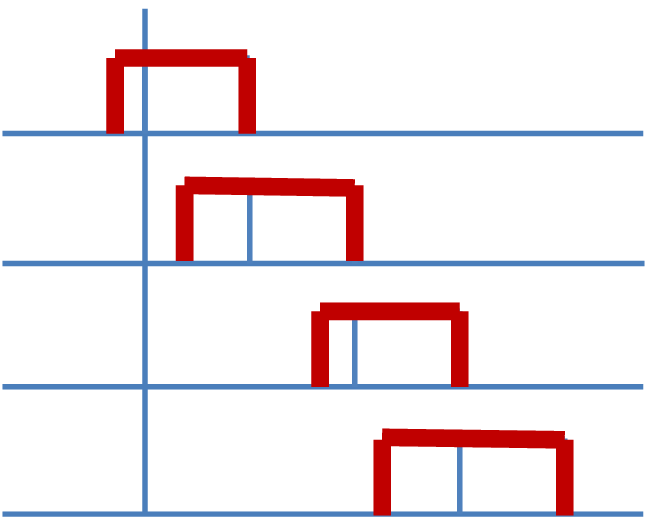}}\\
\subfloat[]{\label{fig:1d}\includegraphics[width = 1.5in]{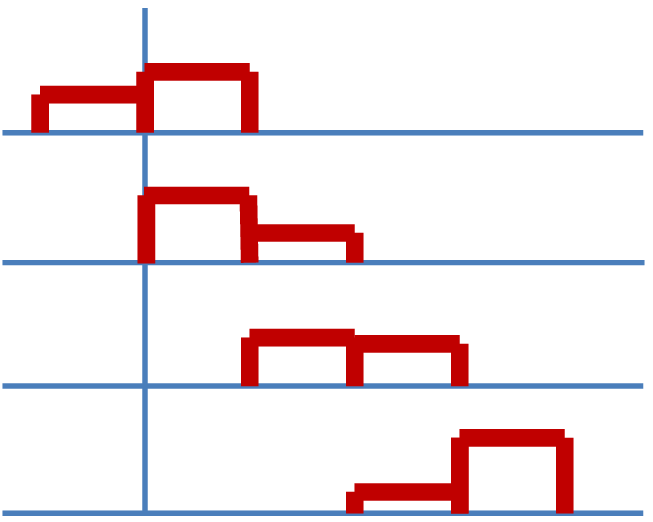}} \quad
\subfloat[]{\label{fig:1e}\includegraphics[width = 1.5in]{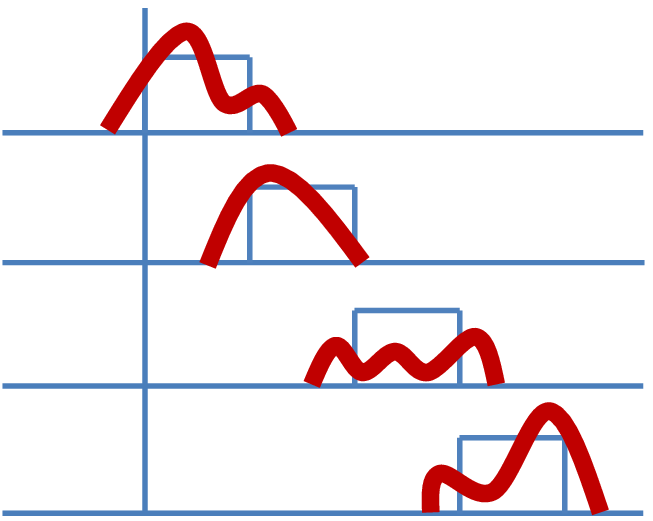}}
\caption{(a): uniform sampling of Theorem~\ref{thm:mainsingle}, (b): nonuniform sampling of Theorem~\ref{thm:nonuniform}, (c): nonuniform sampling of Proposition~\ref{prop:4}, (d): nonuniform filtering of Proposition~\ref{prop:5}, (e): nonuniform sampling with nonuniform waveforms}
\label{fig:1}
\end{figure}

\subsubsection{Sampling with Nonuniform Waveforms}
So far in Theorem~\ref{thm:nonuniform}, Proposition~\ref{prop:4}, and Proposition~\ref{prop:5}, we have seen three different ways of breaking eigenvalue cycles. However, these methods are essentially the same and generalized to nonuniform sampling with nonuniform waveforms.

Fig.~\ref{fig:1} shows the nonuniform sampling methods used to break eigenvalue cycles with respect to their waveforms. First, Fig.~\ref{fig:1a} shows the uniform sampling which is implicitly used to make  discrete-time system~\eqref{eqn:dis:system}, \eqref{eqn:dis:system2} from the underlying continuous-time system. As we saw in Theorem~\ref{thm:mainsingle}, the eigenvalue cycles were not broken in this case. Fig.~\ref{fig:1b} shows the nonuniform sampling by changing the starting time of the integration, which is used in Theorem~\ref{thm:nonuniform}. In this case, the eigenvalue cycles were successfully broken, but we can still observe the regularity in the integration intervals. Due to this regularity, we needed the observability of $(\mathbf{A_c,\mathbf{C}})$ instead of the observability of $(\mathbf{A_c},\mathbf{C_c})$. Fig.~\ref{fig:1c} shows the nonuniform sampling by changing the integration interval, which is used in Proposition~\ref{prop:4}. The eigenvalue cycles were also broken in this case and due to the lack of regularity in sampling intervals the observability of $(\mathbf{A_c},\mathbf{C_c})$ was enough. Fig.~\ref{fig:1d} shows the nonuniform filtering, which is used in Proposition~\ref{prop:5} and successfully break the eigenvalue cycles. Therefore, we can conclude that as long as the sampling waveforms are not uniform as Fig.~\ref{fig:1a} the eigenvalue cycles are broken. In general, nonuniform waveforms shown in Fig.~\ref{fig:1e} can be used to break eigenvalue cycles, and it is an interesting technical equation to find the minimal condition on nonuniform waveforms to break eigenvalue cycles.

\subsubsection{Extension to Parallel Channels}
Theorem~\ref{thm:nonuniform} can also be extended to the multiple sensors that transmit their observations through parallel erasure channels. Consider the following continuous-time system equations.
\begin{align}
&d\mathbf{x_c}(t)=\mathbf{A_c}\mathbf{x_c}(t)dt + \mathbf{B_c}d\mathbf{W_c}(t) \nonumber \\
&\mathbf{y_{c,1}}(t)=\mathbf{C_{c,1}}\mathbf{x_c}(t) + \mathbf{D_{c,1}} \frac{d\mathbf{V_{c,1}(t)}}{dt} \nonumber \\
&\vdots \nonumber \\
&\mathbf{y_{c,d}}(t)=\mathbf{C_{c,d}}\mathbf{x_c}(t) + \mathbf{D_{c,d}} \frac{d\mathbf{V_{c,d}(t)}}{dt} \nonumber
\end{align}
Here $t$ is non-negative real-valued time index. $\mathbf{A_c} \in \mathbb{C}^{m \times m}$, $\mathbf{B_c} \in \mathbb{C}^{m \times g}$ , $\mathbf{C_{c,i}} \in \mathbb{C}^{l_i \times m}$ and $\mathbf{D_{c,i}} \in \mathbb{C}^{l_i \times l_i}$ where $\mathbf{D_{c,i}}$ is invertible. $\mathbf{W_c}(t)$ and $\mathbf{V_{c,1}}(t)$ are independent $g$ and $l_i$-dimensional standard Wiener process respectively.

Like \eqref{eqn:non:4}, the $n$th sample at the sensor $i$ is obtained by integrating $\mathbf{y_{c,i}}(t)$ from $(n-1)I-t_{n,i}$ to $nI-t_{n,i}$:
\begin{align}
\mathbf{y_{o,i}}[n]&=\int^{nI-t_{n,i}}_{(n-1)I-t_{n,i}} \mathbf{y_{c,i}}(t) dt \nonumber \\
&=\underbrace{\left( \int^{I}_{0} \mathbf{C_{c,i}}e^{\mathbf{A_c}(t-I)}dt \right)}_{:=\mathbf{C_i}} \mathbf{x_c}(nI-t_{n,i}) \nonumber \\
&\underbrace{-\int^{nI-t_{n,i}}_{(n-1)I-t_{n,i}} \int^{nI-t_{n,i}}_{t} \mathbf{C_{c,i}} e^{\mathbf{A_c}(t-t')} \mathbf{B_c} d\mathbf{W_c}(t')dt
+\int^{nI-t_{n,i}}_{(n-1)I-t_{n,i}} \mathbf{D_{c,i}} d\mathbf{V_{c,i}}(t)}_{:=\mathbf{v_i}[n]} \nonumber
\end{align}

Since $\mathbf{y_{o,i}}[n]$ are transmitted over the parallel erasure channel, the intermittent system $(\mathbf{A_c},\mathbf{B_c},\mathbf{C_i})$ with parallel channel has the following system equation:
\begin{align}
&d\mathbf{x_c}(t)=\mathbf{A_c}\mathbf{x_c}(t)dt+\mathbf{B_c}d\mathbf{W_c}(t) \nonumber \\
&\mathbf{y_1}[n]=\beta_1[n](\mathbf{C_1}\mathbf{x_c}(nI-t_{n,1})+\mathbf{v_1}[n]) \nonumber \\
&\vdots \nonumber \\
&\mathbf{y_d}[n]=\beta_d[n](\mathbf{C_d}\mathbf{x_c}(nI-t_{n,d})+\mathbf{v_d}[n]) \nonumber
\end{align}
where $\mathbf{y_i}[n] \in \mathbb{C}^{l_i}$ and $\beta_i[n]$ are independent Bernoulli random processes with erasure probability $p_{e,i}$.

Like before, by a change of coordinates, we can rewrite the above system equations to the ones with a Jordan form $\mathbf{A_c}$ without changing the intermittent observability. Therefore, like \eqref{eqn:ac:jordan:thm}, \eqref{eqn:ac2:jordan:thm} and \eqref{eqn:def:lprime:thm} we can write $\mathbf{A_c}$ and $\mathbf{C_i}$ as follows without loss of generality.
\begin{align}
&\mathbf{A_c}=diag\{ \mathbf{A_{1,1}}, \mathbf{A_{1,2}}, \cdots, \mathbf{A_{\mu,\nu_{\mu}}} \} \nonumber \\
&\mathbf{C_i}=\begin{bmatrix} \mathbf{C_{1,1,i}} & \mathbf{C_{1,2,i}} & \cdots & \mathbf{C_{\mu,\nu_\mu,i}} \end{bmatrix} \nonumber \\
&\mbox{where } \nonumber\\
&\quad \mathbf{A_{i,j}} \mbox{ is a Jordan block with eigenvalue $\lambda_i$} \nonumber \\
&\quad \lambda_1, \cdots , \lambda_\mu \mbox{ are pairwise distinct} \nonumber \\
&\quad \mathbf{C_{i,j,k}}\mbox{ is a $l_k \times \dim \mathbf{A_{i,j}}$ complex matrix.} \nonumber
\end{align}

Denote
\begin{align}
&\mathbf{C_{i,j}}=\begin{bmatrix} (\mathbf{C_{i,1,j}})_1 & \cdots & (\mathbf{C_{i,\nu_i,j}})_1 \end{bmatrix}\nonumber \\
&\mbox{where $\left( \mathbf{C_{i,j,k}} \right)_1$ implies the first column of $\mathbf{C_{i,j,k}}$} \nonumber
\end{align}

Let $(l_{i,1},l_{i,2},\cdots,l_{i,d}) \in \{ 0,1 \}^d$ such that
\begin{align}
\begin{bmatrix}
\mathbf{1}(l_{i,1}=0)\mathbf{C_{i,1}} \\
\vdots \\
\mathbf{1}(l_{i,d}=0)\mathbf{C_{i,d}}
\end{bmatrix} \nonumber
\end{align}
is rank deficient, i.e. the rank is strictly less than $\nu_i$.

Denote $L_i$ as the set of such $(l_{i,1},l_{i,2},\cdots,l_{i,d})$ vectors. Then, the intermittent observability of the system $(\mathbf{A_c},\mathbf{B_c},\mathbf{C_i})$ with parallel channel is characterized by the following proposition.
\begin{proposition}
Given an intermittent system $(\mathbf{A_c}, \mathbf{B_c}, \mathbf{C_i})$ with probability of erasures $(p_{e,1}, \cdots, p_{e,d})$, let $(\mathbf{A_c},\mathbf{B_c})$ be controllable, and $t_{n,i}$ be independent random variables uniformly distributed on $[0,T]\ (T>0)$. The intermittent system $(\mathbf{A_c},\mathbf{B_c},\mathbf{C_i})$ with parallel channel is intermittent observable if and only if
\begin{align}
\max_{1 \leq i \leq \mu} \max_{(l_{i,1},l_{i,2},\cdots,l_{i,d}) \in L_i}\left( \prod_{1 \leq j \leq d} p_{e,j}^{l_{i,j}} \right) |e^{2 \lambda_i I}| < 1. \nonumber
\end{align}
\end{proposition}

\section{Proofs}

The proofs of Theorem~\ref{thm:mainsingle} and Theorem~\ref{thm:nonuniform} are quite similar, and we can directly relate them by Weyl's criterion~\cite{Kuipers}. For presentation purposes, we will first present the proof of the nonuniform sampling case, Theorem~\ref{thm:nonuniform}, which is easier than that of Theorem~\ref{thm:mainsingle}.

\label{sec:proof}

\subsection{Sufficiency Proof of Theorem~\ref{thm:nonuniform} (Non-uniform Sampling)}
\label{sec:cont:suf}
We will prove that if $(\mathbf{A_c},\mathbf{C})$ does not have unobservable and unstable eigenvalues and $p_e < \frac{1}{|e^{2\lambda_{max}I}|}$, the system is intermittent observable.

$\bullet$ Reduction to a Jordan form matrix $\mathbf{A_c}$: To simplify the problem, we first restrict to system equations \eqref{eqn:conti:xsample} and \eqref{eqn:conti:ysample} with the following properties. We will also justify that  this restriction is without loss of generality and does not change intermittent observability.\\
(a) The system matrix $\mathbf{A_c}$ is a Jordan form matrix.\\
(b) All eigenvalues of $\mathbf{A_c}$ are unstable, i.e. the real parts are nonnegative.\\
(c) \eqref{eqn:conti:xsample} and \eqref{eqn:conti:ysample} can be extended to two-sided processes.

The restriction (a) can be justified by a similarity transform~\cite{Chen}. As mentioned before, it is known~\cite{Chen} that for any square matrix $\mathbf{A_c}$, there exists an invertible matrix $\mathbf{U}$ and an upper-triangular Jordan matrix $\mathbf{A_c'}$ such that $\mathbf{A_c}=\mathbf{U}\mathbf{A_c'}\mathbf{U}^{-1}$. Then, equations \eqref{eqn:non:0} and \eqref{eqn:non:4} can be rewritten as
\begin{align}
\mathbf{U}^{-1}\mathbf{x_c}(t)&=e^{\mathbf{A_c'}t}\mathbf{U}^{-1}\mathbf{x_c}(0)+\int^t_0 e^{\mathbf{A_c'}(t-t')}\mathbf{U}^{-1}\mathbf{B_c} d\mathbf{W_c}(t') \nonumber \\
\mathbf{y_o}[n]&=
\int^{I}_{0}\mathbf{C_c}\mathbf{U}e^{\mathbf{A_c'}(t-I)}dt  \mathbf{U}^{-1}\mathbf{x_c}(nI-t_n) \nonumber \\
&-\int^{nI-t_n}_{(n-1)I-t_n}\int^{nI-t_n}_{t}\mathbf{C_c}\mathbf{U}e^{\mathbf{A_c'}(t-t')}\mathbf{U}^{-1}\mathbf{B_c}d\mathbf{W_c}(t')dt
+\int^{nI-t_n}_{(n-1)I-t_n} \mathbf{D_c} d\mathbf{V_c}(t). \nonumber
\end{align}
Thus, by denoting $\mathbf{x_c'}(t):=\mathbf{U}^{-1}\mathbf{x_c}(t)$, $\mathbf{B_c'}:=\mathbf{U}^{-1}\mathbf{B_c}$, and $\mathbf{C_c'}:=\mathbf{C_c}\mathbf{U}$, the system equations \eqref{eqn:contistate}, \eqref{eqn:contiob} and \eqref{eqn:conti:ysample} can be written in the following equivalent forms.
\begin{align}
&d\mathbf{x_c'}(t)=\mathbf{A_c'}\mathbf{x_c'}(t) dt + \mathbf{B_c'}d \mathbf{W_c}(t)  \nonumber \\
&\mathbf{y_c}(t)=\mathbf{C_c'}\mathbf{x_c'}(t)+\mathbf{D_c}\frac{d\mathbf{V_c}(t)}{dt}\nonumber \\
&\mathbf{y_o}[n]=\mathbf{C'}\mathbf{x_c'}(nI-t_n)+\mathbf{v}[n] \nonumber
\end{align}
where $\mathbf{C'}:=\int^I_0 \mathbf{C_c'}e^{\mathbf{A_c'}(t-I)}dt=\int^I_0 \mathbf{C_c} \mathbf{U} \mathbf{U}^{-1} e^{\mathbf{A_c}(t-I)} \mathbf{U}dt=\mathbf{C}\mathbf{U}$.

Since $\mathbf{U}$ is invertible, $(\mathbf{A_c},\mathbf{C})$ has an unobservable eigenvalue $\lambda$ if and only if $(\mathbf{A_c'},\mathbf{C'})$ has an unobservable eigenvalue $\lambda$. Moreover, since $\mathbf{x'_c}=\mathbf{U}^{-1}\mathbf{x_c}(t)$, the original intermittent system $(\mathbf{A_c},\mathbf{B_c},\mathbf{C})$ with nonuniform samples is intermittent observable if and only if the new intermittent system $(\mathbf{A_c'},\mathbf{B_c'},\mathbf{C'})$ with nonuniform samples is intermittent observable. Thus, without loss of generality, we can assume $\mathbf{A_c}$ is given in a Jordan form, which justifies (a).

Once $\mathbf{A_c}$ is given in a Jordan form, there is a natural correspondence between the eigenvalues and the states. If there is a stable eigenvalue --- i.e. the real part of the eigenvalue is negative ---, the variance of the corresponding state is uniformly bounded. Thus, we do not have to estimate the state to make the estimation error finite. In the observation $\mathbf{y}[n]$, the stable states can be considered as a part of observation noise $\mathbf{v}[n]$, and the variance of $\mathbf{v}[n]$ is still uniformly bounded (even if $\mathbf{v}[n]$ can be correlated). Therefore, we can assume (b) without loss of generality.

To justify restriction (c), we put $\mathbf{W_c}(t)=0$ for $t < 0$, $\mathbf{V_c}(t)=0$ for $t <0$, and let $\beta[n]$ be a two-sided Bernoulli process with probability $1-p_e$. Then, the resulting two-sided processes $\mathbf{x_c}(t)$ and $\mathbf{y}[n]$ are identical to the original one-sided processes except that $\mathbf{x_c}(t)$ and $\mathbf{y}[n]$ except that $\mathbf{x_c}(t)=0$ for $t \in \mathbb{R}^{--}$ and $\mathbf{y}[n]=0$ for $n \in \mathbb{Z}^{--}$.

In summary, without loss of generality we can assume that $\mathbf{A_c}$ is in a Jordan form, all eigenvalues of $\mathbf{A_c}$ are stable, and \eqref{eqn:conti:xsample} and \eqref{eqn:conti:ysample} are two-sided processes. Thus, we can assume $\mathbf{A_c} \in \mathbb{C}^{m \times m}$ and $\mathbf{C} \in \mathbb{C}^{l \times m}$ is given as follows.
\begin{align}
&\mathbf{A_c}=diag\{\mathbf{A_{1,1}},\mathbf{A_{1,2}},\cdots, \mathbf{A_{1,\nu_{1}}},\cdots,\mathbf{A_{\mu,1}},\cdots,\mathbf{A_{\mu,\nu_{\mu}}}\}
\label{eqn:conti:a2} \\
&\mathbf{C}=\begin{bmatrix}
\mathbf{C_{1,1}} & \mathbf{C_{1,2}} & \cdots & \mathbf{C_{1,\nu_{1}}} & \cdots & \mathbf{C_{\mu,1}} & \cdots & \mathbf{C_{\mu,\nu_{\mu}}}
\end{bmatrix}
\label{eqn:conti:c2} \\
&\mbox{where } \nonumber\\
&\quad\mathbf{A_{i,j}} \mbox{ is a Jordan block with eigenvalue $\lambda_{i}+j\omega_{i}$ and size $m_{i,j}$} \nonumber \\
&\quad m_{i,1} \leq m_{i,2} \leq \cdots \leq m_{i,\nu_i} \mbox{ for all }i=1,\cdots,\mu \nonumber \\
&\quad \lambda_1 \geq \lambda_2 \geq \cdots \geq \lambda_\mu \geq 0 \nonumber \\
&\quad \lambda_1+j\omega_1, \lambda_2+j\omega_2, \cdots , \lambda_\mu+j\omega_\mu \mbox{ are pairwise distinct} \nonumber \\
&\quad \mathbf{C_{i,j}}\mbox{ is a $l \times m_{i,j}$ complex matrix} \nonumber \\
&\quad \mbox{The first columns of $\mathbf{C_{i,1}},\mathbf{C_{i,2}},\cdots,\mathbf{C_{i,\nu_i}}$ are linearly independent}. \nonumber
\end{align}
Here, $\mathbf{A_{i,1}}, \cdots, \mathbf{A_{i,\nu_i}}$ are the Jordan blocks corresponding to the same eigenvalue. The Jordan blocks are sorted in a descending order in the real parts of the eigenvalues. The permutation of Jordan blocks can be justified since they are block diagonal matrices. The linear independence of $\mathbf{C_{i,1}}, \mathbf{C_{i,2}}, \cdots, \mathbf{C_{i,\nu_i}}$ comes from the observability of $(\mathbf{A_c},\mathbf{C})$ (by Theorem~\ref{thm:jordanob}).

$\bullet$ Uniform boundedness of observation noise: To prove the intermittent observability, we will propose a suboptimal maximum likelihood estimator, and analyze it. To upper bound the estimation error, we upper bound the disturbances and observation noises in the system.

By \eqref{eqn:non:1}, we have
\begin{align}
\mathbf{x_c}((n-k)I-t_{n-k})=e^{-\mathbf{A_c}(kI+t_{n-k})}\mathbf{x_c}(nI)
\underbrace{-\int^{nI}_{(n-k)I-t_{n-k}} e^{\mathbf{A_c}((n-k)I-t_{n-k}-t')} \mathbf{B_c} d\mathbf{W_c}(t')}_{:=\mathbf{w'}[n-k]}. \label{eqn:non:5}
\end{align}
By plugging this equation into \eqref{eqn:conti:ysample}, we get
\begin{align}
\mathbf{y}[n-k]&=\mathbf{C}\mathbf{x_c}((n-k)I-t_{n-k})+\mathbf{v}[n-k] \nonumber \\
&=\mathbf{C}e^{-\mathbf{A_c}(kI+t_{n-k})}\mathbf{x_c}(nI)+\underbrace{\mathbf{C}\mathbf{w'}[n-k]+\mathbf{v}[n-k]}_{:=\mathbf{v'}[n-k]}. \label{eqn:nonuniform:1}
\end{align}
We will upper  bound the variance of $\mathbf{v'}[n-k]$. First, consider the variance of $w'[n-k]$. By the assumption (b), all eigenvalues of $\mathbf{A_c}$ are unstable, and since $t_{n-k} \in [0,T]$, $((n-k)I-t_{n-k}-t')$ ranges within $[-(kI+T),0]$. Thus, there exits $p' \in \mathbb{N}$ such that
\begin{align}
\mathbb{E}[\mathbf{w'}[n-k]^\dag \mathbf{w'}[n-k]] \lesssim 1+k^{p'} \label{eqn:pprimedef2}
\end{align}
where $\lesssim$ holds for all $n$. (See Definition~\ref{def:lesssim} for the definition of $\lesssim$.)

By \eqref{eqn:non:4}, the variance of $\mathbf{v}[n]$ is uniformly bounded\footnote{To justify assumption (b), we consider the stable states as a part of observation noise $\mathbf{v}[n]$. However, this does not change the uniform boundedness since the variances of the stable states are also uniformly bounded.} for all $n$. Therefore, we have $\mathbb{E}[\mathbf{v'}[n-k]^\dag \mathbf{v'}[n-k]] \lesssim 1+k^{p'}$ for all $n$.

Moreover, since $W_c(t)$ is a standard Wiener process with unit variance, $\underset{n \in \mathbb{Z}}{\sup} \mathbb{E}[(\mathbf{x}(nI)-\mathbf{\widehat{x}}(nI))^\dag (\mathbf{x}(nI)-\mathbf{\widehat{x}}(nI))] < \infty$ implies $\underset{t \in \mathbb{R}}{\sup} \mathbb{E}[(\mathbf{x}(t)-\mathbf{\widehat{x}}(t))^\dag (\mathbf{x}(t)-\mathbf{\widehat{x}}(t))] < \infty$. Thus, it is enough to estimate the state only at discrete time steps.
%
%
%

$\bullet$ Suboptimal Maximum Likelihood Estimator: Now, we will give the suboptimal state estimator which only uses a finite number of recent observations. We first need the following key lemma.
\begin{lemma}
Let $\mathbf{A_c}$ and $\mathbf{C}$ be given as in \eqref{eqn:conti:a2} and \eqref{eqn:conti:c2}, $\beta[n]$ be a Bernoulli process with probability $1-p_e$, and $t_n$ be i.i.d.~random variables whose distribution is uniform on $[0,T]~(T>0)$.
Then, we can find $m' \in \mathbb{N}$, a polynomial $p(k)$ and a family of stopping times $\{ S(\epsilon,k): k \in \mathbb{Z}^+, 0 < \epsilon < 1 \}$ such that for all $k \in \mathbb{Z}^+$ and $0 < \epsilon < 1$ there exist $k \leq k_1 < k_2 < \cdots < k_{m'} \leq S(\epsilon,k)$ and a $m \times m'l$ matrix $\mathbf{M}$ satisfying the following four conditions:\\
(i) $\beta[k_i] = 1$ for all $1 \leq i \leq m'$\\
(ii)
$
\mathbf{M}
\begin{bmatrix}
\mathbf{C} e^{-(k_1 I + t_{k_1})\mathbf{A_c}} \\
\mathbf{C} e^{-(k_2 I + t_{k_2})\mathbf{A_c}} \\
\vdots \\
\mathbf{C} e^{-(k_{m'} I + t_{k_{m'}})\mathbf{A_c}} \\
\end{bmatrix} = \mathbf{I}_{m \times m}
$\\
(iii)
$
\left| \mathbf{M} \right|_{max} \leq \frac{p(S(\epsilon,k))}{\epsilon} e^{\lambda_1 S(\epsilon,k) I}
$
\\
(iv)
$
\lim_{\epsilon \downarrow 0} \left(\exp \limsup_{s \rightarrow \infty} \sup_{k\in \mathbb{Z^+}} \frac{1}{s} \log \mathbb{P} \{ S(\epsilon,k)-k=s \} \right) \leq p_e
$.
\label{lem:conti:mo}
\end{lemma}
\begin{proof}
See Appendix~\ref{sec:app:2}.
\end{proof}

Since we have $p_e<\frac{1}{|e^{2 \lambda_{max} I}|}=\frac{1}{e^{2 \lambda_1 I}}$, there exists $\delta > 1$ such that $\delta^5 p_e < \frac{1}{e^{2 \lambda_1 I}}$. By Lemma~\ref{lem:conti:mo}, we can find $m' \in \mathbb{N}$, $0 < \epsilon < 1$, a polynomial $p(k)$ and a family of stopping times $\{ S(n) : n \in \mathbb{Z}^+ \}$ such that for all $n$, there exist $0 \leq k_1 < k_2 < \cdots < k_{m'} \leq S(n)$ and a $m \times m'l$ matrix $\mathbf{M_n}$ satisfying the following four conditions:\\
(i') $\beta[n-k_i] = 1$ for $1 \leq i \leq m'$\\
(ii')
$
\mathbf{M_n}
\begin{bmatrix}
\mathbf{C} e^{-(k_1 I + t_{n-k_1})\mathbf{A_c}} \\
\mathbf{C} e^{-(k_2 I + t_{n-k_2})\mathbf{A_c}} \\
\vdots \\
\mathbf{C} e^{-(k_{m'} I + t_{n-k_{m'}})\mathbf{A_c}} \\
\end{bmatrix} = \mathbf{I}_{m \times m}
$\\
(iii')
$
\left| \mathbf{M_n} \right|_{max} \leq \frac{p(S(n))}{\epsilon} e^{\lambda_1 I \cdot S(n)}
$
\\
(iv')
$
\exp \left(\limsup_{s \rightarrow \infty} \sup_{n \in \mathbb{Z^+}} \frac{1}{s} \log \mathbb{P} \{ S(n)=s \} \right) \leq  \sqrt{\delta} p_e
$.

Then, here is the proposed suboptimal maximum likelihood estimator for $\mathbf{x}(nI)$:
\begin{align}
\mathbf{\widehat{x}}(nI)=\mathbf{M_n}\begin{bmatrix}
\mathbf{y}[n-k_1] \\
\mathbf{y}[n-k_2] \\
\vdots \\
\mathbf{y}[n-k_{m'}] \\
\end{bmatrix}. \label{eqn:nonuniform:2}
\end{align}
Here, $k_i$ also depends on $n$, but we omit the dependency in notation for simplicity. Notice that the number of the observation of the estimation, $m'$, is much larger than the dimension of the system, $m$. In other words, the estimator proposed here may use much more number of observations than the number of states (the number of observations that a simple matrix inverse observer needs). This is because we use successive decoding idea in the proof of Lemma~\ref{lem:conti:mo}.

$\mathbf{\bullet}$ Analysis of the estimation error: Now, we will analyze the performance of the proposed estimator.
Remind that $p'$ is defined in \eqref{eqn:pprimedef2} and $\delta > 1$
By (iv') and well-known properties of polynomial and exponential functions, we can find $c > 0$ that satisfies the following three conditions:\\
(i'') $(1+k^{p'}) \leq c \cdot \delta^k$ for all $k \geq 0$\\
(ii'') $p(k) \leq c \cdot \delta^k$ for all $k \geq 0$\\
(iii'') $\sup_{n \in \mathbb{N}}\mathbb{P}\{ S(n) =s  \} \leq c \cdot (\delta \cdot p_e)^s$ for all $s \in \mathbb{Z}^+$

Let $\mathcal{F}_{\beta}$ be the $\sigma$-field generated by $\beta[n]$ and $t_i$. Then, $k_i$, $S(n)$, and $t_i$ are deterministic variables conditioned on $\mathcal{F}_{\beta}$. The estimation error is upper bounded by
\begin{align}
\sup_{n}\mathbb{E}[|\mathbf{x}(nI)-\mathbf{\widehat{x}}(nI)|_2^2]&=\sup_{n}
\mathbb{E}[\mathbb{E}[|\mathbf{x}(nI)-\mathbf{\widehat{x}}(nI)|_2^2| \mathcal{F}_{\beta}]]\nonumber \\
&\overset{(A)}{=}\sup_{n}
\mathbb{E}[\mathbb{E}[\left|\mathbf{x}(nI)-\mathbf{M_n}(\begin{bmatrix} \mathbf{C}e^{-\mathbf{A_c}(k_1 I + t_{n-k_1})} \\ \mathbf{C}e^{-\mathbf{A_c}(k_2 I + t_{n-k_2})} \\ \vdots \\ \mathbf{C}e^{-\mathbf{A_c}(k_{m'} I + t_{n-k_{m'}})} \end{bmatrix}\mathbf{x}(nI) + \begin{bmatrix} \mathbf{v'}[n-k_1] \\ \mathbf{v'}[n-k_2] \\ \vdots \\ \mathbf{v'}[n-k_{m'}] \end{bmatrix})\right|_2^2|\mathcal{F}_{\beta}]]\nonumber \\
&\overset{(B)}{=}\sup_{n}\mathbb{E}[\mathbb{E}[\left|\mathbf{M_n}\begin{bmatrix} \mathbf{v'}[n-k_1] \\ \mathbf{v'}[n-k_2] \\ \vdots \\ \mathbf{v'}[n-k_{\sum_{1 \leq i \leq \mu}m_i'}] \end{bmatrix}\right|_2^2|\mathcal{F}_{\beta}]]   \nonumber\\
&\lesssim
\sup_{n}\mathbb{E}[|\mathbf{M_n}|_{max}^2 \cdot \mathbb{E}[\left|\begin{bmatrix} \mathbf{v'}[n-k_1] \\ \mathbf{v'}[n-k_2] \\ \vdots \\ \mathbf{v'}[n-k_{m'}] \end{bmatrix}\right|_{max}^2|\mathcal{F}_{\beta}]] \nonumber\\
&\overset{(C)}{\lesssim}
\sup_{n}\mathbb{E}[|\mathbf{M_n}|_{max}^2 \cdot (1+S(n)^{p'})^2 ] \nonumber \\
&\overset{(D)}{\leq}
\sup_{n}\mathbb{E}[\left( \frac{p(S(n))}{\epsilon} e^{\lambda_1 I \cdot S(n)} \right)^2 \cdot (1+S(n)^{p'})^2 ] \nonumber\\
&\overset{(E)}{\lesssim}
\sup_{n}\mathbb{E}[\delta^{2S(n)}\cdot e^{2\lambda_1 I \cdot S(n)} \cdot \delta^{2S(n)}] \nonumber\\
&\overset{(F)}{\lesssim}
\sum^{\infty}_{s=0} \delta^{4s}\cdot e^{2 \lambda_1 I \cdot s} \cdot (\delta \cdot p_e)^s \nonumber\\
&\overset{(G)}{=}
\sum^{\infty}_{s=0} (\delta^5 \cdot e^{2 \lambda_1 I} \cdot p_e)^s \nonumber \\
&< \infty \nonumber
\end{align}
where $\lesssim$ holds for all $n$.\\
(A): By \eqref{eqn:nonuniform:1} and \eqref{eqn:nonuniform:2}.\\
(B): By condition (ii').\\
(C): Since $\mathbb{E}[\mathbf{v'}[n-k]^\dag \mathbf{v'}[n-k]] \lesssim 1+k^{p'}$ by definition.\\
(D): By condition (iii').\\
(E): By condition (i'') and (ii'').\\
(F): By condition (iii'').\\
(G): Since we choose $\delta$ so that $\delta^5 p_e \cdot e^{2 \lambda_1 I} < 1$.

Therefore, the estimation error is uniformly bounded over $t\in \mathbb{R}^+$ when $p_e < \frac{1}{e^{2 \lambda_1 I}}$, which finishes the proof.

\subsection{Necessity Proof of Theorem~\ref{thm:nonuniform}}
\label{sec:cont:nec}
The necessity proof divides into two parts. First, we prove that if $p_e \geq \frac{1}{|e^{2 \lambda_{max} I }|}$, then the system is not intermittent observable. Second, we prove that if $(\mathbf{A_c},\mathbf{C})$ has unobservable and unstable eigenvalues --- i.e. $\exists \lambda \in \mathbb{C}^+$ such that $\begin{bmatrix} \lambda \mathbf{I} - \mathbf{A_c} \\ \mathbf{C}\end{bmatrix}$ is rank deficient --- then the system is not intermittent observable.


$\bullet$ When $p_e \geq \frac{1}{|e^{2 \lambda_{max} I }|}$: Intuitively speaking, we will give all states except the one corresponding to the maximum eigenvalue as side-information. Thus, we will reduce the problem to the scalar system discussed in Section~\ref{sec:intui}.

Formally, let $\mathbf{\Sigma_{t|t}}:=\mathbb{E}[(\mathbf{x_c}(t)-\mathbb{E}[\mathbf{x_c}(t)| \mathbf{y}^{\lfloor \frac{t}{I} \rfloor} ])(\mathbf{x_c}(t)-\mathbb{E}[\mathbf{x_c}(t)| \mathbf{y}^{\lfloor \frac{t}{I} \rfloor} ])^\dag | \mathcal{F}_{\beta}]$ where $\mathcal{F}_{\beta}$ is the $\sigma$-field generated by $\beta[n]$ and $t_i$. Notice that $\mathbf{\Sigma_{t|t}}$ is a random variable.

It is known that when $(\mathbf{A_c},\mathbf{B_c})$ is controllable, the estimation error of $\mathbf{x_c}(t)$ even based on all the causally available information $\mathbf{y_c}(0:t)$ is positive definite when $t$ is large enough. Therefore, there exists $t' > 0$ and $\sigma^2 > 0$ such that for all $t \geq t'$,
$\mathbf{\Sigma_{t|t}} \succeq \sigma^2 \mathbf{I}$ with probability one. Let $\mathbf{e}$ be a right eigenvector of $\mathbf{A_c}$ associated with the eigenvalue $\lambda_{max}$, i.e. $\mathbf{A_c}\mathbf{e}=\lambda_{max} \mathbf{e}$. Then, we can find $\sigma'^2 > 0$ such that for all $t \geq t'$, $\mathbf{\Sigma_{t|t}} \succeq \sigma'^2 \mathbf{e}\mathbf{e}^\dag$ with probability one.

Define the stopping time $S_n' := \inf\{ k \in \mathbb{Z}^+| \beta[n-k]=1 \}$ as the time until the most recent observation.

The observations between discrete time $n-S_n'+1$ and $n$ are all erased. This implies the received observations at discrete time $n$ are independent from $\mathbf{y_c}((n-S_n)I:nI)$. 
Thus, conditioned on $(n-S_n')I \geq t'$, $\mathbf{\Sigma_{nI|nI}}$ is lower bounded as follows with probability one.\footnote{The lower bound does not hold for $\Re (\lambda) =0$ which induces $p_e=1$. However, in this case we do not have any observation, so trivially the system is unstable.}
\begin{align}
\mathbb{E}[\mathbf{\Sigma_{nI|nI}}| S_n', (n-S_n')I \geq t'] & \succeq (e^{\mathbf{A_c}(S_n' I)}) \mathbf{\Sigma_{(n-S_n')I|(n-S_n')I}} (e^{\mathbf{A_c}(S_n' I)})^{\dag}\\
& \succeq \sigma'^2 (e^{\mathbf{A_c}(S_n' I)}) \mathbf{e}\mathbf{e}^{\dag} (e^{\mathbf{A_c}(S_n' I)})^{\dag} \\
& \succeq \sigma'^2 |e^{2\lambda_{max}I}|^{S_n'} \mathbf{e} \mathbf{e}^\dag
\end{align}
Here we use the fact that when $\mathbf{e}$ is an eigenvector of $\mathbf{A_c}$ associated with an eigenvalue $\lambda_{max}$, $\mathbf{e}$ is also an eigenvector of $e^{\mathbf{A_c}t}$ associated with the eigenvalue $e^{\lambda_{max}t}$ for all $t$.

Since $p_e \geq \frac{1}{|e^{2 \lambda_{max}I}|}$, the average estimator error is lower bounded as follows:
\begin{align}
&\mathbb{E}[(\mathbf{x_c}(nI)-\mathbb{E}[ \mathbf{x_c}(nI) | \mathbf{y}^n ])^\dag (\mathbf{x_c}(nI)-\mathbb{E}[ \mathbf{x_c}(nI) | \mathbf{y}^n ])]\\
&\geq \mathbb{E}[ \sigma'^2 |e^{2\lambda_{max}I}|^{S_n'} |\mathbf{e}|^2 \cdot \mathbf{1}( (n-S_n')I \geq t')  ]\\
&\geq \sigma'^2 |\mathbf{e}|^2 \cdot \sum_{0 \leq s \leq \lfloor n - \frac{t}{I} \rfloor} |e^{2 \lambda_{max} I}|^s \cdot (1-p_e)p_e^s \\
&\geq \sigma'^2 |\mathbf{e}|^2 \cdot (1-p_e) \cdot (\lfloor n - \frac{t}{I} \rfloor + 1)
\end{align}
Thus, the estimation error goes to infinity as $n \rightarrow \infty$, so the system is not intermittently observable.

$\bullet$ When $(\mathbf{A_c},\mathbf{C})$ has unobservable and unstable eigenvalues: Now, we prove that if $(\mathbf{A_c},\mathbf{C})$ has unobservable and unstable eigenvalues, the system is not intermittent observable. This proof seems trivial, but the original continuous-time system $(\mathbf{A_c},\mathbf{C_c})$ can still be observable while the sampled system $(\mathbf{A_c},\mathbf{C})$ is not. Thus, it still needs justification.

Let $\lambda \in \mathbb{C}^+$ be the unobservable and unstable eigenvalue. Then, $\begin{bmatrix} \lambda \mathbf{I}- \mathbf{A_c} \\ \mathbf{C} \end{bmatrix}$ is rank deficient, and we can find $\mathbf{i}$ such that $\begin{bmatrix} \lambda \mathbf{I}-\mathbf{A_c} \end{bmatrix} \mathbf{i} = \mathbf{0}$. Then, $\mathbf{i}$ satisfies $\mathbf{C}\mathbf{i}=\mathbf{0}$, $\mathbf{A_c}\mathbf{i}=\lambda \mathbf{i}$, and we can notice that $\mathbf{C}e^{\mathbf{A_c}t}\mathbf{i}=e^{\lambda t}\mathbf{C} \mathbf{i}= \mathbf{0}$. We will prove that the uncertainty in the direction $\mathbf{i}$ is not observable by any observations.

By the controllability of $(\mathbf{A_c},\mathbf{B_c})$, as above there exists $t'$ such that for all $t \geq t'$, $\mathbf{x_c}(t)-\mathbb{E}[\mathbf{x_c}(t)|\mathbf{y_c}(0:t)]$ has a positive definite covariance matrix. Therefore, we can write $\mathbf{x_c}(t)-\mathbb{E}[\mathbf{x_c}(t)|\mathbf{y_c}(0:t)]=\mathbf{i} \cdot x_c'(t)+ \mathbf{x_c''}(t)$ where $x_c'(t)$, $\mathbf{x_c''}(t)$ and $\mathbf{y_c}(0:t)$ are independent and $\mathbb{E}[|x_c'(t)|^2]\geq \sigma''^2$ for some $\sigma''^2>0$ and all $t \geq t'$.

Then, we will prove that the sampled observations are independent from $x_c'(t)$. By \eqref{eqn:non:0} and \eqref{eqn:non:4}, for all $\tau \leq (n-1)I-t_n$ we have
\begin{align}
\mathbf{y_o}[n] &= \mathbf{C}(e^{\mathbf{A_c}(nI-t_n-\tau)}(\mathbf{x_c}(\tau)+\int^{nI-t_n}_{\tau} e^{\mathbf{A_c}(\tau - t')}\mathbf{B_c} d \mathbf{W_c}(t'))) \\
&-\int^{nI-t_n}_{(n-1)I-t_n} \int^{nI-t_n}_{t} \mathbf{C_c} e^{\mathbf{A_c}(t-t')} \mathbf{B_c} d\mathbf{W_c}(t')dt
+ \int^{nI-t_n}_{(n-1)I-t_n} \mathbf{D_c}d\mathbf{V_c}(t)\\
&= \mathbf{C}(e^{\mathbf{A_c}(nI-t_n-\tau)}(\mathbf{i} \cdot x_c'(\tau) + \mathbf{x_c''}(\tau) + \mathbf{E}[\mathbf{x_c}(\tau)|\mathbf{y_c}(0:\tau)] +\int^{nI-t_n}_{\tau} e^{\mathbf{A_c}(\tau - t')}\mathbf{B_c} d \mathbf{W_c}(t'))) \\
&-\int^{nI-t_n}_{(n-1)I-t_n} \int^{nI-t_n}_{t} \mathbf{C_c} e^{\mathbf{A_c}(t-t')} \mathbf{B_c} d\mathbf{W_c}(t')dt
+ \int^{nI-t_n}_{(n-1)I-t_n} \mathbf{D_c}d\mathbf{V_c}(t)\\
&= \mathbf{C}(e^{\mathbf{A_c}(nI-t_n-\tau)}( \mathbf{x_c''}(\tau) + \mathbf{E}[\mathbf{x_c}(\tau)|\mathbf{y_c}(0:\tau)] +\int^{nI-t_n}_{\tau} e^{\mathbf{A_c}(\tau - t')}\mathbf{B_c} d \mathbf{W_c}(t'))) \\
&-\int^{nI-t_n}_{(n-1)I-t_n} \int^{nI-t_n}_{t} \mathbf{C_c} e^{\mathbf{A_c}(t-t')} \mathbf{B_c} d\mathbf{W_c}(t')dt
+ \int^{nI-t_n}_{(n-1)I-t_n} \mathbf{D_c}d\mathbf{V_c}(t)
\end{align}
where the last equality comes from $\mathbf{C}e^{\mathbf{A_c}t}\mathbf{i}=0$. Moreover, by the causality and definitions, the last equation is independent from $\mathbf{x_c''}(\tau)$.

Now, we will prove that the uncertainty $\mathbf{x_c''}(\tau)$ can be arbitrarily amplified. Since $t_{i}$ are uniform random variables on $[0,T]$, there exists a positive probability such that $(n-1)I-t_n \leq (n+n'-1)I - t_{n+n'}$ for all $n' \in \mathbb{N}$. Denote such an event as $E$. Then, by choosing $n$ large enough so that $(n-1)I-t_n \geq t'$, we have the following lower bound on the estimation error for all $t \geq (n-1)I-t_n$:
\begin{align}
&\mathbb{E}[|\mathbf{x_c}(t)-\mathbb{E}[\mathbf{x_c}(t)|\mathbf{y}^{\lfloor \frac{t}{I} \rfloor}]|^2] \\
&\geq \mathbb{E}[
|\mathbf{x_c}(t)-\mathbb{E}[\mathbf{x_c}(t)|\mathbf{y}^{\lfloor \frac{t}{I} \rfloor}]|^2
 | E] \mathbb{P}(E) \\
&\overset{(a)}{\geq} \mathbb{E}[ |e^{\mathbf{A_c}(t-((n-1)I-t_n))}\mathbf{i} \cdot x_c''((n-1)I-t_n)|^2 |E]
\mathbb{P}(E) \\
&=|e^{\lambda(t-((n-1)I-T))}  \cdot \mathbf{i}|^2 \sigma''^2  \cdot \mathbb{P}(E) \label{cont:nec:lowerbound}
\end{align}
(a): By \eqref{eqn:non:0}, $\mathbf{x_c}(t)= e^{\mathbf{A_c}(t-((n-1)I-t_n))}\mathbf{x_c}((n-1)I-t_n)+
\int^t_{(n-1)I-t_n} e^{\mathbf{A_c}((n-1)I-t_n -t')} \mathbf{B_c} d \mathbf{W_c}(t')
$. Moreover, $x''_c((n-1)I-t_n)$ is independent from $\mathbf{y_c}(0:(n-1)I-t_n)$ and $y_o[n],y_o[n+1] \cdots$.

Since we can choose $t$ arbitrarily large, this finishes the proof for $\Re (\lambda) > 0$. To prove for the case of $\Re (\lambda) = 0$, we can bound \eqref{cont:nec:lowerbound} more carefully and justify that independent estimation errors accumulates in the direction of $\mathbf{i}$. We omit the proof here since the argument is essentially equivalent to that of the well-known fact that an eigenvalue with zero real part is unstable in continuous-time systems.

\subsection{Sufficiency Proof of Theorem~\ref{thm:mainsingle} (Discrete-Time Systems)}
\label{sec:dis:suff}
We will prove that if $p_e < \frac{1}{\underset{1 \leq i \leq \mu}{\max} |\lambda_{i,1}|^{2 \frac{p_i}{l_i}}}$ then the system is intermittent observable.

$\bullet$ Reduction to a Jordan form matrix $\mathbf{A}$: As in Section~\ref{sec:cont:suf}, we will restrict attention to system equations \eqref{eqn:dis:system} and \eqref{eqn:dis:system2} with the following properties, and justify that such a restriction is without loss of generality and does not change the intermittent observability.\\
(a) The system matrix $\mathbf{A}$ is a Jordan form matrix.\\
(b) All eigenvalues of $\mathbf{A}$ are unstable, i.e. the magnitude of all eigenvalues are greater or equal to $1$.\\
(c) \eqref{eqn:dis:system} and \eqref{eqn:dis:system2} can be extended to two-sided processes.

The restriction (a) can be justifies by a similarity transform~\cite{Chen}. It is known~\cite{Chen} that for any square matrix $\mathbf{A}$, there exists an invertible matrix $\mathbf{U}$ and an upper-triangular Jordan matrix $\mathbf{A'}$ such that $\mathbf{A}=\mathbf{U}\mathbf{A'}\mathbf{U}^{-1}$. Then, the system equations \eqref{eqn:dis:system} and \eqref{eqn:dis:system2} can be rewritten as:
\begin{align}
&\mathbf{U}^{-1}\mathbf{x}[n+1]=\mathbf{A'}\mathbf{U}^{-1}\mathbf{x}[n]+\mathbf{U}^{-1}\mathbf{B}\mathbf{w}[n] \nonumber \\
&\mathbf{y}[n]=\beta[n](\mathbf{C}\mathbf{U}\mathbf{U}^{-1}\mathbf{x}[n]+\mathbf{v}[n]). \nonumber
\end{align}
Thus, by denoting $\mathbf{x'}[n]:=\mathbf{U}^{-1}\mathbf{x}[n]$, $\mathbf{B'}:=\mathbf{U}^{-1}\mathbf{B}$, and $\mathbf{C'}:=\mathbf{C}\mathbf{U}$, we get
\begin{align}
&\mathbf{x'}[n+1]=\mathbf{A'}\mathbf{x'}[n]+\mathbf{B'}\mathbf{w}[n] \nonumber \\
&\mathbf{y}[n]=\beta[n](\mathbf{C'}\mathbf{x'}[n]+\mathbf{v}[n]). \nonumber
\end{align}

Since $\mathbf{U}$ is invertible, the controllability of $(\mathbf{A},\mathbf{B},\mathbf{C})$ remains the same for the new intermittent system $(\mathbf{A'},\mathbf{B'},\mathbf{C'})$. Moreover, since $\mathbf{x'}[n]=\mathbf{U}^{-1}\mathbf{x}[n]$, the original intermittent system is intermittent observable if and only if the new intermittent system is intermittent observable. Thus, without loss of generality, we can assume that $\mathbf{A}$ is given in a Jordan form, which justifies (a).

Once $\mathbf{A}$ is given in Jordan form, there is a natural correspondence between the eigenvalues and the states. If there is a stable eigenvalue --- i.e. the magnitude of the eigenvalue is less than $1$ ---, the variance of the corresponding state is uniformly bounded. Thus, we do not have to estimate that particular state to make the estimation error finite. In the observation $\mathbf{y}[n]$, the stable states can be considered as a part of observation noise $\mathbf{v}[n]$, and the variance of $\mathbf{v}[n]$ is still uniformly bounded (even if $\mathbf{v}[n]$ can be correlated). Therefore, we can assume (b) without loss of generality.

To justify restriction (c), rewrite \eqref{eqn:dis:system} as
\begin{align}
\mathbf{x}[n+1]=\mathbf{A}\mathbf{x}[n]+\mathbf{I}\mathbf{w'}[n] \nonumber
\end{align}
where $\mathbf{w'}[n]=\mathbf{B}\mathbf{w}[n]$ for $n \geq 0$. Let $\mathbf{w'}[-1]=\mathbf{x}[0]$, $\mathbf{w}[n]=\mathbf{0}$ for $n < -1 $, and $\mathbf{v}[n]$ for $n < 0$. We also extend $\beta[n]$ to a two-sided Bernoulli process with probability $1-p_e$. Then, the resulting two-sided processes $\mathbf{x}[n]$ and $\mathbf{y}[n]$ are identical to the original one-sided processes except that $\mathbf{x}[n]=\mathbf{0}$ and $\mathbf{y}[n]=\mathbf{0}$ for $n \in \mathbb{Z}^{--}$.

In summary, without loss of generality we can assume that $\mathbf{A}$ is in a Jordan form, all eigenvalues of $\mathbf{A}$ is stable, and \eqref{eqn:dis:system} and \eqref{eqn:dis:system2} are two-sided process. Therefore, we can assume that $\mathbf{A} \in \mathbb{C}^{m \times m}$ and $\mathbf{C} \in \mathbb{C}^{l \times m}$ are given as
\begin{align}
&\mathbf{A}=diag\{ \mathbf{A_{1,1}}, \mathbf{A_{1,2}}, \cdots, \mathbf{A_{1,\nu_1}}, \cdots, \mathbf{A_{\mu,1}}, \cdots, \mathbf{A_{\mu,\nu_\mu}}\} \nonumber \\
&\mathbf{C}=\begin{bmatrix} \mathbf{C_{1,1}} & \mathbf{C_{1,2}} & \cdots & \mathbf{C_{1,\nu_1}} & \cdots & \mathbf{C_{\mu,1}} & \cdots & \mathbf{C_{\mu,\nu_\mu}} \end{bmatrix} \nonumber \\
&\mbox{where} \nonumber \\
&\quad \mbox{$\mathbf{A_{i,j}}$ is a Jordan block with an eigenvalue $\lambda_{i,j}$ and size $m_{i,j}$} \nonumber \\
&\quad m_{i,1} \geq m_{i,2} \geq \cdots \geq m_{i,\nu_i} \mbox{ for all }i=1,\cdots,\mu \nonumber \\
&\quad |\lambda_{1,1}| \geq |\lambda_{2,1}| \geq \cdots \geq |\lambda_{\mu,1}| \geq 1 \nonumber \\
&\quad \{ \lambda_{i,1},\cdots, \lambda_{i,\nu_i} \} \mbox{ is cycle with length $\nu_i$ and period $p_i$}\nonumber \\
&\quad \mbox{For $i \neq i'$, $\{\lambda_{i,j},\lambda_{i',j'} \}$ is not a cycle} \nonumber \\
&\quad \mbox{$\mathbf{C_{i,j}}$ is a $l \times m_{i,j}$ complex matrix}.\label{eqn:ac:jordan}
\end{align}
Here, $\mathbf{A_{i,1}},\cdots, \mathbf{A_{i,\nu_i}}$ are the Jordan blocks corresponding to the same eigenvalue cycle.  The Jordan blocks are sorted in descending order by the magnitude of the eigenvalues. Such permutation of Jordan blocks can be justified since Jordan forms are block diagonal matrices.

Like \eqref{eqn:ac2:jordan:thm}, \eqref{eqn:def:lprime:thm}, we also define $\mathbf{A_i}$, $\mathbf{C_i}$, and $l_i$ as follows.
\begin{align}
&\mathbf{A_i}=diag\{ \lambda_{i,1},\cdots, \lambda_{i,\nu_i} \}\nonumber \\
&\mathbf{C_i}=\begin{bmatrix} \left(\mathbf{C_{i,1}}\right)_1 & \cdots & \left(\mathbf{C_{i,\nu_i}}\right)_1 \end{bmatrix} \nonumber\\
&\mbox{where $\left(\mathbf{C_{i,j}}\right)_1$ is the first column of $\mathbf{C_{i,j}}$.} \label{eqn:ac2:jordan}
\end{align}
$l_i$ is the minimum cardinality among the sets $S' \subseteq \{ 0,1,\cdots,p_i-1 \}$ whose resulting $S:=\{ 0,1,\cdots, p_i-1 \} \setminus S'=\{s_1,s_2,\cdots,s_{|S|} \}$ makes
\begin{align}
\begin{bmatrix}
\mathbf{C_i}\mathbf{A_i}^{s_1}\\
\mathbf{C_i}\mathbf{A_i}^{s_2}\\
\vdots \\
\mathbf{C_i}\mathbf{A_i}^{s_{|S|}}
\end{bmatrix} \label{eqn:def:lprime}
\end{align}
be rank deficient, i.e. the rank of the matrix~\eqref{eqn:def:lprime} is strictly less than $\nu_i$.

Moreover, in \eqref{eqn:dis:systemconst1}, we already assumed that there exists a finite $\sigma > 0$ such that
\begin{align}
&\sup_{n \in \mathbb{Z}} \mathbb{E}[\mathbf{w}[n]\mathbf{w}[n]^\dag] \preceq \sigma^2 \mathbf{I} \nonumber \\
&\sup_{n \in \mathbb{Z}} \mathbb{E}[\mathbf{v}[n]\mathbf{v}[n]^\dag] \preceq \sigma^2 \mathbf{I}. \label{eqn:dis:suf:1}
\end{align}

$\bullet$ Uniform boundedness of observation noise: To prove intermittent observability, we will propose a suboptimal maximum likelihood estimator, and analyze it. We first have to upper bound the disturbances and observation noises in the system. Following the same steps of \eqref{eqn:intui:3}, we can derive
\begin{align}
\mathbf{y}[n-k]=\mathbf{C}\mathbf{A}^{-k}\mathbf{x}[n]-\underbrace{(\mathbf{C}\mathbf{A}^{-1}\mathbf{w}[n-k]+\cdots+\mathbf{C}\mathbf{A}^{-k}\mathbf{w}[n-1]-\mathbf{v}[n-k])}_{\mathbf{v'}[n-k]}.
\label{eqn:dis:suf:11}
\end{align}
The invertibility of $\mathbf{A}$ is comes from assumption (b). Moreover, since all eigenvalues of $\mathbf{A}$ are unstable, by \eqref{eqn:dis:suf:1} we can find $p' \in \mathbb{N}$ such that
\begin{align}
\mathbb{E}[\mathbf{v'}[n-k]^\dag \mathbf{v'}[n-k]] \lesssim 1 + k^{p'} \label{eqn:pprimedef}
\end{align}
where $\lesssim$ holds for all $n, k (k \leq n)$.

$\bullet$ Suboptimal Maximum Likelihood Estimator: Now, we will give the suboptimal estimator for the state which only uses a finite number of recent observations. We first need the following lemma which plays a parallel role to Lemma~\ref{lem:conti:mo}.
\begin{lemma}
Let $\mathbf{A}$ and $\mathbf{C}$ be given as in \eqref{eqn:ac:jordan}, \eqref{eqn:ac2:jordan} and \eqref{eqn:def:lprime}, and $\beta[n]$ be a Bernoulli process with probability $1-p_e$. Then, we can find $m_1',\cdots,m_\mu' \in \mathbb{N}$, polynomials $p_1(k),\cdots,p_\mu(k)$ and families of stopping times $\{ S_1(\epsilon,k): k \in \mathbb{Z}^+ , 0 < \epsilon < 1 \},\cdots,\{ S_\mu(\epsilon,k): k \in \mathbb{Z}^+ , 0 < \epsilon < 1 \}$ such that for all $k \in \mathbb{Z}^+$ and $0 < \epsilon < 1$  there exist $k \leq k_1 < \cdots < k_{m_1'} \leq S_1(\epsilon,k) < k_{m_1'+1}< \cdots < k_{\sum_{1 \leq i \leq \mu }m_i' } \leq S_{\mu}(\epsilon,k) $ and a $m \times (\sum_{1 \leq i \leq \mu } m_i')l$ matrix $\mathbf{M}$ satisfying the following conditions:\\
(i) $\beta[k_i]=1$ for $1 \leq i \leq \sum_{1 \leq i \leq \mu} m_i'$\\
(ii) $\mathbf{M}
\begin{bmatrix}
\mathbf{C} \mathbf{A}^{-k_1}\\
\mathbf{C} \mathbf{A}^{-k_2}\\
\vdots \\
\mathbf{C} \mathbf{A}^{-k_{\sum_{1 \leq i \leq \mu}m_i'}}\\
\end{bmatrix}= \mathbf{I}_{m \times m}
$\\
(iii)
$
\left| \mathbf{M} \right|_{max} \leq \max_{1 \leq i \leq \mu} \left\{
\frac{p_i(S_i(\epsilon,k))}{\epsilon} |\lambda_{i,1}|^{S_i(\epsilon,k)}
\right\}
$
\\
(iv)
$
\lim_{\epsilon \downarrow 0} \exp \left(\limsup_{s \rightarrow \infty} \sup_{k \in \mathbb{Z}^+}\frac{1}{s} \log \mathbb{P} \{ S_i(\epsilon,k)-k=s \} \right) \leq
\max_{1 \leq j \leq i} \left\{ p_e^{\frac{l_j}{p_j}} \right\}
$ for $1 \leq i \leq \mu$\\
(v)
$
\lim_{\epsilon \downarrow 0} \exp \left(\limsup_{s \rightarrow \infty} \esssup \frac{1}{s} \log \mathbb{P} \{
S_a(\epsilon,k)-S_b(\epsilon,k)=s| \mathcal{F}_{S_b}
\} \right) \leq \max_{b < i \leq a} \left\{ p_e^{\frac{l_i}{p_i}} \right\} $ for $1 \leq b < a \leq \mu$
where $\mathcal{F}_{S_i}$ is the $\sigma$-field generated by $S_i(\epsilon,k)$.
\label{lem:dis:achv}
\end{lemma}
\begin{proof}
See Appendix~\ref{sec:app:cycleproof}.
\end{proof}

Since $p_e < \frac{1}{\underset{1 \leq i \leq \mu}{\max} |\lambda_{i,1}|^{2\frac{p_i}{l_i}}}$, there exists $\delta > 1$ such that $\delta^5 \cdot \underset{1 \leq i \leq \mu}{\max} p_e^{\frac{l_i}{p_i}} |\lambda_{i,1}|^2 < 1$. By Lemma~\ref{lem:dis:achv}, we can find $m_1',\cdots,m_{\mu}' \in \mathbb{N}$, $0<\epsilon<1$, polynomials $p_1(k),\cdots,p_{\mu}(k)$, and a family of stopping times $\{(S_1(n),\cdots, S_\mu(n)):n \in \mathbb{Z}^+ \}$ such that $\forall n$ there exist $0 \leq k_1 < \cdots < k_{m_1'} \leq S_1(n) < k_{m_1'+1}< \cdots < k_{\sum_{1 \leq i \leq \mu }m_i' } \leq S_{\mu}(n)$ and a $m \times (\sum_{1 \leq i \leq \mu } m_i')l$ matrix $\mathbf{M_n}$ satisfying the following conditions:\\
(i') $\beta[n-k_i]=1$ for $1 \leq i \leq \sum_{1 \leq i \leq \mu}m_i'$\\
(ii') $\mathbf{M_n}
\begin{bmatrix}
\mathbf{C} \mathbf{A}^{-k_1}\\
\mathbf{C} \mathbf{A}^{-k_2}\\
\vdots \\
\mathbf{C} \mathbf{A}^{-k_{\sum_{1 \leq i \leq \mu}m_i'}}\\
\end{bmatrix}= \mathbf{I}_{m \times m}$\\
(iii') $
\left| \mathbf{M_n} \right|_{max} \leq \max_{1 \leq i \leq \mu} \left\{
\frac{p_i(S_i(n))}{\epsilon} |\lambda_{i,1}|^{S_i(n)}
\right\}
$\\
(iv')
$
\exp \left( \limsup_{s \rightarrow \infty} \frac{1}{s} \log \mathbb{P} \{ S_i(n)=s \} \right) \leq \sqrt{\delta} \cdot
\max_{1 \leq j \leq i} \left\{ p_e^{\frac{l_j}{p_j}} \right\}
$ for $1 \leq i \leq \mu$\\
(v')
$
\exp \left( \limsup_{s \rightarrow \infty} \esssup \frac{1}{s} \log \mathbb{P} \{
S_a(n)-S_b(n)=s| \mathcal{F}_{S_b}
\} \right) \leq \sqrt{\delta} \cdot \max_{b < i \leq a} \left\{ p_e^{\frac{l_i}{p_i}} \right\} $ for $1 \leq b < a \leq \mu$
where $\mathcal{F}_{S_i}$ is the $\sigma$-field generated by $\beta[n-S_i(n)],\beta[n-S_i(n)+1],\cdots, \beta[n]$.

Then, here is the proposed suboptimal maximum likelihood estimator for $\mathbf{x}[n]$:
\begin{align}
\mathbf{\widehat{x}}[n]=\mathbf{M_n}\begin{bmatrix} \mathbf{y}[n-k_1] \\ \mathbf{y}[n-k_2] \\ \vdots \\ \mathbf{y}[n-k_{\sum_{1 \leq i \leq \mu} m_i'}]  \end{bmatrix}
\label{eqn:dis:suf:12}
\end{align}
Here, $k_i$ also depends on $n$, but we omit the dependency in notation for simplicity. Notice that the number of observations that this estimator uses, $k_{\sum_{1 \leq i \leq \mu} m_i'}$, can be much larger than the dimension of the system, $m$. In other words, the estimator proposed here may use much more number of observations than the number of states (the number of observations that a simple matrix inverse observer needs). This is because we use successive decoding idea in the proof of Lemma~\ref{lem:conti:mo}.

$\bullet$ Analysis of the estimation error: Now, we will analyze the performance of the proposed estimator. Remind that $p'$ is defined in \eqref{eqn:pprimedef} and $\delta > 1$. By  (iv') and (v'), we can find $c>0$ that satisfies the following four conditions:\\
(i'') $(1+k^{p'}) \leq c \cdot \delta^k$ for all $k \geq 0$\\
(ii'') $p_i(k) \leq c \cdot \delta^k$ for all $1 \leq i \leq \mu$ and $k \geq 0$\\
(iii'') $\mathbb{P}\{ S_i(n) =s \} \leq c \cdot (\delta \cdot \max_{1 \leq j \leq i} \left\{ p_e^{\frac{l_j}{p_j}} \right\})^s$ for all $1 \leq i \leq \mu$ and $s \in \mathbb{Z}^+$\\
(iv'') $\mathbb{P}\{ S_a(n)-S_b(n)=s | \mathcal{F}_{S_b} \} \leq c \cdot (\delta \cdot \max_{b < i \leq a} \left\{ p_e^{\frac{l_i}{p_i}} \right\})^s$ for all $1 \leq b < a \leq \mu$ and $s \in \mathbb{Z}^+$.

Let $\mathcal{F}_{\beta}$ be the $\sigma$-field generated by $\beta[n]$. Then, $k_i$ and $S_i$ are deterministic variables conditioned on $\mathcal{F}_{\beta}$. The estimation error is upper bounded by
\begin{align}
\mathbb{E}[|\mathbf{x}[n]-\mathbf{\widehat{x}}[n]|_2^2]&=
\mathbb{E}[\mathbb{E}[|\mathbf{x}[n]-\mathbf{\widehat{x}}[n]|_2^2| \mathcal{F}_{\beta}]] \nonumber \\
&\overset{(A)}{=}
\mathbb{E}[\mathbb{E}[\left|\mathbf{x}[n]-\mathbf{M_n}(\begin{bmatrix} \mathbf{C}\mathbf{A}^{-k_1} \\ \mathbf{C}\mathbf{A}^{-k_2} \\ \vdots \\ \mathbf{C}\mathbf{A}^{-k_{\sum_{1 \leq i \leq \mu}m_i'}} \end{bmatrix}\mathbf{x}[n]-\begin{bmatrix} \mathbf{v'}[n-k_1] \\ \mathbf{v'}[n-k_2] \\ \vdots \\ \mathbf{v'}[n-k_{\sum_{1 \leq i \leq \mu}m_i'}] \end{bmatrix})\right|_2^2|\mathcal{F}_{\beta}]] \nonumber \\
&\overset{(B)}{=}\mathbb{E}[\mathbb{E}[\left|\mathbf{M_n}\begin{bmatrix} \mathbf{v'}[n-k_1] \\ \mathbf{v'}[n-k_2] \\ \vdots \\ \mathbf{v'}[n-k_{\sum_{1 \leq i \leq \mu}m_i'}] \end{bmatrix}\right|_2^2|\mathcal{F}_{\beta}]] \nonumber \\
&\lesssim
\mathbb{E}[|\mathbf{M_n}|_{max}^2 \cdot \mathbb{E}[\left|\begin{bmatrix} \mathbf{v'}[n-k_1] \\ \mathbf{v'}[n-k_2] \\ \vdots \\ \mathbf{v'}[n-k_{\sum_{1 \leq i \leq \mu}m_i'}] \end{bmatrix}\right|_{max}^2|\mathcal{F}_{\beta}]] \nonumber \\
&\overset{(C)}{\lesssim}
\mathbb{E}[|\mathbf{M_n}|_{max}^2 \cdot (1+S_{\mu}^{p'}(n))^2 ] \nonumber \\
&\overset{(D)}{\leq}
\mathbb{E}[\max_{1 \leq i \leq \mu} \{\left( \frac{p_i(S_i(n))}{\epsilon} |\lambda_{i,1}|^{S_i(n)}\right)^2 \} \cdot (1+S_{\mu}^{p'}(n))^2 ] \nonumber \\
&\leq
\sum_{1 \leq i \leq \mu}\mathbb{E}[ \left( \frac{p_i(S_i(n))}{\epsilon} |\lambda_{i,1}|^{S_i(n)}\right)^2 \cdot (1+S_{\mu}^{p'}(n))^2 ]\nonumber \\
&\overset{(E)}{\lesssim}
\sum_{1 \leq i \leq \mu}\mathbb{E}[ \delta^{2 S_i(n)} \cdot |\lambda_{i,1}|^{2 S_i(n)} \cdot \delta^{2 S_\mu(n)} ] \nonumber \\
&=\sum_{1 \leq i \leq \mu}\mathbb{E}[ \delta^{4 S_i(n)} \cdot |\lambda_{i,1}|^{2 S_i(n)} \cdot \mathbb{E}[ \delta^{2 (S_\mu(n)-S_i(n))}|\mathcal{F}_{S_i(n)}] ] \nonumber \\
&\overset{(F)}{\lesssim} \sum_{1 \leq i \leq \mu}\mathbb{E}[
\delta^{4S_i(n)} \cdot |\lambda_{i,1}|^{2S_i(n)} \cdot \sum^{\infty}_{s=0} \delta^{2s} \cdot (\delta \cdot \max_{1 \leq i \leq \mu}\{ p_e^{\frac{l_i}{p_i}} \})^s
] \nonumber \\
&\overset{(G)}{\lesssim} \sum_{1 \leq i \leq \mu}\mathbb{E}[ \delta^{4S_i(n)} \cdot |\lambda_{i,1}|^{2S_i(n)} ] \nonumber\\
&\overset{(H)}{\lesssim} \sum_{1 \leq i \leq \mu} \sum^{\infty}_{s=0} \delta^{4s} \cdot |\lambda_{i,1}|^{2s} \cdot (\delta \cdot \max_{1 \leq j \leq i} \left\{ p_e^{\frac{l_i}{p_i}} \right\} )^s \nonumber \\
&= \sum_{1 \leq i \leq \mu} \sum^{\infty}_{s=0} (\delta^5 \cdot |\lambda_{i,1}|^2 \cdot \max_{1 \leq j \leq i} \left\{ p_e^{\frac{l_j}{p_j}} \right\} )^s \nonumber \\
&\overset{(I)}{<} \infty  \nonumber
\end{align}
where $\lesssim$ holds for all $n$.\\
(A): By \eqref{eqn:dis:suf:11} and \eqref{eqn:dis:suf:12}.\\
(B): By condition (ii').\\
(C): Since  $\mathbb{E}[\mathbf{v'}[n-k]^\dag \mathbf{v'}[n-k]] \lesssim 1 + k^{p'}$ by the definition of $p'$ of \eqref{eqn:pprimedef}, and thus each element of the $\mathbf{v'}[n]$ vector obeys max bound.\\
(D): By condition (iii').\\
(E): By condition (i'') and (ii'').\\
(F): By condition (iv'').\\
(G): Since $\delta^5 \cdot \underset{1 \leq i \leq \mu}{\max} p_e^{\frac{l_i}{p_i}} |\lambda_{i,1}|^2 < 1$.\\
(H): By condition (iii'').\\
(I): Since $\delta^5 \cdot \underset{1 \leq i \leq \mu}{\max} p_e^{\frac{l_i}{p_i}} |\lambda_{i,1}|^2 < 1$.


Therefore, the estimation error is uniformly bounded over $n$ when $p_e < \frac{1}{\underset{1 \leq i \leq \mu}{\max} |\lambda_{i,1}|^{2 \frac{p_i}{l_i}}}$, which finishes the proof.

\subsection{Necessity Proof of Theorem~\ref{thm:mainsingle}}
\label{sec:dis:nece}
Intuitively, we will give all states except the ones that corresponds to the bottleneck eigenvalue cycle as side-information to the estimator. Then, the problem reduces to the single eigenvalue cycle one discussed in Section~\ref{sec:powerproperty}, and we can prove the estimation error diverges similarly. This argument works for $p_e > \frac{1}{\max_i |\lambda_{i,1}|^{2 \frac{p_i}{l_i}}}$, since we can show that a single additional disturbance $\mathbf{w}[n]$ grows exponentially. However, for the equality case $p_e = \frac{1}{\max_i |\lambda_{i,1}|^{2 \frac{p_i}{l_i}}}$, the proof can be more complicated since not a single disturbance but the sum of disturbances algebraically diverges to infinity.

So, to make this argument complete and rigorous, we will analyze the optimal estimator, and prove that its estimation error diverges when the condition of the lemma is violated.

It is well-known that the optimal estimator is the Kalman filter and it can be written in recursive form. Let $\mathcal{F}_{\beta}$ be the $\sigma$-field generated by $\beta[n]$. Denote the one-step prediction error as $\mathbf{\Sigma_{n+1|n}}:=\mathbb{E}[(\mathbf{x}[n+1]-\mathbb{E}[\mathbf{x}[n+1]|\mathbf{y}^n])(\mathbf{x}[n+1]-\mathbb{E}[\mathbf{x}[n+1]|\mathbf{y}^n])^\dag |\mathcal{F}_{\beta}]$. Then, $\mathbf{\Sigma_{n+1|n}}$ follows the following recursive equation~\cite[p.101]{KumarVaraiya}.
\begin{align}
\mathbf{\Sigma_{n+1|n}} &= (\mathbf{A}-\mathbf{A}\mathbf{L_n}\mathbf{\bar{C}_n})\mathbf{\Sigma_{n|n-1}}(\mathbf{A}-\mathbf{A}\mathbf{L_n}\mathbf{\bar{C}_n})^\dag
+\mathbf{A}\mathbf{L_n} \mathbb{E}[\mathbf{v}[n]\mathbf{v}[n]^\dag]\mathbf{L_n}^\dag \mathbf{A}^\dag + \mathbf{B} \mathbb{E}[\mathbf{w}[n]\mathbf{w}[n]^\dag] \mathbf{B}^\dag
\label{eqn:dis:final:2}
\end{align}
Here, $\mathbf{L_n}=\mathbf{\Sigma_{n|n-1}} \mathbf{\bar{C}_n}^\dag \left[ \mathbf{\bar{C}_n} \mathbf{\Sigma_{n|n-1}} \mathbf{\bar{C}_n}^\dag + \mathbb{E}[\mathbf{v}[n]\mathbf{v}[n]^\dag] \right]^{-1}$, and $\mathbf{\bar{C}_n}=\mathbf{C}$ if $\beta[n]=1$ and $\mathbf{C_n}=\mathbf{0}$ otherwise. Notice that $\mathbf{\Sigma_{n+1|n}}$ is a random variable.

Moreover, it is also known that when $(\mathbf{A},\mathbf{B})$ is controllable, the one-step prediction error of $\mathbf{x}[n+1]$ based on $\mathbf{y}[n]$ becomes positive definite for large enough $n$ even if there are no erasures. Therefore, there exists $m \in \mathbb{N}$ and $\sigma^2>0$ such that $\mathbf{\Sigma_{n+1|n}} \succeq \sigma^2 \mathbf{I}$ with probability one for all $n \geq m$.
Therefore, by \eqref{eqn:dis:final:2} for all $n \geq n' \geq m$ we have
\begin{align}
\mathbf{\Sigma_{n+1|n}} &\succeq
(\mathbf{A}-\mathbf{A}\mathbf{L_n}\mathbf{\bar{C}_n}) \cdots
(\mathbf{A}-\mathbf{A}\mathbf{L_{n'}}\mathbf{\bar{C}_{n'}})
\mathbf{\Sigma_{n'|n'-1}}
(\mathbf{A}-\mathbf{A}\mathbf{L_{n'}}\mathbf{\bar{C}_{n'}})^\dag \cdots
(\mathbf{A}-\mathbf{A}\mathbf{L_n}\mathbf{\bar{C}_n})^\dag \\
&\succeq
\sigma^2
(\mathbf{A}-\mathbf{A}\mathbf{L_n}\mathbf{\bar{C}_n}) \cdots
(\mathbf{A}-\mathbf{A}\mathbf{L_{n'}}\mathbf{\bar{C}_{n'}})
\mathbf{I}
(\mathbf{A}-\mathbf{A}\mathbf{L_{n'}}\mathbf{\bar{C}_{n'}})^\dag \cdots
(\mathbf{A}-\mathbf{A}\mathbf{L_n}\mathbf{\bar{C}_n})^\dag. \label{eqn:dis:final:3}
\end{align}

Let's use the definitions of $\mathbf{U}$, $\mathbf{A'}$, $\mathbf{C'}$, $\mathbf{U}$, $\mathbf{A_i}$, $\mathbf{C_i}$, $\lambda_{i,j}$, $p_i$, $l_i$, $\nu_i$ from \eqref{eqn:ac:jordan:thm}, \eqref{eqn:ac2:jordan:thm} and \eqref{eqn:def:lprime:thm}.
Let $i^\star := \underset{1 \leq i \leq \mu}{argmax} |\lambda_{i,1}|^{2 \frac{p_i}{l_i}}$.
Let $S'^\star \subseteq \{0,1,\cdots,p_{i^\star}-1\}$ be a set achieving the minimum cardinality $l_{i^\star}$.  In other words, define $S^\star := \{s_1^\star,s_2^\star,\cdots,s_{|S^\star|}^\star \}= \{ 0, 1, \cdots, p_{i^\star}-1 \} \setminus S'^\star$. Then, $|S'^\star|=l_{i^\star}$ and
\begin{align}
\begin{bmatrix}
\mathbf{C_{i^\star}}\mathbf{A_{i^\star}}^{s_1^\star} \\
\mathbf{C_{i^\star}}\mathbf{A_{i^\star}}^{s_2^\star} \\
\vdots \\
\mathbf{C_{i^\star}}\mathbf{A_{i^\star}}^{s_{|S^\star|}^\star} \\
\end{bmatrix}
\nonumber
\end{align}
is rank deficient, i.e. the rank is strictly less than $\nu_{i^\star}$.

For a given time index $n$, define the stopping time $S_n$ as the most recent observation which does not belong to $S^\star$ in modulo $p_{i^\star}$, i.e.
\begin{align}
S_n:=\inf \{ k p_{i^\star} : k \in \mathbb{Z}^+ \mbox{ and there exists $k'$ such that } \beta[n-k']=1 , k p_{i^\star} \leq k' < (k+1)p_{i^\star}, -k'-1( mod\ p_{i^\star}) \in S'^\star \}.
\end{align}
Then, we can compute that $\mathbb{P}\{S_n = k p_{i^\star} \}=(1-p_e^{l_{i^\star}})(p_e^{l_{i^\star}})^{k}$ for all $k \in \mathbb{Z}^+$. From the definition of $S_n$, we can see that for all $0 \leq k < S_n$, $\beta[n-k]=1$ if and only if $-k-1(mod\ p_{i^\star}) \in S$.

Then, conditioned on $n-S_n \geq m$, by \eqref{eqn:dis:final:3} the following inequality holds with probability one:
\begin{align}
\mathbf{\Sigma_{n+1|n}}&\succeq
\sigma^2 (\mathbf{A}-\mathbf{A}\mathbf{L_n}\mathbf{\bar{C}_n}) \cdots (\mathbf{A}-\mathbf{A}\mathbf{L_{n-S_n+1}}\mathbf{\bar{C}_{n-S_n+1}})
I
(\mathbf{A}-\mathbf{A}\mathbf{L_{n-S_n+1}}\mathbf{\bar{C}_{n-S_n+1}})^\dag \cdots (\mathbf{A}-\mathbf{A}\mathbf{L_n}\mathbf{\bar{C}_n})^\dag. \label{eqn:dis:final:4}
\end{align}
where $\mathbf{\bar{C}_{n-S_n+k}}=\mathbf{C}$ if $-S_n+k-1(mod\ p_{i^\star})=k-1(mod\ p_{i^\star}) \in S$ and
$\mathbf{\bar{C}_{n-S_n+k}}=\mathbf{0}$ otherwise.

We will prove that the L.H.S. of \eqref{eqn:dis:final:4} grows exponentially. For this, we first need the following lemma.

\begin{lemma}
Consider $\mathbf{A}$, $\mathbf{C}$, $\mathbf{U}$, $\mathbf{A'}$, $\mathbf{C'}$, $\mathbf{A_i}$, $\mathbf{C_i}$, $\nu_i$, $p_i$ given as \eqref{eqn:ac:jordan:thm}, \eqref{eqn:ac2:jordan:thm} and \eqref{eqn:def:lprime:thm}.
For a given set $S:=\{s_1,\cdots,s_{|S|} \} \in \{0,1,\cdots,p_i-1 \}$, let
$\begin{bmatrix}\mathbf{C_i}\mathbf{A_i}^{s_1} \\ \mathbf{C_i}\mathbf{A_i}^{s_2} \\ \vdots \\ \mathbf{C_i}\mathbf{A_i}^{s_{|S|}} \end{bmatrix}$ be rank-deficient, i.e. the rank is less than $\nu_i$, and define
\begin{align}
\mathbf{\bar{A}}(\mathbf{K_0},\cdots,\mathbf{K_{p_i-1}}):=(\mathbf{A}-\mathbf{K_{p_i-1}}\mathbf{\bar{C}_{p_i-1}}) \cdots (\mathbf{A}-\mathbf{K_0}\mathbf{\bar{C}_0})\nonumber
\end{align}
where $\mathbf{C_j'}=\mathbf{C}$ when $j \in S$ and $\mathbf{C_j'}=\mathbf{0_{l\times m}}$ otherwise.\\
Then, for all $\mathbf{K_0},\cdots, \mathbf{K_{p_i-1}} \in \mathbb{C}^{m \times l}$, $\mathbf{\bar{A}}(\mathbf{K_0},\cdots,\mathbf{K_{p_i-1}})$ has a common right eigenvector $\mathbf{e}$ whose eigenvalue is $\lambda_{i,1}^{p_i}$.
\label{lem:dis:converse}
\end{lemma}
\begin{proof}
For the simplicity of notation, we will set $i=1$, but the proof for general $i$ is the same. Let $\mathbf{e'}=\begin{bmatrix} e_1 \\ \vdots \\ e_{\nu_1} \end{bmatrix}$ be a nonzero vector that belongs to the right null space of
$\begin{bmatrix}\mathbf{C_1}\mathbf{A_1}^{s_1} \\ \mathbf{C_1}\mathbf{A_1}^{s_2} \\ \vdots \\ \mathbf{C_1}\mathbf{A_1}^{s_{|S|}} \end{bmatrix}$.
Let $\mathbf{e_1'}$ be a $m_{1,1} \times 1$ column vector whose first element is $e_1$ and the rest are $0$. Likewise, $\mathbf{e_2'}$ is a $m_{1,2} \times 1$ column vector with first element $e_2$ and the rest $0$. $\mathbf{e_3'}, \cdots, \mathbf{e_{\nu_1}'}$ are defined in the same way. Let a $m \times 1$ column vector $\mathbf{e''}$ be $\begin{bmatrix} \mathbf{e_1'} \\ \vdots \\ \mathbf{e_{\nu_1}'} \\ \mathbf{0}_{(m-\underset{1 \leq i \leq \nu_1}{\sum} m_{1,i}) \times 1} \end{bmatrix}$. Then, we will prove that $\mathbf{e}:=\mathbf{U}\mathbf{e''}$ is the eigenvector that satisfies the conditions of the lemma.

By construction, we can see that $\mathbf{C_1}\mathbf{A_1}^k \mathbf{e'}=\mathbf{0}$ for $k \in \{ s_1, \cdots, s_{|S|}\}$. Moreover, since $\mathbf{C}\mathbf{A^{k}}\mathbf{e} = \mathbf{C}\mathbf{U}
\mathbf{A'}^{k}\mathbf{U}^{-1}\mathbf{U}\mathbf{e''} = \mathbf{C'}\mathbf{A'}^{k}\mathbf{e''}$, we also have $\mathbf{C}\mathbf{A}^k \mathbf{e}=0$ for $k \in \{ s_1, \cdots, s_{|S|}\}$.
Thus, we can conclude
\begin{align}
&(\mathbf{A}-\mathbf{K_{p_1-1}}\mathbf{C_{p_1-1}'}) \cdots (\mathbf{A}-\mathbf{K_{s_1}}\mathbf{C_{s_1}'}) (\mathbf{A}-\mathbf{K_{s_1-1}}\mathbf{C_{s_1-1}'}) \cdots (\mathbf{A}-\mathbf{K_0}\mathbf{C_0'}) \mathbf{e} \nonumber \\
&=(\mathbf{A}-\mathbf{K_{p_1-1}}\mathbf{C_{p_1-1}'}) \cdots (\mathbf{A}-\mathbf{K_{s_1}}\mathbf{C}) (\mathbf{A}-\mathbf{K_{s_1-1}}\mathbf{0}) \cdots (\mathbf{A}-\mathbf{K_0}\mathbf{0}) \mathbf{e} \nonumber \\
&=(\mathbf{A}-\mathbf{K_{p_1-1}}\mathbf{C_{p_1-1}'}) \cdots (\mathbf{A}-\mathbf{K_{s_1}}\mathbf{C}) \mathbf{A}^{s_1} \mathbf{e} \nonumber \\
&=(\mathbf{A}-\mathbf{K_{p_1-1}}\mathbf{C_{p_1-1}'}) \cdots (\mathbf{A}^{s_1+1}\mathbf{e}-\mathbf{K_{s_1}}\mathbf{C} \mathbf{A}^{s_1}\mathbf{e} )  \nonumber \\
&\overset{(a)}{=}(\mathbf{A}-\mathbf{K_{p_1-1}}\mathbf{C_{p_1-1}'}) \cdots (\mathbf{A}^{s_1+1}\mathbf{e})   \nonumber\\
&\overset{(b)}{=}\mathbf{A}^{p_1}\mathbf{e} = \mathbf{U}\mathbf{A'}^{p_1} \mathbf{U}^{-1} \mathbf{e}= \mathbf{U}\mathbf{A'}^{p_1} \mathbf{e''} \nonumber\\
&\overset{(c)}{=} \mathbf{U}\lambda_{1,1}^{p_1} \mathbf{e''}= \lambda_{1,1}^{p_1} \mathbf{e} \nonumber
\end{align}
(a): $\mathbf{C}\mathbf{A}^{s_1}\mathbf{e}=\mathbf{0}$.\\
(b): Repetitive use of (a) for $s_2, \cdots, s_{|S|}$.\\
(c): $\mathbf{A_1}^{p_1}=\lambda_{1,1}^{p_1} \mathbf{I}$ and the definition of the vector $\mathbf{e''}$.

Thus, the lemma is proved.
\end{proof}

Let the vector $\mathbf{e}$ be the right eigenvector of Lemma~\ref{lem:dis:converse} for $i=i^\star$.
Then, there exists $\sigma' >0$ such that $\mathbf{I} \succeq  \sigma'^2\mathbf{e}\mathbf{e}^\dag$. \eqref{eqn:dis:final:4} is lower bounded as
\begin{align}
\mathbf{\Sigma_{n+1|n}}& \succeq \sigma^2 \sigma'^2  \lambda_{i^\star,1}^{S_n} \mathbf{e}\mathbf{e}^\dag (\lambda_{i^\star,1}^{S_n})^\dag. \nonumber
\end{align}

Since $p_e \geq \frac{1}{|\lambda_{i^\star,1}|^{2 \frac{p_{i^\star}}{l_{i^\star}}}}$, the expected one-step prediction error is lower bounded as follows:\footnote{The lower bound does not hold when $|\lambda_{i^\star,1}|=1$ which induces $p_e =1$. However, in this case we do not have any observation, so trivially the system is unstable.}
\begin{align}
&\mathbb{E}[ ( \mathbf{x}[n+1]-\mathbb{E}[\mathbf{x}[n+1]|\mathbf{y}^n ] )^\dag
( \mathbf{x}[n+1]-\mathbb{E}[\mathbf{x}[n+1]|\mathbf{y}^n ] )]\\
&\geq \mathbb{E}[\sigma^2 \sigma'^2 |\lambda_{i^\star,1}|^{2 S_n} |\mathbf{e}|^2 \cdot \mathbf{1}(n-S_n \geq m)]\\
&\geq \sigma^2 \sigma'^2 |\mathbf{e}|^2 \sum_{0 \leq s \leq \lfloor \frac{n-m}{p_{i^\star}} \rfloor}(1-p_e^{l_{i^\star}})(|\lambda_{i^\star,1}|^{2 p_{i^\star}} p_e^{l_{i^\star}})^s \\
&\geq \sigma^2 \sigma'^2 |\mathbf{e}|^2 \cdot (1-p_e^{l_{i^\star}}) \cdot (\lfloor \frac{n-m}{p_{i^\star}} \rfloor).
\end{align}
Therefore, as $n$ goes to infinity, the one-step prediction error diverges to infinity. The estimation error for the state is not uniformly bounded either, so the system is not intermittent observable.

\begin{figure*}[t]
\begin{center}
\includegraphics[width=6in]{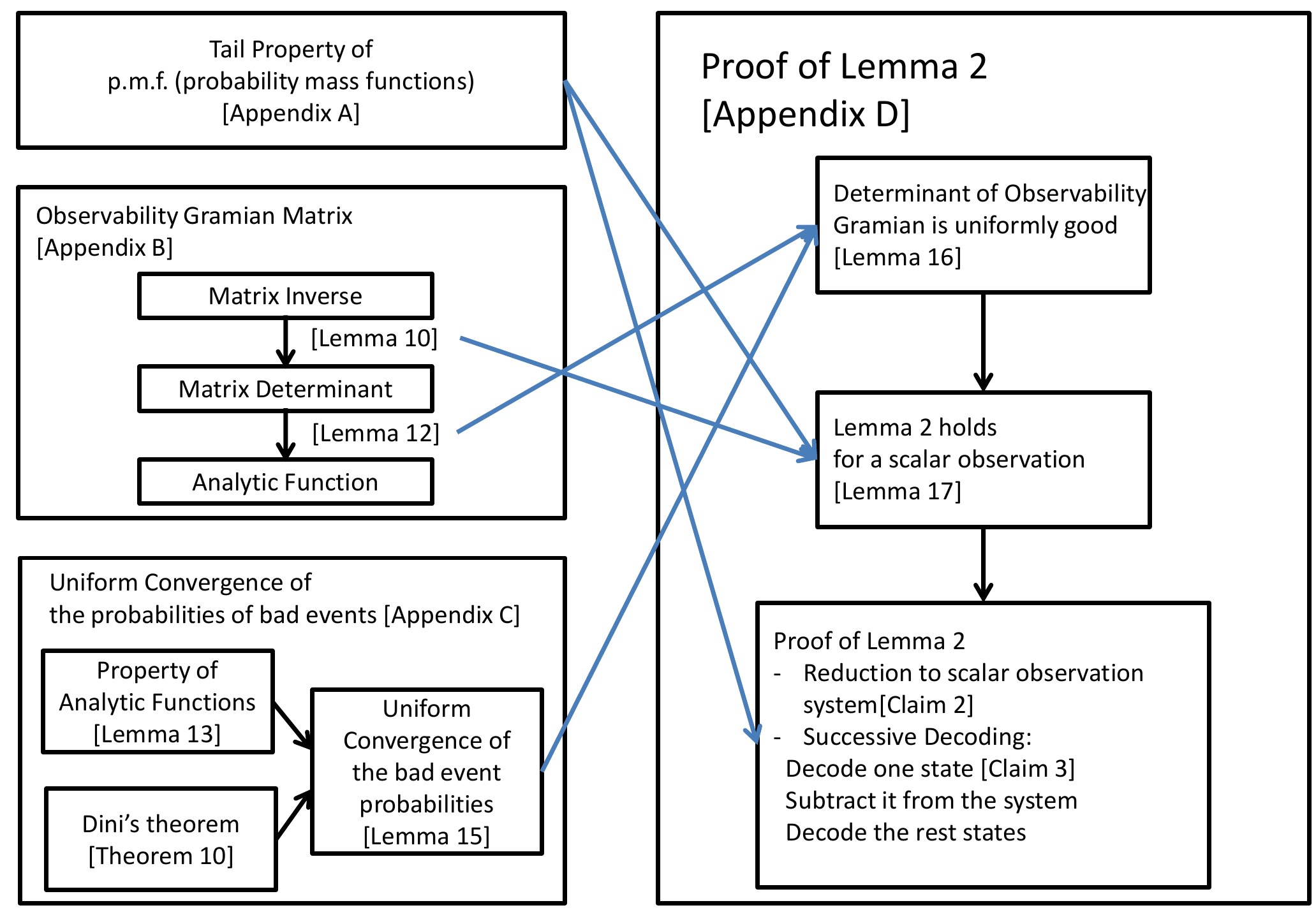}
\caption{Flow diagram of the proof of Lemma~\ref{lem:conti:mo}}
\label{fig:proofflow}
\end{center}
\end{figure*}

\subsection{Proof Outline of Lemma~\ref{lem:conti:mo} and Lemma~\ref{lem:dis:achv}}

Now, the proofs of Theorem~\ref{thm:nonuniform} and \ref{thm:mainsingle} boil down to the proofs of Lemma~\ref{lem:conti:mo} and \ref{lem:dis:achv}. Since the proofs of Lemma~\ref{lem:conti:mo} and \ref{lem:dis:achv} shown in Appendix are too involved, we give the outlines of the proofs in this section.

\subsubsection{Proof Outline of Lemma~\ref{lem:conti:mo}}

The proof flow of Lemma~\ref{lem:conti:mo} is shown in Figure~\ref{fig:proofflow}. As we saw in Section~\ref{sec:intui}, the tail behavior of probability mass functions (p.m.f.) is crucial in the characterization of the critical erasure probability. Thus, in Appendix~\ref{sec:app:1} we first study some properties of the p.m.f. tail.

In the sufficiency proof of Section~\ref{sec:cont:suf}, we analyzed a sub-optimal maximum likelihood estimator whose performance heavily depends on the norm of the inverse of the observability Gramian matrix. In Appendix~\ref{sec:app:3}, we will reduce the question about the norm of the matrix to a question about an analytic function. In Lemma~\ref{lem:conti:inverse2}, we first prove that if the determinant of the observability Gramian matrix is large enough than the norm of the inverse of the observability Gramian matrix is small enough. Thus, we can reduce the question about the norm to an question about the determinant. Since the determinant of the observability Gramian matrix is an analytic function, Lemma~\ref{lem:det:lower} further reduce the question to a question about an analytic function. In other words, if an analytic function is large enough, then the determinant of the observability Gramian matrix is also large enough.

For the intermittent observability, we want to prove that the estimation error is uniformly bounded over all time indexes with nonuniform sampling. It is enough that a set of analytic analytic functions is uniformly away from $0$ with high probability. Lemma~\ref{lem:singleun} of Appendix~\ref{app:unif:conti} captures this insight. In Lemma~\ref{lem:uni:1}, we first prove that each analytic function is away from $0$ with high probability using a property of analytic functions. After this, we apply Dini's theorem which tells pointwise convergence implies uniform convergence when the domain of the functions is compact, and prove the desired uniform convergence of Lemma~\ref{lem:singleun}.

Now, we are ready to prove Lemma~\ref{lem:conti:mo}. By merging the results of Lemma~\ref{lem:det:lower} and \ref{lem:singleun}, we can prove that the determinant of observability Gramian is large enough with high probability uniformly over all time indexes. Together with the properties of the p.m.f. tail, we can first prove Lemma~\ref{lem:conti:mo} for a scalar observation. We can finally prove the general case using the idea of successive decoding. In other words, we reduce the system to the one with a scalar observation, and estimate one state. Then, we subtract the estimation from the system, and repeat the same procedure until we decode all states.

\begin{figure*}[t]
\begin{center}
\includegraphics[width=6in]{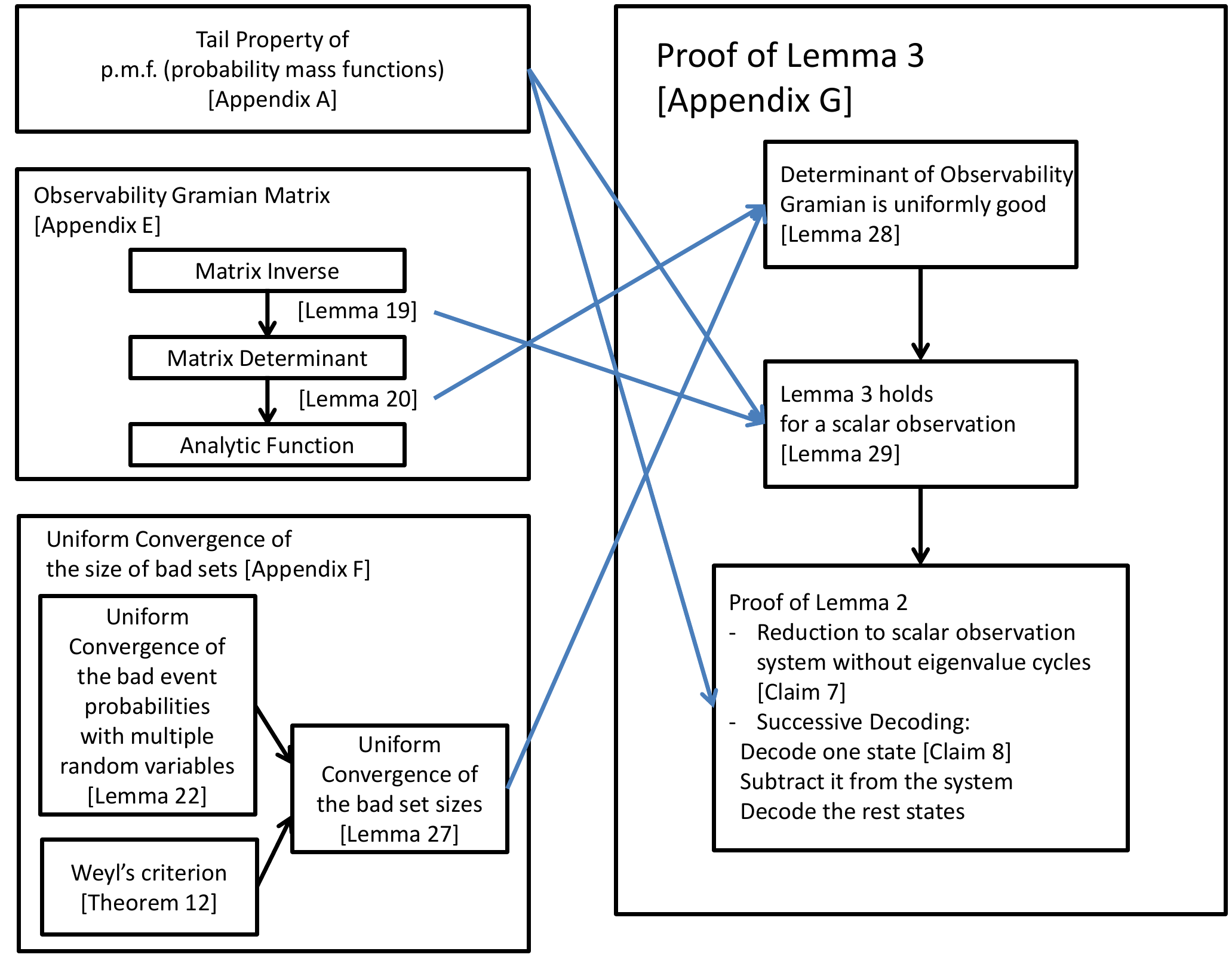}
\caption{Flow diagram of the proof of Lemma~\ref{lem:dis:achv}}
\label{fig:proofflow2}
\end{center}
\end{figure*}

\subsubsection{Proof Outline of Lemma~\ref{lem:dis:achv}}

As we can see in Figure~\ref{fig:proofflow2}, the proof outline of Lemma~\ref{lem:dis:achv} is essentially the same as that of Lemma~\ref{lem:conti:mo}.

We still use the tail properties of p.m.f. shown in Appendix~\ref{sec:app:1}. In Appendix~\ref{sec:dis:gramian}, we will state the lemmas about the observability Gramian matrices of discrete time systems which parallel to the ones of Appendix~\ref{sec:app:3}.

The main difference from the nonuniform sampling case is the uniform convergence shown in Appendix~\ref{sec:dis:uniform}. Consider the system without eigenvalue cycles. In this case, we have to justify that the system essentially reduces to multiple scalar systems, and the critical erasure probability only depends on the largest eigenvalue of the system. However, unlike the nonuniform sampling case, we do not have a random jitter at each observation and the determinant of the observability Gramian is a deterministic sequence in the time indexes. Therefore, we have to prove that the counting measure of the time indexes where the determinant of the observability Gramian is small converges to zero uniformly over all current time indexes.

For this, we apply the Weyl's criterion~\cite{Kuipers} which gives a sufficient condition for deterministic sequences to behave like uniform random variables. Morover, since different eigenvalue cycles behave like independent random variables, we first generalize Lemma~\ref{lem:singleun} of Appendix~\ref{app:unif:conti} for a single random variable to multiple random variables in Lemma~\ref{lem:dis:geo1}. Together with Weyl's criterion, we prove Lemma~\ref{lem:dis:geofinal} which tells the counting measure of the bad time indexes where the determinant of the observability Gramain becomes too small converges to zero uniformly over all current time indexes.

The remaining proof flow of Lemma~\ref{lem:dis:achv} is essentially the same as that of Lemma~\ref{lem:conti:mo}. We first estimate the state corresponds to the largest eigenvalue cycle, subtract the estimation from the system, and successively decode the remaining states.

\section{Comments}
The intermittent Kalman filtering problem was first motivated from
control over communication channels. Therefore, the problem is conventionally believed to fall into the intersection of control
and communication. However, if the plant is unstable the transmission
power of the sensor diverges to infinity if it is really going to
pack an ever increasing number of bits in there. Therefore,
it is hard to say that intermittent Kalman filtering has a direct
connection to communication theory. Instead, we propose that
the intersection of control and signal processing --- especially
sampling theory --- is the right conceptual category for intermittent
Kalman filtering. It should thus be interesting to explore the
connection between the results of this paper with  classical and modern results
of sampling theory.

Arguably, the closest problem to intermittent Kalman filtering is that
of observability after sampling. As we mentioned earlier, the
observability of $(\mathbf{A_c}, \mathbf{C_c})$ in
\eqref{eqn:contistate} and \eqref{eqn:contiob} does not implies the
observability of $(\mathbf{A_c},\mathbf{C})$ in \eqref{eqn:conti:xsample}
and \eqref{eqn:conti:ysample}. The well-known sufficient condition is:
\begin{theorem}[Theorem 6.9. of \cite{Chen}]
Suppose $(\mathbf{A_c},\mathbf{C_c})$ is observable. A sufficient
condition for its discretized system with sampling interval $I$ to be
observable is that $\frac{|\Im (\lambda_i -\lambda_j)I|}{2 \pi} \notin
\mathbb{N}$ whenever $\Re(\lambda_i-\lambda_j)=0$. 
\end{theorem}
Since the eigenvalue of the sampled system is given as $\exp(\lambda_i
I)$, Corollary~\ref{thm:nocycle} can be written as the following
corollary for a sampled system.
\begin{corollary}
Suppose $(\mathbf{A_c},\mathbf{C_c})$ is observable. A sufficient
condition for its discretized system with sampling interval $I$ to
have $\frac{1}{|e^{2\lambda_{max}I}|}$ as a critical erasure
probability is that $\frac{|\Im (\lambda_i -\lambda_j)I|}{2 \pi}
\notin \mathbb{Q}$ whenever $\Re(\lambda_i-\lambda_j)=0$. 
\label{cor:1}
\end{corollary}
The idea of breaking cyclic behavior using non-uniform sampling is
also shown in the context of sampling multiband signals
~\cite{Vaidyanathan_Efficient}. The lower bound on the sampling rate
is known to be the Lebesgue measure of the spectral support of the signal
sampled. To achieve this lower bound for a general multiband signal, a
nonuniform sampling pattern has to be used. Moreover, nonuniform
sampling is also well known as a necessary condition for the currently
hot field of compressed sensing~\cite{Donoho_Compressed}.

As a last comment, we would like to mention that the result is not
sensitive to the norm. In this paper, intermittent
observability is defined using the $l^2$-norm to follow the majority
of the literature. But, if the intermittent observability is defined
by the $l^\eta$-norm, we can simply replace
$2$ in every theorem by $\eta$. For example, the result of
Theorem~\ref{thm:mainsingle} becomes $\frac{1}{ \underset{i}{\max}
  |\lambda_{i,1}|^{\frac{\eta p_i}{l_i'}}}$.

\section{Appendix}
\subsection{Lemmas for Tails of Probability Mass Functions}
\label{sec:app:1}
In this section, we will prove some properties on the tails of probability mass functions (p.m.f.). By the tail, we mean how fast the probability decreases geometrically as we consider rarer and rarer events.

First, we define the essential supremum, $\esssup$.
\begin{definition}
For a given random variable $X$, $\esssup X$ is given as follows.
\begin{align}
\esssup X = \inf\{x \in \mathbb{R} : \mathbb{P}(X > x) = 0 \}.
\end{align}
\end{definition}

The following lemma shows that even if we increase a random variable sub-linearly, its p.m.f. tail remains the same.
\begin{lemma}
Consider $\sigma$-field $\mathcal{F}$ and a nonnegative discrete random variable $k$ whose probability mass function satisfies
\begin{align}
\exp( \limsup_{n \rightarrow \infty} \esssup \frac{1}{n} \log \mathbb{P}\{ k =n | \mathcal{F} \} ) \leq p \nonumber
\end{align}
Then, given a function $f(x)$ such that $f(x) \leq a( \log(x+1) + 1)$ for some $a \in \mathbb{R}^+$, the probability mass function of a random variable $k+f(k)$ satisfies the following:
\begin{align}
\exp( \limsup_{n \rightarrow \infty} \esssup \frac{1}{n} \log \mathbb{P}\{ k+f(k) = n | \mathcal{F} \} ) \leq p. \nonumber
\end{align}
\label{lem:conti:tailpoly}
\end{lemma}
\begin{proof}
Since $\esssup \mathbb{P} \{ k=n | \mathcal{F} \}$ is bounded by $1$, for all $\delta > 0$ such that $p+\delta < 1$ we can find a positive $c$ such that $\esssup \mathbb{P}\{k=n | \mathcal{F} \} \leq c \left(p+\delta \right)^{n} \left(1- \left(p+\delta \right)\right)$. Moreover, since $f(x) \lesssim \log(x+1) + 1$, for all $\delta'>0$ we can find a positive $c'$ such that $f(x) \leq \delta' x + c'$ for all $x \in \mathbb{R}^+$. Then, we have
\begin{align}
\esssup \mathbb{P}\{k+f(k) = n | \mathcal{F} \} & \leq \esssup \mathbb{P}\{k+f(k) \geq n | \mathcal{F} \} \leq \esssup \mathbb{P}\{k+\delta' k + c' \geq n | \mathcal{F} \} \nonumber \\
&\leq \esssup \mathbb{P}\{ k \geq \lfloor \frac{n-c'}{1+\delta'} \rfloor | \mathcal{F} \}
\leq \sum^{\infty}_{i= \lfloor \frac{n-c'}{1+\delta'} \rfloor} \esssup \mathbb{P} \{ k=i | \mathcal{F} \} \nonumber \\
&\leq \sum^{\infty}_{i=\lfloor \frac{n-c'}{1+\delta'} \rfloor}
c(p+\delta)^{i}(1-(p+\delta))\nonumber \\
&=c(1-(p+\delta))\frac{(p+\delta)^{\lfloor \frac{n-c'}{1+\delta'} \rfloor}}{1-(p+\delta)} = c(p+\delta)^{\lfloor \frac{n-c'}{1+\delta'} \rfloor} \nonumber \\
& \leq c(p+\delta)^{\frac{n-c'}{1+\delta'} - 1} = c(p+\delta)^{-\frac{c'}{1+\delta'}-1} (p+\delta)^{\frac{n}{1+\delta'}}. \nonumber
\end{align}
Therefore,
\begin{align}
\exp \left( \limsup_{n \rightarrow \infty} \esssup \frac{1}{n} \log \mathbb{P}\{k+f(k) = n | \mathcal{F} \} \right) \leq (p+\delta)^{\frac{1}{1+\delta'}}. \nonumber
\end{align}
Since we can choose $\delta$ and $\delta'$ arbitrarily close to $0$,
\begin{align}
&\exp \left( \limsup_{n \rightarrow \infty} \esssup \frac{1}{n} \log \mathbb{P}\{k+f(k) = n | \mathcal{F} \} \right) \leq p, \nonumber
\end{align}
which finishes the proof.
\end{proof}

The following lemma tells that if we add independent random variables, the p.m.f. tail of the sum is equal to the heaviest one.
\begin{lemma}
Consider an increasing $\sigma$-fields sequence $\mathcal{F}_0,\mathcal{F}_1,\cdots,\mathcal{F}_{n-1}$ and a sequence of discrete random variables $k_1,k_2,\cdots,k_{n}$ satisfying two properties:\\
(i) $k_i \in \mathcal{F}_i$ for $i \in \{ 1, \cdots, n-1 \}$ \\
(ii) $\exp(\limsup_{k\rightarrow \infty} \esssup \frac{1}{k} \log \mathbb{P}( k_i = k | \mathcal{F}_{i-1} )) \leq p_i$.\\
Let $S=\sum^n_{i=1} k_i$. Then, $\exp(\limsup_{s \rightarrow \infty} \esssup \frac{1}{s} \log \mathbb{P}( S=s | \mathcal{F}_0 )) \leq  \max_{1 \leq i \leq n}\{ p_i\}$.
\label{lem:app:geo}
\end{lemma}
\begin{proof}
Given $\delta > 0$, let $k'_i$ be independent geometric random variables with probability $1-(p_i+\delta)$. Denote $S':=\sum^{n}_{i=1} k_i'$. The moment generating function of $S'$ is
\begin{align}
\mathbb{E}[Z^{-S'}] &= \prod^{n}_{i=1} \frac{\left(1-\left(p_i+\delta \right)\right)}{1-\left(p_i+\delta \right) Z^{-1}}. \nonumber
\end{align}
By \cite{Oppenheim}, the last term can be expanded into a sum of rational functions whose denominators are $1-(p_i+\delta)Z^{-1}$. Therefore, by inverse Z-transform shown in \cite{Oppenheim}, we can prove that $\exp( \limsup_{s \rightarrow \infty} \frac{1}{s} \log \mathbb{P}(S'=s) ) \leq \max_{1 \leq i \leq n} \{ p_i + \delta\}$.\\
On the other hand, since $\esssup \mathbb{P}(k_i=k | \mathcal{F}_{i-1})$ is bounded by $1$, for all $\delta>0$ we can find positive $c_i$ such that
\begin{align}
\esssup \mathbb{P}(k_i=k | \mathcal{F}_{i-1}) \leq c_i \left(p_1+\delta\right)^{k}\left(1-\left(p_1+\delta\right)\right)=c_i \mathbb{P}(k_i'=k) \nonumber
\end{align}
for all $k \in \mathbb{Z}^+$. Then
\begin{align}
&\esssup \mathbb{P}( S = s | \mathcal{F}_0 ) \nonumber \\
&=\esssup \sum_{s=s_1+\cdots+s_n} \mathbb{P}(k_1=s_1|\mathcal{F}_0)\mathbb{P}(k_2=s_2|\mathcal{F}_0,k_1=s_1)\cdots \mathbb{P}(k_n=s_n|\mathcal{F}_0,k_1=s_1,\cdots,k_{n-1}=s_{n-1}) \nonumber \\
&\leq  \sum_{s=s_1+\cdots+s_n} \esssup \mathbb{P}(k_1=s_1|\mathcal{F}_0) \esssup \mathbb{P}(k_2=s_2|\mathcal{F}_1)\cdots \esssup \mathbb{P}(k_n=s_n|\mathcal{F}_{n-1}) \nonumber \\
&\leq \prod_{1 \leq i \leq n}c_i \cdot \sum_{s=s_1+\cdots+s_n} \mathbb{P}(k_1'=s_1) \mathbb{P}(k_2'=s_2)\cdots \mathbb{P}(k_n'=s_n) \nonumber \\
&\leq \prod_{1 \leq i \leq n}c_i \cdot \mathbb{P}(S'=s). \nonumber
\end{align}
Thus, $\exp( \limsup_{s \rightarrow \infty}\esssup \frac{1}{s} \log \mathbb{P}(S=s | \mathcal{F}_0) ) \leq \max_{1 \leq i \leq n}\{ p_i+\delta\}$.

Since this holds for all $\delta>0$, $\exp( \limsup_{s \rightarrow \infty} \esssup \frac{1}{s} \log \mathbb{P}(S=s | \mathcal{F}_0) ) \leq \max_{1 \leq i \leq n}\{ p_i \}$.
\end{proof}

The next lemma tells how the large deviation principle~\cite{Dembo} can be applied to stopping times, i.e. it formally states the ``test channel" and the ``distance idea" shown in the power property of Section~\ref{sec:powerproperty}.

\begin{lemma}
For given $n$, consider discrete random variables $k_1, k_2, \cdots, k_n$ and $\sigma$-algebra $\mathcal{F}$. The probability mass functions of $k_1, k_2 \cdots, k_n$ satisfy
\begin{align}
\exp ( \limsup_{k \rightarrow \infty} \esssup \frac{1}{k} \log \mathbb{P}\{ k_i = k | \mathcal{F} \} ) \leq p_i \nonumber
\end{align}
and $k_1, k_2, \cdots, k_n$ are conditionally independent given $\mathcal{F}$.\\
For given sets $T_1, T_2, \cdots, T_m  \subseteq \{1,2,\cdots, n \}$, define stopping times $M_1, \cdots, M_m$ as
\begin{align}
M_i := \max_{t \in T_i} k_t
\end{align}
and a stopping time $S$ as
\begin{align}
S := \min_{1 \leq i \leq m} M_i.
\end{align}

Then,
\begin{align}
 \exp \left(\limsup_{k \rightarrow \infty} \esssup \frac{1}{k} \log \mathbb{P}\{ S=k | \mathcal{F} \}\right) \leq
\max_{T=\{t_1, t_2, \cdots, t_{|T|}\} \subseteq \{1,2,\cdots, n \}
\small{\mbox{ s.t. }}T \cap T_i \neq \emptyset \small{\mbox{ for all }}i
} p_{t_1} p_{t_2} \cdots p_{t_{|T|}}. \nonumber
\end{align}
\label{lem:dis:geo0}
\end{lemma}
\begin{proof}
Since $\esssup \mathbb{P} \{ k_i = k | \mathcal{F} \} $ is bounded by $1$, for all $\delta > 0$ we can find $c>1$ such that
\begin{align}
\esssup \mathbb{P}\{ k_i = k | \mathcal{F} \} \leq c(p_i + \delta)^{k}\left(1- \left(p_i + \delta \right)\right).\nonumber
\end{align}
Thus, we have
\begin{align}
\esssup \mathbb{P}\{ k_i \geq k | \mathcal{F} \} & \leq  c (p_i + \delta)^{k}.\nonumber
\end{align}
Therefore,
\begin{align}
\esssup \mathbb{P}\{ S=k | \mathcal{F} \} &\leq \esssup \mathbb{P}\{ S \geq k | \mathcal{F} \} \nonumber \\
&= \esssup \mathbb{P}\{ M_1 \geq k, \cdots, M_m \geq k | \mathcal{F} \} \nonumber \\
&= \esssup \mathbb{P}\{ \mbox{There exists $T=\{t_1,t_2,\cdots,t_{|T|} \} \subseteq \{1,\cdots,n \}$ s.t. } T \cap T_i \neq \emptyset \mbox{ for all } i \mbox{ and } k_{t_1} \geq k, \cdots, k_{t_{|T|}} \geq k | \mathcal{F} \} \nonumber \\
&\leq \sum_{
\small{
\begin{array}{c}T=\{t_1,t_2,\cdots,t_{|T|} \} \subseteq \{1,\cdots,n \} \\
\small{\mbox{ s.t. }} T \cap T_i \neq \emptyset \small{\mbox{ for all }} i\end{array}}
}
\esssup \mathbb{P} \{ k_{t_1} \geq k, k_{t_2} \geq k, \cdots, k_{t_{|T|}} \geq k | \mathcal{F} \} \nonumber \\
&\leq
|\{ T=\{t_1,t_2,\cdots,t_{|T|} \} \subseteq \{1,\cdots,n \} {\mbox{ s.t. }} T \cap T_i \neq \emptyset {\mbox{ for all }} i \}|  \label{eqn:geo:large1}  \\
&\quad \cdot \max_{
\small{
\begin{array}{c}
T=\{t_1,t_2,\cdots,t_{|T|} \} \subseteq \{1,\cdots,n \} \\
\small{\mbox{ s.t. }} T \cap T_i \neq \emptyset \small{\mbox{ for all }} i
\end{array}
}
} \esssup \mathbb{P}\{ k_{t_1}\geq k | \mathcal{F} \} \cdots \esssup \mathbb{P}\{ k_{t_{|T|}} \geq k | \mathcal{F} \} \nonumber\\
& \leq c^n |\{ T=\{t_1,t_2,\cdots,t_{|T|} \} \subseteq \{1,\cdots,n \} {\mbox{ s.t. }} T \cap T_i \neq \emptyset {\mbox{ for all }} i \}| \nonumber \\
&\quad \cdot
\max_{
\small{
\begin{array}{c}
T=\{t_1,t_2,\cdots,t_{|T|} \} \subseteq \{1,\cdots,n \} \\
\small{\mbox{ s.t. }} T \cap T_i \neq \emptyset \small{\mbox{ for all }} i
\end{array}
}
}
(p_{t_1}+\delta)^{k-1} (p_{t_2}+\delta)^{k-1} \cdots (p_{t_{|T|}}+\delta)^{k-1}. \nonumber
\end{align}
\eqref{eqn:geo:large1} follows from union bound. Since the above inequality holds for all $\delta>0$,
\begin{align}
\exp\left( \limsup_{k \rightarrow \infty} \esssup \frac{1}{k} \log \mathbb{P} \{ S = k | \mathcal{F} \} \right) \leq
\max_{T=\{t_1, t_2, \cdots, t_{|T|}\} \subseteq \{1,2,\cdots, n \}
\small{\mbox{ s.t. }}T \cap T_i \neq \emptyset \small{\mbox{ for all }}i
} p_{t_1} p_{t_2} \cdots p_{t_{|T|}}. \nonumber
\end{align}
\end{proof}

\subsection{Lemmas about the Observability Gramian of Continuous-Time Systems}
\label{sec:app:3}
In linear system theory~\cite{Chen}, the observability Gramian plays a crucial role in estimating states from observations. Therefore, we also study the behavior of the observability Gramian, especially the norm of the inverse of the observability Gramian.

First, we start with a corollary of the classic rearrangement inequality~\cite{hardy1988inequalities}.
\begin{lemma}[Rearrangement Inequality]
For $\lambda_1 \geq \lambda_2 \geq \cdots \geq \lambda_m \geq 0$, $0 \leq k_1 \leq k_2 \leq \cdots \leq k_m$, and any permutation map $\sigma$, the following inequality is true:
\begin{align}
e^{-\lambda_{\sigma(1)} k_{1}} e^{-\lambda_{\sigma(2)} k_{2}} \cdots e^{-\lambda_{\sigma(m)} k_{m}} \leq e^{-\lambda_1 k_1} e^{- \lambda_2 k_2} \cdots e^{-\lambda_m k_m}. \nonumber
\end{align}
Moreover, the ratio of these two can also be upper bounded as
\begin{align}
\frac{e^{-\lambda_{\sigma(1)} k_{1}} e^{-\lambda_{\sigma(2)} k_{2}} \cdots e^{-\lambda_{\sigma(m)} k_{m}}}{e^{-\lambda_1 k_1} e^{- \lambda_2 k_2} \cdots e^{-\lambda_m k_m}} \leq e^{-(\lambda_{\sigma(m)}-\lambda_m)(k_m-k_{\sigma^{-1}(m)})}. \nonumber
\end{align}
\label{lem:conti:rearr}
\end{lemma}
\begin{proof}
The first inequality directly follows from the classic rearrangement inequality. The second inequality is proved as follows: When $\sigma^{-1}(m)=m$, the inequality is trivial. When $\sigma^{-1}(m)\neq m $, we have
\begin{align}
&e^{-\lambda_{\sigma(1)}k_1} e^{-\lambda_{\sigma(2)}k_2} \cdots e^{-\lambda_m k_{\sigma^{-1}(m)}} \cdots e^{-\lambda_{\sigma(m-1)k_{m-1}} } e^{-\lambda_{\sigma(m)}k_m} \nonumber \\
&=
\underbrace{\left(e^{-\lambda_{\sigma(1)}k_1} e^{-\lambda_{\sigma(2)}k_2} \cdots e^{-\lambda_m k_{\sigma^{-1}(m)}} \cdots e^{-\lambda_{\sigma(m-1)}k_{m-1}} \right)}_{(a)}
\cdot e^{-\lambda_{\sigma(m)}k_m} \nonumber \\
&=
\underbrace{\left(e^{-\lambda_{\sigma(1)}k_1} e^{-\lambda_{\sigma(2)}k_2} \cdots e^{-\lambda_{\sigma(m)} k_{\sigma^{-1}(m)}} \cdots e^{-\lambda_{\sigma(m-1)}k_{m-1}} \right)}_{(b)}
\cdot \left( \frac{e^{-\lambda_m k_{\sigma^{-1}(m)}}}{e^{-\lambda_{\sigma(m)}k_{\sigma^{-1}(m)}}} \right) \cdot e^{-\lambda_{\sigma(m)}k_m}. \label{eqn:arrange:1}
\end{align}
We can notice that the exponent of $(a)$ has $\{ \lambda_1, \lambda_2,\cdots, \lambda_m \} \setminus \{\lambda_{\sigma(m)}\}$ and $\{ k_1, k_2, \cdots, k_m \} \setminus \{ k_m \}$ terms in it, and the exponent of $(b)$ has
\begin{align}
&\left( \{ \lambda_1, \lambda_2,\cdots, \lambda_m \} \setminus \{\lambda_{\sigma(m)}\} \right) \cup \{ \lambda_{\sigma(m)} \} \setminus \{ \lambda_m \} \nonumber \\
&=\{ \lambda_1, \lambda_2,\cdots, \lambda_m \} \setminus \{ \lambda_m \} \nonumber
\end{align}
and $\{ k_1, k_2, \cdots, k_m \} \setminus \{ k_m \}$ terms in it. Thus, by the first inequality of the lemma,
\begin{align}
(b) \leq e^{-\lambda_1 k_1} \cdots e^{-\lambda_{m-1} k_{m-1}}. \nonumber
\end{align}
Together with $\eqref{eqn:arrange:1}$, we have
\begin{align}
&\frac{e^{-\lambda_{\sigma(1)}k_1} e^{-\lambda_{\sigma(2)}k_2} \cdots e^{-\lambda_{\sigma(m)}k_m}}{ e^{-\lambda_1 k_1} e^{-\lambda_2 k_2} \cdots e^{-\lambda_m k_m} }  \nonumber \\
&\leq
\frac{
\left(e^{-\lambda_1 k_1} \cdots e^{-\lambda_{m-1} k_{m-1}}\right)
\cdot \left( \frac{e^{-\lambda_m k_{\sigma^{-1}(m)}}}{e^{-\lambda_{\sigma(m)}k_{\sigma^{-1}(m)}}} \right) \cdot e^{-\lambda_{\sigma(m)}k_m}
}{ e^{-\lambda_1 k_1} e^{-\lambda_2 k_2} \cdots e^{-\lambda_m k_m} }  \nonumber \\
&=
\frac{1}{e^{-\lambda_m k_m}}
\cdot \left( \frac{e^{-\lambda_m k_{\sigma^{-1}(m)}}}{e^{-\lambda_{\sigma(m)}k_{\sigma^{-1}(m)}}} \right) \cdot e^{-\lambda_{\sigma(m)}k_m}=e^{(\lambda_m-\lambda_{\sigma(m)})(k_m - k_{\sigma^{-1}(m)})} \nonumber
\end{align}
which finishes the proof.
\end{proof}


Even though Theorem~\ref{thm:nonuniform} is written for a general matrix $\mathbf{C}$, we will first start from the simpler case of a row vector $\mathbf{C}$. In fact, for the proof of the general case, we will reduce the system with a matrix $\mathbf{C}$ to a system with a row vector $\mathbf{C}$.

First, we introduce the definitions corresponding to \eqref{eqn:conti:a2}, \eqref{eqn:conti:c2} for a row vector $\mathbf{C}$. Let  $\mathbf{A_c}$ be a $m \times m$ Jordan form matrix, and $\mathbf{C}$ be a $1 \times m$ row vector $\mathbf{C}$ which are written as follows:
\begin{align}
&\mathbf{A_c}=diag\{\mathbf{A_{1,1}},\mathbf{A_{1,2}},\cdots, \mathbf{A_{1,\nu_{1}}},\cdots,\mathbf{A_{\mu,1}},\cdots,\mathbf{A_{\mu,\nu_{\mu}}}\}
\label{eqn:conti:a} \\
&\mathbf{C}=\begin{bmatrix}
\mathbf{C_{1,1}} & \mathbf{C_{1,2}} & \cdots & \mathbf{C_{1,\nu_{1}}} & \cdots & \mathbf{C_{\mu,1}} & \cdots & \mathbf{C_{\mu,\nu_{\mu}}}
\end{bmatrix}
\label{eqn:conti:c} \\
&\mbox{where } \mathbf{A_{i,j}} \mbox{ is a Jordan block with eigenvalue $\lambda_{i,j}+\sqrt{-1}\omega_{i,j}$ and size $m_{i,j}$} \nonumber \\
&\quad\quad m_{i,1} \leq m_{i,2} \leq \cdots \leq m_{i,\nu_i} \mbox{ for all }i=1,\cdots,\mu \nonumber \\
&\quad\quad m_i=\sum_{1 \leq j \leq \nu_i} m_{i,j} \mbox{ for all }i=1,\cdots,\mu \nonumber \\
&\quad\quad \lambda_{i,1}=\lambda_{i,2}=\cdots =\lambda_{i,\nu_i} \mbox{ for all }i=1,\cdots,\mu \nonumber \\
&\quad\quad \lambda_{1,1}>\lambda_{2,1} > \cdots > \lambda_{\mu,1} \geq 0 \nonumber \\
&\quad\quad \omega_{i,1}, \cdots ,\omega_{i,\nu_i} \mbox{ are pairwise distinct} \nonumber \\
&\quad\quad \mathbf{C_{i,j}}\mbox{ is a $1 \times m_{i,j}$ complex matrix and its first element is non-zero} \nonumber\\
&\quad\quad \mbox{$\lambda_i+ \sqrt{-1} \omega_i$ is $(i,i)$ element of $\mathbf{A_c}$}.\nonumber
\end{align}
Here, we can notice that the real parts of the eigenvalues of $\mathbf{A_{i,1}}, \cdots, \mathbf{A_{i,\nu_i}}$ are the same, but the eigenvalues of all Jordan blocks $\mathbf{A_{i,j}}$ are distinct. Therefore, by Theorem~\ref{thm:jordanob}, the condition that the first elements of $\mathbf{C_{i,j}}$ are non-zero corresponds to the observability of $(\mathbf{A_c},\mathbf{C})$.

The following lemma upper bounds the determinant of the observability Gramain of the sampled continuous system.
\begin{lemma}
Let $\mathbf{A_c}$ and $\mathbf{C}$ be given as \eqref{eqn:conti:a} and \eqref{eqn:conti:c}. For $0 \leq k_1 \leq k_2 \leq \cdots \leq k_m$, there exists $a > 0$, $p \in \mathbb{Z}^+$ such that
\begin{align}
\left| \det\left( \begin{bmatrix}
\mathbf{C}e^{-k_1 \mathbf{A_c}}\\
\mathbf{C}e^{-k_2 \mathbf{A_c}}\\
\vdots \\
\mathbf{C}e^{-k_m \mathbf{A_c}}\\
\end{bmatrix} \right) \right| \leq a (k_m^p+1) \prod_{1 \leq i \leq m } e^{-k_i \lambda_i} \nonumber
\end{align}
where $\lambda_i$ is the real part of $(i,i)$ component of $\mathbf{A_c}$.
\label{lem:det:upper}
\end{lemma}
\begin{proof}
First consider a diagonal matrix, i.e.
$\mathbf{A_c}=\begin{bmatrix} \lambda_1+ j \omega_1 & 0 & \cdots & 0 \\ 0 & \lambda_2+ j \omega_2 & \cdots & 0 \\ \vdots & \vdots & \ddots & \vdots \\  0 & 0 & \cdots & \lambda_m + j \omega_m \end{bmatrix}$. Then,
\begin{align}
&\left| \det \left( \begin{bmatrix} \mathbf{C}e^{-k_1 \mathbf{A_c}} \\ \mathbf{C}e^{-k_2 \mathbf{A_c}} \\ \vdots \\ \mathbf{C}e^{-k_m \mathbf{A_c}}  \end{bmatrix} \right) \right| \nonumber \\
&=\left| \sum_{\sigma \in S_m} sgn(\sigma) \prod^m_{i=1} c_i e^{-k_{\sigma(i)}(\lambda_i+j \omega_i)} \right| \nonumber \\
&\leq m! \max_{\sigma \in S_m} \left| \prod^m_{i=1} c_i e^{-k_{\sigma(i)}(\lambda_i+j \omega_i)} \right| \nonumber \\
&= m! \left| \prod^m_{i=1}c_i \right| \max_{\sigma \in S_m} \left| \prod^m_{i=1} e^{-k_{\sigma(i)}\lambda_i} \right| \nonumber \\
&= m! \left| \prod^m_{i=1}c_i \right| \prod^m_{i=1} e^{-k_i \lambda_i} (\because Lemma~\ref{lem:conti:rearr}) \nonumber \\
&\lesssim \prod^m_{i=1} e^{-k_i \lambda_i} \label{eqn:matrix:det1}
\end{align}
where $c_i$ are $i$th component of $\mathbf{C}$, $S_m$ is the set of all permutations on $\{1,\cdots,m \}$, and $sgn(\sigma)$ is $+1$ if $\sigma$ is an even permutation $-1$ otherwise. Therefore, the lemma is true for a diagonal $\mathbf{A_c}$.

To extend to a general Jordan matrix $\mathbf{A_c}$, consider a matrix $\mathbf{A_c'}$ which is obtained by erasing the off-diagonal elements of $\mathbf{A_c}$.
Then, we can easily see the ratio between the elements of $\begin{bmatrix}
\mathbf{C} e^{-k_1 \mathbf{A_c}} \\
\vdots \\
\mathbf{C} e^{-k_m \mathbf{A_c}}
\end{bmatrix}$ and the corresponding elements of $\begin{bmatrix}
\mathbf{C} e^{-k_1 \mathbf{A_c'}} \\
\vdots \\
\mathbf{C} e^{-k_m \mathbf{A_c'}}
\end{bmatrix}$ is a polynomial whose degree is less than $m$. Therefore, by repeating the steps of \eqref{eqn:matrix:det1} we can easily obtain
\begin{align}
\left| \det \left(
\begin{bmatrix}
\mathbf{C}e^{-k_1 \mathbf{A_c}} \\ \mathbf{C}e^{-k_2 \mathbf{A_c}} \\ \vdots \\ \mathbf{C}e^{-k_m \mathbf{A_c}}
\end{bmatrix}
\right)\right| \lesssim (1+k_m^{m^2}) \prod^m_{i=1} e^{-k_i \lambda_i},  \nonumber
\end{align}
which finishes the proof.
\end{proof}

The next lemma upper bounds the norm of the inverse of the observability Gramian, given the lower bound on the observability Gramian determinant. Therefore, we can reduce the matrix inverse problem to the matrix determinant problem.

\begin{lemma}
Consider $\mathbf{A_c}$ and $\mathbf{C}$ given as \eqref{eqn:conti:a} and \eqref{eqn:conti:c}. Let $\lambda_i$ be the real part of $(i,i)$ element of $\mathbf{A_c}$. Then, there exists a positive polynomial $p(k)$ such that for all $\epsilon>0$ and $0 \leq k_1 \leq \cdots \leq k_m$, if
\begin{align}
\left| \det\left( \begin{bmatrix} \mathbf{C}e^{-k_1 \mathbf{A_c}} \\ \mathbf{C}e^{-k_2 \mathbf{A_c}} \\ \vdots \\ \mathbf{C}e^{-k_m \mathbf{A_c}}  \end{bmatrix} \right) \right| \geq \epsilon \prod_{1 \leq i \leq m} e^{- k_i \lambda_i} \nonumber
\end{align}
then
\begin{align}
\left| \begin{bmatrix} \mathbf{C}e^{-k_1 \mathbf{A_c}} \\ \mathbf{C}e^{-k_2 \mathbf{A_c}} \\ \vdots \\ \mathbf{C}e^{-k_m \mathbf{A_c}}  \end{bmatrix}^{-1} \right|_{max} \leq \frac{p(k_m)}{\epsilon} e^{\lambda_1 k_m}. \nonumber
\end{align}
\label{lem:conti:inverse2}
\end{lemma}
\begin{proof}
Let $\mathbf{O_{i,j}}$ be the matrix obtained by removing the $i$th row and $j$th column of $\begin{bmatrix} \mathbf{C}e^{-k_1 \mathbf{A_c}} \\ \mathbf{C}e^{-k_2 \mathbf{A_c}} \\ \vdots \\ \mathbf{C}e^{-k_m \mathbf{A_c}} \end{bmatrix}$. Let $\mathbf{A_c(j)}$ be the $(m-1) \times (m-1)$ matrix that we can obtain by removing the $j$th row and column of $\mathbf{A_c}$, and $\mathbf{C(j)}$ be the row vector that we can obtain by removing the $j$th element of $\mathbf{C}$.

First, let's consider the case when $\mathbf{A_c}$ is a diagonal matrix. In this case, using properties of diagonal matrices we can easily check that $\mathbf{O_{i,j}}=\begin{bmatrix} \mathbf{C(j)e^{-k_1 \mathbf{A_c(j)}}} \\ \vdots \\  \mathbf{C(j)e^{-k_{i-1} \mathbf{A_c(j)}}} \\ \mathbf{C(j)e^{-k_{i+1} \mathbf{A_c(j)}}} \\ \vdots \\ \mathbf{C}e^{-k_m \mathbf{A_c}} \end{bmatrix}$. 

In other words, $\mathbf{O_{i,j}}$ are also the observability Gramian of $(\mathbf{A_c(j)}, \mathbf{C(j)})$. Since $C_{i,j}$ is the determinant of $\mathbf{O_{i,j}}$, we can apply Lemma~\ref{lem:det:upper} to conclude that there exists a positive polynomial $p_{i,j}$ such that
\begin{align}
|C_{i,j}| \leq \left\{
\begin{array}{ll}
p_{i,j}(k_m) \left(\prod^{j-1}_{l=1} e^{-\lambda_l k_l} \right)\cdot \left(\prod^{i-1}_{l=j} e^{-\lambda_{l+1}k_l} \right)\cdot \left(\prod^{m}_{l=i+1}e^{-\lambda_l k_l} \right) & \mbox{if }i \geq j \ \\
p_{i,j}(k_m) \left(\prod^{i-1}_{l=1} e^{-\lambda_l k_l} \right)\cdot \left(\prod^{j-1}_{l=i} e^{-\lambda_l k_{l+1}} \right)\cdot \left(\prod^{m}_{l=j+1}e^{-\lambda_l k_l} \right) & \mbox{if }i \leq j
\end{array}
\right. 
\label{eqn:cofactorupper}
\end{align}

Then, let's consider the case when $\mathbf{A_c}$ is a general Jordan form matrix. Compared to the case of diagonal matrix $\mathbf{A_c}$, the elements of $\mathbf{O_{i,j}}$ only differ by polynomials on $k_i$ in ratio. Therefore, by the same argument of the proof of Lemma~\ref{lem:det:upper}, we can still find a positive polynomial $p_{i,j}$ satisfying \eqref{eqn:cofactorupper}.

Moreover, since $\lambda_1 \geq \lambda_2 \geq \cdots \geq \lambda_m \geq 0$ and $0 \leq k_1 \leq k_2 \leq \cdots \leq k_m$, we have
\begin{align}
&\left(\prod^{j-1}_{l=1} e^{-\lambda_l k_l} \right)\cdot \left(\prod^{i-1}_{l=j} e^{-\lambda_{l+1}k_l} \right)\cdot \left(\prod^{m}_{l=i+1}e^{-\lambda_l k_l} \right) \leq \prod^m_{i=2}e^{-\lambda_i k_{i-1}}, \nonumber \\
&\left(\prod^{i-1}_{l=1} e^{-\lambda_l k_l} \right)\cdot \left(\prod^{j-1}_{l=i} e^{-\lambda_l k_{l+1}} \right)\cdot \left(\prod^{m}_{l=j+1}e^{-\lambda_l k_l} \right) \leq \prod^m_{i=2}e^{-\lambda_i k_{i-1}}. \nonumber
\end{align}
Therefore, we can further bound the cofactor as follows:
\begin{align}
|C_{i,j}| \leq \max_{i,j}p_{i,j}(k_m) \prod^m_{i=2}e^{-\lambda_i k_{i-1}}. \nonumber
\end{align}
Then, we have
\begin{align}
&\left| \begin{bmatrix} \mathbf{C}e^{-k_1 \mathbf{A_c}} \\ \mathbf{C}e^{-k_2 \mathbf{A_c}} \\ \vdots \\ \mathbf{C}e^{-k_m \mathbf{A_c}} \end{bmatrix}^{-1} \right|_{max}  = \frac{\max_{i,j} |C_{i,j}|}{\left| \det\left( \begin{bmatrix}\mathbf{C}e^{-k_1 \mathbf{A_c}} \\ \mathbf{C}e^{-k_2 \mathbf{A_c}} \\ \vdots \\ \mathbf{C}e^{-k_m \mathbf{A_c}} \end{bmatrix} \right) \right|} \leq \frac{\max_{i,j} |C_{i,j}|}{\epsilon \prod_{1 \leq i \leq m}e^{-k_i \lambda_i}} \nonumber \\
& \leq \frac{\max_{i,j}p_{i,j}(k_m) \prod^m_{i=2}e^{-\lambda_i k_{i-1}}}{\epsilon \prod_{1 \leq i \leq m}e^{-k_i \lambda_i}}\nonumber \\
&=\frac{\max_{i,j}p_{i,j}(k_m)}{\epsilon} e^{\lambda_1 k_1} \prod^{m}_{i=2}e^{\lambda_i(k_i - k_{i-1})} \nonumber \\
&\leq \frac{\max_{i,j}p_{i,j}(k_m)}{\epsilon} e^{\lambda_1 k_1} \prod^{m}_{i=2}e^{\lambda_1(k_i - k_{i-1})} (\because \lambda_1 \geq \lambda_i \geq 0, k_i - k_{i-1} \geq 0 )\nonumber \\
&=\frac{\max_{i,j}p_{i,j}(k_m)}{\epsilon} e^{\lambda_1 k_m} \nonumber \\
&\leq \frac{\sum_{i,j}p_{i,j}(k_m)}{\epsilon} e^{\lambda_1 k_m} \nonumber
\end{align}
Therefore, the lemma is true.
\end{proof}

Now, the question is reduced to an issue regarding that the observability Gramian determinant has to be large enough. We will find a sufficient condition for the determinant to be large in terms of a simpler analytic function. For this, we first need the following lemma that basically asserts that polynomials increases slower than the exponentials.


\begin{lemma}
For any given polynomial $f(x)$, $\lambda > 0$ and $\epsilon>0$,
there exists $a > 0$ such that
\begin{align}
|f(k+x)| \leq \epsilon e^{\lambda \cdot x}
\end{align}
for all $x \geq a(\log(k+1)+1)$ and $k \geq 0$.
\label{lem:conti:ineq}
\end{lemma}
\begin{proof}
Let the order of $f(x)$ be $p$. Then, there exists $c>0$ such that for all $x \geq 0$,
\begin{align}
|f(x)| \leq c( 1+x^{p+1}). \nonumber
\end{align}
If we consider $\frac{1}{\lambda} \log \frac{c}{\epsilon} + \frac{1}{\lambda} \log(1+(2x)^{p+1})$ and $x$, the former grows logarithmically on $x$ while the later grows linearly on $x$. Therefore, we can find $t>0$ such that
\begin{align}
\frac{1}{\lambda} \log \frac{c}{\epsilon} + \frac{1}{\lambda} \log(1+(2x)^{p+1}) \leq x \nonumber
\end{align}
for all $x \geq t$. We can also finde $a > 0$ such that $a(\log(k+1)+1) \geq \max\left\{ \frac{1}{\lambda} \log \frac{c}{\epsilon} + \frac{1}{\lambda} \log(1+(2k)^{p+1})  , t \right\}$ for all $k \geq 0$.

To check the condition, $|f(k+x)| \leq \epsilon e^{\lambda \cdot x}$, we divide into two cases.

(a) When $x \leq k$,

$|f(k+x)|$ is bounded as follows:
\begin{align}
|f(k+x)| &\leq c\left( 1+ \left(k+x \right)^{p+1} \right) \nonumber \\
&\leq c \left(1+ \left(2k \right)^{p+1} \right) \nonumber \\
&= \epsilon e^{\lambda(\frac{1}{\lambda} \log \frac{c}{\epsilon} + \frac{1}{\lambda} \log(1+(2k)^{p+1}) )} \nonumber \\
&\leq \epsilon e^{\lambda \cdot x} \nonumber
\end{align}
where the last inequality comes from $\frac{1}{\lambda} \log \frac{c}{\epsilon} + \frac{1}{\lambda} \log(1+(2k)^{p+1}) \leq x$.

(b) When $x > k$,

Since $t \leq x$, $\frac{1}{\lambda} \log \frac{c}{\epsilon} + \frac{1}{\lambda} \log(1+(2x)^{p+1}) \leq x$. Then, we can bound $|f(k+x)|$ as follows:
\begin{align}
|f(k+x)| &\leq c\left( 1+ \left(k+x \right)^{p+1} \right) \nonumber \\
& \leq c \left(1+ \left(2x \right)^{p+1} \right) \nonumber \\
&= \epsilon e^{\lambda(\frac{1}{\lambda} \log \frac{c}{\epsilon} + \frac{1}{\lambda} \log(1+(2x)^{p+1})  )} \nonumber \\
&\leq \epsilon e^{\lambda \cdot x}. \nonumber
\end{align}
Therefore, the lemma is proved.
\end{proof}

Now, we give a sufficient condition to guarantee that the determinant of the observability Gramian is large enough.

\begin{lemma}
Let $\mathbf{A_c}$ and $\mathbf{C}$ be given as \eqref{eqn:conti:a} and \eqref{eqn:conti:c}. Denote $a_{i,j}$ and $C_{i,j}$ be the $(i,j)$ element and cofactor of
$\begin{bmatrix} \mathbf{C} e^{-k_1 \mathbf{A_c}}  \\
\mathbf{C} e^{-k_2 \mathbf{A_c}}  \\
\vdots \\
\mathbf{C} e^{-k_m \mathbf{A_c}}\end{bmatrix}$ respectively.
Then there exist $g_{\epsilon}(k):\mathbb{R}^+ \rightarrow \mathbb{R}^+$ and $a \in \mathbb{R}^+$such that for all $\epsilon > 0$ and $k_1,\cdots,k_m$ satisfying
\begin{align}
&(i)~ 0 \leq k_1 < k_2 < \cdots < k_m \nonumber \\
&(ii)~ k_{m}-k_{m-1} \geq g_\epsilon(k_{m-1}) \nonumber\\
&(iii)~ g_\epsilon(k) \leq a(  1+\log (k+1)) \nonumber \\
&(iv)~|\sum_{m-m_{\mu}+1 \leq i \leq m}a_{m,i}C_{m,i}| \geq \epsilon \prod_{1 \leq i \leq m}  e^{- k_i \lambda_{i}} \nonumber
\end{align}
the following inequality holds:
\begin{align}
\left| \det\left(
\begin{bmatrix}
\mathbf{C}e^{-k_1 \mathbf{A_c}} \\
\mathbf{C}e^{-k_2 \mathbf{A_c}} \\
\vdots \\
\mathbf{C}e^{-k_m \mathbf{A_c}}
\end{bmatrix}
\right) \right|
\geq \frac{1}{2} \epsilon \prod_{1 \leq i \leq m}  e^{- k_i \lambda_{i}}. \nonumber
\end{align}
\label{lem:det:lower}
\end{lemma}
\begin{proof}
First of all, because $\mathbf{A_c}$ is in Jordan form, it is well known that the elements of $e^{-k \mathbf{A_c}}$ take a specific form~\cite{Chen}. Thus, we can prove that for all $a_{i,j}$ there exists a polynomial $p_{i,j}(k)$ such that $a_{i,j}=p_{i,j}(k_i)e^{-k_i (\lambda_j+j \omega_j)}$.
Then, we can find $p(k)$ in the form of $a(1+k^b)$ $(a>0)$ such that $p(k) \geq \max_{i,j} |p_{i,j}(k)|$ for all $k \geq 0$.
Denote $\lambda':=\lambda_{\mu-1,1}-\lambda_{\mu,1}>0$.
\begin{align}
&\left| \det\left(\begin{bmatrix} \mathbf{C}e^{-k_1 \mathbf{A_c}} \\
\mathbf{C}e^{-k_2 \mathbf{A_c}} \\
\vdots \\
\mathbf{C}e^{-k_m \mathbf{A_c}}
\end{bmatrix}
\right)\right| = \left| \sum_{ 1 \leq i \leq m } a_{m,i} C_{m,i} \right|  = \left| \sum_{\sigma \in \Sigma_m}  sgn(\sigma) \prod^m_{i=1} a_{i,\sigma(i)} \right|\nonumber \\
& \geq \left| \sum_{ m-m_{\mu}+1 \leq i \leq m } a_{m,i} C_{m,i} \right| - \left| \sum_{ 1 \leq i \leq m-m_{\mu} } a_{m,i} C_{m,i} \right| \nonumber\\
&=
\left| \sum_{\sigma \in S_m, m-m_{\mu}+1 \leq \sigma(m) \leq m} sgn(\sigma) \prod^m_{i=1} a_{i,\sigma(i)} \right| - \left| \sum_{\sigma \in S_m, 1 \leq \sigma(m) \leq m-m_{\mu}} sgn(\sigma) \prod^m_{i=1} a_{i,\sigma(i)} \right|
\nonumber \\
& \geq \epsilon \prod_{1 \leq i \leq m}e^{-k_i \lambda_i} - \left| \sum_{ 1 \leq i \leq m-m_{\mu} } a_{m,i} C_{m,i} \right|
(\because Assumption~(iv))
\nonumber \\
& = \epsilon \prod_{1 \leq i \leq m}e^{-k_i \lambda_i} - \left| \sum_{\sigma \in S_m, 1 \leq \sigma(m) \leq m-m_{\mu}} sgn(\sigma) \prod^m_{i=1} a_{i,\sigma(i)} \right| 
\nonumber \\
& \geq \epsilon \prod_{1 \leq i \leq m}e^{-k_i \lambda_i} - \sum_{\sigma \in S_m, 1 \leq \sigma(m) \leq m-m_{\mu}} \left|  \prod^m_{i=1} a_{i,\sigma(i)} \right| \nonumber \\
& = \epsilon \prod_{1 \leq i \leq m}e^{-k_i \lambda_i} - \sum_{\sigma \in S_m, 1 \leq \sigma(m) \leq m-m_{\mu}} \left|  \prod^m_{i=1} p_{i,\sigma(i)}(k_i)e^{-k_i (\lambda_{\sigma(i)}+j \omega_{\sigma(i)})} \right| \nonumber \\
& \geq \epsilon \prod_{1 \leq i \leq m}e^{-k_i \lambda_i} - \sum_{\sigma \in S_m, 1 \leq \sigma(m) \leq m-m_{\mu}} \left(  e^{(\lambda_m-\lambda_{\sigma(m)})(k_m-k_{\sigma^{-1}(m)})} \cdot \prod^m_{i=1} p(k_i) e^{-k_i \lambda_i}  \right) (\because Lemma~\ref{lem:conti:rearr}) \nonumber \\
& \geq \prod_{1 \leq i \leq m} e^{-k_i \lambda_i} \left( \epsilon - \sum_{\sigma \in S_m, 1 \leq \sigma(m) \leq m-m_{\mu}} p(k_m)^m e^{(\lambda_m-\lambda_{\sigma(m)})(k_m - k_{\sigma^{-1}(m)})} \right) (\because p(k) \mbox{ is an increasing function.}) \nonumber \\
& \geq \prod_{1 \leq i \leq m} e^{-k_i \lambda_i} \left( \epsilon - \sum_{\sigma \in S_m, 1 \leq \sigma(m) \leq m-m_{\mu}} p(k_m)^m e^{-\lambda'(k_m - k_{m-1})} \right)
(\because \lambda_{\sigma(m)} - \lambda_m  \geq \lambda_{\mu-1,1} - \lambda_{\mu,1} = \lambda' ) \nonumber \\
& \geq \prod_{1 \leq i \leq m} e^{-k_i \lambda_i} \left( \epsilon - m! p(k_m)^m e^{-\lambda'(k_m - k_{m-1})} \right) \nonumber
\end{align}
Since $m! p(x)^m$ is a polynomial in $x$, by Lemma~\ref{lem:conti:ineq} there exists $g_{\epsilon}(k):\mathbb{R}^+ \rightarrow \mathbb{R}^+$ such that \\
(i) $g_{\epsilon}(k) \lesssim \log (k+1) + 1$\\
(ii) $| m! p(k+x)^m  | \leq \frac{\epsilon}{2}e^{\lambda' \cdot x}$ for all $x \geq g_\epsilon(k)$ and $k \geq 0$.\\
Therefore, for all $k_m$ such that $k_m-k_{m-1} \geq g_{\epsilon}(k_{m-1})$,
\begin{align}
 &\left| \det \left( \begin{bmatrix} \mathbf{C}e^{-k_1 \mathbf{A_c}} \\
\mathbf{C}e^{-k_2 \mathbf{A_c}} \\
\vdots \\
\mathbf{C}e^{-k_m \mathbf{A_c}}
\end{bmatrix} \right)  \right|
\geq \prod_{1 \leq i \leq m} e^{-k_i \lambda_i} \left( \epsilon - \frac{\epsilon}{2}e^{\lambda' \cdot (k_m - k_{m-1})} e^{-\lambda' \cdot (k_m - k_{m-1})} \right)
\geq \frac{\epsilon}{2} \prod_{1 \leq i \leq m} e^{-k_i \lambda_i}. \nonumber
\end{align}
Thus, the lemma is proved.
\end{proof}

\subsection{Uniform Convergence of a Set of Analytic Functions (Continuous-Time Systems)}
\label{app:unif:conti}
We will prove that after introducing nonuniform sampling, the determinant of the observability Gramian will become large enough regardless of the erasure pattern. Since the determinant of the observability Gramian is an analytic function, to prove that the observability Gramian is large enough it is enough prove that a set of specific analytic functions are large enough. To this end, we will prove a set of analytic functions are uniformly away from $0$.

First, we prove that an analytic function can become zero only on sets of zero Lebesgue-measure set, as long as the function is not zero for all values. The intuition for the lemma is that analytic functions can be locally determined by that Taylor expansions. Thus, if an analytic function is zero for any open interval with non-zero Lebesgue-measure, it is identically zero.
\begin{lemma}
For a given nonnegative integer $p$ and distinct positive reals $\omega_{i,1}, \omega_{i,2}, \cdots, \omega_{i,\nu_i}$, define
\begin{align}
f(x):= \sum^{p}_{i=0} x^i \left( \sum^{\nu_i}_{j=1} a_{R,i,j} \cos(\omega_{i,j}x) + a_{I,i,j} \sin(\omega_{i,j}x) \right) \nonumber
\end{align}
where at least one coefficient among $a_{R,i,j}, a_{I,i,j}$ is non-zero.
Let $X$ be a uniform random variable in $[0,T]~(T>0)$. Then, for all $h \in \mathbb{R}$, the following is true:
\begin{align}
\mathbb{P} \{ | f(X)-h | < \epsilon \} \rightarrow 0 \mbox{ as } \epsilon \downarrow 0. \nonumber
\end{align}
\label{lem:uni:1}
\end{lemma}
\begin{proof}
First, notice that $f(x)-h$ is an analytic function. It is well-known that if an analytic function $f(x)-h$ is not identically zero, the set $\{x \in [0,T] : f(x)-h=0 \}$ is an isolated set~\cite{krantz2002primer}, which is countable. Therefore, $\mathbb{P}\{ |f(X)-h|= 0 \} =0$. Moreover, $\mathbb{P}\{ |f(X)-h| < \epsilon \} \leq \mathbb{P}\{ |f(X)-h| \leq \epsilon \}$, which is a cumulative distribution function. Since cumulative distribution functions are right-continuous, $\lim_{\epsilon \downarrow 0} \mathbb{P}\{|f(X)-h| < \epsilon \}
\leq \lim_{\epsilon \downarrow 0} \mathbb{P}\{|f(X)-h| \leq \epsilon \} = \mathbb{P}\{|f(X)-h| = 0 \}=0$.

Thus, the proof reduces to proving $f(x)-h$ is not zero for all $x$.
Let $i^*$ be the largest $i$ such that either $a_{R,i,j}$ or $a_{I,i,j}$ is non-zero.

(i) When $i^*=0$,

In this case, there are no polynomial terns and only sinusoidal terms exist.
Let's compute the energy of $f(x)-h$ in interval $[s, s+r]$ and prove that $f(x)-h$ is not identically zero for all $s$ as long as $r$ is large enough.
\begin{align}
&\int^{s+r}_{s} \left( \sum^{\nu_{i^*}}_{j=1} (a_{R,i^*,j}\cos(\omega_{i^*,j} x) + a_{I,i^*,j} \sin( \omega_{i^*,j} x)) -h \right)^2  dx \nonumber \\
&= \int^{s+r}_{s}  \sum^{\nu_{i^*}}_{j=1} \left(a^2_{R,i^*,j} \cos^2 ( \omega_{i^*,j} x )+a^2_{I,i^*,j} \sin^2 (\omega_{i^*,j} x) \right) + h^2 +2 \sum_{i\leq j} a_{R,i^*,i}a_{I,i^*,j} \cos(\omega_{i^*,i}x)  \sin(\omega_{i^*,j}x) \nonumber \\
&+ 2 \sum_{i < j} a_{R,i^*,i}a_{R,i^*,j} \cos(\omega_{i^*,i}x)\cos(\omega_{i^*,j}x)
+ 2 \sum_{i < j} a_{I,i^*,i}a_{I,i^*,j} \sin(\omega_{i^*,i}x)\sin(\omega_{i^*,j}x) \nonumber \\
&- 2 \sum^{\nu_{i^*}}_{j=1} (a_{R,i^*,j}\cos(\omega_{i^*,j} x) + a_{I,i^*,j} \sin( \omega_{i^*,j} x)) h
\,\,dx \nonumber \\
&= \int^{s+r}_{s}  \sum^{\nu_{i^*}}_{j=1} \left(a^2_{R,i^*,j} \frac{1+\cos 2\omega_{i^*,j}x}{2}+a^2_{I,i^*,j} \frac{1-\cos 2\omega_{i^*,j}x}{2}\right) \,\,dx\nonumber \\
&+ \int^{s+r}_{s} \sum_{i\leq j} a_{R,i^*,i}a_{I,i^*,j}\left(\sin\left(\left(\omega_{i^*,j}+\omega_{i^*,j}\right)x \right) - \sin\left(\left(\omega_{i^*,j}-\omega_{i^*,j}\right)x \right)\right)\,\,dx  \nonumber \\
&+ \int^{s+r}_{s} \sum_{i < j} a_{R,i^*,i}a_{R,i^*,j} \left(\cos\left(\left(\omega_{i^*,j}-\omega_{i^*,j}\right)x \right) + \cos\left(\left(\omega_{i^*,j}+\omega_{i^*,j}\right)x \right)\right)\,\,dx \nonumber \\
&+ \int^{s+r}_{s} \sum_{i < j} a_{I,i^*,i}a_{I,i^*,j} \left(\cos\left(\left(\omega_{i^*,j}-\omega_{i^*,j}\right)x \right) - \cos\left(\left(\omega_{i^*,j}+\omega_{i^*,j}\right)x \right)\right)\,\,dx  \nonumber \\
&- \int^{s+r}_{s} 2 \sum^{\nu_{i^*}}_{j=1} (a_{R,i^*,j}\cos(\omega_{i^*,j} x) + a_{I,i^*,j} \sin( \omega_{i^*,j} x)) h \,\,dx.  \label{eqn:analyticnotzero}
\end{align}
Therefore, as $r$ increases, the first term in \eqref{eqn:analyticnotzero} arbitrarily increases regardless of $s$, while the remaining terms in \eqref{eqn:analyticnotzero} are sinusoidal and so bounded. Thus, $f(x)-h$ is not identically zero for all $s$ when $r$ is large enough. Thus, there exist $\delta > 0$ and $r > 0$ such that for all $s$, $|f(x)-x| \geq \delta$ holds for some $x \in [s,s+r]$.

(ii) When $i^* \geq 1$,

In this case, we have polynomial terms and we will prove that the term with the highest degree will dominate the reaming terms.
By the argument of (i), we can find $\delta > 0$ and $r > 0$ such that for all $s \geq 0$ we can find $x \in [s,s+r]$ satisfying
\begin{align}
|f(x)-h| \geq \delta x^{i^*}  - \sum^{i^*-1}_{i=0} \left( \sum^{\nu_i}_{j=1} |a_{R,i,j}|+|a_{I,i,j}| \right) x^i - |h|. \nonumber
\end{align}
Since we can choose $s$ arbitrarily large, $|f(x)-h|$ has to be greater than $0$ for some $x$. Thus, $f(x)-h$ is not identically zero.

Therefore, the lemma is true.
\end{proof}

To prove uniform convergence, we need the following Dini's theorem which says that for compact sets, pointwise convergence implies uniform convergence. The intuition behind this theorem is as follows: since we can find a finite open cover for a compact set, we can convert the uniform convergence of an infinite number of functions to the uniform convergence of only finitely many functions when the domain is compact. The uniform convergence of a finite number of functions immediately follows from pointwise convergence.

\begin{theorem}[Dini's Theorem]
\cite[p. 81]{Gelbaum}
If $\{f_n\}$ is a sequence of functions defined on a set $A$ and converging on $A$ to a function $f$, and if\\
(i) the convergence is monotonic,\\
(ii) $f_n$ is continuous on $A$, $n=1,2,\cdots$\\
(iii) $f$ is continuous on $A$,\\
(iv) $A$  is compact,\\
then the convergence is uniform on $A$.\label{thm:dini}
\end{theorem}
\begin{proof}
See \cite[p. 81]{Gelbaum} for the proof.
\end{proof}

Now, using the pointwise convergence of Lemma~\ref{lem:uni:1} and Dini's theorem, we can prove the uniform convergence of the relevant functions over a set of parameters.


\begin{lemma}
Let $p$, $\nu_{0}, \cdots, \nu_p $ be nonnegative integers with $\nu_{p}>0$. Suppose $\gamma$ and $\Gamma$ are strictly positive reals such that $\gamma \leq \Gamma$. For each $0 \leq i \leq p$, $\omega_{i,1},\omega_{i,2}, \cdots, \omega_{i,\nu_i}$ are distinct reals.
Let $X$ be a uniform random variable on $[0,T]$ for some $T > 0$. Then, for all $m,n$ such that $0 \leq m \leq p$ and $1 \leq n \leq \nu_m$, we have the following inequality:
\begin{align}
\sup_{ |a_{m,n}| \geq \gamma, \forall i,j, |a_{i,j}| \leq \Gamma} \mathbb{P}\left\{\left| \sum^{p}_{i=0} X^i \left( \sum^{\nu_i}_{j=1} a_{i,j}e^{j\omega_{i,j}X} \right) \right| < \epsilon \right\} \rightarrow 0 \mbox{ as }\epsilon \downarrow 0 \nonumber
\end{align}
where $a_{i,j}$ are taken from $\mathbb{C}$.
\label{lem:single}
\end{lemma}

\begin{proof}
The purpose of this proof is reducing the lemma to Dini's theorem (Theorem~\ref{thm:dini}).

First, we will assume the $w_{i,j}$ are positive without loss of generality. To justify this, let $\omega_{min}=\min\{ \min_{i,j} { \omega_{i,j}}, 0 \}-\delta$ for some $\delta>0$. Then,
\begin{align}
&\sup_{|a_{m,n}| \geq \gamma, |a_{i,j}| \leq \Gamma} \mathbb{P}\left\{ \left| \sum^{p}_{i=0} X^i \left( \sum^{\nu_i}_{j=1} a_{i,j}e^{j\omega_{i,j}X} \right) \right| < \epsilon \right\} \nonumber\\
&=\sup_{|a_{m,n}| \geq \gamma, |a_{i,j}| \leq \Gamma} \mathbb{P}\left\{ \left| \sum^{p}_{i=0} X^i \left( \sum^{\nu_i}_{j=1} a_{i,j}e^{j(\omega_{i,j}-\omega_{min})X} \right) \right| < \epsilon \right\}. \nonumber
\end{align}
Here, for each $i$, $\omega_{i,1}-\omega_{min},\omega_{i,2}-\omega_{min},\cdots,\omega_{i,\nu_i}-\omega_{min}$ are distinct and strictly positive. Therefore, without loss of generality, we can assume that for each i, $\omega_{i,1},\omega_{i,2},\cdots,\omega_{i,\nu_i}$ are distinct and strictly positive.

Let $a_{i,j}=a_{R,i,j}-j a_{I,i,j}$ where $a_{R,i,j}$ and $a_{I,i,j}$ are real. Since $|a_{m,n}| \geq \gamma$, at least one of $|a_{R,m,n}|$ or $|a_{I,m,n}|$ should be greater than $\frac{\gamma}{\sqrt{2}}$. First, consider the case when $|a_{R,m,n}| \geq \frac{\gamma}{\sqrt{2}}$. It is sufficient to prove that the real part of $\sum^{p}_{i=0} X^i \left( \sum^{\nu_i}_{j=1} a_{i,j}e^{j\omega_{i,j}X} \right)$ satisfies the lemma, i.e.
\begin{align}
\sup_{ a_{R,m,n} \geq \frac{\gamma}{\sqrt{2}} , |a_{R,i,j}| \leq \Gamma, |a_{I,i,j}| \leq \Gamma} \mathbb{P}\left\{\left| \sum^{p}_{i=0} X^i \left( \sum^{\nu_i}_{j=1} a_{R,i,j}\cos(\omega_{i,j}X)+a_{I,i,j}\sin(\omega_{i,j}X) \right) \right| < \epsilon \right\} \rightarrow 0 \mbox{ as }\epsilon \downarrow 0.\nonumber
\end{align}
Here, we take the supremum over $a_{R,m,n}\geq \frac{\gamma}{\sqrt{2}}$ instead of the supremum over $|a_{R,m,n}|\geq \frac{\gamma}{\sqrt{2}}$ by symmetry.

Now, we apply Dini's theorem~\ref{thm:dini} and prove the claim.

Fix a positive sequence $\epsilon_i$ such that $ \epsilon_i \downarrow 0$ as $i \rightarrow \infty$. Define a sequence of functions $\{ f_i \}$ as
\begin{align}
f_i(a_{R,1,1},a_{I,1,1},\cdots,a_{I,p,\nu_p}) := \mathbb{P}\left\{\left| \sum^{p}_{i=0} X^i \left( \sum^{\nu_i}_{j=1} a_{R,i,j}\cos(\omega_{i,j}X)+a_{I,i,j}\sin(\omega_{i,j}X) \right) \right| < \epsilon_i \right\}\nonumber
\end{align}
where the domain $A$ of the functions is $A:=\{(a_{R,1,1},a_{I,1,1},\cdots,a_{I,p,{\nu_p}}): a_{R,m,n} \geq \frac{\gamma}{\sqrt{2}} , |a_{R,i,j}| \leq \Gamma, |a_{I,i,j}| \leq \Gamma\}$. Let $f(a_{R,1,1},a_{I,1,1},\cdots,a_{I,p,\nu_p})$ be the identically zero function.
Then, we will prove that $\{f_i\}$ converges to $f=0$ uniformly on $A$ by checking the conditions of Theorem~\ref{thm:dini}.

$\bullet$ $f_i$ point-wisely converges to $f$:

Since $a_{R,m,n} \geq \frac{\gamma}{\sqrt{2}}$, $\sum^{p}_{i=0} x^i \left( \sum^{\nu_i}_{j=1} a_{R,i,j}\cos(\omega_{i,j}x)+a_{I,i,j}\sin(\omega_{i,j}x) \right)$ satisfies the assumptions of Lemma~\ref{lem:uni:1}. Thus, for all $h$
\begin{align}
\mathbb{P}\left\{\left| \sum^{p}_{i=0} X^i \left( \sum^{\nu_i}_{j=1} a_{R,i,j}\cos(\omega_{i,j}X)+a_{I,i,j}\sin(\omega_{i,j}X) \right) - h \right| < \epsilon \right\} \rightarrow 0 \mbox{ as } \epsilon \downarrow 0.\label{eqn:lem:single:2}
\end{align}
Therefore, by selecting $h=0$, $f_i(a_{R,1,1},a_{I,1,1},\cdots,a_{I,p,\nu_p})$ converges to $f=0$ for all $a_{R,1,1},a_{I,1,1},\cdots,a_{I,p,\nu_p}$ in $A$.

$\bullet$ Convergence is monotone: Since $\epsilon_i$ monotonically converge to $0$, $f_i$ is also a monotonically decreasing function sequence. Thus, the convergence is monotone.

$\bullet$ $f_n$ is continuous on $A$: For continuity (does not have to be uniformly continuous), we will prove that for given $a_{R,1,1},a_{I,1,1},\cdots,a_{I,p,\nu_p}$ and for all $\sigma > 0$, there exists $\delta(\sigma)>0$ such that $|f_i(a_{R,1,1}+\nabla a_{R,1,1},a_{I,1,1}+\nabla a_{I,1,1},\cdots,a_{I,p,\nu_p}+\nabla a_{I,p,\nu_p})-f_i(a_{R,1,1},a_{I,1,1},\cdots,a_{I,p,\nu_p})|<\sigma$ for all $|\nabla a_{R,i,j}|<\delta(\sigma)$ and $|\nabla a_{I,i,j}|<\delta(\sigma)$.

By \eqref{eqn:lem:single:2}, we can find $\delta'(\sigma)$ for all $\sigma$ such that
\begin{align}
&\mathbb{P}\left\{\left| \sum^{p}_{i=0} X^i \left( \sum^{\nu_i}_{j=1} a_{R,i,j}\cos(\omega_{i,j}X)+a_{I,i,j}\sin(\omega_{i,j}X) \right) - \epsilon_i \right| < \delta'(\sigma) \right\} < \frac{\sigma}{2} \mbox{ and } \nonumber \\
&\mathbb{P}\left\{\left| \sum^{p}_{i=0} X^i \left( \sum^{\nu_i}_{j=1} a_{R,i,j}\cos(\omega_{i,j}X)+a_{I,i,j}\sin(\omega_{i,j}X) \right) - \left(-\epsilon_i \right) \right| < \delta'(\sigma) \right\} < \frac{\sigma}{2}\nonumber.
\end{align}
Denote $\delta(\sigma):= \frac{\min\left(\frac{1}{T^p}, 1 \right)}{2 \sum^{p}_{i=0}\nu_i} \delta'(\sigma)$. Then, for all $|\nabla a_{R,i,j}| < \delta(\sigma)$ and $|\nabla a_{I,i,j}| < \delta(\sigma)$, the following inequality is true.
\begin{align}
&\mathbb{P}\left\{\left| \sum^{p}_{i=0} X^i \left( \sum^{\nu_i}_{j=1} \left(a_{R,i,j}+\nabla a_{R,i,j}\right)\cos(\omega_{i,j}X)+\left(a_{I,i,j}+ \nabla a_{R,i,j} \right)\sin(\omega_{i,j}X) \right) \right| < \epsilon_i  \right\} \nonumber \\
&\geq \mathbb{P}\left\{\left| \sum^{p}_{i=0} X^i \left( \sum^{\nu_i}_{j=1} a_{R,i,j}\cos(\omega_{i,j}X)+a_{I,i,j}\sin(\omega_{i,j}X) \right) \right| < \epsilon_i - \left|
\sum^{p}_{i=0} X^i \left( \sum^{\nu_i}_{j=1} \nabla a_{R,i,j}\cos(\omega_{i,j}X)+ \nabla a_{I,i,j}\sin(\omega_{i,j}X) \right)
\right|   \right\} \nonumber \\
&\geq \mathbb{P}\left\{ \left| \sum^{p}_{i=0} X^i \left( \sum^{\nu_i}_{j=1} a_{R,i,j}\cos(\omega_{i,j}X)+a_{I,i,j}\sin(\omega_{i,j}X) \right) \right| < \epsilon_i - \delta'(\sigma) \right\} \label{eqn:deltaprime}\\
&= \mathbb{P}\left\{ \left| \sum^{p}_{i=0} X^i \left( \sum^{\nu_i}_{j=1} a_{R,i,j}\cos(\omega_{i,j}X)+a_{I,i,j}\sin(\omega_{i,j}X) \right) \right| < \epsilon_i \right\} \nonumber \\
&\quad -\mathbb{P}\left\{ \epsilon_i - \delta'(\sigma) \leq \left| \sum^{p}_{i=0} X^i \left( \sum^{\nu_i}_{j=1} a_{R,i,j}\cos(\omega_{i,j}X)+a_{I,i,j}\sin(\omega_{i,j}X) \right) \right| < \epsilon_i \right\} \nonumber \\
&\geq  \mathbb{P}\left\{ \left| \sum^{p}_{i=0} X^i \left( \sum^{\nu_i}_{j=1} a_{R,i,j}\cos(\omega_{i,j}X)+a_{I,i,j}\sin(\omega_{i,j}X) \right) \right| < \epsilon_i \right\} \nonumber \\
&\quad -\mathbb{P}\left\{ \left| \sum^{p}_{i=0} X^i \left( \sum^{\nu_i}_{j=1} a_{R,i,j}\cos(\omega_{i,j}X)+a_{I,i,j}\sin(\omega_{i,j}X) \right)-\epsilon_i \right| < \delta'(\sigma) \right\} \nonumber \\
&\quad -\mathbb{P}\left\{ \left| \sum^{p}_{i=0} X^i \left( \sum^{\nu_i}_{j=1} a_{R,i,j}\cos(\omega_{i,j}X)+a_{I,i,j}\sin(\omega_{i,j}X) \right)-(-\epsilon_i) \right| < \delta'(\sigma) \right\} \nonumber \\
&> \mathbb{P}\left\{ \left| \sum^{p}_{i=0} X^i \left( \sum^{\nu_i}_{j=1} a_{R,i,j}\cos(\omega_{i,j}X)+a_{I,i,j}\sin(\omega_{i,j}X) \right) \right| < \epsilon_i \right\} - \sigma \nonumber.
\end{align}
Here, \eqref{eqn:deltaprime} can be shown as follows:
\begin{align}
&\left| \sum^{p}_{i=0} X^i \left( \sum^{\nu_i}_{j=1} \nabla a_{R,i,j}\cos(\omega_{i,j}X)+ \nabla a_{I,i,j}\sin(\omega_{i,j}X) \right) \right| \nonumber \\
& \leq \sum^{p}_{i=0} |X^i|  \sum^{\nu_i}_{j=1}\left( |\nabla a_{R,i,j}| + |\nabla a_{I,i,j}| \right) \nonumber \\
& \leq \max(T^p, 1) 2 \nu_i \delta(\sigma) (\because 0 \leq X \leq T\mbox{ w.p. 1})\nonumber \\
& = \delta'(\sigma) (\because \mbox{definition of }\delta(\sigma))
\end{align}

Therefore, by the definition of $f_i$ we have
\begin{align}
f_i(a_{R,1,1}+\nabla a_{R,1,1},a_{I,1,1}+\nabla a_{I,1,1},\cdots,a_{I,p,\nu_p}+\nabla a_{I,p,\nu_p})-f_i(a_{R,1,1},a_{I,1,1},\cdots,a_{I,p,\nu_p}) > -\sigma \label{eqn:lem:single:3}.
\end{align}

Likewise, we can prove that
\begin{align}
&\mathbb{P}\left\{\left| \sum^{p}_{i=0} X^i \left( \sum^{\nu_i}_{j=1} \left(a_{R,i,j}+\nabla a_{R,i,j}\right)\cos(\omega_{i,j}X)+\left(a_{I,i,j}+ \nabla a_{R,i,j} \right)\sin(\omega_{i,j}X) \right) \right| < \epsilon_i  \right\} \nonumber \\
&\leq \mathbb{P}\left\{ \left| \sum^{p}_{i=0} X^i \left( \sum^{\nu_i}_{j=1} a_{R,i,j}\cos(\omega_{i,j}X)+a_{I,i,j}\sin(\omega_{i,j}X) \right) \right| < \epsilon_i + \delta'(\sigma) \right\} \nonumber \\
&(\because \mbox{The same step as \eqref{eqn:deltaprime}}) \nonumber \\
&= \mathbb{P}\left\{ \left| \sum^{p}_{i=0} X^i \left( \sum^{\nu_i}_{j=1} a_{R,i,j}\cos(\omega_{i,j}X)+a_{I,i,j}\sin(\omega_{i,j}X) \right) \right| < \epsilon_i \right\} \nonumber \\
&\quad +\mathbb{P}\left\{ \epsilon_i  \leq \left| \sum^{p}_{i=0} X^i \left( \sum^{\nu_i}_{j=1} a_{R,i,j}\cos(\omega_{i,j}X)+a_{I,i,j}\sin(\omega_{i,j}X) \right) \right| < \epsilon_i+ \delta'(\sigma) \right\} \nonumber \\
&\leq  \mathbb{P}\left\{ \left| \sum^{p}_{i=0} X^i \left( \sum^{\nu_i}_{j=1} a_{R,i,j}\cos(\omega_{i,j}X)+a_{I,i,j}\sin(\omega_{i,j}X) \right) \right| < \epsilon_i \right\} \nonumber \\
&\quad +\mathbb{P}\left\{ \left| \sum^{p}_{i=0} X^i \left( \sum^{\nu_i}_{j=1} a_{R,i,j}\cos(\omega_{i,j}X)+a_{I,i,j}\sin(\omega_{i,j}X) \right)-\epsilon_i \right| < \delta'(\sigma) \right\} \nonumber \\
&\quad +\mathbb{P}\left\{ \left| \sum^{p}_{i=0} X^i \left( \sum^{\nu_i}_{j=1} a_{R,i,j}\cos(\omega_{i,j}X)+a_{I,i,j}\sin(\omega_{i,j}X) \right)-(-\epsilon_i) \right| < \delta'(\sigma) \right\} \nonumber \\
&< \mathbb{P}\left\{ \left| \sum^{p}_{i=0} X^i \left( \sum^{\nu_i}_{j=1} a_{R,i,j}\cos(\omega_{i,j}X)+a_{I,i,j}\sin(\omega_{i,j}X) \right) \right| < \epsilon_i \right\} + \sigma \nonumber
\end{align}
which implies
\begin{align}
f_i(a_{R,1,1}+\nabla a_{R,1,1},a_{I,1,1}+\nabla a_{I,1,1},\cdots,a_{I,p,\nu_p}+\nabla a_{I,p,\nu_p})-f_i(a_{R,1,1},a_{I,1,1},\cdots,a_{I,p,\nu_p}) < \sigma. \label{eqn:lem:single:4}
\end{align}
By \eqref{eqn:lem:single:3} and \eqref{eqn:lem:single:4},
\begin{align}
\left| f_i(a_{R,1,1}+\nabla a_{R,1,1},a_{I,1,1}+\nabla a_{I,1,1},\cdots,a_{I,p,\nu_p}+\nabla a_{I,p,\nu_p})-f_i(a_{R,1,1},a_{I,1,1},\cdots,a_{I,p,\nu_p}) \right| < \sigma. \nonumber
\end{align}
Therefore, $f_i(a_{R,1,1},a_{I,1,1},\cdots,a_{I,p,\nu_p})$ is continuous.

$\bullet$ $f$ is continuous on $A$: $f$ is obviously continuous, since $f$ is identically zero.

$\bullet$ $A$ is compact: $A$ is compact since it is closed and bounded.

Thus, by Dini's theorem~\ref{thm:dini}, the convergence is uniform on $A$, which finishes the proof for the case of $|a_{R,m,n}| \geq \frac{\gamma}{\sqrt{2}}$. The proof for the case of $|a_{I,m,n}| \geq \frac{\gamma}{\sqrt{2}}$ follows in an identical manner. Since there are only two cases, the function
\begin{align}
g_i(a_{1,1}, \cdots, a_{p,\nu_p}):= \mathbb{P}\left\{ \left| \sum^{p}_{i=0} X^i \left( \sum^{\nu_i}_{j=1} a_{i,j}e^{j\omega_{i,j}X} \right) \right| < \epsilon_i \right\}
\end{align}
converges uniformly on $\{ a_{i,j} : |a_{m,n}| \geq \gamma, |a_{i,j}| \leq \Gamma \}$. This finishes the proof of the lemma.
\end{proof}

In Lemma~\ref{lem:single}, we have a boundedness condition on the coefficients ($|a_{i,j}| \leq \Gamma$) to guarantee compactness. However, we can easily notice the functions only get larger as $a_{i,j}$ increases. Therefore, we can prove that Lemma~\ref{lem:single} still holds without the boundedness condition.

\begin{lemma}
Let $p$ be nonnegative integer and $\nu_{0}, \cdots, \nu_p $ be also nonnegative integers with $\nu_{p}>0$. $\gamma$ is a strictly positive real. For each $0 \leq i \leq p$, $\omega_{i,1},\omega_{i,2}, \cdots, \omega_{i,\nu_i}$ are distinct reals.
Let $X$ be a uniform random variable on $[0,T]$ for some $T > 0$. Then, for all $m,n$ such that $0 \leq m \leq p$ and $1 \leq n \leq \nu_m$, we have the following inequality:
\begin{align}
\sup_{ |a_{m,n}| \geq \gamma} \mathbb{P}\left\{\left| \sum^{p}_{i=0} X^i \left( \sum^{\nu_i}_{j=1} a_{i,j}e^{j\omega_{i,j}X} \right) \right| < \epsilon \right\} \rightarrow 0 \mbox{ as }\epsilon \downarrow 0 \nonumber
\end{align}
where $a_{i,j}$ are taken from $\mathbb{C}$.
\label{lem:singleun}
\end{lemma}
\begin{proof}
Denote $\nu :=\sum^{p}_{i=0} \nu_i$. The proof is by strong induction on $\nu$.

(i) When $\nu=1$.

\begin{align}
&\sup_{|a_{p,1}| \geq \gamma} \mathbb{P}\left\{ \left| a_{p,1} X^p e^{j \omega_{p,1}X} \right| < \epsilon \right\} \label{eqn:lem:sigleun:0} \\
&= \sup_{|a_{p,1}| \geq \gamma} \mathbb{P}\left\{ \left| \frac{\gamma}{|a_{p,1}|} a_{p,1} X^p e^{j \omega_{p,1}X} \right| < \frac{\gamma}{|a_{p,1}|} \epsilon \right\} \nonumber \\
&\leq \sup_{|a'_{p,1}| = \gamma} \mathbb{P}\left\{ \left| a'_{p,1} X^p e^{j \omega_{p,1}X} \right| < \epsilon \right\} \left( \because \frac{\gamma}{|a_{p,1}|} \leq 1 \right)\label{eqn:lem:sigleun:1}
\end{align}
By lemma \ref{lem:single}, \eqref{eqn:lem:sigleun:1} converges to 0 as $\epsilon \downarrow 0$. Thus, \eqref{eqn:lem:sigleun:0} converges to 0 as $\epsilon \downarrow 0$.

(ii) As an induction hypothesis, we assume the lemma is true for $\nu=1,\cdots,n-1$ and prove that the lemma still holds for $\nu = n$. We will prove this by dividing into two cases: (a) When all $a_{i,j}$ are not much bigger than $a_{m,n}$. In this case, the claim reduces to Lemma~\ref{lem:single}. (b) When there is an $a_{m',n'}$ which is much bigger than $a_{m,n}$. In this case, we can ignore the term associated with $a_{m,n}$ and reduce the number of terms in the functions. Thus, either way the claim reduces to the induction hypothesis.

To prove the lemma for $\nu = n$, it is enough to show that for a fixed $\gamma$ and every $\delta > 0$, there exists $\epsilon(\delta)>0$ such that
\begin{align}
\sup_{ |a_{m,n}| \geq \gamma } \mathbb{P}\left\{\left| \sum^{p}_{i=0} X^i \left( \sum^{\nu_i}_{j=1} a_{i,j}e^{j\omega_{i,j}X} \right) \right| < \epsilon(\delta) \right\} < \delta. \nonumber
\end{align}
By the induction hypothesis for all $(m',n') \neq (m,n)$ we can find $\epsilon_{m',n'}(\delta)>0$ such that
\begin{align}
&\sup_{a_{m,n}=0, |a_{m',n'}|\geq \gamma} \mathbb{P} \left\{ \left| \sum^p_{i=0} X^p \left( \sum^{\nu_i}_{j=1} a_{i,j}e^{j\omega_{i,j}X} \right) \right| < \epsilon_{m',n'}(\delta) \right\} < \delta. \label{eqn:lem:sigleun:2}
\end{align}
We choose $\kappa(\delta)$ as $\min\left\{ \min_{(m',n')\neq (m,n)} \left\{ \frac{\epsilon_{m',n'}(\delta)}{2 \gamma T^m} \right\}, 1 \right\}$. By Lemma \ref{lem:single}, there exists $\epsilon'(\delta)>0$ such that
\begin{align}
&\sup_{|a_{m,n}| = \gamma , a_{i,j} \leq \frac{\gamma}{\kappa(\delta)} } \mathbb{P} \left\{ \left| \sum^p_{i=0} X^p \left( \sum^{\nu_i}_{j=1} a_{i,j}e^{j\omega_{i,j}X} \right) \right| < \epsilon'(\delta) \right\} < \delta. \label{eqn:lem:sigleun:22}
\end{align}
Denote $\epsilon(\delta):=\min\left\{ \epsilon'(\delta), \min_{(m',n')\neq (m,n)} \left\{ \frac{\epsilon_{m',n'}(\delta)}{2} \right\} \right\}$. Then, we have
\begin{align}
&\sup_{ |a_{m,n}| \geq \gamma } \mathbb{P}\left\{\left| \sum^{p}_{i=0} X^i \left( \sum^{\nu_i}_{j=1} a_{i,j}e^{j\omega_{i,j}X} \right) \right| < \epsilon(\delta) \right\} \nonumber \\
&= \max \{  \sup_{ |a_{m,n}| \geq \gamma, \frac{|a_{i,j}|}{|a_{m,n}|} \leq  \frac{1}{\kappa(\delta)} } \mathbb{P}\left\{\left| \sum^{p}_{i=0} X^i \left( \sum^{\nu_i}_{j=1} a_{i,j}e^{j\omega_{i,j}X} \right) \right| < \epsilon(\delta) \right\}, \label{eqn:lem:sigleun:20} \\
& \max_{(m',n') \neq (m,n)} \sup_{ |a_{m,n}| \geq \gamma, \frac{|a_{m',n'}|}{|a_{m,n}|} \geq  \frac{1}{\kappa(\delta)} } \mathbb{P}\left\{\left| \sum^{p}_{i=0} X^i \left( \sum^{\nu_i}_{j=1} a_{i,j}e^{j\omega_{i,j}X} \right) \right| < \epsilon(\delta) \right\} \label{eqn:lem:sigleun:21} \\
&\}. \label{eqn:lem:sigleun:8}
\end{align}

$\bullet$ When $a_{i,j}$ are not too bigger than $a_{m,n}$: Let's bound the first term in \eqref{eqn:lem:sigleun:20}. Set $a'_{i,j}:=\frac{\gamma}{|a_{m,n}|}a_{i,j}$. Then, \eqref{eqn:lem:sigleun:20} is upper bounded as follows:
\begin{align}
&\sup_{ |a_{m,n}| \geq \gamma, \frac{|a_{i,j}|}{|a_{m,n}|} \leq  \frac{1}{\kappa(\delta)} } \mathbb{P}\left\{\left| \sum^{p}_{i=0} X^i \left( \sum^{\nu_i}_{j=1} a_{i,j}e^{j\omega_{i,j}X} \right) \right| < \epsilon(\delta) \right\}\nonumber \\
&= \sup_{ |a_{m,n}| \geq \gamma, \frac{|a_{i,j}|}{|a_{m,n}|} \leq  \frac{1}{\kappa(\delta)} } \mathbb{P}\left\{\left| \sum^{p}_{i=0} X^i \left( \sum^{\nu_i}_{j=1} \frac{\gamma}{|a_{m,n}|}a_{i,j}e^{j\omega_{i,j}X} \right) \right| < \frac{\gamma}{|a_{m,n}|}\epsilon(\delta) \right\} \nonumber \\
&= \sup_{ |a'_{m,n}| = \gamma, |a'_{i,j}| \leq  \frac{\gamma}{\kappa(\delta)} } \mathbb{P}\left\{\left| \sum^{p}_{i=0} X^i \left( \sum^{\nu_i}_{j=1} a'_{i,j}e^{j\omega_{i,j}X} \right) \right| < \frac{\gamma}{|a_{m,n}|}\epsilon(\delta) \right\}\nonumber \\
&\leq \sup_{ |a'_{m,n}| = \gamma, |a'_{i,j}| \leq  \frac{\gamma}{\kappa(\delta)} } \mathbb{P}\left\{\left| \sum^{p}_{i=0} X^i \left( \sum^{\nu_i}_{j=1} a'_{i,j}e^{j\omega_{i,j}X} \right) \right| < \epsilon(\delta) \right\} (\because \frac{\gamma}{|a_{m,n}|}\leq 1) \nonumber \\
&\leq \sup_{ |a'_{m,n}| = \gamma, |a'_{i,j}| \leq  \frac{\gamma}{\kappa(\delta)} } \mathbb{P}\left\{\left| \sum^{p}_{i=0} X^i \left( \sum^{\nu_i}_{j=1} a'_{i,j}e^{j\omega_{i,j}X} \right) \right| < \epsilon'(\delta) \right\} (\because \mbox{definition of $\epsilon(\delta)$})\nonumber \\
&< \delta (\because \eqref{eqn:lem:sigleun:22})\label{eqn:lem:sigleun:3}
\end{align}

$\bullet$ When $a_{m',n'}$ is much bigger than $a_{m,n}$: Let's bound the second term in \eqref{eqn:lem:sigleun:21}. For given $m', n'$, set $a''_{i,j}:=\frac{\gamma}{|a_{m',n'}|}a_{i,j}$. Then, \eqref{eqn:lem:sigleun:21} is upper bounded by
\begin{align}
&\max_{(m',n') \neq (m,n)} \sup_{ |a_{m,n}| \geq \gamma, \frac{|a_{m',n'}|}{|a_{m,n}|} \geq  \frac{1}{\kappa(\delta)} } \mathbb{P}\left\{\left| \sum^{p}_{i=0} X^i \left( \sum^{\nu_i}_{j=1} a_{i,j}e^{j\omega_{i,j}X} \right) \right| < \epsilon(\delta) \right\} \nonumber \\
&= \max_{(m',n') \neq (m,n)} \sup_{ |a_{m,n}| \geq \gamma, \frac{|a_{m',n'}|}{|a_{m,n}|} \geq  \frac{1}{\kappa(\delta)} } \mathbb{P}\left\{\left| \sum^{p}_{i=0} X^i \left( \sum^{\nu_i}_{j=1} \frac{\gamma}{|a_{m',n'}|} a_{i,j}e^{j\omega_{i,j}X} \right) \right| < \frac{\gamma}{|a_{m',n'}|}\epsilon(\delta) \right\} \nonumber \\
&\leq  \max_{(m',n') \neq (m,n)} \sup_{ |a_{m,n}| \geq \gamma, \frac{|a_{m',n'}|}{|a_{m,n}|} \geq  \frac{1}{\kappa(\delta)} } \mathbb{P}\Bigg\{\left| \sum^{p}_{i=0} X^i \left( \sum^{\nu_i}_{j=1} \frac{\gamma}{|a_{m',n'}|} a_{i,j}e^{j\omega_{i,j}X} \right)
- X^m \frac{\gamma}{|a_{m',n'}|}a_{m,n}e^{j \omega_{m,n} X }
\right|  \nonumber \\
&<  \max_{(m',n') \neq (m,n)} \frac{\gamma}{|a_{m',n'}|}\epsilon(\delta) + \frac{\gamma }{|a_{m',n'}|} |a_{m,n}| T^m  \Bigg\} \nonumber \\
&\leq \max_{(m',n') \neq (m,n)} \sup_{ |a_{m,n}| \geq \gamma, \frac{|a_{m',n'}|}{|a_{m,n}|} \geq  \frac{1}{\kappa(\delta)} } \mathbb{P}\left\{\left| \sum^{p}_{i=0} X^i \left( \sum^{\nu_i}_{j=1} \frac{\gamma}{|a_{m',n'}|} a_{i,j}e^{j\omega_{i,j}X} \right)
- X^m \frac{\gamma}{|a_{m',n'}|}a_{m,n}e^{j \omega_{m,n} X }
\right| < \epsilon_{m',n'}(\delta)  \right\} \label{eqn:lem:sigleun:4} \\
&\leq \max_{(m',n') \neq (m,n)} \sup_{ a_{m,n}''=0, |a_{m',n'}''| =  \gamma } \mathbb{P}\left\{\left| \sum^{p-1}_{i=0} X^i \left( \sum^{\nu_i}_{j=1} a''_{i,j}e^{j\omega_{i,j}X} \right)
\right| < \epsilon_{m',n'}(\delta)  \right\} \quad
(\because \mbox{By definition, } a''_{m',n'}=\frac{\gamma}{|a_{m',n'}|} a_{m',n'}) \nonumber \\
&< \delta (\because \eqref{eqn:lem:sigleun:2}) \label{eqn:lem:sigleun:7}
\end{align}
Here, \eqref{eqn:lem:sigleun:4} can be derived as follows:
First, we have
\begin{align}
1 &\geq \kappa(\delta)\quad (\because \mbox{Definition of $\kappa(\delta)$}) \nonumber \\
&\geq \frac{\gamma \cdot \kappa(\delta)}{|a_{m,n}|} \quad (\because |a_{m,n}|\geq \gamma) \nonumber \\
&\geq \frac{\gamma}{|a_{m',n'}|}. \quad (\because \frac{|a_{m',n'}|}{|a_{m,n}|} \geq \frac{1}{\kappa(\delta)} ) \label{eqn:lem:sigleun:5}
\end{align}
We also have
\begin{align}
\frac{\gamma}{|a_{m',n'}|} |a_{m,n}| T^m &\leq \gamma  \cdot \kappa(\delta) T^m \quad (\because \frac{|a_{m',n'}|}{|a_{m,n}|} \geq \frac{1}{\kappa(\delta)} )\nonumber \\
&\leq \gamma \frac{\epsilon_{m',n'}(\delta)}{2 \gamma T^m} T^m \quad(\because \mbox{By definition, } \kappa(\delta) \leq \frac{\epsilon_{m',n'}(\delta)}{2 \gamma T^m}) \nonumber \\
&= \frac{\epsilon_{m',n'}(\delta)}{2}. \label{eqn:lem:sigleun:6}
\end{align}
Therefore,
\begin{align}
\frac{\gamma}{|a_{m',n'}|}\epsilon(\delta) + \frac{\gamma }{|a_{m',n'}|} |a_{m,n}| T^m &\leq \epsilon(\delta)+\frac{\epsilon_{m',n'}(\delta)}{2} \quad (\because \eqref{eqn:lem:sigleun:5},\eqref{eqn:lem:sigleun:6} )\nonumber \\
&\leq \epsilon_{m',n'}(\delta). \quad (\because \mbox{By definition, }\epsilon(\delta) \leq \frac{\epsilon_{m',n'}(\delta)}{2}) \nonumber
\end{align}
Therefore, \eqref{eqn:lem:sigleun:4} is true.

By plugging \eqref{eqn:lem:sigleun:3} and \eqref{eqn:lem:sigleun:7} into \eqref{eqn:lem:sigleun:8}, we get
\begin{align}
&\sup_{ |a_{m,n}| \geq \gamma } \mathbb{P}\left\{\left| \sum^{p}_{i=0} X^i \left( \sum^{\nu_i}_{j=1} a_{i,j}e^{j\omega_{i,j}X} \right) \right| < \epsilon(\delta) \right\} < \delta, \nonumber
\end{align}
which finishes the proof.
\end{proof}

\subsection{Proof of Lemma~\ref{lem:conti:mo}}
\label{sec:app:2}

In this section, we will merge the properties about the observability Gramian shown in Section~\ref{sec:app:3} with the uniform convergence of Section~\ref{app:unif:conti}, and prove Lemma~\ref{lem:conti:mo} of page~\pageref{lem:conti:mo}.


We first prove the following lemma which tells us that the determinant of the observability Gramian is large with high probability under a cofactor condition on the Gramian.
\begin{lemma}
Let $\mathbf{A_c}$ and $\mathbf{C}$ be given as \eqref{eqn:conti:a} and \eqref{eqn:conti:c}. Let $a_{i,j}$ and $C_{i,j}$ be the $(i,j)$ element and cofactor of
$\begin{bmatrix}
\mathbf{C} e^{-(k_1 I + t_1)\mathbf{A_c}} \\
\vdots \\
\mathbf{C} e^{-(k_{m-1} I + t_{m-1})\mathbf{A_c}} \\
\mathbf{C} e^{-(k_m I + t)\mathbf{A_c}}
\end{bmatrix}$ respectively, where $t$ is a random variable which is uniformly distributed on $[0,T]$ and $I$ is the sampling interval defined in \eqref{eqn:non:4}.
Then, there exist $a \in \mathbb{R}^+$ and a family of increasing functions $\{g_{\epsilon}(\cdot) : \epsilon > 0, g_{\epsilon}:\mathbb{R}^+ \rightarrow \mathbb{R}^+ \}$ satisfying:\\
(i) For all $\epsilon>0$, $k_1 < k_2 < \cdots < k_{m-1}$, $0\leq t_i \leq T$ if $|C_{m,m}| > \epsilon \prod_{1 \leq i \leq m-1} e^{-k_i I \cdot \lambda_i}$ the following is true:
\begin{align}
\sup_{k_m \in \mathbb{Z}, k_m - k_{m-1} \geq g_{\epsilon}(k_{m-1})} \mathbb{P} \left\{ \left| \det\left(
\begin{bmatrix}
\mathbf{C}e^{-(k_1 I + t_1)\mathbf{A_c}} \\
\vdots \\
\mathbf{C}e^{-(k_{m-1} I + t_{m-1})\mathbf{A_c}} \\
\mathbf{C}e^{-(k_{m} I + t)\mathbf{A_c}} \\
\end{bmatrix}
\right) \right| < \epsilon^2 \prod_{1 \leq i \leq m} e^{-k_i I \cdot \lambda_i} \right\} \rightarrow 0 \mbox{ as } \epsilon \downarrow 0 \nonumber
\end{align}
(ii) For all $\epsilon>0$, $g_{\epsilon}(k) \leq a( 1 + \log (k+1))$.
\label{lem:conti:single}
\end{lemma}
\begin{proof}
Let $\epsilon' = 2 \epsilon^2 \prod_{1 \leq i \leq m}e^{\lambda_i T}$. Define $a'_{i,j}$, $C'_{i,j}$ as the $(i,j)$ element and cofactor of $\begin{bmatrix} \mathbf{C}e^{- \kappa_1 \mathbf{A_c}} \\ \vdots \\ \mathbf{C}e^{- \kappa_m \mathbf{A_c}}  \end{bmatrix}$.

Then, by Lemma~\ref{lem:det:lower}, we can find a function $g'_{\epsilon'}(k)$ such that for all $0 \leq \kappa_1 < \kappa_2 < \cdots < \kappa_m$ satisfying:\\
(i') $\kappa_m - \kappa_{m-1} \geq g'_{\epsilon'}(\kappa_{m-1})$\\
(ii') $g'_{\epsilon'}(\kappa) \lesssim 1 + \log(\kappa+1)$\\
(iii') $| \sum_{m-m_{\mu}+1 \leq i \leq m} a'_{m,i} C'_{m,i} | \geq \epsilon' \prod_{1 \leq i \leq m} e^{-\kappa_i  \lambda_i}$\\
the following inequality holds:
\begin{align}
\left| \det\left( \begin{bmatrix} \mathbf{C}e^{- \kappa_1 \mathbf{A_c}} \\ \vdots \\ \mathbf{C}e^{- \kappa_m \mathbf{A_c}} \end{bmatrix} \right) \right| \geq \frac{1}{2} \epsilon' \prod_{1 \leq i \leq m}e^{- \kappa_i \lambda_i}. \label{eqn:equalitycomment}
\end{align}

Let's use $t_m$ and $t$ interchangeably which are in $[0,T]$ with probability one. Ideally, we want to plug $k_i I + t_i$ into $\kappa_i$. However, even though the sequence $k_1, \cdots, k_m$ is sorted, the sequence $k_1 I + t_1, \cdots,  k_m I + t_m$ may not be sorted. Therefore, we define $k_{(1)}I +t_{(1)}, \cdots, k_{(m)}I + t_{(m)}$ as the sorted sequence of $k_1 I + t_1, \cdots,  k_m I + t_m$. Then, we can see this sorted sequence has the following property.
\begin{claim}
Consider two sequences, $\alpha_1, \alpha_2, \cdots, \alpha_n$ and $\beta_1, \beta_2, \cdots, \beta_n$ where $\alpha_1 \leq \alpha_2 \leq \cdots \leq \alpha_n$ and $\beta_i \in [0, T]$ $(T > 0)$.  Let $\alpha_{(1)}+\beta_{(1)}, \alpha_{(2)}+\beta_{(2)}, \cdots, \alpha_{(n)}+\beta_{(n)}$ be the ascending ordered set of $\alpha_{1}+\beta_{1}, \alpha_{2}+\beta_{2}, \cdots, \alpha_{n}+\beta_{n}$. In other words,

Then, for all $i \in \{1, \cdots, n \}$, we have
\begin{align}
0 \leq \alpha_{(i)}+ \beta_{(i)} - \alpha_i \leq T.
\end{align}
\label{claim:sort}
\end{claim}
\begin{proof}
We will prove this by contradiction. Let's say there exists $i$ such that
\begin{align}
\alpha_{(i)} + \beta_{(i)} - \alpha_i < 0.
\end{align}
Then, we have
\begin{align}
\alpha_{(i)} + \beta_{(i)} < \alpha_i \leq \alpha_{i+1} \leq \cdots \leq \alpha_n.
\end{align}
Since $\beta_1, \cdots, \beta_n \geq 0$, we can conclude
$\alpha_{(i)} + \beta_{(i)} < \alpha_i + \beta_i$, $\cdots$, $\alpha_{(i)} + \beta_{(i)} < \alpha_n + \beta_n$. Thus, in the sequence $\alpha_1+\beta_1, \cdots, \alpha_n + \beta_n$, there exists $n-i+1$ elements which are larger than $\alpha_{(i)}+\beta{(i)}$. This contradicts to the fact that $\alpha_{(i)}+\beta_{(i)}$ is $i$th largest element among $\alpha_1+\beta_1, \cdots, \alpha_n + \beta_n$.

Likewise, let's say there exists $i$ such that
\begin{align}
\alpha_{(i)}+\beta_{(i)}-\alpha_i > T.
\end{align}
Then, we have
\begin{align}
\alpha_{(i)}+\beta_{(i)} > \alpha_i + T \leq \alpha_{i-1} + T \leq \alpha_{1} + T.
\end{align}
Since $\beta_1, \cdots, \beta_n \leq T$, we can conclude
$\alpha_{(i)}+\beta_{(i)} > \alpha_i + \beta_i$, $\cdots$, $\alpha_{(i)}+\beta_{(i)} > \alpha_1 + \beta_1$. Thus, in the sequence $\alpha_1+\beta_1, \cdots, \alpha_n + \beta_n$, there exists $i$ elements which are smaller than $\alpha_{(i)}+\beta{(i)}$. This contradicts to the fact that $\alpha_{(i)}+\beta_{(i)}$ is $i$th smallest element among $\alpha_1+\beta_1, \cdots, \alpha_n + \beta_n$.
\end{proof}

Therefore, by the claim, we have
\begin{align}
\prod_{1 \leq i \leq m} e^{-\lambda_i T} \prod_{1 \leq i \leq m} e^{-k_i I \cdot \lambda_i} \leq \prod_{1 \leq i \leq m} e^{-(k_{(i)}I + t_{(i)}) \lambda_i}
\leq \prod_{1 \leq i \leq m} e^{-k_i I \cdot \lambda_i}. \label{eqn:sortedineq}
\end{align}


Finally, we can plug $k_{(i)}I +t_{(i)}$ into $\kappa_i$ to conclude the following statement. For all $0 \leq k_1 < \cdots < k_m$, $0 \leq t_i \leq T$, $0 \leq t \leq T$ such that\footnote{Here, we select $g''_{\epsilon'}(k)$ large enough so that when $k_m - k_{m-1} \geq g''_{\epsilon'}(k_{m-1})$, we always have $k_mI + t \geq k_{m-1}I + t_{m-1}$, i.e. $k_mI+t$ becomes the largest.}\\
(i'') $k_m - k_{m-1} \geq g''_{\epsilon'}(k_{m-1})$ \\
(ii'') $g''_{\epsilon'}(k) \lesssim 1 + \log(k+1)$ \\
(iii'') $\left| \sum_{m-m_{\mu}+1 \leq i \leq m} a_{m,i}C_{m,i} \right| \geq
\epsilon' \prod_{1 \leq i \leq m} e^{- k_i I \cdot \lambda_i} \overset{(A)}{\geq} \epsilon' \prod_{1 \leq i \leq m} e^{- (k_{(i)} I + t_{(i)} ) \lambda_i}$\\
the following inequality holds:
\begin{align}
\left| \det \left(
\begin{bmatrix}
\mathbf{C}e^{-(k_1 I + t_1)\mathbf{A_c}} \\
\vdots \\
\mathbf{C}e^{-(k_{m-1} I + t_{m-1})\mathbf{A_c}} \\
\mathbf{C}e^{-(k_m I + t)\mathbf{A_c}}
\end{bmatrix}
\right)  \right| & \geq \frac{1}{2}\epsilon' \prod_{1 \leq i \leq m} e^{- (k_{(i)} I + t_{(i)} ) \lambda_i} \overset{(B)}{\geq} \frac{1}{2} \epsilon' \prod_{1 \leq i \leq m}e^{-\lambda_i T} \prod_{1 \leq i \leq m} e^{-k_i I \cdot \lambda_i} \\
& \overset{(C)}{=} \epsilon^2 \prod_{1 \leq i \leq m} e^{-k_i I \cdot \lambda_i} .
\end{align}
Here, (A) and (B) always hold by \eqref{eqn:sortedineq}. (C) follows from the definition of $\epsilon'$.
%

Let $g_{\epsilon}(k)$ be $g''_{e'}(k)$. Then, we can easily check such $g_{\epsilon}(k)$ satisfies (ii) of the lemma. Let's show that such $g_{\epsilon}(k)$ also satisfies (i) of the lemma.
\begin{align}
&\sup_{k_m \in \mathbb{Z}, k_m - k_{m-1} \geq g_{\epsilon}(k_{m-1})} \mathbb{P} \left\{ \left| \det \left(
\begin{bmatrix}
\mathbf{C} e^{-(k_1 I + t_1)\mathbf{A_c}} \\
\vdots \\
\mathbf{C} e^{-(k_{m-1} I + t_{m-1})\mathbf{A_c}} \\
\mathbf{C} e^{-(k_m I + t)\mathbf{A_c}}
\end{bmatrix}
\right) \right|  < \epsilon^2 \prod_{1 \leq i \leq m} e^{-k_i I \cdot \lambda_i}\right\} \nonumber \\
& \leq \sup_{k_m \in \mathbb{Z}, k_m - k_{m-1} \geq g_{\epsilon}(k_{m-1})} \mathbb{P} \left\{
\left|
\sum_{m-m_{\mu}+1 \leq i \leq m} C_{m,i} a_{m,i}
\right| < 2 \epsilon^2 \prod_{1 \leq i \leq m} e^{\lambda_i T}\cdot \prod_{1 \leq i \leq m} e^{-k_i I \cdot \lambda_i}
\right\} \nonumber \\
& = \sup_{k_m \in \mathbb{Z}, k_m - k_{m-1} \geq g_{\epsilon}(k_{m-1})} \mathbb{P} \left\{
\left|
\sum_{m-m_{\mu}+1 \leq i \leq m}
\frac{C_{m,i}}{\epsilon \prod_{1 \leq i \leq m-1} e^{-k_i I \cdot \lambda_i} } \frac{a_{m,i}}{e^{-(k_m  I+t ) \lambda_m}}
\right| < 2 \epsilon \cdot e^{\lambda_m t} \prod_{1 \leq i \leq m} e^{\lambda_i T}
\right\} \nonumber \\
&\leq \sup_{|b_m|\geq 1} \mathbb{P}\left\{ \left| \sum_{m-m_{\mu}+1 \leq i \leq m} b_i \frac{a_{m,i}}{e^{-( k_m I +t )\lambda_m}} \right| < 2 \epsilon \cdot  e^{\lambda_m T} \prod_{1 \leq i \leq m} e^{\lambda_i T}  \right\}. \label{eqn:lem:dettail:1}
\end{align}
where the last inequality comes from the assumption of (ii), $|C_{m,m}| > \epsilon \prod_{1 \leq i \leq m-1} e^{-k_i I \cdot \lambda_i}$, and $t \in [0,T]$ with probability one.

Now, it is enough to prove that \eqref{eqn:lem:dettail:1} converges to $0$ as $\epsilon \downarrow 0$. To this end, let's study $a_{m,i}$ which are the elements of the observability gramian. Let the $\mathbf{C_{\mu,\nu_{\mu}}}$ defined in \eqref{eqn:conti:c} be $\begin{bmatrix} c'_{1} & \cdots & c'_{m_{\mu,\nu_{\mu}}} \end{bmatrix}$. Then, we have
\begin{align}
&e^{-(k_m I + t)\mathbf{A_{\mu,\nu_{\mu}}}}\nonumber \\
&= \begin{bmatrix}
e^{-(k_m I + t)(\lambda_{\mu,\nu_\mu} + j \omega_{\mu,\nu_\mu})} & -(k_m I + t) e^{-(k_m I + t)(\lambda_{\mu,\nu_\mu} + j \omega_{\mu,\nu_\mu})} & \cdots &
\frac{(-1)^{m_{\mu,\nu_\mu}-1} (k_m I + t)^{m_{\mu,\nu_\mu}-1}}{(m_{\mu,\nu_\mu}-1)!} e^{-(k_m I + t)(\lambda_{\mu,\nu_\mu} + j \omega_{\mu,\nu_\mu})}\\
0 & e^{-(k_m I + t)(\lambda_{\mu,\nu_\mu} + j \omega_{\mu,\nu_\mu})} & \cdots &
\frac{(-1)^{m_{\mu,\nu_\mu}-2} (k_m I + t)^{m_{\mu,\nu_\mu}-2}}{(m_{\mu,\nu_\mu}-2)!} e^{-(k_m I + t)(\lambda_{\mu,\nu_\mu} + j \omega_{\mu,\nu_\mu})}\\
\vdots & \vdots & \ddots & \vdots \\
0 & 0 & \cdots & e^{-(k_m I + t)(\lambda_{\mu,\nu_\mu} + j \omega_{\mu,\nu_\mu})}
\end{bmatrix}. \nonumber
\end{align}
Thus, we can see that
\begin{align}
a_{m,m}&=\sum_{1 \leq i \leq m_{\mu,\nu_{\mu}}} c'_{i} \frac{(-1)^{m_{\mu,\nu_{\mu}}-i}(k_m I +t)^{m_{\mu,\nu_{\mu}}-i} }{(m_{\mu,\nu_{\mu}}-i)!} e^{-(k_m I + t)(\lambda_{m}+j \omega_{\mu,\nu_\mu})}.  \nonumber
\end{align}
Therefore,
\begin{align}
\frac{a_{m,m}}{e^{-(k_m I + t)\lambda_m}} &=\sum_{1 \leq i \leq m_{\mu,\nu_{\mu}}} c'_{i} \frac{(-1)^{m_{\mu,\nu_{\mu}}-i}(k_m I +t)^{m_{\mu,\nu_{\mu}}-i} }{(m_{\mu,\nu_{\mu}}-i)!} e^{-(k_m I + t)(j \omega_{\mu,\nu_\mu})}. \nonumber
\end{align}
Moreover, when $a_{m,i}$ is considered as a function of $t$, the $t^{m_{\mu,\nu_\mu}-1} e^{-j \omega_{\mu,\nu_\mu} t }$ term only shows up in $\frac{a_{m,m}}{e^{-(k_m I + t)\lambda_m}}$ among $\frac{a_{m,m-m_{\mu}+1}}{e^{-(k_m I + t)\lambda_m}}, \cdots, \frac{a_{m,m}}{e^{-(k_m I + t)\lambda_m}}$,
and the coefficient is $c_1' \frac{(-1)^{m_{\mu,\nu_\mu}-1}}{(m_{\mu,\nu_\mu}-1)!}e^{-j \omega_{\mu,\nu_\mu} k_m I }$. Since we put $|b_m| \geq 1$ in \eqref{eqn:lem:dettail:1}, by defining $c':=\frac{|c_1'|}{(m_{\mu,\nu_\mu}-1)!}$ we can see that the magnitude of the corresponding coefficient is greater or equal to $c'$. Furthermore, the remaining terms $\frac{a_{m,m-m_{\mu}+1}}{e^{-(k_m I + t)\lambda_m}}, \cdots, \frac{a_{m,m-1}}{e^{-(k_m I + t)\lambda_m}}$ only have $e^{-j \omega_{\mu,1}t}, \cdots, t^{m_{\mu,1}-1} e^{-j \omega_{\mu,1}t}$, $e^{-j \omega_{\mu,2}t}, \cdots, t^{m_{\mu,2}-1} e^{-j \omega_{\mu,2}t}$, $\cdots$, $e^{-j \omega_{\mu,\nu_\mu}t}, \cdots, t^{m_{\mu,\nu_\mu}-2} e^{-j \omega_{\mu,\nu_\mu}t}$ when they are considered as functions in $t$. Thus, using the assumption that $m_{\nu,1} \leq \cdots \leq m_{\nu,\mu_\nu}$,\eqref{eqn:lem:dettail:1} can be upper bounded as follows:
\begin{align}
\eqref{eqn:lem:dettail:1} \leq \sup_{|a'_{m_{{\mu},\nu_{\mu}},\nu_\mu}|\geq c' } \mathbb{P} \left\{
\left| \sum^{m_{\mu,\nu_\mu}}_{i=1}
t^{i-1}\left(
\sum^{\nu_{\mu}}_{j=1}
a'_{i,j} e^{-j \omega_{\mu,j}t}
\right)
\right|
\leq
2 \epsilon e^{\lambda_m T} \cdot \prod_{1 \leq i \leq m} e^{\lambda_i T}
\right\}. \label{eqn:lem:dettail:2}
\end{align}
By Lemma~\ref{lem:singleun} (by setting $\gamma$ as $c'$, $(m,n)$ as $(m_{\mu,\nu_\mu}, \nu_\mu)$, $p$ as $m_{\mu,\nu_\mu}$, $\nu_0, \cdots, \nu_p$ as $\nu_\mu$, $\omega_{0,j}, \cdots, \omega_{p,j}$ as $-\omega_{\mu,j}$, and $\epsilon$ as $2 \epsilon \prod_{1 \leq i \leq m} e^{\lambda_i T} \cdot e^{\lambda_m T}$), we get
\begin{align}
\sup_{|a'_{m_{{\mu},\nu_{\mu}},\nu_\mu}|\geq c' } \mathbb{P} \left\{
\left| \sum^{m_{\mu,\nu_\mu}}_{i=1}
t^{i-1}\left(
\sum^{\nu_{\mu}}_{j=1}
a'_{i,j} e^{-j \omega_{\mu,j}t}
\right)
\right|
\leq
2 \epsilon e^{\lambda_m T} \cdot \prod_{1 \leq i \leq m} e^{\lambda_i T}
\right\} \rightarrow 0 \mbox{ as } \epsilon \downarrow 0. \label{eqn:lem:dettail:3}
\end{align}
Therefore, by \eqref{eqn:lem:dettail:1}, \eqref{eqn:lem:dettail:2}, \eqref{eqn:lem:dettail:3} we can say that
\begin{align}
&\sup_{k_m \in \mathbb{Z}, k_m - k_{m-1} \geq g_{\epsilon}(k_{m-1})} \mathbb{P} \left\{ \left| \det \left(
\begin{bmatrix}
\mathbf{C} e^{-(k_1 I + t_1)\mathbf{A_c}} \\
\vdots \\
\mathbf{C} e^{-(k_{m-1} I + t_{m-1})\mathbf{A_c}} \\
\mathbf{C} e^{-(k_m I + t)\mathbf{A_c}}
\end{bmatrix}
\right) \right|  < \epsilon^2 \prod_{1 \leq i \leq m} e^{-k_i I \cdot \lambda_i}\right\} \rightarrow 0 \mbox{ as } \epsilon \downarrow 0 \nonumber
\end{align}
which finishes the proof.
\end{proof}

Based on the previous lemma, we will integrate the properties of p.m.f. tails shown in Section~\ref{sec:app:1} with the properties of the observability Gramian discussed in Section~\ref{sec:app:3}, and prove Lemma~\ref{lem:conti:mo} for the case of a row vector $\mathbf{C}$.

\begin{lemma}
Let $\mathbf{A_c}$ and $\mathbf{C}$ be given as \eqref{eqn:conti:a} and \eqref{eqn:conti:c}. Let $\beta[n]~(n \in \mathbb{Z}^+)$ be a Bernoulli random process with probability $1-p_e$ and $t_n$ be i.i.d.~random variables which are uniformly distributed on $[0,T]~(T>0)$. Then, we can find a polynomial $p(k)$ and a family of stopping times $\{ S(\epsilon,k): k \in \mathbb{Z}^+, \epsilon > 0 \}$ such that for all $\epsilon > 0$, $k \in \mathbb{Z}^+$ there exist
 $k \leq k_1 < k_2 < \cdots < k_m \leq S(\epsilon,k)$ and $\mathbf{M}$ satisfying the following conditions:\\
(i) $\beta[k_i]=1$ for $1 \leq i \leq m$\\
(ii) $\mathbf{M} \begin{bmatrix}
\mathbf{C}e^{-(k_1 I + t_{k_1})\mathbf{A_c}} \\
\mathbf{C}e^{-(k_2 I + t_{k_2})\mathbf{A_c}} \\
\vdots \\
\mathbf{C}e^{-(k_m I + t_{k_m})\mathbf{A_c}} \\
\end{bmatrix}=\mathbf{I}$\\
(iii) $|\mathbf{M}|_{max} \leq \frac{p(S(\epsilon,k))}{\epsilon} e^{\lambda_1 S(\epsilon,k) I}$\\
(iv) $\lim_{\epsilon \downarrow 0} \exp \limsup_{s \rightarrow \infty} \sup_{k \in \mathbb{Z}^+} \frac{1}{s} \log \mathbb{P}\left\{ S(\epsilon,k)-k=s\right\} \leq p_e$.
\label{lem:conti:singlec}
\end{lemma}

\begin{proof}
By Lemma~\ref{lem:conti:inverse2}, instead of conditions (ii) and (iii), it is enough to prove that
\begin{align}
\left| \det \left(
\begin{bmatrix}
\mathbf{C}e^{-(k_1 I + t_{k_1})\mathbf{A_c}} \\
\mathbf{C}e^{-(k_2 I + t_{k_2})\mathbf{A_c}} \\
\vdots \\
\mathbf{C}e^{-(k_m I + t_{k_m})\mathbf{A_c}} \\
\end{bmatrix}\right)\right| \geq \epsilon \prod_{1 \leq i \leq m} e^{-(k_i I + t_{k_i}) \lambda_i}. \nonumber
\end{align}
Furthermore, since $t_i \geq 0$ it is sufficient to prove that
\begin{align}
\left| \det \left(
\begin{bmatrix}
\mathbf{C}e^{-(k_1 I + t_{k_1})\mathbf{A_c}} \\
\mathbf{C}e^{-(k_2 I + t_{k_2})\mathbf{A_c}} \\
\vdots \\
\mathbf{C}e^{-(k_m I + t_{k_m})\mathbf{A_c}} \\
\end{bmatrix}\right)\right| \geq \epsilon \prod_{1 \leq i \leq m} e^{-k_i I \cdot \lambda_i}. \nonumber
\end{align}
Therefore, it is enough to prove the following claim:

We can find a family of stopping times $\{ S(\epsilon,k) : k \in \mathbb{Z}^+, \epsilon > 0 \}$ such that for all $\epsilon>0$ and $k \in \mathbb{Z}^+$
there exist $k \leq k_1 < k_2 < \cdots < k_m \leq S(\epsilon,k)$ satisfying the following condition:\\
(a) $\beta[k_i]=1$ for $1 \leq i \leq m$ \\
(b) $\left| \det \left(
\begin{bmatrix}
\mathbf{C}e^{-(k_1 I + t_{k_1})\mathbf{A_c}} \\
\mathbf{C}e^{-(k_2 I + t_{k_2})\mathbf{A_c}} \\
\vdots \\
\mathbf{C}e^{-(k_m I + t_{k_m})\mathbf{A_c}} \\
\end{bmatrix}\right)\right| \geq \epsilon  \prod_{1 \leq i \leq m} e^{-k_i I \cdot \lambda_i}$ \\
(c) $\lim_{\epsilon \downarrow 0} \exp \limsup_{s \rightarrow \infty} \sup_{k \in \mathbb{Z}^+} \frac{1}{s} \log \mathbb{P}\left\{ S(\epsilon,k)-k=s\right\} \leq p_e$

We will prove the claim by induction on $m$, the size of the $\mathbf{A_c}$ matrix.

(i) When $m=1$,

Since we only have to care about small enough $\epsilon$, let $\epsilon \leq |c_1| e^{-2 T \lambda_1}$. Denote $S(\epsilon,k) := \inf \{ n \geq k : \beta[n]=1 \}$ and $k_1=S(\epsilon,k)$. Then, $\beta[k_1]=1$ and $\left| \det\left( \begin{bmatrix}
c_1 e^{-(k_1 I + t_{k_1})(\lambda_1 + j \omega_1)}
\end{bmatrix} \right) \right|\geq |c_1| e^{-T \lambda_1} e^{-k_1 I \cdot \lambda_1} \geq \epsilon e^{-k_1 I \cdot \lambda_1}$.\\
Moreover, since $S(\epsilon,k)-k$ is a geometric random variable with probability $1-p_e$,
\begin{align}
\exp \limsup_{s \rightarrow \infty} \sup_{k \in \mathbb{Z}^+} \log \mathbb{P}\left\{ S(\epsilon,k)-k=s \right\} = p_e.  \nonumber
\end{align}
Therefore, $S(\epsilon,k)$ satisfies all the conditions of the lemma.

(ii) Now, we assume that the lemma is true for $m-1$ and prove the lemma still holds for $m$. We will see that the induction hypothesis corresponds to the cofactor condition of Lemma~\ref{lem:conti:single}, which tells us that the determinant of the observability Gramian is large enough with high probability.

Let $\mathbf{A_c'}$ be the $(m-1) \times (m-1)$ matrix obtained by removing $m$th row and column of $\mathbf{A_c}$. Likewise, $\mathbf{C'}$ is a $1 \times (m-1)$ vector obtained by removing $m$th element of $\mathbf{C}$. Then, since $\mathbf{A_c}$ is given in a Jordan form, we can easily check that once we remove the last element from the row vector $\mathbf{C}e^{-(k_i I + t_{k_i})\mathbf{A_c}}$, we get $\mathbf{C'}e^{-(k_i I + t_{k_i})\mathbf{A_c'}}$. Therefore, we can see that
\begin{align}
\det\left( \begin{bmatrix}
\mathbf{C'} e^{-(k_1 I + t_{k_1})\mathbf{A_c'}} \\
\vdots \\
\mathbf{C'} e^{-(k_{m-1} I + t_{k_{m-1}})\mathbf{A_c'}}
\end{bmatrix} \right)
= cof_{m,m}\left(
\begin{bmatrix}
\mathbf{C} e^{-(k_1 I + t_{k_1})\mathbf{A_c}} \\
\vdots \\
\mathbf{C} e^{-(k_m I + t_{k_m})\mathbf{A_c}} \\
\end{bmatrix}
\right) \label{eqn:cofdet1}
\end{align}
where $cof_{i,j}(\mathbf{A})$ implies the cofactor matrix of $\mathbf{A}$ with respect to $(i,j)$ element.

By the induction hypothesis, there exists a stopping time $S'(\epsilon,k)$ such that we can find $k \leq k_1 < k_2 < \cdots < k_{m-1} \leq S'(\epsilon,k)$ satisfying:\\
(a) $\beta[k_i]=1$ for $1 \leq i \leq m-1$ \\
(b) $\left| \det \left(
\begin{bmatrix}
\mathbf{C'}e^{-(k_1 I + t_{k_1})\mathbf{A_c'}} \\
\vdots \\
\mathbf{C'}e^{-(k_{m-1} I + t_{k_{m-1}})\mathbf{A_c'}}
\end{bmatrix}
\right) \right| \geq \epsilon  \prod_{1 \leq i \leq m-1} e^{-k_i I \cdot \lambda_i}$\\
(c) $\lim_{\epsilon \downarrow 0} \exp \limsup_{s \rightarrow \infty} \sup_{k \in \mathbb{Z}^+} \frac{1}{s} \log \mathbb{P} \left\{ S'(\epsilon,k)-k=s \right\} \leq p_e$.

Let $\mathcal{F}_{i}$ be a $\sigma$-field generated by $\beta[0],\cdots,\beta[i]$, and $t_0,\cdots, t_i$. Let $g_{\epsilon}: \mathbb{R}^+ \rightarrow \mathbb{R}^+$ be the function of Lemma~\ref{lem:conti:single}. Denote
\begin{align}
p'(\epsilon):=\esssup \sup_{k_m \in \mathbb{Z}, k_m-S'(\epsilon, k) \geq g_{\epsilon}(S'(\epsilon, k))} \mathbb{P}_t \left\{
\left|
\det \left(
\begin{bmatrix}
\mathbf{C}e^{-(k_1 I + t_{k_1})\mathbf{A_c}}\\
\vdots\\
\mathbf{C}e^{-(k_{m-1} I + t_{k_{m-1}})\mathbf{A_c}}\\
\mathbf{C}e^{-(k_{m} I + t)\mathbf{A_c}}
\end{bmatrix}
\right)
\right| < \epsilon^2 \prod_{1 \leq i \leq m} e^{-k_i I \cdot \lambda_i}
| \mathcal{F}_{S'(\epsilon,k)}
\right\}. \label{eqn:stoppingtimedef1}
\end{align}

Here, given $\mathcal{F}_{S'(\epsilon,k)}$, $k_1, \cdots, k_{m-1}$, $t_{k_1}, \cdots, t_{k_{m-1}}$, $S'(\epsilon,k)$ are all fixed, we took the supremum over $k_m$ such that $k_m - S'(\epsilon, k) \geq g_{\epsilon}(S'(\epsilon, k))$, and $t$ is a uniform random variable on $[0,T]$ which we computed the probability over.

Since $k_m \geq S'(\epsilon, k)+ g_{\epsilon}(S'(\epsilon, k)) \geq k_{m-1}+g_{\epsilon}(k_{m-1})$, and by \eqref{eqn:cofdet1}, (b) implies $cof_{m,m}\left(
\begin{bmatrix}
\mathbf{C} e^{-(k_1 I + t_{k_1})\mathbf{A_c}} \\
\vdots \\
\mathbf{C} e^{-(k_{m} I + t_{k_m})\mathbf{A_c}} \\
\end{bmatrix}
\right) \geq \epsilon \prod_{1 \leq i \leq m-1} e^{- k_i I \cdot \lambda_i}$, by Lemma~\ref{lem:conti:single} we have $\lim_{\epsilon \downarrow 0}p'(\epsilon) = 0$.

Denote $S''(\epsilon,k) := \lceil S'(\epsilon,k)+g_{\epsilon}(S'(\epsilon,k)) \rceil$. From (ii) of Lemma~\ref{lem:conti:single}  we know $g_{\epsilon}(k) \lesssim 1 + \log (k+1)$ for all $\epsilon>0$.
Therefore, by (c) and Lemma~\ref{lem:conti:tailpoly} we have
\begin{align}
\lim_{\epsilon \downarrow 0} \exp \limsup_{s \rightarrow \infty} \sup_{k \in \mathbb{Z}^+} \frac{1}{s} \log \mathbb{P}\{ S''(\epsilon,k)-k=s \} \leq p_e. \label{eqn:lem:single:5}
\end{align}
Denote a stopping time
\begin{align}
&S'''(\epsilon,k) \nonumber \\
&:=\inf \left\{n \geq S''(\epsilon): \beta[n]=1 \mbox{ and }
\left|
\det \left(
\begin{bmatrix}
\mathbf{C}e^{-(k_1 I + t_{k_1})\mathbf{A_c}} \\
\vdots \\
\mathbf{C}e^{-(k_{m-1} I + t_{k_{m-1}})\mathbf{A_c}} \\
\mathbf{C}e^{-(n I + t_n)\mathbf{A_c}} \\
\end{bmatrix}
\right)
\right|
\geq \epsilon^2  e^{-nI \cdot \lambda_m } \prod_{1 \leq i \leq m-1} e^{-k_i I \cdot \lambda_i}
\right\}. \label{eqn:stoppingtimedef2}
\end{align}

Since $\beta[n]$ and $t_n$ are independent processes, for $S'''(\epsilon,k)=n$ to hold, $\beta[n]=1$ and the determinant of \eqref{eqn:stoppingtimedef2} has to be large enough. By \eqref{eqn:stoppingtimedef1}, we already know the probability for the determinant not being large enough is upper bounded by $p'(\epsilon)$. Therefore, given that $S'''(\epsilon,k) \geq n$, the probability that $S'''(\epsilon,k) \neq n$ is upper bounded by $(p_e + (1-p_e)p'(\epsilon))$ --- (erasure) or (not erased but small determinant). Thus, for all $s \in \mathbb{Z}^+$, we have
\begin{align}
\esssup \mathbb{P}\{ S'''(\epsilon,k)- S''(\epsilon,k) \geq s | \mathcal{F}_{S''(\epsilon,k)} \} \leq \left(p_e + \left(1-p_e\right)p'(\epsilon)\right)^{s}. \nonumber
\end{align}

Since we know $\lim_{\epsilon \downarrow 0}p'(\epsilon) = 0$, we have
\begin{align}
\lim_{\epsilon \downarrow 0} \exp \limsup_{s \rightarrow \infty} \esssup \frac{1}{s} \log \mathbb{P}\{ S'''(\epsilon,k)-  S''(\epsilon,k)  = s | \mathcal{F}_{S''(\epsilon,k)}  \} \leq p_e. \label{eqn:lem:single:6}
\end{align}

By applying Lemma~\ref{lem:app:geo} to \eqref{eqn:lem:single:5} and \eqref{eqn:lem:single:6}, we can conclude that
\begin{align}
\lim_{\epsilon \downarrow 0} \exp \limsup_{s \rightarrow \infty} \sup_{k \in \mathbb{Z}^+} \frac{1}{s} \log \mathbb{P}\{S'''(\epsilon,k)-k=s \} \leq p_e. \nonumber
\end{align}
Therefore, if we denote $S(\epsilon,k):=S'''(\epsilon^{\frac{1}{2}},k)$, $S(\epsilon,k)$ satisfies all the conditions of the claim.
\end{proof}

Before we prove Lemma~\ref{lem:conti:mo}, we will first prove the following lemma which allows to merge two Jordan blocks associated with the same eigenvalue into one Jordan block.

\begin{lemma}
Let $\mathbf{A}$ be a Jordan block matrix with an eigenvalue $\lambda \in \mathbb{C}$ and a size $m \in \mathbb{N}$, i.e.
$\mathbf{A}=\begin{bmatrix}
\lambda & 1 & \cdots & 0 \\
0 & \lambda & \cdots & 0 \\
\vdots & \vdots & \ddots & \vdots \\
0 & 0 & \cdots & \lambda \\
\end{bmatrix}$. $\mathbf{C}$ and $\mathbf{C'}$ are $1 \times m$ matrices such that
\begin{align}
&\mathbf{C}=\begin{bmatrix} c_1 & c_2 & \cdots & c_m\end{bmatrix} \nonumber\\
&\mathbf{C'}=\begin{bmatrix} c'_1 & c'_2 & \cdots & c'_m\end{bmatrix}
\end{align}
where $c_i, c'_i \in \mathbb{C}$ and $c_1 \neq 0$.\\
For all $k \in \mathbb{R}$ and $m \times 1$ matrices $\mathbf{X}=\begin{bmatrix} x_1 \\ x_2 \\ \vdots \\ x_m \end{bmatrix}$ and $\mathbf{X'}=\begin{bmatrix} x'_1 \\ x'_2 \\ \vdots \\ x'_m \end{bmatrix}$, there exists $\mathbf{T}$ such that\\
\begin{align}
&(i) \mathbf{T} \mbox{ is an upper triangular matrix.} \nonumber \\
&(ii) \mathbf{C}e^{k\mathbf{A}}\mathbf{X}+\mathbf{C'}e^{k\mathbf{A}}\mathbf{X'}=\mathbf{C}e^{k\mathbf{A}}\left(\mathbf{X}+\mathbf{T}\mathbf{X'}\right) \nonumber
\end{align}
Moreover, the diagonal elements of $\mathbf{T}$ are $\frac{c_1'}{c_1}$.
\label{lem:conti:jordan}
\end{lemma}
\begin{proof}
The proof is an induction on $m$, the size of the $\mathbf{A}$ matrix. The lemma is trivial when $m=1$. Thus, we can assume the lemma is true for $m$ as an induction hypothesis, and consider $m+1$ as the dimension of $\mathbf{A}$.
\begin{align}
&\mathbf{C} e^{k\mathbf{A}} \mathbf{X} + \mathbf{C'} e^{k\mathbf{A}} \mathbf{X'} \nonumber \\
&=\mathbf{C} \begin{bmatrix}
e^{k \lambda} & \frac{k}{1!} e^{k \lambda} & \cdots & \frac{k^m}{m!} e^{k \lambda} \\
0 & e^{k \lambda} & \cdots & \frac{k^{m-1}}{(m-1)!} e^{k \lambda} \\
\vdots & \vdots & \ddots & \vdots \\
0 & 0 & \cdots & e^{k \lambda}\\
\end{bmatrix}\mathbf{X} +
\mathbf{C'} \begin{bmatrix}
e^{k \lambda} & \frac{k}{1!} e^{k \lambda} & \cdots & \frac{k^m}{m!} e^{k \lambda} \\
0 & e^{k \lambda} & \cdots & \frac{k^{m-1}}{(m-1)!} e^{k \lambda} \\
\vdots & \vdots & \ddots & \vdots \\
0 & 0 & \cdots & e^{k \lambda}\\
\end{bmatrix}\mathbf{X'} \nonumber \\
&=\mathbf{C} \begin{bmatrix}
e^{k \lambda} & \frac{k}{1!} e^{k \lambda} & \cdots & \frac{k^m}{m!} e^{k \lambda} \\
0 & e^{k \lambda} & \cdots & \frac{k^{m-1}}{(m-1)!} e^{k \lambda} \\
\vdots & \vdots & \ddots & \vdots \\
0 & 0 & \cdots & e^{k \lambda}\\
\end{bmatrix}\mathbf{X} \nonumber \\
&+
\left(
\frac{c_1'}{c_1}\mathbf{C}+\begin{bmatrix}
0 & c_2'-\frac{c_1'}{c_1}c_2 & \cdots & c_{m+1}'-\frac{c_1'}{c_1}c_{m+1}
\end{bmatrix}
\right)
\begin{bmatrix}
e^{k \lambda} & \frac{k}{1!} e^{k \lambda} & \cdots & \frac{k^m}{m!} e^{k \lambda} \\
0 & e^{k \lambda} & \cdots & \frac{k^{m-1}}{(m-1)!} e^{k \lambda} \\
\vdots & \vdots & \ddots & \vdots \\
0 & 0 & \cdots & e^{k \lambda}\\
\end{bmatrix}\mathbf{X'} \nonumber \\
&=\mathbf{C} \begin{bmatrix}
e^{k \lambda} & \frac{k}{1!} e^{k \lambda} & \cdots & \frac{k^m}{m!} e^{k \lambda} \\
0 & e^{k \lambda} & \cdots & \frac{k^{m-1}}{(m-1)!} e^{k \lambda} \\
\vdots & \vdots & \ddots & \vdots \\
0 & 0 & \cdots & e^{k \lambda}\\
\end{bmatrix}\left(\mathbf{X}+\frac{c_1'}{c_1}\mathbf{X'} \right) \nonumber \\
&+
\begin{bmatrix}
0 & c_2'-\frac{c_1'}{c_1}c_2 & \cdots & c_m'-\frac{c_1'}{c_1}c_m
\end{bmatrix}
\begin{bmatrix}
e^{k \lambda} & \frac{k}{1!} e^{k \lambda} & \cdots & \frac{k^m}{m!} e^{k \lambda} \\
0 & e^{k \lambda} & \cdots & \frac{k^{m-1}}{(m-1)!} e^{k \lambda} \\
\vdots & \vdots & \ddots & \vdots \\
0 & 0 & \cdots & e^{k \lambda}\\
\end{bmatrix}\mathbf{X'} \nonumber \\
&=\mathbf{C} \begin{bmatrix}
e^{k \lambda} & \frac{k}{1!} e^{k \lambda} & \cdots & \frac{k^m}{m!} e^{k \lambda} \\
0 & e^{k \lambda} & \cdots & \frac{k^{m-1}}{(m-1)!} e^{k \lambda} \\
\vdots & \vdots & \ddots & \vdots \\
0 & 0 & \cdots & e^{k \lambda}\\
\end{bmatrix}\left(\begin{bmatrix} 0 \\ 0 \\ \vdots \\ x_{m+1}+\frac{c_1'}{c_1}x'_{m+1} \end{bmatrix}+ \begin{bmatrix}x_1+\frac{c_1'}{c_1}x'_1 \\ x_2+\frac{c_1'}{c_1}x'_2 \\ \vdots \\ 0 \end{bmatrix} \right) \nonumber \\
&+\begin{bmatrix}
c_2'-\frac{c_1'}{c_1}c_2 & c_3'-\frac{c_1'}{c_1}c_3 & \cdots & c_{m+1}'-\frac{c_1'}{c_1}c_{m+1}
\end{bmatrix}
\begin{bmatrix}
e^{k \lambda} & \frac{k}{1!} e^{k \lambda} & \cdots & \frac{k^{m-1}}{(m-1)!} e^{k \lambda} \\
0 & e^{k \lambda} & \cdots & \frac{k^{m-2}}{(m-2)!} e^{k \lambda} \\
\vdots & \vdots & \ddots & \vdots \\
0 & 0 & \cdots & e^{k \lambda}\\
\end{bmatrix}
\begin{bmatrix}
x_2' \\
x_3' \\
\vdots \\
x_{m+1}' \\
\end{bmatrix}
\nonumber
\end{align}
\begin{align}
&=\mathbf{C} \begin{bmatrix}
e^{k \lambda} & \frac{k}{1!} e^{k \lambda} & \cdots & \frac{k^m}{m!} e^{k \lambda} \\
0 & e^{k \lambda} & \cdots & \frac{k^{m-1}}{(m-1)!} e^{k \lambda} \\
\vdots & \vdots & \ddots & \vdots \\
0 & 0 & \cdots & e^{k \lambda}\\
\end{bmatrix}\begin{bmatrix} 0 \\ 0 \\ \vdots \\ x_{m+1}+\frac{c_1'}{c_1}x'_{m+1} \end{bmatrix}\nonumber\\
&+\begin{bmatrix} c_1 & c_2 & \cdots & c_m \end{bmatrix}
\begin{bmatrix}
e^{k \lambda} & \frac{k}{1!} e^{k \lambda} & \cdots & \frac{k^{m-1}}{(m-1)!} e^{k \lambda} \\
0 & e^{k \lambda} & \cdots & \frac{k^{m-2}}{(m-2)!} e^{k \lambda} \\
\vdots & \vdots & \ddots & \vdots \\
0 & 0 & \cdots & e^{k \lambda}\\
\end{bmatrix}\begin{bmatrix}x_1+\frac{c_1'}{c_1}x'_1 \\ x_2+\frac{c_1'}{c_1}x'_2 \\ \vdots \\ x_{m}+\frac{c_1'}{c_1}x'_{m} \end{bmatrix}\nonumber\\
&+\begin{bmatrix}
c_2'-\frac{c_1'}{c_1}c_2 & c_3'-\frac{c_1'}{c_1}c_3 & \cdots & c_{m+1}'-\frac{c_1'}{c_1}c_{m+1}
\end{bmatrix}
\begin{bmatrix}
e^{k \lambda} & \frac{k}{1!} e^{k \lambda} & \frac{k}{2!} e^{k \lambda} & \cdots & \frac{k^{m-1}}{(m-1)!} e^{k \lambda} \\
0 & e^{k \lambda} & \frac{k}{1!} e^{k \lambda} & \cdots & \frac{k^{m-2}}{(m-2)!} e^{k \lambda} \\
\vdots & \vdots & \vdots & \ddots & \vdots \\
0 & 0 & 0 & \cdots & e^{k \lambda}\\
\end{bmatrix}
\begin{bmatrix}
x_2' \\
x_3' \\
\vdots \\
x_{m+1}' \\
\end{bmatrix}
\nonumber \\
&=
\mathbf{C} \begin{bmatrix}
e^{k \lambda} & \frac{k}{1!} e^{k \lambda} & \cdots & \frac{k^m}{m!} e^{k \lambda} \\
0 & e^{k \lambda} & \cdots & \frac{k^{m-1}}{(m-1)!} e^{k \lambda} \\
\vdots & \vdots & \ddots & \vdots \\
0 & 0 & \cdots & e^{k \lambda}\\
\end{bmatrix}\begin{bmatrix} 0 \\ 0 \\ \vdots \\ x_{m+1}+\frac{c_1'}{c_1}x'_{m+1} \end{bmatrix}\nonumber\\
&+\begin{bmatrix} c_1 & c_2 & \cdots & c_{m} \end{bmatrix}
\begin{bmatrix}
e^{k \lambda} & \frac{k}{1!} e^{k \lambda} & \cdots & \frac{k^{m-1}}{(m-1)!} e^{k \lambda} \\
0 & e^{k \lambda} & \cdots & \frac{k^{m-2}}{(m-2)!} e^{k \lambda} \\
\vdots & \vdots & \ddots & \vdots \\
0 & 0 & \cdots & e^{k \lambda}\\
\end{bmatrix}
\left(
\begin{bmatrix}x_1+\frac{c_1'}{c_1}x'_1 \\ x_2+\frac{c_1'}{c_1}x'_2 \\ \vdots \\ x_{m}+\frac{c_1'}{c_1}x'_{m} \end{bmatrix} +
\begin{bmatrix}
t'_{1,1} & t'_{1,2} & \cdots & t'_{1,m} \\
0 & t'_{2,2} & \cdots & t'_{2,m} \\
\vdots & \vdots & \ddots & \vdots \\
0 & 0 & \cdots & t'_{m,m}
\end{bmatrix}
\begin{bmatrix}
x_2' \\
x_3' \\
\vdots \\
x_{m+1}' \\
\end{bmatrix}\right)
\label{eqn:dis:matrix:1}\\
&=
\mathbf{C} \begin{bmatrix}
e^{k \lambda} & \frac{k}{1!} e^{k \lambda} & \cdots & \frac{k^m}{m!} e^{k \lambda} \\
0 & e^{k \lambda} & \cdots & \frac{k^{m-1}}{(m-1)!} e^{k \lambda} \\
\vdots & \vdots & \ddots & \vdots \\
0 & 0 & \cdots & e^{k \lambda}\\
\end{bmatrix} \left(
\begin{bmatrix} 0 \\ 0 \\ \vdots \\ x_{m+1}+\frac{c_1'}{c_1}x'_{m+1} \end{bmatrix} +
\begin{bmatrix}x_1+\frac{c_1'}{c_1}x'_1 \\ x_2+\frac{c_1'}{c_1}x'_2 \\ \vdots \\ 0 \end{bmatrix} +
\begin{bmatrix}
t'_{1,1} & t'_{1,2} & \cdots & t'_{1,m} \\
0 & t'_{2,2} & \cdots & t'_{2,m} \\
\vdots & \vdots & \ddots & \vdots \\
0 & 0 & \cdots & 0
\end{bmatrix}
\begin{bmatrix}
x_2' \\
x_3' \\
\vdots \\
x_{m+1}' \\
\end{bmatrix}
\right)\nonumber\\
&=
\mathbf{C} \begin{bmatrix}
e^{k \lambda} & \frac{k}{1!} e^{k \lambda} & \cdots & \frac{k^m}{m!} e^{k \lambda} \\
0 & e^{k \lambda} & \cdots & \frac{k^{m-1}}{(m-1)!} e^{k \lambda} \\
\vdots & \vdots & \ddots & \vdots \\
0 & 0 & \cdots & \frac{k}{1!} e^{k \lambda}\\
0 & 0 & \cdots & e^{k \lambda}\\
\end{bmatrix}
\left(
\begin{bmatrix}
x_1 \\
x_2 \\
\vdots \\
x_{m+1} \\
\end{bmatrix}
+
\begin{bmatrix}
\frac{c'_1}{c_1} & t'_{1,1} & \cdots & t'_{1,m} \\
0 & \frac{c'_1}{c_1} & \cdots & t'_{2,m} \\
\vdots & \vdots & \ddots & \vdots \\
0 & 0 & \cdots & \frac{c'_1}{c_1}
\end{bmatrix}
\begin{bmatrix}
x'_1 \\
x'_2 \\
\vdots \\
x'_{m+1} \\
\end{bmatrix}
\right)
\nonumber
\end{align}
where \eqref{eqn:dis:matrix:1} follows from the induction hypothesis. The lemma is true.
\end{proof}

Now, we are ready to prove Lemma~\ref{lem:conti:mo}.
\begin{proof}[Proof of Lemma~\ref{lem:conti:mo}]
The proof is an induction on $m$, the size of matrix $\mathbf{A_c}$.  Remind that here $\mathbf{C}$ can be a general matrix, so we use the definitions of $\mathbf{A_c},\mathbf{C}$ given as \eqref{eqn:conti:a2}, \eqref{eqn:conti:c2}.

(i) When $m=1$,

In this case, the system is scalar, and  the lemma is trivially true. A rigorous proof goes as follows: Since $(\mathbf{A_c},\mathbf{C})$ is observable, we can find a $1 \times l$ matrix $\mathbf{L}$ such that $\mathbf{L}\mathbf{C}$ is not zero. Then, $(\mathbf{A_c},\mathbf{L}\mathbf{C})$ is observable, and the lemma is reduced to Lemma~\ref{lem:conti:singlec}.

(ii) We will assume that the lemma holds for $(m-1)$-dimensional systems as an induction hypothesis, and prove the lemma holds for $m$.

The proof goes in three steps. First, we reduce the system to reducing a system with scalar observations to apply Lemma~\ref{lem:conti:singlec}. Then, we estimate one of the states, and subtract the estimation from the system --- this procedure is known as successive decoding in information theory. Now, the system reduces to the $(m-1)$-dimensional one, so we apply the induction hypothesis.

For this, we define $\mathbf{x}:=\begin{bmatrix} \mathbf{x_{1,1}} \\ \mathbf{x_{1,2}} \\ \vdots \\ \mathbf{x_{\mu,\nu_{\mu}}} \end{bmatrix}$ where $\mathbf{x_{i,j}}$ are $m_{i,j} \times 1$ vectors, and $(\mathbf{x_{1,\nu_1}})_{m_{1,\nu_1}}$ as the $m_{1,\nu_1}$th element of $\mathbf{x_{1,\nu_1}}$. We also define $(\mathbf{x})_k$ as the $k$th element of a vector $\mathbf{x}$ in general. Here, $\mathbf{x}$ can be thought as the states of the system. We first decode $(\mathbf{x_{1,\nu_1}})_{m_{1,\nu_1}}$, and decode the remaining elements in $\mathbf{x}$.


$\bullet$ Reduction to Systems with Scalar Observations: By Lemma~\ref{lem:conti:singlec}, we already know that the lemma is true for systems with scalar observations. Therefore, we will reduce the general systems with vector observations to system with scalar observations.

\begin{claim}
There exist $\mathbf{L}, \mathbf{C'}, \mathbf{A'}, \mathbf{x'}$ that satisfies the following conditions.\\
(i) $\mathbf{L}$ is a $1 \times l$ row vector.\\
(ii) $\mathbf{A'}$ is a $m' \times m'$ square matrix given in a Jordan form. The eigenvalues of $\mathbf{A'}$ belong to $\{ \lambda_1 + j \omega_1, \cdots, \lambda_{\mu} + j \omega_{\mu}\}$ which is the set of the eigenvalues of $\mathbf{A}$. The first Jordan block of $\mathbf{A'}$ is equal to $\mathbf{A_{1,\nu_1}}$.\\
(iii) $\mathbf{C'}$ is a $l \times m'$ matrix and $(\mathbf{A'}, \mathbf{L}\mathbf{C'})$ is observable.\\
(iv) $\mathbf{x'}$ is a $m' \times l$ column vector. $(\mathbf{x'})_{m_{1,\nu_1}} = (\mathbf{x_{1,\nu_1}})_{m_1, \nu_1}$. \\
(v) $\mathbf{L}\mathbf{C}e^{-k \mathbf{A_c}} \mathbf{x} = \mathbf{L} \mathbf{C'} e^{-k \mathbf{A'}}\mathbf{x'}$.
\label{claim:dummy}
\end{claim}

What this claim implies is the following. By multiplying the matrix $\mathbf{L}$ to the vector observations, we can reduce the vector observations to the scalar observations. However, the resulting system may not be observable any more. Therefore, we will carefully design $\mathbf{L}$ matrix and reduced system matrices $\mathbf{A'}$, $\mathbf{C'}$, so that the system remains observable even with a scalar observation and the information about $(\mathbf{x_{1,\nu_1}})_{m_{1,\nu_1}}$ remains intact.

\begin{proof}
Since the first columns of $\mathbf{C_{1,1}}, \mathbf{C_{1,2}}, \cdots, \mathbf{C_{1,\nu_1}}$ are linearly independent, there exists a $1 \times l $ matrix $\mathbf{L}$ such that the first elements of $\mathbf{L}\mathbf{C_{1,1}},\mathbf{L}\mathbf{C_{1,2}},\cdots, \mathbf{L}\mathbf{C_{1,\nu_1-1}}$ are zeros and the first element of $\mathbf{L}\mathbf{C_{1,\nu_1}}$ is non-zero. Then, we can observe that
\begin{align}
\mathbf{L}\mathbf{C}e^{-k \mathbf{A_c}} \mathbf{x}&=
\mathbf{L}
\begin{bmatrix}
\mathbf{C_{1,1}} & \cdots & \mathbf{C_{\mu,\nu_{\mu}}}
\end{bmatrix}
\begin{bmatrix}
e^{-k \mathbf{A_{1,1}}} & \cdots & 0 \\
\vdots & \ddots & \vdots \\
0 & \cdots & e^{-k \mathbf{A_{\mu,\nu_{\mu}}}}
\end{bmatrix}
\begin{bmatrix}
\mathbf{x_{1,1}} \\
\vdots \\
\mathbf{x_{\mu,\nu_\mu}}
\end{bmatrix}
\nonumber \\
&=\mathbf{L} \mathbf{C_{1,1}} e^{-k \mathbf{A_{1,1}}} \mathbf{x_{1,1}} + \mathbf{L} \mathbf{C_{1,2}} e^{-k \mathbf{A_{1,2}}} \mathbf{x_{1,2}} + \cdots +
\mathbf{L} \mathbf{C_{\mu,\nu_\mu}} e^{-k \mathbf{A_{\mu,\nu_\mu}}} \mathbf{x_{\mu,\nu_\mu}}
\label{eqn:lem:idontknow}
\end{align}
Remind that the Jordan blocks $\mathbf{A_{i,1}}, \cdots ,\mathbf{A_{i,\nu_i}}$ correspond to the same eigenvalue. We will merge these Jordan blocks into one Jordan block. However, since the size of Jordan blocks $\mathbf{A_{i,1}},  \cdots ,\mathbf{A_{i,\nu_i}}$ are distinct, we will extend the small Jordan block to the size of the largest one by adding zero elements. Let the dimension of $\mathbf{A_{i,\bar{\nu}_i}}$ be the largest among $\mathbf{A_{i,1}},  \cdots ,\mathbf{A_{i,\nu_i}}$, and $m_{i,\bar{\nu}_i}$ be the corresponding dimension. Then, we define $\mathbf{\bar{C}_{i,j}}$ as a matrix where the first $m_{i,\bar{\nu}_i} - m_{i,j}$ vectors are all zeros, and the remaining vectors are the same as those of $\mathbf{C_{i,j}}$. $\mathbf{\bar{A}_{i,j}}$ is defined as the same matrix as $\mathbf{{A}_{i,\bar{\nu}_i}}$. $\mathbf{\bar{x}_{i,j}}$ is defined as a column vector whose first $m_{i,\bar{\nu}_i} - m_{i,j}$ elements are all zeros, and the remaining elements are those of $\mathbf{x_{i,j}}$.

Then, by the construction, we know
\begin{align}
\eqref{eqn:lem:idontknow}=\mathbf{L} \mathbf{\bar{C}_{1,1}} e^{-k \mathbf{\bar{A}_{1,1}}} \mathbf{\bar{x}_{1,1}} + \mathbf{L} \mathbf{\bar{C}_{1,2}} e^{-k \mathbf{\bar{A}_{1,2}}} \mathbf{\bar{x}_{1,2}} + \cdots +
\mathbf{L} \mathbf{\bar{C}_{\mu,\nu_\mu}} e^{-k \mathbf{\bar{A}_{\mu,\nu_\mu}}} \mathbf{\bar{x}_{\mu,\nu_\mu}}.
\end{align}
Furthermore, $\mathbf{A_{1,\nu_1}}=\mathbf{\bar{A}_{1,\nu_1}}$, $\mathbf{C_{1,\nu_1}}=\mathbf{\bar{C}_{1,\nu_1}}$, $\mathbf{x_{1,\nu_1}}=\mathbf{\bar{x}_{1,\nu_1}}$.
The first elements of $\mathbf{L}\mathbf{\mathbf{C_{1,1}}},\mathbf{L}\mathbf{\mathbf{C_{1,2}}}, \cdots, \mathbf{L}\mathbf{C_{1,\nu_1-1}}$ are zeros and the first element of $\mathbf{L}\mathbf{C_{1,\nu_1}}$ is non-zero.

Now, we get the same dimension systems $(\mathbf{\bar{A}_{i,1}},\mathbf{L}\mathbf{\bar{C}_{i,1}})$, $\cdots$, $(\mathbf{\bar{A}_{i,\nu_i}},\mathbf{L}\mathbf{\bar{C}_{i,\nu_i}})$. However, none of them might be observable. Thus, we will truncate the matrices to make sure that at least one of them is observable. Remind that since $\mathbf{L}\mathbf{\bar{C}_{i,j}}$ is a row vector and $\mathbf{\bar{A}_{i,j}}$ is a single Jordan block, the system is observable as long as the first element of $\mathbf{L} \mathbf{\bar{C}_{i,j}}$ is not zero. Thus, we will truncate the matrices until we see at least one nonzero element among the first elements of  $\mathbf{L}\mathbf{\bar{C}_{i,1}}$, $\cdots$, $\mathbf{L}\mathbf{\bar{C}_{i,\nu_i}}$. Let $k_i$ be the smallest number such that at least one of the $k_i$th elements of $\mathbf{L}\mathbf{\bar{C}_{i,1}}, \cdots, \mathbf{L}\mathbf{\bar{C}_{i,\nu_i}}$ becomes nonzero, and let $\mathbf{L}\mathbf{\bar{C}_{i,\nu_i^\star}}$ be the vector that achieves the minimum. 

Then, we will reduce the dimensions of $(\mathbf{\bar{A}_{i,j}},\mathbf{L}\mathbf{\bar{C}_{i,j}})$ by truncating the first $(k_i-1)$ vectors. Define $\mathbf{C_{i,j}'}$ as a matrix obtained by removing the first $(k_i-1)$ columns from $\mathbf{\bar{C}_{i,j}}$, $\mathbf{A_{i,j}'}$ as a matrix obtained by removing the first $(k_i-1)$ rows and columns from $\mathbf{\bar{A}_{i,j}}$, and $\mathbf{x_{i,j}'}$ as a matrix obtained by removing the first $(k_i-1)$ elements from $\mathbf{\bar{x}_{i,j}}$.

Then, by the construction, the resulting systems $(\mathbf{A_{i,\nu_i^\star}'}, \mathbf{L}\mathbf{C_{i,\nu_i^\star}'})$ are observable. 
We can also see that $\nu_{1}^\star=\nu_1$, $\mathbf{C_{1,\nu_1^\star}'}=\mathbf{\bar{C}_{1,\nu_1}}=\mathbf{C_{1,\nu_1}}$, $\mathbf{A_{1,\nu_1^\star}'}=\mathbf{\bar{A}_{1,\nu_1}}=\mathbf{A_{1,\nu_1}}$, and $\mathbf{x'_{1,\nu_1^\star}}=\mathbf{\bar{x}_{1,\nu_1}}=\mathbf{x_{1,\nu_1}}$. In words, the Jordan block $\mathbf{A_{1,\nu_1}}$ was not affected by the above manipulations. Moreover, by the construction, the first elements of $\mathbf{L}\mathbf{C'_{1,1}},\cdots,\mathbf{L}\mathbf{C'_{1,\nu_1-1}}$ are all zero.

Denote $\mathbf{C'}:=\begin{bmatrix} \mathbf{C'_{1,\nu_1^\star}} & \mathbf{C'_{2,\nu_2^\star}} & \cdots &  \mathbf{C'_{\mu,\nu_\mu^\star}} \end{bmatrix}$ and $\mathbf{A'}:= diag \{ \mathbf{A'_{1,\nu_1^\star}}, \mathbf{A'_{2,\nu_2^\star}}, \cdots, \mathbf{A'_{\mu,\nu_\mu^\star}} \}$. Then, \eqref{eqn:lem:idontknow} can be written as follows:
\begin{align}
\eqref{eqn:lem:idontknow}&=\mathbf{L} \mathbf{C'_{1,1}} e^{-k \mathbf{A'_{1,1}}} \mathbf{x'_{1,1}} + \mathbf{L} \mathbf{C'_{1,2}} e^{-k \mathbf{A'_{1,2}}} \mathbf{x'_{1,2}} + \cdots +
\mathbf{L} \mathbf{C'_{\mu,\nu_\mu}} e^{-k \mathbf{A'_{\mu,\nu_\mu}}} \mathbf{x'_{\mu,\nu_\mu}} \nonumber \\
&=\mathbf{L} \mathbf{C'_{1,\nu_1^\star}} e^{-k \mathbf{A'_{1,\nu_1^\star}}} (
\mathbf{x'_{1,\nu_1^\star}}+\sum_{j \in \{1,\cdots,\nu_1 \} \setminus \nu_1^\star} \mathbf{T_{1,j}}\mathbf{x_{1,j}'})+ \cdots  \nonumber \\
&\quad+\mathbf{L} \mathbf{C'_{\mu,\nu_\mu^\star}} e^{-k \mathbf{A'_{\mu,\nu_\mu^\star}}} (
\mathbf{x'_{\mu,\nu_\mu^\star}}+\sum_{j \in \{1,\cdots,\nu_\mu \} \setminus \nu_\mu^\star} \mathbf{T_{\mu,j}}\mathbf{x_{\mu,j}'}
) \label{eqn:lem:contigeo:10}\\
&=
\begin{bmatrix}
\mathbf{L}\mathbf{C'_{1,\nu_1^\star}} & \cdots & \mathbf{L}\mathbf{C'_{1,\nu_\mu^\star}}
\end{bmatrix}
\begin{bmatrix}
e^{-k \mathbf{A'_{1,\nu_1^\star}}} & \cdots & 0 \\
\vdots & \ddots & \vdots \\
0 & \cdots & e^{-k \mathbf{A'_{\mu,\nu_\mu^\star}}}
\end{bmatrix}
\underbrace{
\begin{bmatrix}
\mathbf{x'_{1,\nu_1^\star}}+\sum_{j \in \{1,\cdots,\nu_1 \} \setminus \nu_1^\star} \mathbf{T_{1,j}}\mathbf{x_{1,j}'}\\
\vdots \\
\mathbf{x'_{\mu,\nu_\mu^\star}}+\sum_{j \in \{1,\cdots,\nu_\mu \} \setminus \nu_\mu^\star} \mathbf{T_{\mu,j}}\mathbf{x_{\mu,j}'}
\end{bmatrix}
}_{:= \mathbf{x'} } \nonumber \\
&=\mathbf{L}\mathbf{C'} e^{-k \mathbf{A'}} \mathbf{x'}\label{eqn:lem:contigeo:3}
\end{align}
where \eqref{eqn:lem:contigeo:10} follows from Lemma~\ref{lem:conti:jordan}. Here, we can easily see that $\mathbf{A'}$ satisfies the condition (ii) of the claim, and $(\mathbf{A'}, \mathbf{L}\mathbf{C'})$ is observable since each $(\mathbf{A_{i,\nu_i^\star}'}, \mathbf{L}\mathbf{C_{i,\nu_i^\star}'})$ is observable.

Moreover, by Lemma~\ref{lem:conti:jordan}, we know that $\mathbf{T_{1,1}}, \cdots, \mathbf{T_{1,\nu_1-1}}$ are upper triangular matrices whose diagonal elements are zeros. Therefore, $(\mathbf{x'})_{m_{1,\nu_1}}=(\mathbf{x_{1,\nu_1}'})_{m_{1,\nu_1}}=(\mathbf{x_{1,\nu_1}})_{m_{1,\nu_1}}$. Therefore, the condition (iv) of the claim is also satisfied.
\end{proof}

$\bullet$ Decoding $(\mathbf{x_{1,\nu_1}})_{m_{1,\nu_1}}$: Now, we reduced the system to a system with a scalar observation. Then, we can apply Lemma~\ref{lem:conti:singlec} to decode $(\mathbf{x_{1,\nu_1}})_{m_{1,\nu_1}}$.

\begin{claim}
We can find a polynomial $p'(k)$ and a family of stopping time $\{S'(\epsilon,k): k \in \mathbb{Z^+}, \epsilon > 0\}$ such that for all $\epsilon>0$, $k \in \mathbb{Z}^+$ there exist $k \leq k_1 < k_2 < \cdots < k_{m'} \leq S'(\epsilon,k)$ and $\mathbf{M_1'}$ satisfying:\\
(i) $\beta[k_i]=1$ for $1 \leq i \leq m'$\\
(ii) $\mathbf{M_1'}
\begin{bmatrix}
\mathbf{L} & 0 & \cdots & 0 \\
0 & \mathbf{L} & \cdots & 0 \\
\vdots & \vdots & \ddots & \vdots \\
0 & 0 & \cdots & \mathbf{L}
\end{bmatrix}
\begin{bmatrix}
\mathbf{C} e^{-(k_1 I + t_{k_1}) \mathbf{A}} \\
\mathbf{C} e^{-(k_2 I + t_{k_2}) \mathbf{A}} \\
\vdots \\
\mathbf{C} e^{-(k_{m'} I + t_{k_{m'}}) \mathbf{A}} \\
\end{bmatrix} \mathbf{x}=(\mathbf{x_{1,\nu_1}})_{m_{1,\nu_1}}
$\\
(iii) $\left|\mathbf{M_1'}\right|_{max} \leq  \frac{p'(S'(\epsilon,k))}{\epsilon} e^{\lambda_1 S'(\epsilon,k)I}$\\
(iv) $\lim_{\epsilon \downarrow 0} \exp \limsup_{s \rightarrow \infty} \sup_{k \in \mathbb{Z^+}} \frac{1}{s} \log \mathbb{P}\{ S'(\epsilon,k)-k = s\} \leq p_e$.
\label{claim:donknow}
\end{claim}

This claim is showing that there exist an estimator $\mathbf{M_1'} diag\{ \mathbf{L}, \cdots, \mathbf{L} \}$ which can estimate the state $(\mathbf{x_{1,\nu_1}})_{m_{1,\nu_1}}$ with observations at time $k_1, \cdots, k_m$.

\begin{proof}
By the construction, $(\mathbf{A'},\mathbf{L}\mathbf{C'})$ is observable and $\mathbf{L}\mathbf{C'}$ is a row vector. Thus, by Lemma~\ref{lem:conti:singlec} we can find a polynomial $p'(k)$ and a family of stopping time $\{S'(\epsilon,k): k \in \mathbb{Z^+}, \epsilon > 0\}$ such that for all $\epsilon>0$, $k \in \mathbb{Z}^+$ there exist $k \leq k_1 < k_2 < \cdots < k_{m'} \leq S'(\epsilon,k)$ and $\mathbf{M'}$ satisfying:\\
(i) $\beta[k_i]=1$ for $1 \leq i \leq m'$\\
(ii) $\mathbf{M'}
\begin{bmatrix}
\mathbf{L}\mathbf{C'} e^{-(k_1 I + t_{k_1}) \mathbf{A'}} \\
\mathbf{L}\mathbf{C'} e^{-(k_2 I + t_{k_2}) \mathbf{A'}} \\
\vdots \\
\mathbf{L}\mathbf{C'} e^{-(k_{m'} I + t_{k_{m'}}) \mathbf{A'}} \\
\end{bmatrix}=\mathbf{I}
$\\
(iii) $\left|\mathbf{M'}\right|_{max} \leq  \frac{p'(S'(\epsilon,k))}{\epsilon} e^{\lambda_1 S'(\epsilon,k)I}$\\
(iv) $\lim_{\epsilon \downarrow 0} \exp \limsup_{s \rightarrow \infty} \sup_{k \in \mathbb{Z^+}} \frac{1}{s} \log \mathbb{P}\{ S'(\epsilon,k)-k = s\} \leq p_e$.

Let $\mathbf{M'_{1}}$ be the $m_{1,\nu_1}$th row of $\mathbf{M'}$. Then,
\begin{align}
&\mathbf{M'_{1}}
\begin{bmatrix}
\mathbf{L} & 0 & \cdots & 0 \\
0 & \mathbf{L} & \cdots & 0 \\
\vdots & \vdots & \ddots & \vdots \\
0 & 0 & \cdots & \mathbf{L}
\end{bmatrix}
\begin{bmatrix}
\mathbf{C}e^{-(k_1 I + t_{k_1})\mathbf{A_c}} \\
\mathbf{C}e^{-(k_2 I + t_{k_2})\mathbf{A_c}} \\
\vdots \\
\mathbf{C}e^{-(k_{m'} I + t{k_{m'}})\mathbf{A_c}} \\
\end{bmatrix}
\mathbf{x}
=
\mathbf{M'_{1}}
\begin{bmatrix}
\mathbf{L}\mathbf{C}e^{-(k_1 I + t_{k_1})\mathbf{A_c}}\mathbf{x} \\
\mathbf{L}\mathbf{C}e^{-(k_2 I + t_{k_2})\mathbf{A_c}}\mathbf{x} \\
\vdots \\
\mathbf{L}\mathbf{C}e^{-(k_{m'} I + t_{k_{m'}})\mathbf{A_c}}\mathbf{x} \\
\end{bmatrix} \nonumber \\
&=
\mathbf{M'_{1}}
\begin{bmatrix}
\mathbf{L'}\mathbf{C'}e^{-(k_1 I + t_{k_1})\mathbf{A'}}\mathbf{x'} \\
\mathbf{L'}\mathbf{C'}e^{-(k_2 I + t_{k_2})\mathbf{A'}}\mathbf{x'} \\
\vdots \\
\mathbf{L'}\mathbf{C'}e^{-(k_{m'} I + t_{k_{m'}})\mathbf{A'}}\mathbf{x'} \\
\end{bmatrix} (\because Claim~\ref{claim:dummy}\mbox{ (v)})\nonumber \\
&=
\mathbf{M'_{1}}
\begin{bmatrix}
\mathbf{L'}\mathbf{C'}e^{-(k_1 I + t_{k_1})\mathbf{A'}} \\
\mathbf{L'}\mathbf{C'}e^{-(k_2 I + t_{k_2})\mathbf{A'}} \\
\vdots \\
\mathbf{L'}\mathbf{C'}e^{-(k_{m'} I + t_{k_{m'}})\mathbf{A'}} \\
\end{bmatrix}
\mathbf{x'} = (\mathbf{x'})_{m_{1,\nu_1}} = (\mathbf{x_{1,\nu_1}})_{m_{1,\nu_1}} (\because Claim~\ref{claim:dummy}\mbox{ (iv)}).
\end{align}
\end{proof}

$\bullet$ Subtracting $(\mathbf{x_{1,\nu_1}})_{m_{1,\nu_1}}$ from the observations: Now, we have an estimation for $(\mathbf{x_{1,\nu_1}})_{m_{1,\nu_1}}$. We will remove it from the system. $\mathbf{A''}$,$\mathbf{C''}$ and $\mathbf{x''}$ are the system matrices after the removal. Formally, $\mathbf{A''}$,$\mathbf{C''}$ and $\mathbf{x''}$ are obtained by removing $\sum_{1 \leq i \leq \nu_i} m_{1,i}$th row and column from $\mathbf{A_c}$, removing $\sum_{1 \leq i \leq \nu_i} m_{1,i}$th row from $\mathbf{C}$ and removing $\sum_{1 \leq i \leq \nu_i} m_{1,i}$th component from $\mathbf{x}$ respectively.

Obviously, $\mathbf{A''} \in \mathbb{C}^{(m-1) \times (m-1)}$ and $\mathbf{C''} \in \mathbb{C}^{l \times (m-1)}$. Moreover, since the last element of the Jordan block $\mathbf{A_{1,\nu_1}}$ is removed and the observability only depends on the first element, $(\mathbf{A''},\mathbf{C''})$ is observable. Denote $\lambda_1''+\omega_1''$ be the eigenvalue of $\mathbf{A''}$ with the largest real part. Then, trivially $\lambda_1'' \leq \lambda_1$.

The new system $(\mathbf{A''},\mathbf{C''})$ and the original system $(\mathbf{A},\mathbf{C})$ are related as follows. Denote the $\sum_{1 \leq i \leq \nu_i} m_{1,i}$th column of $\mathbf{C}e^{-k \mathbf{A_c}}$ as $\mathbf{R}(k)$. Then, we have
\begin{align}
\mathbf{C}\mathbf{e^{-k \mathbf{A_c}}} \mathbf{x} -\mathbf{R}(k)(\mathbf{x_{1,\nu_1}})_{m_{1,\nu_1}} &= \mathbf{C''}\mathbf{e^{-k \mathbf{A''}}} \mathbf{x''} \label{eqn:lem:contigeo:4}
\end{align}
which can be easily proved from the block diagonal structure of $\mathbf{A_c}$. We can further see that there exists a polynomial $p'''(k)$ such that $\left| \mathbf{R}(k) \right|_{max} \leq p'''(k)e^{-k \lambda_1}$.

$\bullet$ Decoding the remaining element of $\mathbf{x}$:
We decoded and subtracted the state $(\mathbf{x_{1,\nu_1}})_{m_{1,\nu_1}}$ from the system. Now, we can apply the induction hypothesis to the remaining $(m-1)$-dimensional system and estimate the remaining states.

By induction hypothesis, for given $S'(\epsilon,k)$, we can find $m'' \in \mathbb{Z}$ and a polynomial $p''(k)$ and a family of stopping time $\{S''(\epsilon,S'(\epsilon,k)): S'(\epsilon,k) \in \mathbb{Z^+}, 0 < \epsilon < 1 \}$ such that for all $0< \epsilon < 1$ there exist $S'(\epsilon,k) < k_{m'+1} < \cdots < k_{m''} \leq S''(\epsilon,S'(\epsilon,k))$ and a $(m-1) \times (m''-m')l$ matrix $\mathbf{M''}$ satisfying the following conditions:\\
(i) $\beta[k_i] = 1$ for $m'+1 \leq i \leq m''$\\
(ii)
$
\mathbf{M''}
\begin{bmatrix}
\mathbf{C''} e^{-(k_{m'+1} I + t_{k_{m'+1}})\mathbf{A''}} \\
\mathbf{C''} e^{-(k_{m'+2} I + t_{k_{m'+2}})\mathbf{A''}} \\
\vdots \\
\mathbf{C''} e^{-(k_{m''} I + t_{k_{m''}})\mathbf{A''}} \\
\end{bmatrix} = \mathbf{I}_{(m-1) \times (m-1)}
$ \\
(iii)
$
\left| \mathbf{M''} \right|_{max} \leq \frac{p''(S''(\epsilon,S'(\epsilon,k)))}{\epsilon} e^{\lambda_1'' S''(\epsilon,S'(\epsilon,k)) I}
$
\\
(iv)
$
\lim_{\epsilon \downarrow 0} \exp \limsup_{s \rightarrow \infty} \esssup \frac{1}{s} \log \mathbb{P} \{ S''(\epsilon,S'(\epsilon,k))-S'(\epsilon,k)=s | \mathcal{F}_{S'(\epsilon,k)} \} \leq p_e
$\\
where $\mathcal{F}_n$ is the $\sigma$-field generated by $\beta[0], \cdots, \beta[n]$ and $t_0, \cdots, t_n$.

Then,
\begin{align}
\mathbf{x''}&=\mathbf{M''}
\begin{bmatrix}
\mathbf{C''}e^{-(k_{m'+1}I + t_{k_{m'+1}})\mathbf{A''}} \\
\mathbf{C''}e^{-(k_{m'+2}I + t_{k_{m'+2}})\mathbf{A''}} \\
\vdots \\
\mathbf{C''}e^{-(k_{m''}I + t_{k_{m''}})\mathbf{A''}} \\
\end{bmatrix}
\mathbf{x''} \nonumber \\
&=
\mathbf{M''}
\begin{bmatrix}
\mathbf{C}e^{-(k_{m'+1}I + t_{k_{m'+1}})\mathbf{A_c}}\mathbf{x}-\mathbf{R}(k_{m'+1}I + t_{k_{m'+1}})(\mathbf{x_{1,\nu_1}})_{m_{1,\nu_1}} \\
\mathbf{C}e^{-(k_{m'+2}I + t_{k_{m'+2}})\mathbf{A_c}}\mathbf{x}-\mathbf{R}(k_{m'+2}I + t_{k_{m'+2}})(\mathbf{x_{1,\nu_1}})_{m_{1,\nu_1}} \\
\vdots \\
\mathbf{C}e^{-(k_{m''}I + t_{k_{m''}})\mathbf{A_c}}\mathbf{x}-\mathbf{R}(k_{m''}I + t_{k_{m''}})(\mathbf{x_{1,\nu_1}})_{m_{1,\nu_1}} \\
\end{bmatrix}
\label{eqn:lem:contigeo:11} \\
&=
\mathbf{M''} \left(
\begin{bmatrix}
\mathbf{C}e^{-(k_{m'+1}I + t_{k_{m'+1}})\mathbf{A_c}}\\
\mathbf{C}e^{-(k_{m'+2}I + t_{k_{m'+2}})\mathbf{A_c}}\\
\vdots \\
\mathbf{C}e^{-(k_{m''}I + t_{k_{m''}})\mathbf{A_c}}\\
\end{bmatrix}
\mathbf{x}
-
\begin{bmatrix}
\mathbf{R}(k_{m'+1}I + t_{k_{m'+1}}) \\
\mathbf{R}(k_{m'+2}I + t_{k_{m'+2}}) \\
\vdots \\
\mathbf{R}(k_{m''}I + t_{k_{m''}}) \\
\end{bmatrix}
(\mathbf{x_{1,\nu_1}})_{m_{1,\nu_1}}
\right)
\nonumber \\
&=\mathbf{M''} \left(
\begin{bmatrix}
\mathbf{C}e^{-(k_{m'+1}I + t_{k_{m'+1}})\mathbf{A_c}}\\
\mathbf{C}e^{-(k_{m'+2}I + t_{k_{m'+2}})\mathbf{A_c}}\\
\vdots \\
\mathbf{C}e^{-(k_{m''}I + t_{k_{m''}})\mathbf{A_c}}\\
\end{bmatrix}
\mathbf{x}
-
\begin{bmatrix}
\mathbf{R}(k_{m'+1}I + t_{k_{m'+1}}) \\
\mathbf{R}(k_{m'+2}I + t_{k_{m'+2}}) \\
\vdots \\
\mathbf{R}(k_{m''}I + t_{k_{m''}}) \\
\end{bmatrix}
\mathbf{M_1'}
\begin{bmatrix}
\mathbf{L} & 0 & \cdots & 0 \\
0 & \mathbf{L} & \cdots & 0 \\
\vdots & \vdots & \ddots & \vdots \\
0 & 0 & \cdots & \mathbf{L}
\end{bmatrix}
\begin{bmatrix}
\mathbf{C}e^{-(k_1 I + t_{k_1})\mathbf{A_c}} \\
\mathbf{C}e^{-(k_2 I + t_{k_2})\mathbf{A_c}} \\
\vdots \\
\mathbf{C}e^{-(k_{m'} I + t_{k_{m'}})\mathbf{A_c}} \\
\end{bmatrix}
\mathbf{x}
\right)
\label{eqn:lem:contigeo:2} \\
&=
\mathbf{M''}
\begin{bmatrix}
-
\begin{bmatrix}
\mathbf{R}(k_{m'+1}I + t_{k_{m'+1}}) \\
\mathbf{R}(k_{m'+2}I + t_{k_{m'+2}}) \\
\vdots \\
\mathbf{R}(k_{m''}I + t_{k_{m''}}) \\
\end{bmatrix}
\mathbf{M_1'}
\begin{bmatrix}
\mathbf{L} & 0 & \cdots & 0 \\
0 & \mathbf{L} & \cdots & 0 \\
\vdots & \vdots & \ddots & \vdots \\
0 & 0 & \cdots & \mathbf{L}
\end{bmatrix}
&
\mathbf{I}
\end{bmatrix}
\begin{bmatrix}
\mathbf{C}e^{-(k_1 I + t_{k_1})\mathbf{A_c}} \\
\mathbf{C}e^{-(k_2 I + t_{k_2})\mathbf{A_c}} \\
\vdots \\
\mathbf{C}e^{-(k_{m''} I + t_{k_{m''}})\mathbf{A_c}} \\
\end{bmatrix}
\mathbf{x} \nonumber
\end{align}
where \eqref{eqn:lem:contigeo:11} follows from \eqref{eqn:lem:contigeo:4}, and \eqref{eqn:lem:contigeo:2} follows from the condition (ii) of Claim~\ref{claim:donknow}.
Therefore, we can recover the remaining states of $\mathbf{x}$.

Moreover, we have
\begin{align}
&\left| \mathbf{M''}
\begin{bmatrix}
-
\begin{bmatrix}
\mathbf{R}(k_{m'+1}I + t_{k_{m'+1}}) \\
\mathbf{R}(k_{m'+2}I + t_{k_{m'+2}}) \\
\vdots \\
\mathbf{R}(k_{m''}I + t_{k_{m''}}) \\
\end{bmatrix}
\mathbf{M_1'}
\begin{bmatrix}
\mathbf{L} & 0 & \cdots & 0 \\
0 & \mathbf{L} & \cdots & 0 \\
\vdots & \vdots & \ddots & \vdots \\
0 & 0 & \cdots & \mathbf{L}
\end{bmatrix}
&
\mathbf{I}
\end{bmatrix} \right|_{max}
\nonumber \\
&\lesssim
\left| \mathbf{M''} \right|_{max} \cdot \max\left\{
\left|
\begin{bmatrix}
\mathbf{R}(k_{m'+1}I + t_{k_{m'+1}}) \\
\mathbf{R}(k_{m'+2}I + t_{k_{m'+2}}) \\
\vdots \\
\mathbf{R}(k_{m''}I + t_{k_{m''}}) \\
\end{bmatrix}
\right|_{max}
\left| \mathbf{M_1'} \right|_{max}
\left| \mathbf{L} \right|_{max}
,1
\right\} \nonumber \\
&\lesssim
\frac{p''(S''(\epsilon,S'(\epsilon,k)))}{\epsilon} e^{\lambda_1'' S''(\epsilon,S'(\epsilon,k)) I}
\max\left\{
p'''(k_{m''}I+t_{k_{m''}}) e^{-\lambda_1 ( k_{m'+1}I + t_{k_{m'+1}} )}
\cdot \frac{p'(S'(\epsilon,k))}{\epsilon}e^{\lambda_1 S'(\epsilon,k) I}
\cdot \left| \mathbf{L} \right|_{max}
,1
\right\} \nonumber \\
&\lesssim \frac{\bar{p}(S''(\epsilon,S'(\epsilon,k)))}{\epsilon^2} e^{\lambda_1 S''(\epsilon,S'(\epsilon,k))I}\ (\because S'(\epsilon,k)< k_{m'+1} < k_{m''} \leq S''(\epsilon,S'(\epsilon,k)),\ \lambda_1'' \leq \lambda_1) \nonumber
\end{align}
for some polynomial $\bar{p}(k)$. Since for some $\bar{\bar{p}}(k)$
\begin{align}
\left| \mathbf{M_1'} \right|_{max} \leq \frac{p'(S'(\epsilon,k))}{\epsilon} e^{\lambda_1 S'(\epsilon,k)I} \leq \frac{\bar{\bar{p}}(S''(\epsilon,S'(\epsilon,k)))}{\epsilon^2} e^{\lambda_1 S''(\epsilon,S'(\epsilon,k))I} \nonumber
\end{align}
and we can recover $\mathbf{x}$ from $\mathbf{x''}$ and $(\mathbf{x_{1,\nu_1}})_{m_{1,\nu_1}}$, there exists $\mathbf{M}$ and a polynomial $p(k)$ such that
\begin{align}
\mathbf{M}
\begin{bmatrix}
\mathbf{C}e^{-(k_1 I + t_{k_1})\mathbf{A_c}} \\
\vdots \\
\mathbf{C}e^{-(k_{m''} I + t_{k_{m''}})\mathbf{A_c}}
\end{bmatrix}
=\mathbf{I}_{m \times m} \nonumber
\end{align}
and
\begin{align}
\left| \mathbf{M} \right|_{max} \leq \frac{p(S''(\epsilon,S'(\epsilon,k)))}{\epsilon^2} e^{\lambda_1 S''(\epsilon,S'(\epsilon,k)) I}. \nonumber
\end{align}
Moreover, since
\begin{align}
\lim_{\epsilon \downarrow 0} \exp \limsup_{s \rightarrow \infty} \esssup \log \mathbb{P}\{ S''(\epsilon,S'(\epsilon,k)) - S'(\epsilon,k) |  \mathcal{F}_{S'(\epsilon,k)}\} \leq p_e \nonumber
\end{align}
and
\begin{align}
\lim_{\epsilon \downarrow 0} \exp \limsup_{s \rightarrow \infty} \sup_{k \in \mathbb{Z^+}} \log \mathbb{P}\{ S'(\epsilon,k) - k \} \leq p_e, \nonumber
\end{align}
by Lemma~\ref{lem:app:geo}
\begin{align}
\lim_{\epsilon \downarrow 0} \exp \limsup_{s \rightarrow \infty} \sup_{k \in \mathbb{Z^+}} \log \mathbb{P}\{ S''(\epsilon,S'(\epsilon,k)) - k \} \leq p_e.\nonumber
\end{align}
Therefore, by putting $S(\epsilon,k):=S''(\epsilon^{\frac{1}{2}},S'(\epsilon^{\frac{1}{2}},k))$, $S(\epsilon,k)$ satisfies all the conditions of the lemma.
\end{proof}

\subsection{Lemmas about the Observability Gramian of Discrete-Time Systems}
\label{sec:dis:gramian}
Now, we will consider the discrete-time systems discussed in Section~\ref{sec:interob}.
Like the continuous time case, we start from a simpler case when $\mathbf{C}$ is a row vector and $\mathbf{A}$ has no eigenvalue cycles. The definitions corresponding to \eqref{eqn:ac:jordan} for the row vector case are given as follows:
Let $\mathbf{A}$ be a $m \times m$ Jordan form matrix and $\mathbf{C}$ be $1 \times m$ row vector which can be written as
\begin{align}
&\mathbf{A}=diag\{ \mathbf{A_{1,1}}, \mathbf{A_{1,2}}, \cdots, \mathbf{A_{1,\nu_1}}, \cdots, \mathbf{A_{\mu,1}}, \cdots, \mathbf{A_{\mu,\nu_\mu}}\} \label{eqn:ac:jordansingle} \\
&\mathbf{C}=\begin{bmatrix} \mathbf{C_{1,1}}, \mathbf{C_{1,2}}, \cdots, \mathbf{C_{1,\nu_1}}, \cdots, \mathbf{C_{\mu,1}}, \cdots, \mathbf{C_{\mu,\nu_\mu}} \end{bmatrix} \label{eqn:ac:jordansinglec} \\
&\mbox{where} \nonumber \\
&\quad \mbox{$\mathbf{A_{i,j}}$ is a Jordan block with eigenvalue $\lambda_{i,j}e^{j 2 \pi \omega_{i,j}}$ and size $m_{i,j}$} \nonumber \\
&\quad m_{i,1} \leq m_{i,2} \leq \cdots \leq m_{i,\nu_i}\mbox{ for all }i=1,\cdots,\mu \nonumber \\
&\quad \lambda_{i,1}=\lambda_{i,2}=\cdots=\lambda_{i,\nu_i}\mbox{ for all }i=1,\cdots,\mu \nonumber \\
&\quad \lambda_{1,1} > \lambda_{2,1} > \cdots > \lambda_{\mu,1} \geq 1  \nonumber \\
&\quad \{ \lambda_{i,1},\cdots, \lambda_{i,\nu_i} \} \mbox{ is cycle with length $\nu_i$ and period $p_i$}\nonumber \\
&\quad \mbox{For all $(i,j )\neq (i',j')$, $\omega_{i,j}-\omega_{i',j'} \notin \mathbb{Q}$} \nonumber \\
&\quad \mbox{$\mathbf{C_{i,j}}$ is a $1 \times m_{i,j}$ complex matrix and its first element is non-zero} \nonumber \\
&\quad \mbox{$\lambda_i e^{j 2 \pi \omega_i}$ is $(i,i)$ element of $\mathbf{A}$}. \nonumber
\end{align}
Here, we can notice that $\mathbf{A}$ has no eigenvalue cycles since $\omega_{i,j}-\omega_{i',j'} \notin \mathbb{Q}$ for all $(i,j )\neq (i',j')$, and $\mathbf{C}$ is a row vector. By Theorem~\ref{thm:jordanob}, the condition that the first elements of $\mathbf{C_{i,j}}$ are non-zero corresponds to the observability condition of $(\mathbf{A},\mathbf{C})$ since $\mathbf{C}$ is a row vector.

Let's state lemmas which parallel Lemma~\ref{lem:conti:inverse2} and Lemma~\ref{lem:det:lower}. In fact, the proofs of the lemmas are very similar to those of Lemma~\ref{lem:conti:inverse2} and Lemma~\ref{lem:det:lower} and we omit the proofs here.
\begin{lemma}
Let $\mathbf{A}$ and $\mathbf{C}$ be given as \eqref{eqn:ac:jordansingle} and \eqref{eqn:ac:jordansinglec}. Then, there exists a polynomial $p(k)$ such that for all $\epsilon>0$ and $0 \leq k_1 \leq \cdots \leq k_m$, if
\begin{align}
\left| \det\left(
\begin{bmatrix}
\mathbf{C} \mathbf{A}^{-k_1} \\
\mathbf{C} \mathbf{A}^{-k_2} \\
\vdots \\
\mathbf{C} \mathbf{A}^{-k_m}
\end{bmatrix}
\right)  \right| \geq \epsilon \prod_{1 \leq i \leq m} \lambda_i^{-k_i} \nonumber
\end{align}
then
\begin{align}
\left|
\begin{bmatrix}
\mathbf{C} \mathbf{A}^{-k_1} \\
\mathbf{C} \mathbf{A}^{-k_2} \\
\vdots \\
\mathbf{C} \mathbf{A}^{-k_m} \\
\end{bmatrix}
\right|_{max} \leq \frac{p(k_m)}{\epsilon} \lambda_1^{k_m} \nonumber
\end{align}
\label{lem:dis:inverse}
\end{lemma}
\begin{proof}
It can be easily proved in a similar way to Lemma~\ref{lem:conti:inverse2}
\end{proof}

\begin{lemma}
Let $\mathbf{A}$ and $\mathbf{C}$ be given as \eqref{eqn:ac:jordansingle} and \eqref{eqn:ac:jordansinglec}.
Define $a_{i,j}$ and $C_{i,j}$ as the $(i,j)$ element and cofactor of
$\begin{bmatrix} \mathbf{C}\mathbf{A}^{-k_1} \\ \mathbf{C}\mathbf{A}^{-k_2} \\ \vdots \\ \mathbf{C}\mathbf{A}^{-k_m} \end{bmatrix}$ respectively.
Then there exists $g_{\epsilon}(k):\mathbb{R}^+ \rightarrow \mathbb{R}^+$ and $a \in \mathbb{R}^+$ such that for all $\epsilon>0$  and $k_1,\cdots, k_m$ satisfying
\begin{align}
&(i) 0 \leq k_1 < k_2 < \cdots < k_m \nonumber \\
&(ii) k_m - k_{m-1} \geq g_{\epsilon}(k_{m-1})  \nonumber \\
&(iii) g_{\epsilon}(k) \leq a(1+\log (k+1)) \nonumber \\
&(iv)\left| \sum_{m-m_{\mu}+1 \leq i \leq m} a_{m,i} C_{m,i} \right| \geq \epsilon \prod_{1 \leq i \leq m} {\lambda_i}^{-k_i} \nonumber
\end{align}
the following inequality holds:
\begin{align}
\left| \det \left(
\begin{bmatrix}
\mathbf{C}\mathbf{A}^{-k_1} \\
\mathbf{C}\mathbf{A}^{-k_2} \\
\vdots \\
\mathbf{C}\mathbf{A}^{-k_m} \\
\end{bmatrix}
\right) \right| \geq \frac{1}{2} \epsilon \prod_{1 \leq i \leq m} {\lambda_i}^{-k_i}. \nonumber
\end{align}
\label{lem:dis:det:lower}
\end{lemma}
\begin{proof}
It can be easily proved in a similar way to Lemma~\ref{lem:det:lower}.
\end{proof}
Like the continuous-time case, these lemmas reduce questions about the inverse of the observability Gramian to questions about the determinant of the observability Gramian.

\subsection{Uniform Convergence of Sequences satisfying Weyl's criterion (Discrete-Time Systems)}
\label{sec:dis:uniform}
As we did in the continuous-time case, we will prove that the determinant of the observability matrix is large enough regardless of the erasure pattern. The main difference from the continuous-time case of Appendix~\ref{app:unif:conti} is the measure that must be used. While we used the Lebesgue measure to measure the bad event ---the event that the determinant of the observability matrix is small---, we use the counting measure in this section.

The main idea of this section is approximating aperiodic deterministic sequences by random variables using ergodic theory~\cite{Kuipers}. The necessary and sufficient condition for a sequence to behave like uniformly distributed random variables in $[0,1]$ is known as Weyl's criterion. We first state a general ergodic theorem, and derive the Weyl's criterion as a corollary.
\begin{theorem}[Koksma and Szusz inequality~\cite{Kuipers}]
Consider a $s$-dimensional sequence $\mathbf{x_1}, \mathbf{x_2}, \cdots \in \mathbb{R}^s$, and let $\alpha := (\alpha_1, \cdots, \alpha_s)$ and $\beta:=(\beta_1, \cdots, \beta_s)$. For any positive integer $m$, we have
\begin{align}
\sup_{0 \leq \alpha_i < \beta_i \leq 1} \left| \frac{A\left([\mathbf{\alpha},\mathbf{\beta});N, \{\mathbf{x_n}\}\right)}{N} - \prod_{1 \leq i \leq s} (\beta_i - \alpha_i)  \right| \leq 2s^2 3^{s+1} \left( \frac{1}{m} + \sum_{\mathbf{h} \in \mathbb{Z}^s, 0 < |\mathbf{h}|_{\infty} \leq m} \frac{1}{r(\mathbf{h})} \left| \frac{1}{N} \sum_{1 \leq n \leq N} e^{2 \pi \sqrt{-1} \left<\mathbf{h},\mathbf{x_n}\right>} \right| \right) \nonumber
\end{align}
where
\begin{align}
&A\left([\mathbf{\alpha},\mathbf{\beta});N, \{\mathbf{x_n}\}\right):=\sum_{1 \leq n \leq N} \mathbf{1}\left\{\mathbf{x_n} \in [\alpha_1, \beta_1)\times [\alpha_2, \beta_2) \cdots \times [\alpha_s, \beta_s)  \right\} \label{eqn:defcount} \\
&r(\mathbf{h}):= \prod_{1 \leq j \leq s} \max\{|h_j|,1\}. \nonumber
\end{align}
\label{thm:koksma}
\end{theorem}
\begin{proof}
See \cite{Kuipers} for the proof.
\end{proof}
Here, we can see $A([\alpha, \beta);N, \{\mathbf{x_n})$ is the counting measure of the event that a sequence falls in the set $[\alpha,\beta)$. The theorem tells us that the counting measure is close to the Lebesgue measure of the set $[\alpha,\beta)$ uniformly over all $\alpha, \beta$.

Using this theorem, we can easily derive\footnote{The original Weyl's criterion is shown for only one sequence. But, here we extend Weyl's criterion to a family of sequences. For this, we state a generalized theorem of the Weyl's criterion and prove it.} the Weyl's criterion for a family of sequences.

\begin{definition}
Consider a family of $s$-dimensional sequences $\mathcal{J}=\{ (\mathbf{x_{1,\sigma}}, \mathbf{x_{2,\sigma}}, \cdots): \sigma \in J , x_{i,\sigma} \in \mathbb{R}^s  \}$. Here, the index set for the sequences, $J$, can be infinite. If for all $\mathbf{h} \in \mathbb{Z}^s \setminus \{\mathbf{0} \}$,
\begin{align}
\lim_{N \rightarrow \infty} \sup_{\sigma \in \mathcal{J}} \left| \frac{1}{N} \sum_{1 \leq n \leq N} e^{j 2 \pi \left<\mathbf{h},\mathbf{x_{n,\sigma}}\right>}\right| = 0 \nonumber
\end{align}
then the family of sequences is said to satisfy Weyl's criterion.
\end{definition}

\begin{theorem}[Weyl's criterion~\cite{Kuipers}]
Consider a family of $s$-dimensional sequences $\mathcal{J}=\{ (\mathbf{x_{1,\sigma}}, \mathbf{x_{2,\sigma}}, \cdots): \sigma \in J , x_{i,\sigma} \in \mathbb{R}^s  \}$, which satisfy the Weyl's criterion. Then, this family of sequences satisfies
\begin{align}
\lim_{N \rightarrow \infty} \sup_{\sigma \in \mathcal{J}} \sup_{0 \leq \alpha_i < \beta_i \leq 1} \left| \frac{A\left([\mathbf{\alpha},\mathbf{\beta});N, \{\mathbf{x_{n,\sigma}} \}\right)}{N} - \prod_{1 \leq i \leq s} (\beta_i - \alpha_i)  \right| =0, \nonumber
\end{align}
where the definition of $A\left([\mathbf{\alpha},\mathbf{\beta});N, \{\mathbf{x_{n,\sigma}} \}\right)$ is given in \eqref{eqn:defcount}.
\label{thm:weyl}
\end{theorem}
\begin{proof}
By Theorem~\ref{thm:koksma}, for any positive integer $m$, we have
\begin{align}
&\sup_{\sigma \in \mathcal{J}} \sup_{0 \leq \alpha_i < \beta_i \leq 1} \left| \frac{A\left([\mathbf{\alpha},\mathbf{\beta});N, \{\mathbf{x_{n,\sigma}}\}\right)}{N} - \prod_{1 \leq i \leq s} (\beta_i - \alpha_i)  \right|\\
&\leq \sup_{\sigma \in \mathcal{J}}
2s^2 3^{s+1} \left( \frac{1}{m} + \sum_{0 < |\mathbf{h}|_{\infty} \leq m} \frac{1}{r(\mathbf{h})} \left| \frac{1}{N} \sum_{1 \leq n \leq N} e^{2 \pi j \left<\mathbf{h},\mathbf{x_{n,\sigma}}\right>} \right| \right)  \label{eqn:weyl:new:1}
\end{align}
To prove the theorem, it is enough to show that  for all $\delta > 0$ there exists $N'$ such that for all $N > N'$
\begin{align}
 \sup_{\sigma \in \mathcal{J}} \sup_{0 \leq \alpha_i < \beta_i \leq 1} \left| \frac{A\left([\mathbf{\alpha},\mathbf{\beta});N, \{\mathbf{x_{n,\sigma}}\}\right)}{N} - \prod_{1 \leq i \leq s} (\beta_i - \alpha_i)  \right| < \delta.\label{eqn:weyl:new:2}
\end{align}

Let's choose $m:=\frac{4s^2 3^{s+1}}{\delta}$ so that
\begin{align}
\frac{2s^2 3^{s+1}}{m} < \frac{\delta}{2}. \label{eqn:weyl:new:3}
\end{align}

Once we fix $m$, there are only $(2m+1)^s$ number of $\mathbf{h} \in \mathbb{Z}^s$ such that $|\mathbf{h}|_{\infty} \leq m$. Furthermore, by the definition of Weyl's criterion, we can find $N''$ such that for all $N > N''$,
\begin{align}
\sup_{\sigma \in \mathcal{J}} \left| \frac{1}{N} \sum_{1 \leq n \leq N} e^{j 2 \pi \left<\mathbf{h},\mathbf{x_{n,\sigma}}\right>}\right| < \frac{1}{(2m+1)^s 2 s^2 3^{s+1}} \frac{\delta}{2}.
\end{align}

Thus, we can find $N''$ such that for all $N > N''$ the following holds:
\begin{align}
2s^2 3^{s+1} s^{m+1} \max_{0 < |\mathbf{h}|_{\infty} \leq m} \sup_{\sigma \in \mathcal{J}} \left| \frac{1}{N} \sum_{1 \leq n \leq N} e^{j 2 \pi \left<\mathbf{h},\mathbf{x_{n,\sigma}}\right>}\right| < \frac{\delta}{2} \label{eqn:weyl:new:4}
\end{align}
Therefore, by plugging \eqref{eqn:weyl:new:3}, \eqref{eqn:weyl:new:4} into \eqref{eqn:weyl:new:1}, we can prove \eqref{eqn:weyl:new:2}. Thus, the theorem is true.
\end{proof}

Since we are mainly interested in the fractional part of sequences, it will be helpful to denote $\left<x \right> := x - \lfloor x \rfloor$. Although $\left<\mathbf{x},\mathbf{y}\right>$ is the inner product between two vectors, these two definitions can be distinguished by counting the number of arguments. Let's consider some specific sequences, and see whether they satisfies the Weyl's criterion.
\begin{example}
$\left(\left< \sqrt{2}n \right>, \left< \sqrt{3}n \right> \right)$ satisfies Weyl's criterion and $\left( \left<  \sqrt{2}n \right>, \left<(\sqrt{2}+\sqrt{3})n \right> \right)$ does too.\\
$\left(\left<\sqrt{2}n \right>, \left< \left( \sqrt{2} + 0.5 \right)n \right> \right)$ does not satisfy Weyl's criterion and neither does $\left(\left<\sqrt{2}n \right>, \left< \frac{\sqrt{2}}{2} n \right> \right)$.
\end{example}
Therefore, among general sequences in the form of $( \left<\omega_1 n \right>, \left<\omega_2 n \right>, \cdots, \left<\omega_m n \right> )$, there are sequences which satisfy Weyl's criterion and others do not. However, the following lemma reveals all sequences can be written as linear combinations of basis sequences which satisfy Weyl's criterion. This idea is very similar to that linear-algebraic concepts like linear decomposition and basis.

\begin{lemma}
Consider an $m$-dimensional sequence $( \left<\omega_1 n\right>, \left<\omega_2 n\right>, \cdots, \left<\omega_m n\right> )$. Then, there exists $k \leq m$ and $p \in \mathbb{N}$ such that
\begin{align}
\omega_i= \frac{q_{i,0}}{p}+\sum_{1 \leq j \leq k} q_{i,j}\gamma_j \nonumber
\end{align}
where
\begin{align}
&q_{i,j} \in \mathbb{Z},\nonumber \\
&( \left<\gamma_1 n\right>, \left<\gamma_2 n\right>, \cdots, \left<\gamma_k n\right> ) \mbox{ satisfies Weyl's criterion.} \nonumber
\end{align}
\label{lem:dis:weyl2}
\end{lemma}
\begin{proof}
Before the proof, we can observe the following two facts.

First, since as long as $\left<\mathbf{h}, \mathbf{w}\right>$ is not an integer,
\begin{align}
\frac{1}{N} \sum_{1 \leq n \leq N} e^{j 2 \pi \left<\mathbf{h},(\left<\omega_1 n\right>, \left<\omega_2 n\right>, \cdots, \left<\omega_m n\right> ) \right> } = \frac{1}{N} \frac{e^{j2 \pi (h_1 \omega_1 + h_2 \omega_2 + \cdots + h_m \omega_m )}\left(1-e^{j2 \pi N(h_1 \omega_1 + h_2 \omega_2 + \cdots + h_m \omega_m )}\right)}{1- e^{j 2 \pi(h_1 \omega_1 + h_2  \omega_2 + \cdots + h_m \omega_m)}},\nonumber
\end{align}
the statement that the sequence $(\left<\omega_1 n\right>,\left<\omega_2 n\right>,\cdots,\left<\omega_m n\right>)$ does not satisfy Weyl's criterion is equivalent to there being $h_1, h_2, \cdots , h_m \in \mathbb{Z}$ that are not identically zero and make
\begin{align}
h_1 \omega_1 + h_2 \omega_2 + \cdots + h_m \omega_m \in \mathbb{Z}.\label{eqn:weyl:1}
\end{align}

The second observation is that if $( \left< \omega_1 n \right> , \left< \omega_2 n \right>, \cdots, \left< \omega_m n\right> )$ satisfies Weyl's criterion then for all $a_1,\cdots,a_m \in \mathbb{N}$, $(\left< \frac{\omega_1}{a_1} n\right>,\left< \frac{\omega_2}{a_2} n\right>,\cdots, \left< \frac{\omega_m}{a_m} n\right>)$ also satisfies Weyl's criterion. To see this, suppose $(\left< \frac{\omega_1}{a_1} n\right>,\left< \frac{\omega_2}{a_2} n\right>,\cdots, \left< \frac{\omega_m}{a_m} n\right>)$ did not satisfy Weyl's criterion. Then, by \eqref{eqn:weyl:1} there would exist $(h_1,h_2,\cdots,h_m) \in \mathbb{Z}^m \setminus \{\mathbf{0}\}$ such that $h_1 \frac{\omega_1}{a_1} + h_2 \frac{\omega_2}{a_2} + \cdots + h_m \frac{\omega_m}{a_m} \in \mathbb{Z}$. So, $ \frac{h_1 \prod_{1 \leq i \leq m}a_i}{a_1} {\omega_1} + \frac{h_2 \prod_{1 \leq i \leq m} a_i}{a_2} {\omega_2} + \cdots + \frac{h_m \prod_{1 \leq i \leq m} a_i}{a_m} {\omega_m} \in \mathbb{Z}$ as well as $(\frac{h_1 \prod_{1 \leq i \leq m}a_i}{a_1}, \cdots, \frac{h_m \prod_{1 \leq i \leq m}a_i}{a_m}) \in \mathbb{Z}^m \setminus \{ \mathbf{0} \}$. But since $(\left<\omega_1 n\right>,\left<\omega_2 n\right>, \cdots, \left<\omega_m n\right> )$ would not satisfy Weyl's criterion, this causes a contradiction.

Now, we will prove the lemma by induction on $m$.

(i) When $m=1$,

If $\left<\omega_1 n\right>$ satisfies Weyl's criterion, the lemma is trivially true by selecting $\gamma_1=\omega_1$ and $q_{1,1}=1$. If $\left<\omega_1 n\right>$ does not satisfy Weyl's criterion, then by \eqref{eqn:weyl:1}, $\omega_1$ is a rational number. So we can find $q_{1,0}$ and $p$ such that $\omega_1=\frac{q_{1,0}}{p}$, and set the $k=0$.

(ii) Assume that the lemma is true for $m-1$.

If $(\left<\omega_1 n\right>, \left<\omega_2 n\right>,\cdots, \left<\omega_m n\right>)$ satisfies Weyl's criterion, the lemma follows by selecting $k=m$, $\gamma_i=\omega_i$ and $q_{i,i}=1$.

If $(\left<\omega_1 n\right>, \left<\omega_2 n\right>,\cdots, \left<\omega_m n\right>)$ does not satisfy Weyl's criterion, by \eqref{eqn:weyl:1} there exists $(h_1,h_2,\cdots, h_m) \in \mathbb{Z}^m \setminus \{ \mathbf{0} \}$ and $h \in \mathbb{Z}$ such that $h_1 \omega_1 + h_2 \omega_2 + \cdots + h_m \omega_m = h$. Without loss of generality, let's say $h_1 \neq 0$. Then
\begin{align}
\omega_1 = - \frac{h_2}{h_1} \omega_2 - \frac{h_3}{h_1} \omega_3 - \cdots - \frac{h_m}{h_1} \omega_m + \frac{h}{h_1}. \label{eqn:dis:weyl:1}
\end{align}

By induction hypothesis, we know that there exists $k' \leq m-1$, $p' \in \mathbb{N}$, $q_{i,j}' \in \mathbb{Z}$, $\gamma_i'$ such that
\begin{align}
&\omega_2 = \frac{q_{2,0}'}{p'}+\sum_{1 \leq j \leq k'} q_{2,j}'\gamma_j' \nonumber \\
&\vdots \nonumber \\
&\omega_m = \frac{q_{m,0}'}{p'}+\sum_{1 \leq j \leq k'} q_{m,j}'\gamma_j'. \label{eqn:dis:weyl:2}
\end{align}
where $(\left<\gamma_1' n \right>, \left<\gamma_2' n\right>, \cdots, \left<\gamma_k' n\right>)$ satisfies Weyl's criterion. Therefore, by plugging \eqref{eqn:dis:weyl:2} to \eqref{eqn:dis:weyl:1} we can find $q'_{1,j} \in \mathbb{Z}$ such that
\begin{align}
\omega_1 = \frac{q'_{1,0}}{|h_1 \cdot p'|}+ \sum_{1 \leq i \leq k } q'_{1,i} \frac{\gamma_i'}{h_1}. \nonumber 
\end{align}
By the second observation, $(\left<\frac{\gamma_1'}{h_1}n\right>, \left<\frac{\gamma_2'}{h_1}n\right>,\cdots, \left<\frac{\gamma_k'}{h_1}n\right> )$ satisfies Weyl's criterion, so we can use $p=|h_1 \cdot p'|$ and $\gamma_i = \frac{\gamma_i'}{h_1}$ to show that the lemma also holds for $m$.

Therefore, by induction the lemma is true.
\end{proof}

Now, we can decompose the sequences into basis sequences which satisfy Weyl's criterion, and so behave like uniform random variables.
The main difference from the uniform convergence discussion of Appendix~\ref{app:unif:conti} is the number of random variables. In other words, in continuous-time systems with random jitter, only one random variable is introduced at each sample for the random jitter. However, this is not the case in discrete-time systems.

Let $\mathbf{A_1}=\begin{bmatrix} e^{j \sqrt{2}} & 0 \\ 0 & e^{j 2 \sqrt{2}} \end{bmatrix}$, $\mathbf{A_2}=\begin{bmatrix} e^{j \sqrt{2}} & 0 \\ 0 & e^{j \sqrt{3}} \end{bmatrix}$, $\mathbf{C}=\begin{bmatrix} 1 & 1 \end{bmatrix}$. The row of the observability gramian of $(\mathbf{\mathbf{A_1}}, \mathbf{C})$ is $\mathbf{C}\mathbf{A_1}^n = \begin{bmatrix} e^{j \sqrt{2} n} & e^{j 2\sqrt{2} n}\end{bmatrix}$. In this case, the elements of $\mathbf{C}\mathbf{A_1}^n$ do not satisfy Weyl's criterion. Thus, it can be approximated by $\begin{bmatrix} e^{j X} & e^{j 2X} \end{bmatrix}$ where $X$ is uniform in $[0, 2\pi]$, which involves only one random variable. 

However, the row of the observability gramian of $(\mathbf{\mathbf{A_2}}, \mathbf{C})$ is $\mathbf{C}\mathbf{A_2}^n = \begin{bmatrix} e^{j \sqrt{2} n} & e^{j \sqrt{3} n}\end{bmatrix}$ whose elements satisfy Weyl's criterion. Thus, it can be approximated by $\begin{bmatrix} e^{j X_1} & e^{j X_2} \end{bmatrix}$ where $X_1$, $X_2$ are independent uniform random variables in $[0, 2\pi]$, which involves two random variables. 

Therefore, the lemmas derived in Appendix~\ref{app:unif:conti} have to be generalized to multiple random variables, and then the multiple random variables can be used to model deterministic sequences. 

Intuitively, adding more randomness should not cause any problems, so generalization to multiple random variables must be possible.
We first extend Lemma~\ref{lem:single} which was written for a single random variable to multiple random variables.

\begin{lemma}
Let $\mathbf{X}$ be $(X_1,X_2,\cdots,X_{\nu})$ where $X_i$ are i.i.d.~random variables whose distribution is uniform between $0$ and $2 \pi$. Let $\mathbf{k_1},\mathbf{k_2},\cdots,\mathbf{k_{\mu}} \in \mathbb{R}^{\nu}$ be distinct. 
Then, for strictly positive $\gamma$, $\Gamma$ $(\gamma \leq \Gamma)$, and $m \in \{ 1, \cdots, \mu \}$
\begin{align}
\sup_{|a_{m}| \geq \gamma, |a_{i}| \leq \Gamma, a_i \in \mathbb{C}} \mathbb{P} \{ | \sum^{\mu}_{i=1} a_i e^{j<\mathbf{k_i},\mathbf{X}>} | < \epsilon \} \rightarrow 0 \mbox{ as } \epsilon \downarrow 0. \nonumber
\end{align}
\label{lem:dis:geo1}
\end{lemma}
\begin{proof}
We will prove the lemma by induction on $\nu$, the number of random variables.

(i) When $\nu=1$. The lemma reduces to Lemma~\ref{lem:single}.

(ii) Let's assume the lemma is true for $1,\cdots,\nu-1$.

Without loss of generality, we can assume $m=1$ by symmetry. We will prove the lemma by dividing into cases based on $\mathbf{k_i}$. Let the $j$th component of $\mathbf{k_i}$ be denoted as $k_{ij}$.

First, consider the case when $k_{1,1}=k_{2,1}=\cdots=k_{\mu, 1}$. Then, 
\begin{align}
&\sup_{|a_{1}| \geq \gamma, |a_{i}| \leq \Gamma} \mathbb{P} \{ | \sum^{\mu}_{i=1} a_i e^{j<\mathbf{k_i},\mathbf{X}>} | < \epsilon \}
=\sup_{|a_{1}| \geq \gamma, |a_{i}| \leq \Gamma} \mathbb{P} \{ | \sum^{\mu}_{i=1} a_i e^{j \sum_{1 \leq j \leq \nu} k_{i,j} X_j} | < \epsilon \} \nonumber \\
&=\sup_{|a_{1}| \geq \gamma, |a_{i}| \leq \Gamma} \mathbb{P} \{ | e^{j k_{1,1}X_1}| \cdot | \sum^{\mu}_{i=1} a_i e^{j \sum_{2 \leq j \leq \nu} k_{i,j} X_j} | < \epsilon \} \nonumber\\
&=\sup_{|a_{1}| \geq \gamma, |a_{i}| \leq \Gamma} \mathbb{P} \{ | \sum^{\mu}_{i=1} a_i e^{j \sum_{2 \leq j \leq \nu} k_{i,j} X_j} | < \epsilon \}
\rightarrow 0\ (\because \mbox{induction hypothesis}) \nonumber .
\end{align}

Second, consider the case when $k_{i,1} \neq k_{j,1}$ for some $i,j$. Without loss of generality, we can assume that $k_{1,1}=k_{2,1}=\cdots =k_{\mu_1,1}$ and $k_{1,1} \neq k_{j,1}$ for all $\mu_1 <j\leq \mu$. Then, for all $\epsilon' > 0$, we have
\begin{align}
&\sup_{|a_{1}| \geq \gamma, |a_{i}| \leq \Gamma} \mathbb{P} \{ | \sum^{\mu}_{i=1} a_i e^{j <\mathbf{k_i},\mathbf{X}>} | < \epsilon \} \nonumber\\
&=\sup_{|a_{1}| \geq \gamma, |a_{i}| \leq \Gamma} \mathbb{P} \{ | \sum^{\mu_1}_{i=1} a_i e^{j <\mathbf{k_i},\mathbf{X}>}+ \sum^{\mu}_{i=\mu_1+1} a_i e^{j <\mathbf{k_i},\mathbf{X}>} | < \epsilon \} \nonumber\\
&\leq \sup_{|a_{1}| \geq \gamma, |a_{i}| \leq \Gamma} \mathbb{P}
\{ | \sum^{\mu_1}_{i=1} a_i e^{j <\mathbf{k_i},\mathbf{X}>}+ \sum^{\mu}_{i=\mu_1+1} a_i e^{j <\mathbf{k_i},\mathbf{X}>} | < \epsilon \Big| |\sum^{\mu_1}_{i=1} a_i e^{j \sum_{2 \leq j \leq \nu} k_{i,j}X_j} | \geq \epsilon' \} +
\mathbb{P}\{ |\sum^{\mu_1}_{i=1} a_i e^{j \sum_{2 \leq j \leq \nu} k_{i,j}X_j} | < \epsilon' \} \nonumber \\
&= \sup_{|a_{1}| \geq \gamma, |a_{i}| \leq \Gamma} \mathbb{P}
\{ | ( \sum^{\mu_1}_{i=1} a_i e^{j \sum_{2 \leq j \leq \nu} k_{i,j} X_j} )e^{j k_{1,1}X_1}  + \sum^{\mu}_{i=\mu_1+1} a_i e^{j <\mathbf{k_i},\mathbf{X}>} | < \epsilon \Big| |\sum^{\mu_1}_{i=1} a_i e^{j \sum_{2 \leq j \leq \nu} k_{i,j}X_j} | \geq \epsilon' \}  \nonumber\\
&+\mathbb{P}\{ |\sum^{\mu_1}_{i=1} a_i e^{j \sum_{2 \leq j \leq \nu} k_{i,j}X_j} | < \epsilon' \} \nonumber \\
&\leq
\sup_{|a'_{1}| \geq \epsilon', |a'_{i}| \leq \mu\Gamma} \mathbb{P}_{X_1}
\{ |a'_{1} e^{j k_{1,1}X_1}+ \sum^{\mu}_{i=\mu_1+1} a'_i e^{j k_{i,1}X_1} | < \epsilon  \} +\sup_{|a_{1}| \geq \gamma, |a_{i}| \leq \Gamma} \mathbb{P}\{ |\sum^{\mu_1}_{i=1} a_i e^{j \sum_{2 \leq j \leq \nu} k_{i,j}X_j} | < \epsilon' \}. \nonumber \\
\end{align}
Therefore, by the induction hypothesis (since the first term has only one random variable, and the second term has $\nu-1$ random variables)
\begin{align}
&\lim_{\epsilon \rightarrow 0} \sup_{|a_{1}| \geq \gamma, |a_{i}| \leq \Gamma} \mathbb{P} \{ | \sum^{\mu}_{i=1} a_i e^{j <\mathbf{k_i},\mathbf{X}>} | < \epsilon \} \nonumber\\
&\leq \lim_{\epsilon' \rightarrow 0} \lim_{\epsilon \rightarrow 0} \sup_{|a'_{1}| \geq \epsilon', |a'_{i}| \leq \mu\Gamma} \mathbb{P}
\{ |a'_{1} e^{j k_{1,1}X_1}+ \sum^{\mu}_{i=\mu_1+1} a'_i e^{j k_{i,1}X_1} | < \epsilon  \} +\sup_{|a_{1}| \geq \gamma, |a_{i}| \leq \Gamma} \mathbb{P}\{ |\sum^{\mu_1}_{i=1} a_i e^{j \sum_{2 \leq j \leq \nu} k_{i,j}X_j} | < \epsilon' \} \nonumber\\
&=0. \nonumber
\end{align}
Therefore, the lemma is true.
\end{proof}

Now, we will consider a deterministic sequence in the form of $(<\omega_1 n>, \cdots, <\omega_{\mu}n>)$. As we have shown in Lemma~\ref{lem:dis:weyl2}, this sequence can be thought of as a linear combination of basis sequences which satisfy Weyl's criterion. Thus, we can approximate the deterministic sequence as a linear combination of multiple uniform random variables considered in Lemma~\ref{lem:dis:geo1}.

\begin{lemma}
Let $\omega_1,\omega_2,\cdots,\omega_{\mu}$ be real numbers such that $\omega_i - \omega_j \notin \mathbb{Q}$ for all $i \neq j$. Then, for strictly positive numbers $\gamma$ and $\Gamma$ $(\gamma \leq \Gamma)$, and $m \in \{1, \cdots, \mu\}$
\begin{align}
\lim_{\epsilon \downarrow 0}\lim_{N \rightarrow \infty} \sup_{|a_m| \geq \gamma, |a_{i}| \leq \Gamma, k \in \mathbb{Z}} \frac{1}{N} \sum^N_{n=1} \mathbf{1} \{ | \sum^{\mu}_{i=1} a_{i}e^{j2 \pi \omega_i(n+k)}| < \epsilon \}\rightarrow 0. \nonumber
\end{align}
\label{lem:dis:geo2}
\end{lemma}
\begin{proof}
By Lemma~\ref{lem:dis:weyl2}, $\omega_i$ can be written as $\left<\mathbf{q_i},\mathbf{\rho}\right>$ where $\mathbf{q_i}=(q_{i,0},q_{i,1},\cdots,q_{i,r}) \in \mathbb{Z}^{r+1}$, $\mathbf{\rho}=(\frac{1}{s},\rho_1,\cdots,\rho_r) \in \mathbb{R}^{r+1}$ and $s \in \mathbb{N}$. Here, $(\left<\rho_1 n\right>,\left<\rho_2 n\right>,\cdots,\left<\rho_r n\right>)$ satisfies Weyl's criterion. Since $\omega_i-\omega_j \notin \mathbb{Q}$ for all $i\neq j$, $(q_{i,1},q_{i,2},\cdots,q_{i,r}) \neq (q_{j,1},q_{j,2},\cdots,q_{j,r})$.

For given $k, N, M  \in \mathbb{N}$, and $m_1, \cdots, m_r \in \{1, \cdots, M\}$, define a set $S_{m_1, \cdots, m_r}$ as\footnote{Notice that the definition of $S_{m_1, \cdots, m_r}$ also depends on $k, N, M$ as well as $m_1, \cdots, m_r$. However, we omit the dependence on $k, N, M$ in the definition for simplicity.}
\begin{align}
\left\{n \in \{1, \cdots, N \}: \frac{m_1-1}{M} \leq \left<\rho_1 (n+k)\right> < \frac{m_1}{M},\cdots,\frac{m_r-1}{M} \leq \left<\rho_r (n+k)\right> < \frac{m_r}{M} \right\}. \nonumber
\end{align}

Then, for all $k, N, M  \in \mathbb{N}$ and $\epsilon > 0$, we have the following:
\begin{align}
&\sum^N_{n=1} \mathbf{1} \{ |\sum^{\mu}_{i=1}  a_i e^{j 2 \pi \omega_i (n+k)} | < \epsilon  \} \nonumber \\
&= \sum^N_{n=1} \sum_{1 \leq m_1 \leq M, \cdots, 1 \leq m_r \leq M} \mathbf{1} \{ | \sum^{\mu}_{i=1} a_i e^{j 2 \pi \omega_i (n+k) } | < \epsilon, n \in S_{m_1, \cdots, m_r} \} \nonumber \\
&\leq \sum^N_{n=1} \sum_{1 \leq m_1 \leq M, \cdots, 1 \leq m_r \leq M} \mathbf{1} \{ \min_{n \in S_{m_1, \cdots, m_r}} | \sum^{\mu}_{i=1} a_i e^{j 2 \pi \omega_i (n+k) } | < \epsilon,  n \in S_{m_1, \cdots, m_r}  \} \nonumber \\
&= \sum^N_{n=1} \sum_{1 \leq m_1 \leq M, \cdots, 1 \leq m_r \leq M} \mathbf{1} \{ \min_{n \in S_{m_1, \cdots, m_r}} | \sum^{\mu}_{i=1} a_i e^{j 2 \pi \omega_i (n+k) } | < \epsilon \} \cdot \mathbf{1}\{  n \in S_{m_1, \cdots, m_r} \}.
\label{eqn:dis:geo1:1}
\end{align}

Moreover, we also know by the definitions of $\mathbf{q_i}$ and $\mathbf{\rho}$,
\begin{align}
\sum^{\mu}_{i=1} a_i e^{j 2 \pi \omega_i(n+k)}&=\sum^{\mu}_{i=1} a_i e^{j 2 \pi \left<\mathbf{q_i},\mathbf{\rho}\right>(n+k)} \nonumber \\
&=\sum^{\mu}_{i=1} a_i e^{j 2 \pi \left(\frac{q_{i,0}}{s}(n+k)+q_{i,1}\rho_1(n+k)+\cdots + q_{i,r}\rho_r(n+k) \right)} \nonumber \\
&=\sum^{\mu}_{i=1} a_i e^{j 2 \pi \left(\frac{q_{i,0}}{s}(n+k)+q_{i,1}\left<\rho_1(n+k)\right>+\cdots + q_{i,r}\left<\rho_r(n+k)\right> \right)}  (\because q_{i,j} \in \mathbb{Z}).\nonumber
\end{align}

Thus, by defining $\mathbf{X_{m_1,\cdots,m_r}}$ as a random vector which is uniformly distributed over $[\frac{m_1-1}{M} , \frac{m_1}{M} ) \times \cdots \times [ \frac{m_r-1}{M} , \frac{m_r}{M} )$ and $\mathbf{q_i'}=(q_{i,1},q_{i,2},\cdots,q_{i,r})$, $\mathbf{\rho'}=(\rho_{1},\rho_{2},\cdots,\rho_{r})$, we can conclude
\begin{align}
\max_{n \in S_{m_1, \cdots, m_r}} | \sum^{\mu}_{i=1} a_i e^{j 2 \pi \omega_i (n+k) } |
&=\max_{n \in S_{m_1, \cdots, m_r}} | \sum^{\mu}_{i=1} a_i e^{j 2 \pi \left(\frac{q_{i,0}}{s}(n+k)+q_{i,1}\left<\rho_1(n+k)\right>+\cdots + q_{i,r}\left<\rho_r(n+k)\right> \right) } |\nonumber \\
&\geq
| \sum^{\mu}_{i=1} a_i e^{j 2 \pi \left( \frac{q_{i,0}}{s} (n+k) + \left<\mathbf{q_i'},\mathbf{X_{m_1,\cdots,m_r}}\right> \right) } |
\quad a.e. \label{eqn:weylupper1}
\end{align}

By \eqref{eqn:weylupper1}, \eqref{eqn:dis:geo1:1} can be upper bounded as follows:
\begin{align}
&\eqref{eqn:dis:geo1:1}\leq \sum^N_{n=1} \sum_{1 \leq m_1 \leq M, \cdots, 1 \leq m_r \leq M} \mathbb{P} \{ \min_{ n \in S_{m_1, \cdots, m_r}} | \sum^{\mu}_{i=1} a_i e^{j 2 \pi \omega_i (n+k) } | - \max_{ n \in S_{m_1, \cdots, m_r}} | \sum^{\mu}_{i=1} a_i e^{j 2 \pi \omega_i (n+k) } |
\nonumber \\
&+ | \sum^{\mu}_{i=1} a_i e^{j 2 \pi \left( \frac{q_{i,0}}{s} (n+k) + \left<\mathbf{q_i'},\mathbf{X_{m_1,\cdots,m_r}}\right> \right) } | < \epsilon \}
\cdot \mathbf{1}\{  n \in S_{m_1, \cdots, m_r} \}
\nonumber \\
&= \sum^N_{n=1} \sum_{1 \leq m_1 \leq M, \cdots, 1 \leq m_r \leq M} \mathbb{P} \{ \min_{ n \in S_{m_1, \cdots, m_r}} | \sum^{\mu}_{i=1} a_i e^{j 2 \pi \left( \frac{q_{i,0}}{s}(n+k)+\left<\mathbf{q_i'},\mathbf{\rho'}\right>(n+k) \right) } | - \max_{ n \in S_{m_1, \cdots, m_r}} | \sum^{\mu}_{i=1} a_i e^{j 2 \pi \left( \frac{q_{i,0}}{s}(n+k)+\left<\mathbf{q_i'},\mathbf{\rho'}\right>(n+k) \right) } | \nonumber \\
&+ | \sum^{\mu}_{i=1} a_i e^{j 2 \pi \left( \frac{q_{i,0}}{s} (n+k) + \left<\mathbf{q_i'},\mathbf{X_{m_1,\cdots,m_r}}\right> \right) } | < \epsilon \} \cdot \mathbf{1}\{  n \in S_{m_1, \cdots, m_r} \} \nonumber \\
&\leq \sum^N_{n=1} \sum_{1 \leq m_1 \leq M, \cdots, 1 \leq m_r \leq M} \max_{0 \leq s' < s} \mathbb{P} \{ \min_{ n \in S_{m_1, \cdots, m_r}} | \sum^{\mu}_{i=1} a_i e^{j 2 \pi \left( \frac{s'}{s}+\left<\mathbf{q_i'},\mathbf{\rho'}\right>(n+k) \right) } | - \max_{ n \in S_{m_1, \cdots, m_r}} | \sum^{\mu}_{i=1} a_i e^{j 2 \pi \left( \frac{s'}{s}+\left<\mathbf{q_i'},\mathbf{\rho'}\right>(n+k) \right) } | \nonumber \\
&+ | \sum^{\mu}_{i=1} a_i e^{j 2 \pi \left( \frac{s'}{s}  + \left<\mathbf{q_i'},\mathbf{X_{m_1,\cdots,m_r}}\right> \right) } | < \epsilon \}
 \cdot \mathbf{1}\{  n \in S_{m_1, \cdots, m_r} \}\nonumber \\
&\leq \sum^N_{n=1} \sum_{1 \leq m_1 \leq M, \cdots, 1 \leq m_r \leq M} \sum_{0 \leq s' < s} \mathbb{P} \{ \min_{ n \in S_{m_1, \cdots, m_r}} | \sum^{\mu}_{i=1} a_i e^{j 2 \pi \left( \frac{s'}{s}+\left<\mathbf{q_i'},\mathbf{\rho'}\right>(n+k) \right) } | - \max_{ n \in S_{m_1, \cdots, m_r}} | \sum^{\mu}_{i=1} a_i e^{j 2 \pi \left( \frac{s'}{s}+\left<\mathbf{q_i'},\mathbf{\rho'}\right>(n+k) \right) } | \nonumber \\
&+ | \sum^{\mu}_{i=1} a_i e^{j 2 \pi \left( \frac{s'}{s}  + \left<\mathbf{q_i'},\mathbf{X_{m_1,\cdots,m_r}}\right> \right) } | < \epsilon \}
 \cdot \mathbf{1}\{  n \in S_{m_1, \cdots, m_r} \}. \label{eqn:dis:geo1:3}
\end{align}

Here, we have
\begin{align}
&\max_{n \in S_{m_1, \cdots, m_r}}
| \sum^{\mu}_{i=1} a_i e^{j 2 \pi \left( \frac{s'}{s}+\left<\mathbf{q_i'},\mathbf{\rho'}\right>(n+k) \right) } |  \nonumber \\
&=\max_{n \in S_{m_1, \cdots, m_r}} | \sum^{\mu}_{i=1} a_i e^{j 2 \pi \left( \frac{s'}{s}+q_{i,1}\left<\rho_1(n+k)\right>+\cdots+q_{i,r}\left<\rho_r(n+k)\right> \right)}|
(\because q_{i,j} \in \mathbb{Z})\nonumber \\
&\leq \sup_{ 0 \leq \Delta_i < \frac{1}{M}} | \sum^{\mu}_{i=1} a_i e^{j 2 \pi \left( \frac{s'}{s}+q_{i,1}\frac{m_1-1}{M}+\cdots+q_{i,r}\frac{m_r-1}{M}+q_{i,1}\Delta_1+\cdots+q_{i,r}\Delta_r \right)}| \nonumber \\
&=\sup_{ 0 \leq \Delta_i < \frac{1}{M}} | \sum^{\mu}_{i=1} a_i e^{j 2 \pi \left( \frac{s'}{s}+q_{i,1}\frac{m_1-1}{M}+\cdots+q_{i,r}\frac{m_r-1}{M}\right)}  + a_i e^{j 2 \pi \left( \frac{s'}{s}+q_{i,1}\frac{m_1-1}{M}+\cdots+q_{i,r}\frac{m_r-1}{M}\right)} \nonumber\\
&\quad(-1+\cos 2\pi (q_{i,1}\Delta_1+\cdots+q_{i,r}\Delta_r )+j \sin 2\pi(q_{i,1}\Delta_1+\cdots+q_{i,r}\Delta_r) )| \nonumber \\
&\leq |\sum^{\mu}_{i=1} a_i e^{j 2 \pi \left( \frac{s'}{s}+q_{i,1}\frac{m_1-1}{M}+\cdots+q_{i,r}\frac{m_r-1}{M}\right)} | \nonumber \\
&+\sum^{\mu}_{i=1} |a_i e^{j 2 \pi \left( \frac{s'}{s}+q_{i,1}\frac{m_1-1}{M}+\cdots+q_{i,r}\frac{m_r-1}{M}\right)}|  \nonumber \\
&\cdot(\sup_{ 0 \leq \Delta_i < \frac{1}{M}} |-1+\cos2\pi(q_{i,1}\Delta_1+\cdots+q_{i,r}\Delta_r)|+ \sup_{ 0 \leq \Delta_i < \frac{1}{M}}|\sin2\pi(q_{i,1}\Delta_1+\cdots+q_{i,r}\Delta_r)|) \nonumber \\
&\leq |\sum^{\mu}_{i=1} a_i e^{j 2 \pi \left( \frac{s'}{s}+q_{i,1}\frac{m_1-1}{M}+\cdots+q_{i,r}\frac{m_r-1}{M}\right)} |
+4 \pi \sum^{\mu}_{i=1} |a_i| \sup_{0 \leq \Delta_i < \frac{1}{M}}|q_{i,1}\Delta_1+\cdots+q_{i,r}\Delta_r| \label{eqn:dis:geo1:2} \\
&\leq |\sum^{\mu}_{i=1} a_i e^{j 2 \pi \left( \frac{s'}{s}+q_{i,1}\frac{m_1-1}{M}+\cdots+q_{i,r}\frac{m_r-1}{M}\right)} |
+ \frac{4 \pi \Gamma}{M} \sum^{\mu}_{i=1}  \sum^{r}_{j=1} |q_{i,j}|.  (\because \mbox{We assumed }|a_i| \leq \Gamma)\nonumber
\end{align}
where \eqref{eqn:dis:geo1:2} comes from the fact that $|\sin x| \leq |x|$ and $|-1+\cos x| \leq |x|$ for all $x \in \mathbb{R}$. 

Likewise, we also have
\begin{align}
&\min_{n \in S_{m_1, \cdots, m_r}} | \sum^{\mu}_{i=1} a_i e^{j 2 \pi \left( \frac{s_i'}{s}+\left<\mathbf{q_i'},\mathbf{\rho'}\right>(n+k) \right) } |  \nonumber \\
&\geq |\sum^{\mu}_{i=1} a_i e^{j 2 \pi \left( \frac{s_i'}{s}+q_{i,1}\frac{m_1-1}{M}+\cdots+q_{i,r}\frac{m_r-1}{M}\right)} |
- \frac{4 \pi \Gamma}{M} \sum^{\mu}_{i=1}  \sum^{r}_{j=1} |q_{i,j}|. \nonumber
\end{align}

Therefore,
\begin{align}
&\sup_{\frac{m_i-1}{M} \leq {\left<\rho_i (n+k)\right>} < \frac{m_i}{M}} | \sum^{\mu}_{i=1} a_i e^{j 2 \pi \left( \frac{s_i'}{s}+\left<\mathbf{q_i'},\mathbf{\rho'}\right>(n+k) \right) } | -
\inf_{\frac{m_i-1}{M} \leq {\left<\rho_i (n+k)\right>} < \frac{m_i}{M}} | \sum^{\mu}_{i=1} a_i e^{j 2 \pi \left( \frac{s_i'}{s}+\left<\mathbf{q_i'},\mathbf{\rho'}\right>(n+k) \right) } |  \nonumber \\
&\leq \frac{8 \pi \Gamma}{M} \sum^{\mu}_{i=1}  \sum^{r}_{j=1} |q_{i,j}|. \nonumber
\end{align}

By selecting $M$ such that $\frac{8 \pi \Gamma}{M} \sum^{\mu}_{i=1}  \sum^{r}_{j=1} |q_{i,j}| \leq \epsilon$, \eqref{eqn:dis:geo1:3} is upper bounded by
\begin{align}
&\eqref{eqn:dis:geo1:3} \leq
\sum^N_{n=1} \sum_{1 \leq m_1 \leq M, \cdots, 1 \leq m_r \leq M} \sum_{0 \leq s' < s} \mathbb{P} \{  | \sum^{\mu}_{i=1} a_i e^{j 2 \pi \left( \frac{s'}{s}  + \left<\mathbf{q_i'},\mathbf{X_{m_1,\cdots,m_r}}\right> \right) } | < 2 \epsilon \} \cdot \mathbf{1}\{  n \in S_{m_1, \cdots, m_r} \}. \label{eqn:dis:geo1:4}
\end{align}
Since $(\left<\rho_1 n\right>,\cdots,\left<\rho_k n\right> )$ satisfies Weyl's criterion, by Theorem~\ref{thm:weyl}
\begin{align}
&\lim_{N \rightarrow \infty} \sup_{k \in \mathbb{Z}} \frac{1}{N}  \sum^N_{n=1} \mathbf{1}\{   n \in S_{m_1, \cdots, m_r} \} = \frac{1}{M^r}. \label{eqn:dis:geo1:5}
\end{align}
Therefore, if we let $\mathbf{X}$ be a $1 \times r$ random vector whose distribution is uniform on $[0,1)^r$, by \eqref{eqn:dis:geo1:4} and \eqref{eqn:dis:geo1:5}
\begin{align}
&\lim_{N \rightarrow \infty} \sup_{|a_{m}| \geq \gamma, |a_{i}| \leq \Gamma, k \in \mathbb{Z}} \frac{1}{N} \sum^N_{n=1} \mathbf{1} \{ | \sum^{\mu}_{i=1} a_{i}e^{j2 \pi \omega_i(n+k)}| < \epsilon \} \nonumber \\
&\leq \sup_{|a_{m}| \geq \gamma, |a_{i}| \leq \Gamma, k \in \mathbb{Z}} \sum_{1 \leq m_1 \leq M, \cdots, 1 \leq m_r \leq M} \sum_{0 \leq s' < s} \mathbb{P} \{  | \sum^{\mu}_{i=1} a_i e^{j 2 \pi \left( \frac{s'}{s}  + \left<\mathbf{q_i'},\mathbf{X_{m_1,\cdots,m_r}}\right> \right) } | < 2 \epsilon \} \cdot \frac{1}{M^r}
\\
&\leq \sup_{|a_{m}| \geq \gamma, |a_{i}| \leq \Gamma, k \in \mathbb{Z}} \sum_{0 \leq s' < s}
\mathbb{P} \{  | \sum^{\mu}_{i=1} a_i e^{j 2 \pi \left( \frac{s'}{s}  + \left<\mathbf{q_i'},\mathbf{X}\right> \right) } | < 2 \epsilon \}
(\because \mbox{definitions of } \mathbf{X_{m_1, \cdots, m_r}}, \mathbf{X})
\nonumber \\
&\leq \sup_{|a_{m}| \geq \gamma, |a_{i}| \leq \Gamma} s \cdot
\mathbb{P} \{  | \sum^{\mu}_{i=1} a_i e^{j 2 \pi \left(  \left<\mathbf{q_i'},\mathbf{X}\right> \right) } | < 2 \epsilon \}.
(\because e^{j 2 \pi \frac{s'}{s}} \mbox{ only rotates the phase.})
 \label{eqn:dis:geo1:6}
\end{align}
Since $\mathbf{q_i'}$ are distinct, by Lemma~\ref{lem:dis:geo1}, \eqref{eqn:dis:geo1:6} goes to 0 as $\epsilon \downarrow 0$.
\end{proof}

So far, we put the restriction that $|a_i| \leq \Gamma$. However, the functions are growing as $|a_i|$ increases. Therefore, Lemma~\ref{lem:dis:geo2} holds even after we remove such restrictions. The proof is similar to that of Lemma~\ref{lem:singleun}.

\begin{lemma}
Let $\omega_1,\omega_2,\cdots,\omega_{\mu}$ be real numbers such that $\omega_i - \omega_j \notin \mathbb{Q}$ for all $i \neq j$. Then, for strictly positive numbers $\gamma$, and any $m \in \{ 1, \cdots, \mu\}$
\begin{align}
\lim_{\epsilon \downarrow 0}\lim_{N \rightarrow \infty} \sup_{|a_m| \geq \gamma, a_i \in \mathbb{C}, k \in \mathbb{Z}} \frac{1}{N} \sum^N_{n=1} \mathbf{1} \{ | \sum^{\mu}_{i=1} a_{i}e^{j2 \pi \omega_i(n+k)}| < \epsilon \}\rightarrow 0. \nonumber
\end{align}
\label{lem:dis:geo3}
\end{lemma}
\begin{proof}
The proof is by induction on $\mu$, the number of terms in the inner sum.

(i) When $\mu=1$.

Denote $a'_1$ as $\gamma \frac{a_1}{|a_1|}$. Then,
\begin{align}
&\lim_{N \rightarrow \infty} \sup_{|a_1| \geq \gamma,  k \in \mathbb{Z}} \frac{1}{N} \sum^N_{n=1} \mathbf{1} \{ |  a_{1}e^{j2 \pi \omega_1(n+k)}| < \epsilon \}\label{eqn:dis:geo3:1} \\
&=\lim_{N \rightarrow \infty} \sup_{|a_1| \geq \gamma,  k \in \mathbb{Z}} \frac{1}{N} \sum^N_{n=1} \mathbf{1} \{ | \frac{\gamma}{|a_1|}a_{1}e^{j2 \pi \omega_1(n+k)}| < \frac{\gamma}{|a_1|} \epsilon \}\nonumber \\
&\leq \lim_{N \rightarrow \infty} \sup_{|a'_1| = 1,  k \in \mathbb{Z}} \frac{1}{N} \sum^N_{n=1} \mathbf{1} \{ | a'_{1}e^{j2 \pi \omega_1(n+k)}| < \epsilon \} (\because \frac{\gamma}{|a_1|} \leq 1 )\label{eqn:dis:geo3:2}
\end{align}
By Lemma~\ref{lem:dis:geo2}, \eqref{eqn:dis:geo3:2} converges to $0$ as $\epsilon \downarrow 0$. Thus, \eqref{eqn:dis:geo3:1} converges to $0$ as $\epsilon \downarrow 0$.

(ii) As an induction hypothesis, we assume the lemma is true until $\mu-1$.

To prove the lemma for $\mu$, it is enough to show that for all $\delta > 0$ there exists $\epsilon(\delta)>0$ such that
\begin{align}
&\lim_{N \rightarrow \infty} \sup_{|a_m| \geq \gamma,  k \in \mathbb{Z}} \frac{1}{N} \sum^N_{n=1} \mathbf{1} \{ | \sum^{\mu}_{i=1} a_{i}e^{j2 \pi \omega_i(n+k)}| < \epsilon(\delta) \} < \delta. \nonumber
\end{align}

By the induction hypothesis, for all $m' \neq m$ we can find $\epsilon_{m'}(\delta) > 0$ such that
\begin{align}
&\lim_{N \rightarrow \infty} \sup_{|a_{m'}| \geq \gamma,  k \in \mathbb{Z}} \frac{1}{N} \sum^N_{n=1} \mathbf{1} \{ | \sum_{1 \leq i \leq \mu, i \neq m} a_{i}e^{j2 \pi \omega_i(n+k)}| < \epsilon_{m'}(\delta) \} < \delta. \label{eqn:limmax3}
\end{align}

Let $\kappa(\delta):=\min \left\{ \min_{m'\neq m} \left\{\frac{\epsilon_{m'}(\delta)}{2 \gamma } \right\},1 \right\}$. By Lemma~\ref{lem:dis:geo2}, there exists $\epsilon'(\delta)>0$  such that
\begin{align}
&\lim_{N \rightarrow \infty} \sup_{|a_m| \geq \gamma, |a_{i}| \leq \frac{\gamma}{\kappa(\delta)},  k \in \mathbb{Z}} \frac{1}{N} \sum^N_{n=1} \mathbf{1} \{ | \sum^{\mu}_{i=1} a_{i}e^{j2 \pi \omega_i(n+k)}| < \epsilon'(\delta) \} < \delta. \label{eqn:limmax2}
\end{align}
Set $\epsilon(\delta):= \min \left\{ \epsilon'(\delta) , \min_{m' \neq m} \left\{ \frac{\epsilon_{m'}(\delta)}{2} \right\} \right\}$. Then, we have
\begin{align}
&\lim_{N \rightarrow \infty} \sup_{|a_m| \geq \gamma,  k \in \mathbb{Z}} \frac{1}{N} \sum^N_{n=1} \mathbf{1} \{ | \sum^{\mu}_{i=1} a_{i}e^{j2 \pi \omega_i(n+k)}| < \epsilon(\delta) \}  \nonumber \\
&\leq
\lim_{N \rightarrow \infty} \max \{ \sup_{|a_m| \geq \gamma, \frac{|a_i|}{|a_m|} \leq \frac{1}{\kappa(\delta)},  k \in \mathbb{Z}} \frac{1}{N} \sum^N_{n=1} \mathbf{1} \{ | \sum^{\mu}_{i=1} a_{i}e^{j2 \pi \omega_i(n+k)}| < \epsilon({\delta}) \}, \nonumber \\
&\max_{m' \neq m}  \sup_{|a_m| \geq \gamma, \frac{|a_{m'}|}{|a_m|} \geq \frac{1}{\kappa(\delta)},  k \in \mathbb{Z}} \frac{1}{N} \sum^N_{n=1} \mathbf{1} \{ | \sum^{\mu}_{i=1} a_{i}e^{j2 \pi \omega_i(n+k)}| < \epsilon(\delta) \} \} \nonumber \\
&=\max \{ \lim_{N \rightarrow \infty} \sup_{|a_m| \geq \gamma, \frac{|a_i|}{|a_m|} \leq \frac{1}{\kappa(\delta)},  k \in \mathbb{Z}} \frac{1}{N} \sum^N_{n=1} \mathbf{1} \{ | \sum^{\mu}_{i=1} a_{i}e^{j2 \pi \omega_i(n+k)}| < \epsilon({\delta}) \}, \nonumber \\
&\max_{m' \neq m} \lim_{N \rightarrow \infty}  \sup_{|a_m| \geq \gamma, \frac{|a_{m'}|}{|a_m|} \geq \frac{1}{\kappa(\delta)},  k \in \mathbb{Z}} \frac{1}{N} \sum^N_{n=1} \mathbf{1} \{ | \sum^{\mu}_{i=1} a_{i}e^{j2 \pi \omega_i(n+k)}| < \epsilon(\delta) \} \}. \label{eqn:limmax1}
\end{align}

Let $a'_i := \frac{\gamma}{|a_m|} a_i$. Then, the first term in \eqref{eqn:limmax1} is upper bounded by
\begin{align}
&\lim_{N \rightarrow \infty}  \sup_{|a_m| \geq \gamma, \frac{|a_i|}{|a_m|} \leq \frac{1}{\kappa(\delta)},  k \in \mathbb{Z}} \frac{1}{N} \sum^N_{n=1} \mathbf{1} \{ | \sum^{\mu}_{i=1} a_{i}e^{j2 \pi \omega_i(n+k)}| < \epsilon({\delta}) \} \nonumber \\
&=\lim_{N \rightarrow \infty}  \sup_{|a_m| \geq \gamma, \frac{|a_i|}{|a_m|} \leq \frac{1}{\kappa(\delta)},  k \in \mathbb{Z}} \frac{1}{N} \sum^N_{n=1} \mathbf{1} \{ | \sum^{\mu}_{i=1} \frac{\gamma}{|a_m|} a_{i}e^{j2 \pi \omega_i(n+k)}| < \frac{\gamma}{|a_m|} \epsilon({\delta}) \} \nonumber \\
&=\lim_{N \rightarrow \infty}  \sup_{|a'_m| = \gamma, |a'_i| \leq \frac{\gamma}{\kappa(\delta)},  k \in \mathbb{Z}} \frac{1}{N} \sum^N_{n=1} \mathbf{1} \{ | \sum^{\mu}_{i=1} a'_{i}e^{j2 \pi \omega_i(n+k)}| < \frac{\gamma}{|a_m|} \epsilon({\delta}) \} \nonumber \\
&\leq \lim_{N \rightarrow \infty}  \sup_{|a'_m| = \gamma, |a'_i| \leq \frac{\gamma}{\kappa(\delta)},  k \in \mathbb{Z}} \frac{1}{N} \sum^N_{n=1} \mathbf{1} \{ | \sum^{\mu}_{i=1} a'_{i}e^{j2 \pi \omega_i(n+k)}| < \epsilon({\delta}) \} (\because \frac{\gamma}{|a_m|} \leq 1) \nonumber \\
&\leq \lim_{N \rightarrow \infty}  \sup_{|a'_m| = \gamma, |a'_i| \leq \frac{\gamma}{\kappa(\delta)},  k \in \mathbb{Z}} \frac{1}{N} \sum^N_{n=1} \mathbf{1} \{ | \sum^{\mu}_{i=1} a'_{i}e^{j2 \pi \omega_i(n+k)}| < \epsilon'(\delta) \} (\because \epsilon' \geq \epsilon)\nonumber \\
&< \delta. (\because \eqref{eqn:limmax2}) \label{eqn:limmax5}
\end{align}

Let $a''_i := \frac{\gamma}{|a_{m'}|} a_i$. Then, the second term in \eqref{eqn:limmax1} is upper bounded by
\begin{align}
&\lim_{N \rightarrow \infty}  \sup_{|a_m| \geq \gamma, \frac{|a_{m'}|}{|a_m|} \geq \frac{1}{\kappa(\delta)},  k \in \mathbb{Z}} \frac{1}{N} \sum^N_{n=1} \mathbf{1} \{ | \sum^{\mu}_{i=1} a_{i}e^{j2 \pi \omega_i(n+k)}| < \epsilon({\delta}) \} \nonumber \\
&=\lim_{N \rightarrow \infty}  \sup_{|a_m| \geq \gamma, \frac{|a_{m'}|}{|a_m|} \geq \frac{1}{\kappa(\delta)},  k \in \mathbb{Z}} \frac{1}{N} \sum^N_{n=1} \mathbf{1} \{ | \sum^{\mu}_{i=1} \frac{\gamma}{|a_{m'}|} a_{i}e^{j2 \pi \omega_i(n+k)}| < \frac{\gamma}{|a_{m'}|} \epsilon({\delta}) \} \nonumber \\
&\leq \lim_{N \rightarrow \infty}  \sup_{|a_m| \geq \gamma, \frac{|a_{m'}|}{|a_m|} \geq \frac{1}{\kappa(\delta)},  k \in \mathbb{Z}} \frac{1}{N} \sum^N_{n=1} \mathbf{1} \{ | \sum^{\mu}_{i=1} \frac{\gamma}{|a_{m'}|} a_{i}e^{j2 \pi \omega_i(n+k)} - \frac{\gamma}{|a_{m'}|} a_m e^{j 2 \pi \omega_m(n+k)} | < \frac{\gamma}{|a_{m'}|} \epsilon({\delta})+ \frac{\gamma}{|a_{m'}|} | a_m | \} \nonumber \\
&\leq \lim_{N \rightarrow \infty}  \sup_{|a_m| \geq \gamma, \frac{|a_{m'}|}{|a_m|} \geq \frac{1}{\kappa(\delta)},  k \in \mathbb{Z}} \frac{1}{N} \sum^N_{n=1} \mathbf{1} \{ | \sum^{\mu}_{i=1} \frac{\gamma}{|a_{m'}|} a_{i}e^{j2 \pi \omega_i(n+k)} - \frac{\gamma}{|a_{m'}|} a_m e^{j 2 \pi \omega_m(n+k)} | <
\epsilon_{m'}(\delta) \} \label{eqn:limmax4} \\
&\leq \lim_{N \rightarrow \infty}  \sup_{|a_{m'}''| = \gamma,  k \in \mathbb{Z}} \frac{1}{N} \sum^N_{n=1} \mathbf{1} \{ | \sum_{1 \leq i \leq \mu, i \neq m} a_i'' e^{j 2 \pi \omega_i(n+k)} | < \epsilon_{m'}(\delta) \} (\because \mbox{definition of }a_i'') \nonumber \\
&< \delta.  (\because \eqref{eqn:limmax3})  \label{eqn:limmax6}
\end{align}
Here, \eqref{eqn:limmax4} is justified as follows:
\begin{align}
&\frac{\gamma}{|a_m'|} \epsilon(\delta) + \frac{\gamma}{|a_m'|}|a_m| \\
&\leq \frac{\gamma}{|a_m|} \epsilon(\delta) + \gamma \kappa(\delta) (\because \frac{|a_{m'}|}{|a_m|} \geq \frac{1}{\kappa(\delta)} \mbox{, and by definition } \kappa(\delta) \leq 1)\\
&\leq \epsilon(\delta) + \gamma \kappa(\delta) (\because |a_m| \geq \gamma)\\
&\leq \frac{\epsilon_{m'}(\delta)}{2} + \frac{\epsilon_{m'}(\delta)}{2}. (\because \mbox{definitions of }\epsilon(\delta), \kappa(\delta))
\end{align}

Therefore, by plugging \eqref{eqn:limmax5} and \eqref{eqn:limmax6} into \eqref{eqn:limmax1},we get
\begin{align}
&\lim_{N \rightarrow \infty} \sup_{|a_m| \geq \gamma,  k \in \mathbb{Z}} \frac{1}{N} \sum^N_{n=1} \mathbf{1} \{ | \sum^{\mu}_{i=1} a_{i}e^{j2 \pi \omega_i(n+k)}| < \epsilon(\delta) \} < \delta, \nonumber
\end{align}
which finishes the proof.
\end{proof}

Now, we will generalize Lemma~\ref{lem:dis:geo3} by introducing polynomial terms. First, we prove that  a  set of polynomials is uniformly bounded away from $0$ when there is nonzero coefficient.

\begin{lemma}
For all $n\in \mathbb{N}$, $n' \in \mathbb{Z}^+$, $m \in \{1,\cdots,n \}$, $\gamma>0$ and $k > 0$,
\begin{align}
\lim_{T \rightarrow \infty} \sup_{|a_m| \geq \gamma, a_i \in \mathbb{C}}
\frac{| \{ x \in (0,T] : |\sum^{n}_{i=-n'} a_i x^i | < k \} |_{\mathbb{L}}}{T}=0 \nonumber
\end{align}
where $| \cdot |_{\mathbb{L}}$ is the Lebesgue measure of the set.
\label{lem:dis:leb}
\end{lemma}
\begin{proof}
Let $X$ be a uniform random variable on $(0,1]$. Then, we have
\begin{align}
&\sup_{|a_m| \geq \gamma} \frac{|\{ x \in (0,T] : | \sum^{n}_{i=-n'} a_i x^i | < k \}|_{\mathbb{L}}}{T} \nonumber \\
&=\sup_{|a_m| \geq \gamma} \frac{|\{ x \in (0,T] : | \sum^{n}_{i=-n'} a_i \frac{x^i}{T^m} | < \frac{k}{T^m} \}|_{\mathbb{L}}}{T} \nonumber \\
&=\sup_{|a_m| \geq \gamma} \frac{|\{ x \in (0,T] : | \sum^{n}_{i=-n'} a_i \left(\frac{x}{T}\right)^i | < \frac{k}{T^m} \}|_{\mathbb{L}}}{T} \nonumber \\
&=\sup_{|a_m| \geq \gamma} | \{ x \in (0,1] : |\sum^{n}_{i=-n'}a_i x^i | < \frac{k}{T^m} \} |_{\mathbb{L}} \nonumber \\
&=\sup_{|a_m| \geq \gamma} \mathbb{P} \{ |\sum^n_{i=-n'} a_i X^i | < \frac{k}{T^m} \} \nonumber\\
&=\sup_{|a_{m+n'}| \geq \gamma} \mathbb{P} \{ |\sum^{n+n'}_{i=0} a_i X^i | < \frac{k X^{n'}}{T^m} \} \nonumber\\
&\leq \sup_{|a_{m+n'}| \geq \gamma} \mathbb{P} \{ |\sum^{n+n'}_{i=0} a_i X^i | < \frac{k }{T^m} \}. (\because 0 < X \leq 1 \mbox{ w.p. 1}) \nonumber\\
\end{align}

Therefore, by Lemma~\ref{lem:singleun}
\begin{align}
&\lim_{T \rightarrow \infty} \sup_{|a_m| \geq \gamma, a_i \in \mathbb{C}}
\frac{| \{ x \in [0,T] : |\sum^{n}_{i=-n'} a_i x^i | < k \} |_{\mathbb{L}}}{T} \nonumber \\
&=\lim_{T \rightarrow \infty} \sup_{|a_{m+n'}| \geq \gamma, a_i \in \mathbb{C}}
\mathbb{P} \{ |\sum^{n+n'}_{i=0} a_i X^i | < \frac{k}{T^m} \}=0, \nonumber
\end{align}
which finishes the proof.
\end{proof}

The following lemma shows that the above lemma still holds even if we change Lebesgue measure to counting measure.
\begin{lemma}
For all $n \in \mathbb{N}$, $n' \in \mathbb{Z}^+$, $m \in \{1,\cdots,n \}$, $\gamma>0$ and $k > 0$,
\begin{align}
\lim_{N \rightarrow \infty} \sup_{|a_m| \geq \gamma, a_i \in \mathbb{C}}
\frac{| \{ x \in \{1,\cdots,N\} : |\sum^{n}_{i=-n'} a_i x^i | < k \} |_{\mathbb{C}}}{N}=0 \nonumber
\end{align}
where $| \cdot |_{\mathbb{C}}$ implies the counting measure of the set, the cardinality of the set.
\label{lem:dis:cnt}
\end{lemma}
\begin{proof}
First, we will prove the following claim which relates Lebesgue measure with counting measure. 
\begin{claim}
Let $f:\mathbb{R^+} \rightarrow \mathbb{R}$ be a $\mathcal{C}^{\infty}$ function with $l$ local maxima and minima. Then,
\begin{align}
\left| \left\{x \in [1,N]  :  f(x) > 0 \right\} \right|_{\mathbb{L}} \leq \left|\left\{x \in \{1,\cdots, N \} : f(x) > 0 \right\}\right|_{\mathbb{C}}+ 3l +2. \nonumber
\end{align}
\end{claim}
\begin{proof}
Since $f(x)$ is a continuous function with $l$ local maxima and minima, we can prove that there exist $l' \leq l+1$, $s_i$ and $t_i$ $(1 \leq i \leq l')$ such that
\begin{align}
\left\{x \in \{1,\cdots, N \} : f(x) > 0 \right\}  = \{ s_1 , s_1+1 ,\cdots , s_1+t_1 \} \cup \cdots \cup \{ s_{l'} ,s_{l'}+1 , \cdots,  s_{l'}+t_{l'} \}.  \nonumber
\end{align}
One way to justify this is by contradiction, i.e. if we assume $l' > l+1$, there should exist more than $l$ local maxima and minima by the mean value theorem. Moreover, since the number of local maxima and minima is bounded by $l$, we have
\begin{align}
\left| \left\{x \in [1,N]  :  f(x) > 0 \right\} \right|_{\mathbb{L}} & \leq \left|[s_1-1,s_1+t_1+1]\right|_{\mathbb{L}}+  \cdots + \left|[s_{l'}-1,s_{l'}+t_{l'}+1]\right|_{\mathbb{L}} + l \nonumber \\
&\leq (t_1+2) + \cdots + (t_{l'}+2)+l \nonumber \\
&\leq  \left|\left\{x \in \{1,\cdots, N \} : f(x) > 0 \right\}\right|_{\mathbb{C}} +2l'+l \nonumber \\
&\leq \left|\left\{x \in \{1,\cdots, N \} : f(x) > 0 \right\}\right|_{\mathbb{C}}+ 3l +2. \nonumber
\end{align}
Thus, the claim is true.
\end{proof}
%

To prove the lemma, let $a_i=a_{R,i}+j a_{I,i}$ where $a_{R,i}, a_{I,i} \in \mathbb{R}$. Then,
\begin{align}
&|\sum^n_{i=-n'}a_i x^i| < k\\
& (\Leftrightarrow) |\sum^{n+n'}_{i=0}a_{i-n'} x^i| < k x^{n'} \\
& (\Leftrightarrow) (\sum^{n+n'}_{i=0}a_{R,i-n'} x^i)^2 + (\sum^{n+n'}_{i=0}a_{I,i-n'} x^i)^2 < k^2 x^{n'}.
\end{align}

Since $k^2 x^{2n'} - (\sum^{n+n'}_{i=0}a_{R,i-n'} x^i)^2 - (\sum^{n+n'}_{i=0}a_{I,i-n'} x^i)^2$ is a continuous function with at most $2(n+n')$ local maxima and minima, by the claim we have
\begin{align}
&\lim_{N \rightarrow \infty} \sup_{|a_m| \geq \gamma, a_i \in \mathbb{C}}
\frac{| \{ x \in \{1,\cdots,N\} : |\sum^{n}_{i=0} a_i x^i | < k \} |_{\mathbb{C}}}{N} \nonumber \\
& \leq \lim_{N \rightarrow \infty} \sup_{|a_m| \geq \gamma, a_i \in \mathbb{C}}
\frac{|\{ x \in [1,N] : |\sum^{n}_{i=0} a_i x^i | < k \}|_{\mathbb{L}}+6(n+n')+2 }{N} \nonumber \\
& = \lim_{N \rightarrow \infty} \sup_{|a_m| \geq \gamma, a_i \in \mathbb{C}}
\frac{|\{ x \in (0,N] : |\sum^{n}_{i=0} a_i x^i | < k \}|_{\mathbb{L}}}{N}=0\ (\because Lemma~\ref{lem:dis:leb}) \nonumber
\end{align}
Therefore, the lemma is proved.
\end{proof}

Now, we merge Lemma~\ref{lem:dis:cnt} with Lemma~\ref{lem:dis:geo3} to prove that Lemma~\ref{lem:dis:geofinal} still holds even after we introduce polynomial terms to the functions.

\begin{lemma}
Let $\omega_1,\omega_2,\cdots,\omega_{\mu}$ be real numbers such that $\omega_i - \omega_j \notin \mathbb{Q}$ for all $i \neq j$. Then, for strictly positive numbers $\gamma$,
\begin{align}
\lim_{\epsilon \downarrow 0}\lim_{N \rightarrow \infty} \sup_{|a_{1\nu_1}| \geq \gamma, a_{ij} \in \mathbb{C},k \in \mathbb{Z}} \frac{1}{N} \sum^N_{n=1} \mathbf{1} \left\{ \left| \sum^{\mu}_{i=1} \left( \sum^{\nu_i}_{j=0} a_{ij}n^j \right) e^{j2 \pi \omega_i (n+k)} \right| < \epsilon \right\}\rightarrow 0. \nonumber
\end{align}
\label{lem:dis:geofinal}
\end{lemma}
\begin{proof}
To prove the lemma, it is enough to show that for all $\delta > 0$, there exist $\epsilon > 0$ and $N \in \mathbb{N}$ such that
\begin{align}
\sup_{|a_{1\nu_1}| \geq \gamma, a_{ij} \in \mathbb{C},k \in \mathbb{Z}} \frac{1}{N} \sum^{N}_{n=1} \mathbf{1} \left\{ \left| \sum^{\mu}_{i=1} \left( \sum^{\nu_i}_{j=0} a_{ij}n^j \right) e^{j2 \pi \omega_i (n+k)} \right| < \epsilon \right\} < \delta. \label{eqn:dis:target}
\end{align}

Since $\mu$ is finite, by Lemma~\ref{lem:dis:geo3}, there exist $\epsilon' > 0$ and $M \in \mathbb{N}$ such that
\begin{align}
\max_{d \in \{1, \cdots, \mu \}} \left(
\sup_{k \in \mathbb{Z}, a_i \in \mathbb{C}, |a_d| \geq 1}
\frac{1}{M} \sum_{c=1}^{M}
\mathbf{1} \left\{ \left| \sum^{\mu}_{i=1} a_i e^{j 2 \pi \omega_i (c+k)}  \right| < \epsilon'
\right\}
\right) < \frac{\delta}{2}. \label{eqn:dis:target2}
\end{align}

By Lemma~\ref{lem:dis:cnt}, there exists $B' \in \mathbb{N}$ such that
\begin{align}
\sup_{|a_{1\nu_1}'| \geq \gamma} \frac{
\left|
\left\{ b \in \{1,\cdots,B' \} :
\left|\sum^{\nu_1}_{j=0} a_{1j}' b^j \right| \leq 2
\right\}
\right|_{\mathbb{C}}
}{B'} < \frac{\delta}{4}.  \label{eqn:dis:target3}
\end{align}

Define $\kappa' := \frac{2 \sum^{\mu}_{i=1} \sum^{\nu_i}_{j=1}   \sum^j_{k=1} {{j}\choose{k}} }{\epsilon'}$.
By Lemma~\ref{lem:dis:cnt}, there exists $B'' \in \mathbb{N}$ such that
\begin{align}
\sum_{1 \leq i \leq \mu, 1 \leq j \leq \nu_i, 1 \leq k \leq j} \sup_{|a_{k}|=1} \frac{| \{ b \in \{1,\cdots,B'' \} : \kappa'  \geq |\sum^{\nu_i-j+k}_{j'=-j+k} a_{j'}b^{j'} |\}|_{\mathbb{C}}}{B''} < \frac{\delta}{4}. \label{eqn:dis:target4}
\end{align}

Define $B:=\max(B', B'')$. We will show that the choice of $\epsilon=\epsilon'$ and $N=M \cdot B$ satisfies \eqref{eqn:dis:target}.
\begin{align}
&\sup_{|a_{1\nu_1}| \geq \gamma, a_{ij} \in \mathbb{C}, k \in \mathbb{Z}} \frac{1}{N} \sum^N_{n=1}
\mathbf{1} \left\{ \left| \sum^{\mu}_{i=1} \left( \sum^{\nu_i}_{j=0} a_{ij}n^j \right) e^{j2 \pi \omega_i (n+k)} \right| < \epsilon \right\} \nonumber \\
&=\sup_{|a_{1\nu_1}| \geq \gamma, a_{ij} \in \mathbb{C}} \frac{1}{N} \sum^{N}_{n=1} \mathbf{1} \left\{ \left| \sum^{\mu}_{i=1}\left( \sum^{\nu_i}_{j=0} a_{ij}n^j \right)e^{j 2 \pi \omega_i n}  \right| < \epsilon \right\} 
(\because e^{j 2\pi \omega_i k} \mbox{ can be absorbed into the $a_{ij}$.})
\nonumber \\
&=\sup_{|a_{1\nu_1}| \geq \gamma, a_{ij} \in \mathbb{C}} \frac{1}{B \cdot M} \sum^{B-1}_{b=0} \sum^{M}_{c=1} \mathbf{1} \left\{ \left| \sum^{\mu}_{i=1}\left( \sum^{\nu_i}_{j=0} a_{ij}(bM+c)^j \right)e^{j 2 \pi \omega_i (bM+c)}  \right| < \epsilon \right\} (\because n \mbox{ is rewritten as $bM+c$.}) \nonumber \\
&=\sup_{|a_{1\nu_1}| \geq \gamma, a_{ij} \in \mathbb{C}} \frac{1}{B \cdot M} \sum^{B-1}_{b=0} \sum^{M}_{c=1} \mathbf{1} \left\{ \left| \sum^{\mu}_{i=1}\left( \sum^{\nu_i}_{j=0} a_{ij}\left((bM)^j+\sum^j_{k=1} {{j}\choose{k}}(bM)^{j-k}c^k \right) \right)e^{j 2 \pi \omega_i (bM+c)}  \right| < \epsilon \right\} \nonumber \\
&\leq \sup_{|a_{1\nu_1}| \geq \gamma, a_{ij} \in \mathbb{C}} \frac{1}{B \cdot M} \sum^{B-1}_{b=0} \sum^{M}_{c=1} \mathbf{1} \left\{ \left| \sum^{\mu}_{i=1}\left( \sum^{\nu_i}_{j=0} a_{ij} (bM)^j \right)e^{j 2 \pi \omega_i (bM+c)}  \right| < \epsilon
+\sum^{\mu}_{i=1} \sum^{\nu_i}_{j=1} |a_{ij}|  \sum^j_{k=1} {{j}\choose{k}}(bM)^{j-k}c^k
\right\} \nonumber \\
&\leq \sup_{|a_{1\nu_1}| \geq \gamma, a_{ij} \in \mathbb{C}} \frac{1}{B \cdot M} \sum^{B-1}_{b=0} \sum^{M}_{c=1} \mathbf{1} \left\{ \left| \sum^{\mu}_{i=1}\left( \sum^{\nu_i}_{j=0} a_{ij} (bM)^j \right)e^{j 2 \pi \omega_i (bM+c)}  \right| < \epsilon
+\sum^{\mu}_{i=1} \sum^{\nu_i}_{j=1}  \sum^j_{k=1} |a_{ij}|  {{j}\choose{k}}(bM)^{j-k}M^k
\right\} \label{eqn:geo:identity} \\
&\leq \sup_{|a_{1\nu_1}| \geq \gamma, a_{ij} \in \mathbb{C}} \frac{1}{B \cdot M} \sum^{B-1}_{b=0} \sum^{M}_{c=1} \mathbf{1} \Bigg\{ \left| \sum^{\mu}_{i=1}\left( \frac{\sum^{\nu_i}_{j=0} a_{ij} (bM)^j}{M_b} \right)e^{j 2 \pi \omega_i (bM+c)}  \right| < \nonumber \\
& \frac{\epsilon}{M_b}
+\frac{\sum^{\mu}_{i=1} \sum^{\nu_i}_{j=1}   \sum^j_{k=1} |a_{ij}| {{j}\choose{k}}(bM)^{j-k}M^k}{M_b}
\Bigg\}\nonumber\\
&\leq \sup_{|a_{1\nu_1}| \geq \gamma, a_{ij} \in \mathbb{C}} \frac{1}{B} \sum^{B-1}_{b=0} \Bigg\{ \frac{1}{M} \sum^{M}_{c=1} \mathbf{1} \Bigg\{ \left| \sum^{\mu}_{i=1}\left( \frac{\sum^{\nu_i}_{j=0} a_{ij} (bM)^j}{M_b} \right)e^{j 2 \pi \omega_i (bM+c)}  \right| < \epsilon \Bigg\} \nonumber \\
&+\mathbf{1}\Bigg\{ \frac{\epsilon}{M_b}
+\frac{\sum^{\mu}_{i=1} \sum^{\nu_i}_{j=1}   \sum^j_{k=1} |a_{ij}| {{j}\choose{k}}(bM)^{j-k}M^k}{M_b} \geq \epsilon  \Bigg\} \Bigg\} \nonumber \\
&\leq \sup_{|a_{1\nu_1}| \geq \gamma, a_{ij} \in \mathbb{C}} \frac{1}{B} \sum^{B-1}_{b=0} \Bigg\{ \frac{1}{M} \sum^{M}_{c=1} \mathbf{1} \Bigg\{ \left| \sum^{\mu}_{i=1}\left( \frac{\sum^{\nu_i}_{j=0} a_{ij} (bM)^j}{M_b} \right)e^{j 2 \pi \omega_i (bM+c)}  \right| < \epsilon \Bigg\} \nonumber \\
&+\mathbf{1}\Bigg\{ \frac{\epsilon}{M_b} \geq \frac{\epsilon}{2} \Bigg\} + \mathbf{1} \Bigg\{\frac{\sum^{\mu}_{i=1} \sum^{\nu_i}_{j=1}   \sum^j_{k=1} |a_{ij}| {{j}\choose{k}}(bM)^{j-k}M^k}{M_b} \geq \frac{\epsilon}{2}  \Bigg\} \Bigg\} 
\label{eqn:dis:geofinal:1}
\end{align}
where $M_b:=\max_{i} \left\{ \left| \sum^{\nu_i}_{j=0} a_{ij}\left(bM \right)^j \right| \right\}$ and when $M_b = 0$ the value of the indicator function is set to be $0$ since in this case, the indicator function of \eqref{eqn:geo:identity} is already $0$. 

First, let's prove that the first term of \eqref{eqn:dis:geofinal:1} is small enough. For all $a_{ij} \in \mathbb{C}$ such that $|a_{1 \nu_1}| \geq \gamma$ and $b \in \{ 0, \cdots, B\}$,  we have
\begin{align}
&\frac{1}{M} \sum^{M}_{c=1} \mathbf{1} \Bigg\{ \left| \sum^{\mu}_{i=1}\left( \frac{\sum^{\nu_i}_{j=0} a_{ij} (bM)^j}{M_b} \right)e^{j 2 \pi \omega_i (bM+c)}  \right| < \epsilon \Bigg\} \\
&\leq \max_{d \in \{1, \cdots, \mu \}} \left(
\sup_{k \in \mathbb{Z}, a_i \in \mathbb{C}, |a_d| \geq 1}
\frac{1}{M} \sum_{c=1}^{M}
\mathbf{1} \left\{ \left| \sum^{\mu}_{i=1} a_i e^{j \omega_i (c+k)}  \right| < \epsilon
\right\}
\right) \\
&(\because \mbox{By the definition of $M_b$, $\left|\frac{\sum^{\nu_i}_{j=0} a_{ij} (bM)^j}{M_b} \right|=1$ for some $i$})\\
&< \frac{\delta}{2}. (\because \eqref{eqn:dis:target2})\label{eqn:dis:target:100}
\end{align}

Let's prove that the second term of \eqref{eqn:dis:geofinal:1} is small enough. 
\begin{align}
&  \sup_{|a_{1 \nu_1}| \geq \gamma} \frac{\left| \left\{ b \in \{1,\cdots, B \} : M_b < 2 \right\} \right|_{\mathbb{C}}}{B} \nonumber \\
&\leq  \sup_{|a_{1\nu_1}| \geq \gamma} \frac{\left|
\left\{ b \in \{1,\cdots,B \} :
\left|\sum^{\nu_1}_{j=0} a_{1j} (bM)^j \right| < 2
\right\}
\right|_{\mathbb{C}}}{B} (\because \mbox{definition of $M_b$})\nonumber \\
&\leq \sup_{|a_{1\nu_1}'| \geq \gamma} \frac{
\left|
\left\{ b \in \{1,\cdots,B \} :
\left|\sum^{\nu_1}_{j=0} a_{1j}' b^j \right| < 2
\right\}
\right|_{\mathbb{C}}
}{B} (\because \mbox{putting $a_{1j}':=a_{1j}M^j$ and $M$ goes to infinity.})\nonumber \\
&< \frac{\delta}{4}. (\because \eqref{eqn:dis:target3}) \label{eqn:dis:target:101}
\end{align}

Now, we will prove that the third term of \eqref{eqn:dis:geofinal:1} is small enough.  
\begin{align}
&\sup_{|a_{1 \nu_1}| \geq \gamma, a_{ij} \in \mathbb{C}} \frac{1}{B} \sum_{b=0}^{B-1} \mathbf{1} \Bigg\{\frac{\sum^{\mu}_{i=1} \sum^{\nu_i}_{j=1}   \sum^j_{k=1} |a_{ij}| {{j}\choose{k}}(bM)^{j-k}M^k}{M_b} \geq \frac{\epsilon}{2}  \Bigg\} \\
&\leq \sup_{|a_{1 \nu_1}| \geq \gamma, a_{ij} \in \mathbb{C}} \frac{1}{B} \sum_{b=0}^{B-1} \mathbf{1} \Bigg\{
(\sum^{\mu}_{i'=1} \sum^{\nu_{i'}}_{j'=1}   \sum^{j'}_{k'=1}{{j'}\choose{k'}}) \cdot 
\max_{1 \leq i \leq \mu, 1 \leq j \leq \nu_i, 1\leq k \leq j} |a_{ij}| (bM)^{j-k}M^k \geq \frac{\epsilon}{2} {M_b} \Bigg\} \\
&\leq \sum_{1 \leq i \leq \mu, 1 \leq j \leq \nu_i, 1\leq k \leq j} \sup_{|a_{1 \nu_1}| \geq \gamma, a_{ij} \in \mathbb{C}} \frac{1}{B} \sum_{b=0}^{B-1} \mathbf{1} \Bigg\{ \kappa'  |a_{ij}| (bM)^{j-k}M^k \geq  M_b \Bigg\} \\
&\leq \sum_{1 \leq i \leq \mu, 1 \leq j \leq \nu_i, 1\leq k \leq j} \sup_{|a_{1 \nu_1}| \geq \gamma, a_{ij} \in \mathbb{C}} \frac{1}{B} \sum_{b=0}^{B-1} \mathbf{1} \Bigg\{ \kappa'  |a_{ij}| (bM)^{j-k}M^k \geq  |\sum^{\nu_i}_{j'=0} a_{ij'}(bM)^{j'} | \Bigg\} (\because \mbox{definition of $M_b$})\\
&\leq \sum_{1 \leq i \leq \mu, 1 \leq j \leq \nu_i, 1\leq k \leq j} \sup_{|a_{1 \nu_1}| \geq \gamma, a_{ij} \in \mathbb{C}} \frac{1}{B} \sum_{b=0}^{B-1} \mathbf{1} \Bigg\{ \kappa' \geq |\sum^{\nu_i}_{j'=0} \frac{a_{ij'}(bM)^{j'}}{|a_{ij}| b^{j-k} M^k} | \Bigg\} \\
&\leq \sum_{1 \leq i \leq \mu, 1 \leq j \leq \nu_i, 1\leq k \leq j} \sup_{|a_{k}| =1} \frac{1}{B} \sum_{b=0}^{B-1} \mathbf{1} \Bigg\{ \kappa' \geq  |\sum^{\nu_i-j+k}_{j'=-j+k} a_{j'}b^{j'} | \Bigg\}  \\
&< \frac{\delta}{4}. (\because \eqref{eqn:dis:target4})\label{eqn:dis:target:102}
\end{align}

Therefore, by \eqref{eqn:dis:target:100}, \eqref{eqn:dis:target:101}, \eqref{eqn:dis:target:102}, we can see $\eqref{eqn:dis:geofinal:1} < \delta$, which finishes the proof.
\end{proof}

\subsection{Proof of Lemma~\ref{lem:dis:achv}}
\label{sec:app:cycleproof}

In this section, we will merge the properties about the observability Gramian shown in Appendix~\ref{sec:dis:gramian} with the uniform convergence of Appendix~\ref{sec:dis:uniform}, and prove Lemma~\ref{lem:dis:achv} of page~\pageref{lem:dis:achv}.

Just as we did in Appendix~\ref{sec:app:2}, we must first prove the following lemma which tells that the determinant of the observability Gramian is large except a negligible set under a cofactor condition the Gramian matrix. The proof of the lemma is very similar to that of Lemma~\ref{lem:conti:single}.

\begin{lemma}
Let $\mathbf{A}$ and $\mathbf{C}$ be given as \eqref{eqn:ac:jordansingle} and \eqref{eqn:ac:jordansinglec}. 
Define $a_{i,j}$ and $C_{i,j}$ as the $(i,j)$ element and cofactor of $\begin{bmatrix} \mathbf{C}\mathbf{A}^{-k_1}  \\ \vdots \\ \mathbf{C}\mathbf{A}^{-k_{m-1}} \\ \mathbf{C}\mathbf{A}^{-n} \end{bmatrix}$ respectively.
Then, there exists a family of functions $\{ g_{\epsilon} : \epsilon > 0, g_{\epsilon}:\mathbb{R}^+ \rightarrow \mathbb{R}^+ \}$ satisfying:\\
(i) For all $\epsilon>0$, $k_1 < k_2 < \cdots < k_{m-1}$ and $|C_{m,m}| \geq \epsilon \prod_{1 \leq i \leq m-1} \lambda_i^{-k_i}$, the following is true.\\
\begin{align}
\lim_{N \rightarrow \infty} \sup_{k \in \mathbb{Z}, k-k_{m-1} \geq g_{\epsilon}(k_{m-1})}
\frac{1}{N}\sum_{n=k+1}^{k+N}
\mathbf{1} \left\{\left|
\det\left(
\begin{bmatrix}
\mathbf{C} \mathbf{A}^{-k_1} \\
\vdots \\
\mathbf{C} \mathbf{A}^{-k_{m-1}} \\
\mathbf{C} \mathbf{A}^{-n}
\end{bmatrix}
\right) \right|
 < \epsilon^2 \lambda_m^{-n} \prod_{1 \leq i \leq m-1} \lambda_i^{-k_i} \right\} \rightarrow 0
\mbox{ as } \epsilon \downarrow 0. \nonumber
\end{align}
(ii) For each $\epsilon>0$, $g_{\epsilon}(k) \lesssim 1 + \log(k+1)$.
\label{lem:dis:single}
\end{lemma}
\begin{proof}
By Lemma~\ref{lem:dis:det:lower}, we can find a function $g'_{2\epsilon^2}(k)$ such that for all $0 \leq k_1 < k_2 < \cdots < k_{m-1} < n$ satisfying:\\
(i) $n-k_{m-1} \geq g'_{2\epsilon^2}(k_{m-1})$ \\
(ii) $g'_{2\epsilon^2}(k) \lesssim 1 + \log (k+1)$ \\
(iii) $\left| \sum_{m-m_{\mu}+1 \leq i \leq m} a_{m,i}C_{m,i} \right| \geq 2 \epsilon^2 \lambda_m^{-n} \prod_{1 \leq i \leq {m-1}} \lambda_i^{-k_i}$\\
the following inequality holds:
\begin{align}
\left| \det\left(
\begin{bmatrix}
\mathbf{C}\mathbf{A}^{-k_1} \\
\vdots \\
\mathbf{C}\mathbf{A}^{-k_{m-1}} \\
\mathbf{C}\mathbf{A}^{-n} \\
\end{bmatrix}
\right)  \right| & \geq \epsilon^2 \lambda_m^{-n} \prod_{1 \leq i \leq {m-1}} \lambda_i^{-k_i}. \nonumber
\end{align}
Let $g_{\epsilon}(k)$ be $g'_{2\epsilon^2}(k)$. Then, we have
\begin{align}
&\sup_{k \in \mathbb{Z}, k - k_{m-1} \geq g_{\epsilon}(k_{m-1})}
\frac{1}{N} \sum_{n=k+1}^{k+N}
\mathbf{1}
\left\{ \left| \det \left(
\begin{bmatrix}
\mathbf{C} \mathbf{A}^{-k_1} \\
\vdots \\
\mathbf{C} \mathbf{A}^{-k_{m-1}} \\
\mathbf{C} \mathbf{A}^{-n}
\end{bmatrix}
\right) \right|  < \epsilon^2 \lambda_m^{-n} \prod_{1 \leq i \leq m-1} \lambda_i^{-k_i}\right\} \nonumber \\
& \leq \sup_{k \in \mathbb{Z}, k - k_{m-1} \geq g_{\epsilon}(k_{m-1})}
\frac{1}{N} \sum_{n=k+1}^{k+N} \mathbf{1} \left\{
\left|
\sum_{m-m_{\mu}+1 \leq i \leq m} a_{m,i} C_{m,i}
\right| < 2 \epsilon^2 \lambda_m^{-n} \prod_{1 \leq i \leq m-1} \lambda_i^{-k_i}
\right\}  \label{eqn:lem:detail2:0} \\ 
& = \sup_{k \in \mathbb{Z}, k - k_{m-1} \geq g_{\epsilon}(k_{m-1})}
\frac{1}{N} \sum_{n=k+1}^{k+N}
\mathbf{1} \left\{
\left|
\sum_{m-m_{\mu}+1 \leq i \leq m}
\frac{a_{m,i}}{\lambda_m^{-n}} \frac{C_{m,i}}{\epsilon \prod_{1 \leq i \leq m-1} \lambda_i^{-k_i} }
\right| < 2 \epsilon
\right\} \nonumber \\
&\leq \sup_{k \in \mathbb{Z}, |b_m|\geq 1}
\frac{1}{N} \sum_{n=k+1}^{k+N}\mathbf{1}
\left\{ \left| \sum_{m-m_{\mu}+1 \leq i \leq m} b_i \frac{a_{m,i}}{\lambda_m^{-n}} \right| < 2 \epsilon \right\} \label{eqn:lem:detail2:1}
\end{align}
where \eqref{eqn:lem:detail2:0} is by the definition of $g_{\epsilon}(k)$ and Lemma~\ref{lem:dis:det:lower}, and \eqref{eqn:lem:detail2:1} is by  $|C_{m,m}| \geq \epsilon \prod_{1 \leq i \leq m-1} \lambda_i^{-k_i} $.

Let $\mathbf{C_{\mu,\nu_{\mu}}}$ denoted in \eqref{eqn:ac:jordansinglec} be $\begin{bmatrix} c'_{1} & \cdots & c'_{m_{\mu,\nu_{\mu}}} \end{bmatrix}$.\\
Moreover,
\begin{align}
&\mathbf{A_{\mu,\nu_\mu}}^{-n}\nonumber \\
&=\begin{bmatrix}
(\lambda_{\mu,\nu_\mu}e^{j 2 \pi \omega_{\mu,\nu_\mu}})^{-n} & {{-n}\choose{1}} (\lambda_{\mu,\nu_\mu}e^{j 2 \pi \omega_{\mu,\nu_\mu}})^{-n-1} & \cdots &
{{-n}\choose{m_{\mu,\nu_\mu}-1}} (\lambda_{\mu,\nu_\mu}e^{j 2 \pi \omega_{\mu,\nu_\mu}})^{-n-m_{\mu,\nu_\mu}+1} \\
0 & (\lambda_{\mu,\nu_\mu}e^{j 2 \pi \omega_{\mu,\nu_\mu}})^{-n} & \cdots &
{{-n}\choose{m_{\mu,\nu_\mu}-2}} (\lambda_{\mu,\nu_\mu}e^{j 2 \pi \omega_{\mu,\nu_\mu}})^{-n-m_{\mu,\nu_\mu}+2} \\
\vdots & \vdots & \ddots & \vdots \\
0 & 0 & \cdots & (\lambda_{\mu,\nu_\mu}e^{j 2 \pi \omega_{\mu,\nu_\mu}})^{-n}
\end{bmatrix}.\nonumber
\end{align}
Thus, we can see that
\begin{align}
a_{m,m}&=\sum_{1 \leq i \leq m_{\mu,\nu_{\mu}}} c'_i { -n \choose m_{\mu,\nu_\mu}-i }(\lambda_{\mu,\nu_\mu}e^{j 2 \pi \omega_{\mu,\nu_\mu}})^{-n-m_{\mu,\nu_\mu}+i} .\nonumber
\end{align}
Therefore,
\begin{align}
\frac{a_{m,m}}{\lambda_m^{-n}} &=\sum_{1 \leq i \leq m_{\mu,\nu_{\mu}}} c'_{i} { -n \choose m_{\mu,\nu_\mu}-i }\lambda_{\mu,\nu_\mu}^{-m_{\mu,\nu_\mu}+i}(e^{j 2 \pi \omega_{\mu,\nu_\mu}})^{-n-m_{\mu,\nu_\mu}+i}.\nonumber
\end{align}
Moreover, when $a_{m,i}$ is considered as a function of $n$, $n^{m_{\mu,\nu_\mu}-1} e^{-j 2 \pi \omega_{\mu,\nu_\mu} n}$ term is only shown up in $\frac{a_{m,m}}{\lambda_m^{-n}}$
among $\frac{a_{m,m-m_{\mu}+1}}{\lambda_m^{-n}},\cdots ,\frac{a_{m,m}}{\lambda_m^{-n}}$, and the associated coefficient is
$\frac{c_1'(-1)^{m_{\mu,\nu_\mu}-1}}{(m_{\mu,\nu_\mu}-1)!} \lambda_{\mu,\nu_\mu}^{-m_{\mu,\nu_\mu}+1} e^{j 2 \pi \omega_{\mu,\nu_\mu}(-m_{\mu,\nu_\mu}+1)}$.

Let $c':=\frac{|c_1'|}{(m_{\mu,\nu_\mu}-1)!} \lambda_{\mu,\nu_\mu}^{-m_{\mu,\nu_\mu}+1}$. Then, \eqref{eqn:lem:detail2:1} can be upper bounded as follows:
\begin{align}
\eqref{eqn:lem:detail2:1} &\leq \sup_{k \in \mathbb{Z},|a_{\nu_\mu,m_{\mu,\nu_\mu}}| \geq c'} \frac{1}{N}
\sum^{k+N}_{n=k+1} \mathbf{1} \left\{ \left| \sum_{1 \leq i \leq \nu_\mu} \left( \sum_{1 \leq j \leq m_{\mu,i}} a_{ij}n^{j-1} \right)e^{j 2 \pi (-\omega_{\mu,i})n} \right| < 2 \epsilon  \right\} \nonumber \\
&
=\sup_{k \in \mathbb{Z},|a_{\nu_\mu,m_{\mu,\nu_\mu}}| \geq  c'} \frac{1}{N}
\sum^{N}_{n=1} \mathbf{1} \left\{ \left| \sum_{1 \leq i \leq \nu_\mu} \left( \sum_{1 \leq j \leq m_{\mu,i}} a_{ij}(n+k)^{j-1} \right)e^{j 2 \pi (-\omega_{\mu,i})(n+k)} \right| < 2 \epsilon  \right\} \nonumber \\
&
\leq \sup_{k \in \mathbb{Z},|a_{\nu_\mu,m_{\mu,\nu_\mu}}| \geq c'} \frac{1}{N}
\sum^{N}_{n=1} \mathbf{1} \left\{ \left| \sum_{1 \leq i \leq \nu_\mu} \left( \sum_{1 \leq j \leq m_{\mu,i}} a_{ij}n^{j-1} \right)e^{j 2 \pi (-\omega_{\mu,i})(n+k)} \right| < 2 \epsilon  \right\} \label{eqn:lem:detail2:2}
\end{align}
The last inequality comes from the fact that the coefficient of $n^{m_{\mu,\nu_{\mu}}-1}$ is the same for both $\sum_{1 \leq j \leq m_{\mu,\nu_{\mu}}} a_{\nu_{\mu},j}(n+k)^{j-1}$ and $\sum_{1 \leq j \leq m_{\mu,\nu_{\mu}}} a_{\nu_{\mu},j}n^{j-1}$.

By Lemma~\ref{lem:singleun}, we get
\begin{align}
\lim_{N \rightarrow \infty} \sup_{k \in \mathbb{Z},|a_{\nu_\mu,m_{\mu,\nu_\mu}}| \geq c'} \frac{1}{N}
\sum^{N}_{n=1} \mathbf{1} \left\{ \left| \sum_{1 \leq i \leq \nu_\mu} \left( \sum_{1 \leq j \leq m_{\mu,i}} a_{ij}n^{j-1} \right)e^{j 2 \pi (-\omega_{\mu,i})(n+k)} \right| < 2 \epsilon  \right\} \rightarrow 0 \mbox{ as } \epsilon \downarrow 0. \nonumber
\end{align}
Therefore, by \eqref{eqn:lem:detail2:2} we can say that
\begin{align}
\lim_{N \rightarrow \infty} \sup_{k \in \mathbb{Z}, k-k_{m-1} \geq g_{\epsilon}(k_{m-1})}
\frac{1}{N}\sum_{n=k+1}^{k+N}
\mathbf{1} \left\{\left|
\det\left(
\begin{bmatrix}
\mathbf{C} \mathbf{A}^{-k_1} \\
\vdots \\
\mathbf{C} \mathbf{A}^{-k_{m-1}} \\
\mathbf{C} \mathbf{A}^{-n}
\end{bmatrix}
\right) \right|
 < \epsilon^2 \lambda_m^{-n} \prod_{1 \leq i \leq m-1} \lambda_i^{-k_i} \right\} \rightarrow 0
\mbox{ as } \epsilon \downarrow 0 \nonumber
\end{align}
which finishes the proof.
\end{proof}

Based on the previous lemma,  the properties of p.m.f. tails shown in Section~\ref{sec:app:1} and the properties of the observability Gramian discussed in Section~\ref{sec:dis:gramian}, we can prove Lemma~\ref{lem:dis:achv} for the case when the system has no eigenvalue cycles. Moreover, we will prove a lemma with multiple systems. This will turn out to be helpful in proving Lemma~\ref{lem:dis:achv} for the general systems with eigenvalue cycles.

Consider pairs of matrices $(\mathbf{A_1},\mathbf{C_1}),(\mathbf{A_2},\mathbf{C_2}),\cdots,(\mathbf{A_r},\mathbf{C_r})$ defined as follows:
\begin{align}
&\mbox{$\mathbf{A_i}$ is a $m_i \times m_i$ Jordan form matrix and $\mathbf{C_i}$ is a $1 \times m_i$ row vector}  \label{eqn:ac:jordansingle2} \\
&\mbox{Each $\mathbf{A_i}$ has no eigenvalues cycles and $(\mathbf{A_i},\mathbf{C_i})$ is observable} \nonumber \\
&\mbox{$\lambda_j^{(i)} e^{j 2 \pi \omega_j^{(i)}}$ is $(j,j)$ element of $\mathbf{A_i}$} \nonumber \\
&\lambda_1^{(i)} \geq \lambda_2^{(i)} \geq \cdots \geq \lambda_{m_i}^{(i)} \geq 1. \nonumber
\end{align}
Then, the following lemma holds.
\begin{lemma}
Consider systems $(\mathbf{A_1}, \mathbf{C_1}),(\mathbf{A_2}, \mathbf{C_2}),\cdots,(\mathbf{A_r}, \mathbf{C_r})$ given as \eqref{eqn:ac:jordansingle2}.
Then, we can find a polynomial $p(k)$ and a family of random variable $\{ S(\epsilon,k) : k \in \mathbb{Z}^+, \epsilon>0 \}$ such that for all $\epsilon>0$, $k \in \mathbb{Z}^+$ and $1 \leq i \leq r$
there exist $k \leq k_{i,1} < k_{i,2} < \cdots < k_{i,m_i} \leq S(\epsilon,k)$ and $\mathbf{M_i}$ satisfying the following conditions:\\
(i) $\beta[k_{i,j}]=1$ for $1 \leq i \leq \mu $ and $1 \leq j \leq m_i$\\
(ii) $\mathbf{M_i}\begin{bmatrix} \mathbf{C_i}\mathbf{A_i}^{-k_{i,1}} \\ \mathbf{C_i}\mathbf{A_i}^{-k_{i,2}} \\ \vdots \\ \mathbf{C_i}\mathbf{A_i}^{-k_{i,m_i}}  \end{bmatrix}=\mathbf{I}$\\
(iii) $|\mathbf{M_i}|_{max} \leq \frac{p(S(\epsilon,k))}{\epsilon} (\lambda_{1}^{(i)})^{S(\epsilon,k)}$\\
(iv) $\lim_{\epsilon \downarrow 0} \exp \limsup_{s \rightarrow \infty} \sup_{k \in \mathbb{Z}^+} \frac{1}{s} \log \mathbb{P}\{ S(\epsilon,k)-k=s \} = p_e $.
\label{lem:dis:geodet}
\end{lemma}
\begin{proof}
By Lemma~\ref{lem:dis:inverse}, instead of the condition (ii) and (iii) it is enough to prove that
\begin{align}
\left| \det \left(
\begin{bmatrix}
\mathbf{C_i}\mathbf{A_i}^{-k_{i,1}} \\
\mathbf{C_i}\mathbf{A_i}^{-k_{i,2}} \\
\vdots \\
\mathbf{C_i}\mathbf{A_i}^{-k_{i,m_i}} \\
\end{bmatrix}\right)\right| \geq \epsilon \prod_{1 \leq j \leq m_i}  (\lambda_j^{(i)})^{-k_{i,j}}. \nonumber
\end{align}
Therefore, it is enough to prove the following claim:
\begin{claim}
We can find a family of stopping times $\{ S(\epsilon,k) : k \in \mathbb{Z}^+, \epsilon > 0 \}$ such that for all $\epsilon>0$, $k \in \mathbb{Z}^+$ and $1 \leq i \leq r$
there exist $k \leq k_{i,1} < k_{i,2} < \cdots < k_{i,m_i} \leq S(\epsilon,k)$ satisfying the following condition:\\
(a) $\beta[k_{i,j}]=1$ for $1 \leq i \leq \mu$ and $1 \leq j \leq m_i$ \\
(b) $\left| \det \left(
\begin{bmatrix}
\mathbf{C_i}\mathbf{A_i}^{-k_{i,1}} \\
\mathbf{C_i}\mathbf{A_i}^{-k_{i,2}} \\
\vdots \\
\mathbf{C_i}\mathbf{A_i}^{-k_{i,m_i}} \\
\end{bmatrix}\right)\right| \geq \epsilon \prod_{1 \leq j \leq m_i}  (\lambda_j^{(i)})^{-k_{i,j}}$ \\
(c) $\lim_{\epsilon \downarrow 0} \exp \limsup_{s \rightarrow \infty} \sup_{k \in \mathbb{Z}^+} \frac{1}{s} \log \mathbb{P}\left\{ S(\epsilon,k)-k=s\right\} \leq p_e$.
\label{claim:dis:1}
\end{claim}
Before we prove the above claim, we first prove the claim for a single system.
\begin{claim}
We can find a family of stopping times $\{ S_1(\epsilon,k) : k \in \mathbb{Z}^+, \epsilon > 0 \}$ such that for all $\epsilon>0$ and $k \in \mathbb{Z}^+$
there exist $k \leq k_1' < k_2' < \cdots < k_{m_1}' \leq S_1(\epsilon,k)$ satisfying the following condition:\\
(a') $\beta[k_j']=1$ for $1 \leq j \leq m_1$ \\
(b') $\left| \det \left(
\begin{bmatrix}
\mathbf{C_1}\mathbf{A_1}^{-k_1'} \\
\mathbf{C_1}\mathbf{A_1}^{-k_2'} \\
\vdots \\
\mathbf{C_1}\mathbf{A_1}^{-k_{m_i}'} \\
\end{bmatrix}\right)\right| \geq \epsilon \prod_{1 \leq j \leq m_1}  (\lambda_j^{(1)})^{-k_j'}$ \\
(c') $\lim_{\epsilon \downarrow 0} \exp \limsup_{s \rightarrow \infty} \sup_{k \in \mathbb{Z}^+} \frac{1}{s} \log \mathbb{P}\left\{ S_1(\epsilon,k)-k=s\right\} \leq p_e$.
\label{claim:dis:2}
\end{claim}

$\bullet$ Proof of Claim~\ref{claim:dis:2}: The proof of Claim~\ref{claim:dis:2} is an induction on $m$.

(i) First consider the case $m_1=1$.

In this case, $\mathbf{A_1}$ and $\mathbf{C_1}$ is scalar, so denote $\mathbf{A_1}:=\lambda_1^{(1)}e^{j 2 \pi \omega_1^{(1)}}$ and $\mathbf{C_1}:=c_1$.
Since we only care about small enough $\epsilon$, let $\epsilon \leq |c_1|$.
Denote $S_1(\epsilon,k) := \inf \{ n \geq k : \beta[n]=1 \}$ and $k_1'=S_1(\epsilon,k)$.
Then, $\beta[k_1']=1$ and $\left| \det\left( \begin{bmatrix}
c_1 (\lambda_1^{(1)} e^{j 2 \pi \omega_1^{(1)}} )^{-k_1'}
\end{bmatrix} \right) \right| = |c_1| (\lambda_1^{(1)})^{-k_1'} \geq \epsilon (\lambda_1^{(1)})^{-k_1'}$.
Moreover, since $S_1(\epsilon,k)-k$ is a geometric random variable with probability $1-p_e$,
\begin{align}
\exp \limsup_{s \rightarrow \infty} \sup_{k \in \mathbb{Z}^+} \log \mathbb{P}\left\{ S_1(\epsilon,k)-k=s \right\} = p_e.  \nonumber
\end{align}
Therefore, $S_1(\epsilon,k)$ satisfies all the conditions of the claim.

(ii) As an induction hypothesis, we assume the claim is true for $m_1-1$ and prove the claim hold for $m_1$.

Denote $\mathbf{A_1'}$ be a $(m_1-1) \times (m_1-1)$ matrix obtained by removing $m_1$th row and column of $\mathbf{A_1}$. Likewise, $\mathbf{C_1'}$ is a $1 \times (m_1-1)$ vector obtained by removing $m_1$th element of $\mathbf{C_1}$. Then, we can observe that
\begin{align}
\det\left( \begin{bmatrix}
\mathbf{C_1'} \mathbf{A_1'}^{-k_1'} \\
\vdots \\
\mathbf{C_1'} \mathbf{A_1'}^{-k_{m_1-1}'}
\end{bmatrix} \right)
= cof_{m_1,m_1}\left(
\begin{bmatrix}
\mathbf{C_1} \mathbf{A_1}^{-k_1'} \\
\vdots \\
\mathbf{C_1} \mathbf{A_1}^{-k_{m_1}'} \\
\end{bmatrix}
\right) \nonumber
\end{align}
where $cof_{i,j}(\mathbf{A})$ implies the cofactor matrix of $\mathbf{A}$ with respect to $(i,j)$ element.

By the induction hypothesis, we can find a stopping time $S_1'(\epsilon,k)$ such that there exist $k \leq k_1' < k_2' < \cdots < k_{m_1-1}' \leq S_1'(\epsilon,k)$ satisfying:\\
(a'') $\beta[k_j']=1$ for $1 \leq j \leq m_1-1$ \\
(b'') $\left| \det \left(
\begin{bmatrix}
\mathbf{C_1'}\mathbf{A_1'}^{-k_1'} \\
\vdots \\
\mathbf{C_1'}\mathbf{A_1'}^{-k_{m_1-1}'}
\end{bmatrix}
\right) \right| \geq \epsilon \prod_{1 \leq j \leq m_1-1} (\lambda_{j}^{(1)})^{-k_j'}$\\
(c'') $\lim_{\epsilon \downarrow 0} \exp \limsup_{s \rightarrow \infty} \sup_{k \in \mathbb{Z}^+} \frac{1}{s} \log \mathbb{P} \left\{ S_1'(\epsilon,k)-k=s \right\} \leq p_e$.

Let $\mathcal{F}_{i}$ be a $\sigma$-field generated by $\beta[0],\cdots,\beta[i]$ and $g_{\epsilon}: \mathbb{R}^+ \rightarrow \mathbb{R}^+$ be the function of Lemma~\ref{lem:dis:single}.
Denote a random variable $d(\epsilon,N)$ as following:
\begin{align}
d(\epsilon,N):=\sup_{k \in \mathbb{Z}, k-S_1'(\epsilon,k) \geq g_{\epsilon}(S_1'(\epsilon,k))}
\frac{1}{N} \sum_{n=k+1}^{k+N}
\mathbf{1}\left\{
\left|
\det \left(
\begin{bmatrix}
\mathbf{C_1}\mathbf{A_1}^{-k_1'} \\
\vdots  \\
\mathbf{C_1}\mathbf{A_1}^{-k_{m_1-1}'} \\
\mathbf{C_1}\mathbf{A_1}^{-n}
\end{bmatrix}
\right)
\right|
< \epsilon^2 (\lambda_{m_1}^{(1)})^{-n} \prod_{1 \leq j \leq m_1 -1} ( \lambda_j^{(1)} )^{-k_j}
| \mathcal{F}_{S_1'(\epsilon,k)}
\right\}. \nonumber
\end{align}
Since (b'') implies $cof_{m_1,m_1}\left(
\begin{bmatrix}
\mathbf{C_1} \mathbf{A_1}^{-k_1'} \\
\vdots \\
\mathbf{C_1} \mathbf{A_1}^{-k_{m_1-1}'} \\
\mathbf{C_1} \mathbf{A_1}^{-n} \\
\end{bmatrix}
\right) \geq \epsilon \prod_{1 \leq j \leq m_1-1} (\lambda_i^{(1)})^{-k_j'}$, by Lemma~\ref{lem:dis:single} we have \begin{align}
\lim_{\epsilon \downarrow 0} \lim_{N \rightarrow \infty} \esssup d(\epsilon,N) = 0. \nonumber
\end{align}
Denote $S_1''(\epsilon,k) := S_1'(\epsilon,k)+g_{\epsilon}(S_1'(\epsilon,k))$.
From (ii) of Lemma~\ref{lem:dis:single}  we know $g_{\epsilon}(k) \lesssim 1 + \log(k+1)$ for all $\epsilon>0$.
Therefore, by Lemma~\ref{lem:conti:tailpoly} we have
\begin{align}
\lim_{\epsilon \downarrow 0} \exp \limsup_{s \rightarrow \infty} \sup_{k \in \mathbb{Z}^+} \frac{1}{s} \log \mathbb{P}\{ S_1''(\epsilon,k)-k=s \} \leq p_e. \label{eqn:dis:single:5}
\end{align}
Denote a stopping time
\begin{align}
S_1'''(\epsilon,k):=\inf \left\{n > S_1''(\epsilon,k): \beta[n]=1 \mbox{ and }
\left|
\det \left(
\begin{bmatrix}
\mathbf{C_1}\mathbf{A_1}^{-k_1'} \\
\vdots \\
\mathbf{C_1}\mathbf{A_1}^{-k_{m_1-1}'} \\
\mathbf{C_1}\mathbf{A_1}^{-n} \\
\end{bmatrix}
\right)
\right|
\geq \epsilon^2 (\lambda_{m_1}^{(1)})^{-n} \prod_{1 \leq j \leq m_1-1} (\lambda_j^{(1)})^{-k_j'}
\right\}. \nonumber
\end{align}
Since $\beta[n]$ is a Bernoulli process,
\begin{align}
\mathbb{P}\{ S_1'''(\epsilon,k)-S_1''(\epsilon,k) \geq N | \mathcal{F}_{S_1''(\epsilon,k)} \} \leq p_e^{N(1-d(\epsilon,N))}. \nonumber
\end{align}
Therefore,
\begin{align}
\lim_{\epsilon \downarrow 0} \exp \limsup_{N \rightarrow 0} \esssup \frac{1}{N} \log \mathbb{P}\{ S_1'''(\epsilon,k)-S_1''(\epsilon,k) \geq N | \mathcal{F}_{S_1''(\epsilon,k)} \} \leq \lim_{\epsilon \downarrow 0} \lim_{N \rightarrow \infty}\esssup p_e^{1-d(\epsilon,N)} \leq p_e \nonumber
\end{align}
i.e.
\begin{align}
\lim_{\epsilon \downarrow 0} \exp \limsup_{s \rightarrow 0} \esssup \frac{1}{s} \log \mathbb{P}\{ S_1'''(\epsilon,k)-S_1''(\epsilon,k) = s | \mathcal{F}_{S_1''(\epsilon,k)} \} \leq p_e. \label{eqn:dis:single:6}
\end{align}
By applying Lemma~\ref{lem:app:geo} to \eqref{eqn:dis:single:5} and \eqref{eqn:dis:single:6}, we can conclude that
\begin{align}
\lim_{\epsilon \downarrow 0} \exp \limsup_{s \rightarrow \infty} \sup_{k \in \mathbb{Z}^+} \frac{1}{s} \log \mathbb{P}\{S_1'''(\epsilon,k)-k=s \} \leq p_e. \nonumber
\end{align}
Therefore, if we denote $S_1(\epsilon,k):=S_1'''(\epsilon^{\frac{1}{2}},k)$, $S_1(\epsilon,k)$ satisfies all the conditions of Claim~\ref{claim:dis:2}.

$\bullet$ Proof of Claim~\ref{claim:dis:1}: By recursive use of Claim~\ref{claim:dis:2}, we can find stopping times $S_2(\epsilon,k),\cdots,S_r(\epsilon,k)$ such that for all $\epsilon>0$ and $2 \leq i \leq r$
there exist $S_{i-1}(\epsilon,k) < k_{i,1} < k_{i,2} < \cdots < k_{i,m_i}\leq S_{i}(\epsilon,k)$ satisfying the following condition:\\
(a) $\beta[k_{i,j}]=1$ for $1 \leq j \leq m_{i}$ \\
(b) $\left| \det\left(
\begin{bmatrix}
\mathbf{C_i} \mathbf{A_i}^{-k_{i,1}} \\
\mathbf{C_i} \mathbf{A_i}^{-k_{i,2}} \\
\vdots \\
\mathbf{C_i} \mathbf{A_i}^{-k_{i,m_i}} \\
\end{bmatrix}
\right) \right| \geq \epsilon \prod_{1 \leq j \leq m_i} (\lambda_j^{(i)})^{-k_{i,j}}$\\
(c) $\lim_{\epsilon \downarrow 0} \exp \limsup_{s \rightarrow \infty} \esssup \frac{1}{s} \log \mathbb{P}\{ S_i(\epsilon,k)-S_{i-1}(\epsilon,k)=s | \mathcal{F}_{S_{i-1}(\epsilon,k)} \} \leq p_e$.

Then, by Lemma~\ref{lem:app:geo}
\begin{align}
\lim_{\epsilon \downarrow 0} \exp \limsup_{s \rightarrow \infty} \sup_{k \in \mathbb{Z}^+} \frac{1}{s} \log \mathbb{P}\{ S_r(\epsilon,k)-k=s \} \leq p_e. \nonumber
\end{align}
Therefore, if we denote $S(\epsilon,k):=S_r(\epsilon,k)$, $S(\epsilon,k)$ satisfies all the conditions of Claim~\ref{claim:dis:1}. Thus, Claim~\ref{claim:dis:1} is true and the lemma is also true.
\end{proof}

We prove some properties about matrices which will be helpful in the proof of Lemma~\ref{lem:dis:achv}.

\begin{lemma}
Let $\mathbf{A}$ and $\mathbf{A'}$ be Jordan block matrices with eigenvalues $\lambda, \alpha \lambda (\alpha \neq 0)$ respectively and the same size $m \in \mathbb{N}$, i.e. $\mathbf{A} = \begin{bmatrix}
\lambda & 1 & \cdots & 0 \\
0 & \lambda & \cdots & 0 \\
\vdots & \vdots & \ddots & \vdots \\
0 & 0 & \cdots & \lambda
\end{bmatrix}$ and
$\mathbf{A'}= \begin{bmatrix}
\alpha\lambda & 1 & \cdots & 0 \\
0 & \alpha\lambda & \cdots & 0 \\
\vdots & \vdots & \ddots & \vdots \\
0 & 0 & \cdots & \alpha\lambda
\end{bmatrix}$. Then, for all $n \in \mathbb{Z}$
\begin{align}
\mathbf{A'}^{n}
=
\begin{bmatrix}
\alpha^{-(m-1)} & 0 & \cdots & 0 \\
0 & \alpha^{-(m-2)} & \cdots & 0 \\
\vdots & \vdots & \ddots & \vdots \\
0 & 0 & \cdots & 1
\end{bmatrix}
\mathbf{A}^n
\begin{bmatrix}
\alpha^{n+(m-1)} & 0 & \cdots & 0 \\
0 & \alpha^{n+(m-2)} & \cdots & 0 \\
\vdots & \vdots & \ddots & \vdots \\
0 & 0 & \cdots & \alpha^{n}
\end{bmatrix}.
 \nonumber
\end{align}
\label{lem:dis:jordan1}
\end{lemma}

\begin{proof}
\begin{align}
\mathbf{A'}^n &=
\begin{bmatrix}
(\alpha \lambda )^{n} & {n \choose 1} (\alpha \lambda)^{n-1} & {n \choose 2}(\alpha \lambda)^{n-2} & \cdots & {n \choose m} (\alpha \lambda)^{n-(m-1)} \\
0 & (\alpha \lambda)^{n} & {n \choose 1 }(\alpha \lambda)^{n-1} & \cdots & {n \choose m-1}(\alpha \lambda)^{n-(m-2)} \\
0 & 0 & (\alpha \lambda)^{n}  & \cdots & {n \choose m-2}(\alpha \lambda)^{n-(m-3)} \\
\vdots & \vdots & \vdots & \ddots & \vdots \\
0 & 0 & 0 & \cdots & (\alpha \lambda)^n
\end{bmatrix}\nonumber \\
&=
\begin{bmatrix}
\alpha^{-(m-1)} & 0 & 0 & \cdots & 0 \\
0 & \alpha^{-(m-2)} & 0 & \cdots & 0 \\
0 & 0 & \alpha^{-(m-3)} &  \cdots & 0 \\
\vdots & \vdots & \vdots & \ddots & \vdots \\
0 & 0 & 0 & \cdots & 1
\end{bmatrix} \nonumber \\
& \cdot
\begin{bmatrix}
\alpha^{n+m-1} \lambda^{n} & {n \choose 1} \alpha^{n-1+m-1} \lambda^{n-1} & {n \choose 2}\alpha^{n-2+m-1} \lambda^{n-2} & \cdots & {n \choose m} \alpha^{n-(m-1)+m-1} \lambda^{n-(m-1)} \\
0 & \alpha^{n+m-2} \lambda^{n} & {n \choose 1 }\alpha^{n-1+m-2} \lambda^{n-1} & \cdots & {n \choose m-1}\alpha^{n-(m-2)+m-2} \lambda^{n-(m-2)} \\
0 & 0 & \alpha^{n+m-3} \lambda^{n}  & \cdots & {n \choose m-2}\alpha^{n-(m-3)+m-3} \lambda^{n-(m-3)} \\
\vdots & \vdots & \vdots & \ddots & \vdots \\
0 & 0 & 0 & \cdots & \alpha^n \lambda^n
\end{bmatrix}\nonumber \\
&=
\begin{bmatrix}
\alpha^{-(m-1)} & 0 & 0 & \cdots & 0 \\
0 & \alpha^{-(m-2)} & 0 & \cdots & 0 \\
0 & 0 & \alpha^{-(m-3)} &  \cdots & 0 \\
\vdots & \vdots & \vdots & \ddots & \vdots \\
0 & 0 & 0 & \cdots & 1
\end{bmatrix} \nonumber \\
& \cdot
\begin{bmatrix}
\alpha^{n+m-1} \lambda^{n} & {n \choose 1} \alpha^{n+m-2} \lambda^{n-1} & {n \choose 2}\alpha^{n+m-3} \lambda^{n-2} & \cdots & {n \choose m} \alpha^{n} \lambda^{n-m} \\
0 & \alpha^{n+m-2} \lambda^{n} & {n \choose 1 }\alpha^{n+m-3} \lambda^{n-1} & \cdots & {n \choose m-1}\alpha^{n} \lambda^{n-(m-1)} \\
0 & 0 & \alpha^{n+m-3} \lambda^{n}  & \cdots & {n \choose m-2}\alpha^{n} \lambda^{n-(m-2)} \\
\vdots & \vdots & \vdots & \ddots & \vdots \\
0 & 0 & 0 & \cdots & \alpha^n \lambda^n
\end{bmatrix}\nonumber \\
&=
\begin{bmatrix}
\alpha^{-(m-1)} & 0 & 0 & \cdots & 0 \\
0 & \alpha^{-(m-2)} & 0 & \cdots & 0 \\
0 & 0 & \alpha^{-(m-3)} &  \cdots & 0 \\
\vdots & \vdots & \vdots & \ddots & \vdots \\
0 & 0 & 0 & \cdots & 1
\end{bmatrix} \nonumber \\
& \cdot
\begin{bmatrix}
\lambda^{n} & {n \choose 1} \lambda^{n-1} & {n \choose 2} \lambda^{n-2} & \cdots & {n \choose m}  \lambda^{n-m} \\
0 &  \lambda^{n} & {n \choose 1 } \lambda^{n-1} & \cdots & {n \choose m-1} \lambda^{n-(m-1)} \\
0 & 0 &  \lambda^{n}  & \cdots & {n \choose m-2} \lambda^{n-(m-2)} \\
\vdots & \vdots & \vdots & \ddots & \vdots \\
0 & 0 & 0 & \cdots &  \lambda^n
\end{bmatrix}
\cdot
\begin{bmatrix}
\alpha^{n+(m-1)} & 0 & 0 & \cdots & 0 \\
0 & \alpha^{n+(m-2)} & 0 & \cdots & 0 \\
0 & 0 & \alpha^{n+(m-3)} & \cdots & 0 \\
\vdots & \vdots & \vdots & \ddots & \vdots \\
0 & 0 & 0 & \cdots & \alpha^n
\end{bmatrix}
\nonumber \\
&=
\begin{bmatrix}
\alpha^{-(m-1)} & 0 & \cdots & 0 \\
0 & \alpha^{-(m-2)} & \cdots & 0 \\
\vdots & \vdots & \ddots & \vdots \\
0 & 0 & \cdots & 1
\end{bmatrix}
\mathbf{A}^n
\begin{bmatrix}
\alpha^{n+(m-1)} & 0 & \cdots & 0 \\
0 & \alpha^{n+(m-2)} & \cdots & 0 \\
\vdots & \vdots & \ddots & \vdots \\
0 & 0 & \cdots & \alpha^{n}
\end{bmatrix}\nonumber
\end{align}
This finishes the proof.
\end{proof}

\begin{lemma}
Let $\mathbf{A}$ be a Jordan block with eigenvalue $\lambda$ and dimension $m \times m$.
Then, the Jordan decomposition of the matrix $\mathbf{A}^k$ for $k \in \mathbb{N}$ is $\mathbf{U}\mathbf{\Lambda}\mathbf{U}^{-1}$ where $\mathbf{U}$ is an invertible upper triangular matrix ---so the diagonal elements of $\mathbf{U}$ are non-zero---
and $\mathbf{\Lambda}$ is a Jordan block with eigenvalue $\lambda^k$ and dimension $m \times m$.
\label{lem:dis:jordan3}
\end{lemma}
\begin{proof}
We can see that $\mathbf{A}^k$ is a upper triangular toeplitz matrix whose diagonal elements are $\lambda^{k}$.
Thus, $\det(s \mathbf{I}- \mathbf{A}^k)= (s-\lambda^k)^m$ and all eigenvalues of $\mathbf{A}^k$ are $\lambda^k$. Moreover, the rank of $\mathbf{A}^k-\lambda^k \mathbf{I}$ is $m-1$. Thus, $\mathbf{\Lambda}$ has to be a Jordan block matrix with eigenvalue $\lambda^k$ and dimension $m \times m$.

Moreover, $Ker\left(\left(\mathbf{A}-\lambda^k \mathbf{I}\right)^p \right) \supseteq span\{ \mathbf{e_1}, \mathbf{e_2} , \cdots , \mathbf{e_p}\}$.
Therefore, $i$th column of $\mathbf{U}^{-1}$ has to belong to the vector space $\{ \mathbf{e_1}, \cdots, \mathbf{e_i} \}$ and $\mathbf{U}^{-1}$ is upper diagonal matrix. Here, the existence of Jordan form of arbitrary matrix guarantee the invertibility of $\mathbf{U}$. Therefore, $\mathbf{U}$ is also upper triangular matrix and the invertibility condition of an upper triangular matrix is its diagonal elements are non-zero.
\end{proof}

\begin{lemma}
Let $\mathbf{A}$ be a Jordan block matrix with eigenvalue $\lambda \in \mathbb{C}$ and size $m \in \mathbb{N}$, i.e.
$\mathbf{A}=\begin{bmatrix}
\lambda & 1 & \cdots & 0 \\
0 & \lambda & \cdots & 0 \\
\vdots & \vdots & \ddots & \vdots \\
0 & 0 & \cdots & \lambda \\
\end{bmatrix}$. $\mathbf{C}$ and $\mathbf{C'}$ are $1 \times m$ matrices such that
\begin{align}
&\mathbf{C}=\begin{bmatrix} c_1 & c_2 & \cdots & c_m\end{bmatrix} \nonumber\\
&\mathbf{C'}=\begin{bmatrix} c'_1 & c'_2 & \cdots & c'_m\end{bmatrix}
\end{align}
where $c_i, c'_i \in \mathbb{C}$ and $c_1 \neq 0$.\\
For all $k \in \mathbb{R}$ and $m \times 1$ matrices $\mathbf{X}=\begin{bmatrix} x_1 \\ x_2 \\ \vdots \\ x_m \end{bmatrix}$ and $\mathbf{X'}=\begin{bmatrix} x'_1 \\ x'_2 \\ \vdots \\ x'_m \end{bmatrix}$, there exists $\mathbf{T}$ such that\\
\begin{align}
&(i) \mathbf{T} \mbox{ is an upper triangular matrix.} \nonumber \\
&(ii) \mathbf{C} \mathbf{A}^k \mathbf{X}+ \mathbf{C'}\mathbf{A}^k \mathbf{X'}=\mathbf{C} \mathbf{A}^k \left(\mathbf{X}+\mathbf{T}\mathbf{X'}\right) \nonumber
\end{align}
Moreover, the diagonal elements of $\mathbf{T}$ are $\frac{c_1'}{c_1}$.
\label{lem:dis:jordan2}
\end{lemma}
\begin{proof}
Similar to Lemma \ref{lem:conti:jordan}.
\end{proof}
%

Now, we can prove Lemma~\ref{lem:dis:achv}.
\begin{proof}[Proof of Lemma~\ref{lem:dis:achv}]
We will prove the lemma by an induction on $m$, the dimension of the system. Remind that here we are using the definitions of \eqref{eqn:ac:jordan}, \eqref{eqn:ac2:jordan} for the system matrices $\mathbf{A}$, $\mathbf{C}$, $\mathbf{A_i}$, $\mathbf{C_i}$, $\cdots$.

(i) When $m=1$,

In this case, the lemma reduces to the scalar problem and is trivially true.
Precisely, if we choose $S_1(\epsilon,k)$ as $\inf\{s \geq k : \beta[s]=1 \}$, we can check all the conditions of the lemma are satisfied.

(ii) Now, we will assume the lemma is true when the system dimension is $m-1$ as an induction hypothesis, and prove the lemma holds for the system with dimension $m$.

Let $\mathbf{x_{i,j}}$ be a $m_{i,j} \times 1$ column vector, and $\mathbf{x}$ be $\begin{bmatrix} \mathbf{x_{1,1}} \\ \mathbf{x_{1,2}} \\ \vdots \\ \mathbf{x_{\mu,\nu_\mu}} \end{bmatrix}$. Here, $\mathbf{x}$ can be thought as the state of the system, and $\mathbf{x_{i,j}}$ corresponds to the states associated with the Jordan block $\mathbf{A_{i,j}}$. Remind that $\mathbf{A_{1,1}}$ is the Jordan block with the largest eigenvalue and size.

The purpose of this proof is following: By Lemma~\ref{lem:dis:geodet}, we already know that the lemma holds for systems with scalar observations and without eigenvalue cycles. Therefore, we first reduce the system to one with scalar observations and without eigenvalue cycles. To reduce the system to the one without eigenvalue cycles, we will use down-sampling ideas (polyphase decomposition) form signal processing~\cite{Oppenheim}. To reduce the system to the one with scalar observations, we will multiple a proper post-processing matrix which combines vector observations to scalar observations.
Then, we estimate the $m_{1,1}$th element of $\mathbf{x_{1,1}}$, which associated with the largest eigenvalue. Then, we subtract the estimation from the system. The resulting system becomes a $(m-1)$-dimensional system, and by the induction hypothesis, we can estimate the remaining states. As we mentioned before, this idea is called successive decoding in information theory~\cite{Cover}.

Let's start from the down-sampling and reduction to scalar observation systems. 

$\bullet$ Down-sampling the System by $p$ and Reduction to Scalar Observation Systems: The main difficulty in estimating the $m_{1,1}$th element of $\mathbf{x_{1,1}}$ is the periodicity of the system. To handle this difficulty, we down sample the system. Let $p=\prod_{1 \leq i \leq \mu} p_i$. Remind that in \eqref{eqn:ac2:jordan}, $p_i$ was the period of each eigenvalue cycles. We can see when the system is down sampled by $p$, the resulting system becomes aperiodic. Thus, we can reduce the original periodic system to $p$ number of aperiodic systems.

We can further reduce vector observation systems to scalar observation systems. Thus, the system reduces to aperiodic system with scalar observations, and by Lemma~\ref{lem:dis:geodet} we can estimate the $m_{1,1}$th element of $\mathbf{x_{1,1}}$.

Since we are using induction for the proof, we can focus on the first eigenvalue cycle of the system. 

Let $T_1, \cdots, T_R$ be all the sets $T$ such that $T :=\{t_1, \cdots, t_{|T|} \} \subseteq \{0,1,\cdots, p_1-1 \}$ and
\begin{align}
\begin{bmatrix}
\mathbf{C_1} \mathbf{A_1}^{-t_1} \\
\mathbf{C_1} \mathbf{A_1}^{-t_2} \\
\vdots \\
\mathbf{C_1} \mathbf{A_1}^{-t_{|T|}}
\end{bmatrix} \text{ is full rank.}
\label{eqn:dis:geofinal:0}
\end{align}
Here, the definition of $\mathbf{A_1}$ and $\mathbf{C_1}$ is given in \eqref{eqn:ac2:jordan} and
$
\begin{bmatrix}
\mathbf{C_1} diag\{\alpha_{1,1},\cdots, \alpha_{1,\nu_1} \}^{-t_1} \\
\mathbf{C_1} diag\{\alpha_{1,1},\cdots, \alpha_{1,\nu_1} \}^{-t_2} \\
\vdots \\
\mathbf{C_1} diag\{\alpha_{1,1},\cdots, \alpha_{1,\nu_1} \}^{-t_{|T|}}
\end{bmatrix}
$ is also full rank.
The number of such sets, $R$, is finite since $p_1$ is finite.

Therefore, for each $T_r := \{ t_{r,1},\cdots,t_{r,|T_r|} \}$ $(1 \leq r \leq R)$, we can find a matrix $\mathbf{L_r}$ such that
\begin{align}
\mathbf{L_r}
\begin{bmatrix}
\mathbf{C_1} diag\{\alpha_{1,1},\cdots, \alpha_{1,\nu_1} \}^{-t_{r,1}} \\
\mathbf{C_1} diag\{\alpha_{1,1},\cdots, \alpha_{1,\nu_1} \}^{-t_{r,2}} \\
\vdots \\
\mathbf{C_1} diag\{\alpha_{1,1},\cdots, \alpha_{1,\nu_1} \}^{-t_{r,|T_i|}} \\
\end{bmatrix}=\mathbf{I}. \nonumber
\end{align}
Denote $
\begin{bmatrix}
\mathbf{L_{t_{r,1},r}} & \mathbf{L_{t_{r,2},r}} & \cdots & \mathbf{L_{t_{r,|T_r|},r}}
\end{bmatrix}
$ be the first row of $\mathbf{L_r}$ where $\mathbf{L_{t,r}}$ are $1 \times l$ matrices. Then,
\begin{align}
\begin{bmatrix}
\mathbf{L_{t_{r,1},r}} & \mathbf{L_{t_{r,2},r}} & \cdots & \mathbf{L_{t_{r,|T_r|},r}}
\end{bmatrix}
\begin{bmatrix}
\mathbf{C_1} diag\{\alpha_{1,1},\cdots, \alpha_{1,\nu_1} \}^{-t_{r,1}} \\
\mathbf{C_1} diag\{\alpha_{1,1},\cdots, \alpha_{1,\nu_1} \}^{-t_{r,2}} \\
\vdots \\
\mathbf{C_1} diag\{\alpha_{1,1},\cdots, \alpha_{1,\nu_1} \}^{-t_{r,|T_r|}} \\
\end{bmatrix}=
\begin{bmatrix}
1 & 0 & \cdots & 0
\end{bmatrix}
. \label{eqn:dis:thm6}
\end{align}

We also extend this definition of $\mathbf{L_{q,r}}$ to all $q \in \{0,\cdots, p-1 \}, r \in \{ 1,\cdots,R \} $ by putting $\mathbf{L_{q,r}}:=\mathbf{L_{q (\mod p_1),r}}$ for $q \geq p_1$. Then, we can easily check that \eqref{eqn:dis:thm6} still holds as long as $t_{r,i}$ remains the same in $\mod p_1$.

\begin{claim}
For given $q \in \{0,\cdots, p-1 \}$ and $r \in \{ 1,\cdots, R\}$, let $\mathbf{L_{q,r}} \mathbf{C_1}$ be not $\mathbf{0}$. 
Then, there exists $\mathbf{\bar{C}_{q,r}}$, $\mathbf{\bar{A}_{q,r}}$, $\mathbf{\bar{U}_{q,r}}$, $\mathbf{\bar{x}_{q,r}}$ that satisfies the following conditions:\\
(i) $\mathbf{\bar{A}_{q,r}}$ is a $\bar{\bar{m}}_{q,r} \times \bar{\bar{m}}_{q,r}$ square matrix given in a Jordan form. The eigenvalues of $\mathbf{\bar{A}_{q,r}}$ belong to $\{ \lambda_{1,1}^p, \lambda_{2,1}^p, \cdots, \lambda_{\mu,1}^p \}$, and no two different Jordan blocks have the same eigenvalue. Therefore, $\mathbf{\bar{A}_{q,r}}$ has no eigenvalue cycles. Furthermore, the first Jordan block(left-top) of $\mathbf{\bar{A}_{q,r}}$ is a $m_{1,1} \times m_{1,1}$ Jordan block associated with eigenvalue $\lambda_{1,1}^p$.\\
(ii) $\mathbf{\bar{C}_{q,r}}$ is a $1 \times \bar{\bar{m}}_{q,r}$ row vector and $(\mathbf{\bar{A}_{q,r}}, \mathbf{\bar{C}_{q,r}})$ is observable.\\
(iii) $\mathbf{\bar{U}_{q,r}}$ is a $\bar{\bar{m}}_{q,r} \times \bar{\bar{m}}_{q,r}$ invertible upper triangular matrix.\\
(iv) $\mathbf{\bar{x}_{q,r}}$ is a $\bar{\bar{m}}_{q,r} \times 1$ column vector. There exists a nonzero constant $g_{q,r}$ such that 
\begin{align}
(\mathbf{\bar{x}_{q,r}})_{m_{1,1}}=g_{q,r}
\left(
\mathbf{L_{q,r}}\mathbf{C_1}
diag\{
\alpha_{1,1},\cdots, \alpha_{1,\nu_1}
\}^{-(q+(m_{1,1}-1))}
\right)
\begin{bmatrix}
(\mathbf{x_{1,1}})_{m_{1,1}}\\
(\mathbf{x'_{1,2}})_{m_{1,1}}\\
\vdots \\
(\mathbf{x'_{1,\nu_1}})_{m_{1,1}}
\end{bmatrix}.\nonumber\end{align}
where $(\mathbf{x'_{1,i}})_{m_{1,1}}=(\mathbf{x_{1,i}})_{m_{1,1}}$ when the size of $\mathbf{x_{1,i}}$ is greater or equal to $m_{1,1}$, and $(\mathbf{x'_{1,i}})_{m_{1,1}}=0$ otherwise.\\
(v) For all $k \in \mathbb{Z}^+$, $\mathbf{L_{q,r}}\mathbf{C}\mathbf{A}^{-(pk+q)}\mathbf{x}=\mathbf{\bar{C}_{q,r}}\mathbf{\bar{A}_{q,r}}^{-k}\mathbf{\bar{U}_{q,r}}\mathbf{\bar{x}_{q,r}}$.
\label{claim:donknow2}
\end{claim}

This claim tells that by sub-sampling with rate $p$, we get systems without eigenvalue cycles. Moreover, by multiplying proper row vector to observations, we can reduce the system to a scalar observation system while keeping required information to estimate $(\mathbf{x_{1,1}})_{m_{1,1}}$. When $\mathbf{L_{q,r}} \mathbf{C_1}$ is $\mathbf{0}$, the observation is not useful in estimation $(\mathbf{x_{1,1}})_{m_{1,1}}$. Thus, we can ignore it.

\begin{proof}
The proof of the claim consists of two parts, down-sampling and reduction to a scalar observation system.

(1) Down-sampling the System by $p$:

By the definition of $\mathbf{C}$, $\mathbf{A}$, $\mathbf{C_{i,j}}$, $\mathbf{A_{i,j}}$, for all $k \in \mathbb{Z}$, $q \in \{0,\cdots, p-1 \}$ we have
\begin{align}
&\mathbf{C} \mathbf{A}^{-(pk+q)} \mathbf{x} =
\mathbf{C_{1,1}}\mathbf{A_{1,1}}^{-(pk+q)} \mathbf{x_{1,1}}
+ \mathbf{C_{1,2}}\mathbf{A_{1,2}}^{-(pk+q)} \mathbf{x_{1,2}}+ \cdots
+ \mathbf{C_{\mu,\nu_\mu}}\mathbf{A_{\mu,\nu_\mu}}^{-(pk+q)} \mathbf{x_{\mu,\nu_\mu}}
\label{eqn:downsample:1}
\end{align}

Since the dimensions of $\mathbf{x_{i,1}}, \cdots , \mathbf{x_{i,\nu_i}}$ may be different, we will make them equal by extending the dimensions to the maximum, i.e. $m_{i,1}$. For the extension, we will append zeros at the end of the matrices.
Let $\mathbf{C'_{i,j}}$ be a $l \times m_{i,1}$ matrix given as $\begin{bmatrix}
\mathbf{C_{i,j}} & \mathbf{0}_{l \times (m_{i,1}-m_{i,j})}
\end{bmatrix}
$, $\mathbf{A'_{i,j}}$ be a $m_{i,1} \times m_{i,1}$ Jordan block matrix with eigenvalue $\lambda_{i,j}$, and $\mathbf{x'_{i,j}}$ be a $m_{i,1} \times 1$ column vector given as
$\begin{bmatrix}
\mathbf{x_{i,j}} \\
\mathbf{0}_{(m_{i,1}-m_{i,j}) \times 1}
\end{bmatrix}$.
Then, by the construction, we can see that $(\mathbf{x_{1,1}'})_{m_{1,1}}=(\mathbf{x_{1,1}})_{m_{1,1}}$, and if $m_{1,i}$ is greater or equal to $m_{1,1}$ $(\mathbf{x_{1,i}'})_{m_{1,1}}=(\mathbf{x_{1,i}})_{m_{1,1}}$ and otherwise $(\mathbf{x_{1,i}'})_{m_{1,1}}=0$. Therefore, $\mathbf{x_{i,j}'}$ satisfies the condition (iv) of the claim. Furthermore, the first column of $\mathbf{C_{i,j}'}$ is equal to the first column of $\mathbf{C_{i,j}}$ by construction.

We also define $\alpha_{i,j}$ to be $\frac{\lambda_{i,j}}{\lambda_{i,1}}$. Remind that $\lambda_{i,j}$ was defined as the eigenvalue corresponds to $\mathbf{A_{i,j}}$ in \eqref{eqn:ac:jordan}. Then, by the definitions $\alpha_{i,j}^{p_i}=1$.


Then, \eqref{eqn:downsample:1} can be written as follows:
\begin{align}
&\mathbf{C} \mathbf{A}^{-(pk+q)} \mathbf{x}=
 \mathbf{C'_{1,1}} \mathbf{A'_{1,1}}^{-(pk+q)} \mathbf{x'_{1,1}}
+\mathbf{C'_{1,2}} \mathbf{A'_{1,2}}^{-(pk+q)} \mathbf{x'_{1,2}}+ \cdots
+ \mathbf{C'_{\mu,\nu_\mu}} \mathbf{A'_{\mu,\nu_\mu}}^{-(pk+q)}\mathbf{x'_{\mu,\nu_\mu}} \nonumber \\
&=\mathbf{C'_{1,1}} \mathbf{A'_{1,1}}^{-(pk+q)} \mathbf{x'_{1,1}}\nonumber \\
&+ \mathbf{C'_{1,2}}
\begin{bmatrix}
\alpha_{1,2}^{-(m_{1,1}-1)} & 0 & \cdots & 0 \\
0 & \alpha_{1,2}^{-(m_{1,1}-2)} & \cdots & 0 \\
\vdots & \vdots & \ddots & \vdots \\
0 & 0 & \cdots & 1 \\
\end{bmatrix}
\mathbf{A'_{1,1}}^{-(pk+q)} \begin{bmatrix}
\alpha_{1,2}^{-(pk+q)+(m_{1,1}-1)} & 0 & \cdots & 0 \\
0 & \alpha_{1,2}^{-(pk+q)+(m_{1,1}-2)} & \cdots & 0 \\
\vdots & \vdots & \ddots & \vdots \\
0 & 0 & \cdots & \alpha_{1,2}^{-(pk+q)}
\end{bmatrix}
\mathbf{x'_{1,2}}+ \cdots \nonumber \\
&+ \mathbf{C'_{\mu,\nu_\mu}}
\begin{bmatrix}
\alpha_{\mu,\nu_\mu}^{-(m_{\mu,1}-1)} & 0 & \cdots & 0 \\
0 & \alpha_{\mu,\nu_\mu}^{-(m_{\mu,1}-2)} & \cdots & 0 \\
\vdots & \vdots & \ddots & \vdots \\
0 & 0 & \cdots & 1  \\
\end{bmatrix}
\mathbf{A'_{\mu,1}}^{-(pk+q)}  \begin{bmatrix}
\alpha_{\mu,\nu_\mu}^{-(pk+q)+(m_{\mu,1}-1)} & 0 & \cdots & 0 \\
0 & \alpha_{\mu,\nu_\mu}^{-(pk+q)+(m_{\mu,1}-2)} & \cdots & 0 \\
\vdots & \vdots & \ddots & \vdots \\
0 & 0 & \cdots &  \alpha_{\mu,\nu_\mu}^{-(pk+q)}\\
\end{bmatrix}
\mathbf{x'_{\mu,\nu_\mu}} \label{eqn:dis:thm1} \\
&=\mathbf{C'_{1,1}} \mathbf{A'_{1,1}}^{-(pk+q)} \mathbf{x'_{1,1}}\nonumber \\
&+\mathbf{C'_{1,2}}
\begin{bmatrix}
\alpha_{1,2}^{-(m_{1,1}-1)} & 0 & \cdots & 0 \\
0 & \alpha_{1,2}^{-(m_{1,1}-2)} & \cdots & 0 \\
\vdots & \vdots & \ddots & \vdots \\
0 & 0 & \cdots & 1 \\
\end{bmatrix}
\mathbf{A'_{1,1}}^{-(pk+q)}\begin{bmatrix}
\alpha_{1,2}^{-q+(m_{1,1}-1)} & 0 & \cdots & 0 \\
0 & \alpha_{1,2}^{-q+(m_{1,1}-2)} & \cdots & 0 \\
\vdots & \vdots & \ddots & \vdots \\
0 & 0 & \cdots & \alpha_{1,2}^{-q}
\end{bmatrix}
\mathbf{x'_{1,2}}
+ \cdots \nonumber \\
&+\mathbf{C'_{\mu,\nu_\mu}}
\begin{bmatrix}
\alpha_{\mu,\nu_\mu}^{-(m_{\mu,1}-1)} & 0 & \cdots & 0 \\
0 & \alpha_{\mu,\nu_\mu}^{-(m_{\mu,1}-2)} & \cdots & 0 \\
\vdots & \vdots & \ddots & \vdots \\
0 & 0 & \cdots & 1  \\
\end{bmatrix}
\mathbf{A'_{\mu,1}}^{-(pk+q)}
\begin{bmatrix}
\alpha_{\mu,\nu_\mu}^{-q+(m_{\mu,1}-1)} & 0 & \cdots & 0 \\
0 & \alpha_{\mu,\nu_\mu}^{-q+(m_{\mu,1}-2)} & \cdots & 0 \\
\vdots & \vdots & \ddots & \vdots \\
0 & 0 & \cdots &  \alpha_{\mu,\nu_\mu}^{-q}\\
\end{bmatrix}
\mathbf{x'_{\mu,\nu_\mu}}.\label{eqn:dis:thm2}
\end{align}
Here, \eqref{eqn:dis:thm1} follows from Lemma~\ref{lem:dis:jordan1} and \eqref{eqn:dis:thm2} follows from $\alpha_{i,j}^p=\left( \alpha_{i,j}^{p_i} \right)^{\prod_{j \neq i}p_j }=1$. Remind that $m_{i,j}$ was defined as the size of $\mathbf{A_{i,j}}$ in \eqref{eqn:ac:jordan}.

Define 
\begin{align}
&\mathbf{C''_{i,j}} := \mathbf{C'_{i,j}}
\begin{bmatrix}
\alpha_{i,j}^{-(m_{i,1}-1)} & 0 & \cdots & 0 \\
0 & \alpha_{i,j}^{-(m_{i,1}-2)} & \cdots & 0 \\
\vdots & \vdots & \ddots & \vdots \\
0 & 0 & \cdots & 1
\end{bmatrix}, \label{eqn:evidence3} \\
&\mathbf{x''_{i,j}} :=
\begin{bmatrix}
\alpha_{i,j}^{-q+(m_{i,1}-1)} & 0 & \cdots & 0 \\
0 & \alpha_{i,j}^{-q+(m_{i,1}-2)} & \cdots & 0 \\
\vdots & \vdots & \ddots & \vdots \\
0 & 0 & \cdots & \alpha_{i,j}^{-q}
\end{bmatrix}
\mathbf{x'_{i,j}}. \label{eqn:evidence2}
\end{align}

Here, we can notice that the first column of $\mathbf{C''_{i,j}}$ is $\alpha_{i,j}^{-(m_{i,1}-1)}$ times the first column of $\mathbf{C'_{i,j}}$. Here, we know the first column of $\mathbf{C'_{i,j}}$ is equal to the first column of $\mathbf{C_{i,j}}$. The last element of $\mathbf{x''_{i,j}}$ is $\alpha_{i,j}^{-q}$ times the last element of $\mathbf{x'_{i,j}}$. \eqref{eqn:dis:thm2} can be written as
\begin{align}
&\mathbf{C}\mathbf{A}^{-(pk+q)}\mathbf{x}=\mathbf{C''_{1,1}}\mathbf{A'_{1,1}}^{-(pk+q)}\mathbf{x''_{1,1}}
+\mathbf{C''_{1,2}}\mathbf{A'_{1,1}}^{-(pk+q)}\mathbf{x''_{1,2}}+ \cdots +
\mathbf{C''_{\mu,\nu_\mu}}\mathbf{A'_{\mu,1}}^{-(pk+q)}\mathbf{x''_{\mu,\nu_\mu}}. \label{eqn:dis:thm3}
\end{align}
We can see all $\mathbf{x''_{i,1}}, \cdots, \mathbf{x''_{i,\nu_i}}$ are multiplied by the same matrix $\mathbf{A_{1,1}'}$. Eventually, we will merge $\mathbf{x''_{i,1}}, \cdots, \mathbf{x''_{i,\nu_i}}$ by taking linear combinations.

(2) Reduction to the scalar observation: Now, we reduce $\mathbf{C''_{i,j}}$ to row vectors by multiply $\mathbf{L_{q,r}}$ to \eqref{eqn:dis:thm3}.
\begin{align}
&\mathbf{L_{q,r}}\mathbf{C}\mathbf{A}^{-(pk+q)}\mathbf{x}=
\mathbf{L_{q,r}}\mathbf{C''_{1,1}}\mathbf{A'_{1,1}}^{-(pk+q)}\mathbf{x''_{1,1}}
+\mathbf{L_{q,r}}\mathbf{C''_{1,2}}\mathbf{A'_{1,1}}^{-(pk+q)}\mathbf{x''_{1,2}}+ \cdots +
\mathbf{L_{q,r}}\mathbf{C''_{\mu,\nu_\mu}}\mathbf{A'_{\mu,1}}^{-(pk+q)}\mathbf{x''_{\mu,\nu_\mu}}. \label{eqn:dis:thm33}
\end{align}

Here, the systems $(\mathbf{A'_{i,1}},\mathbf{L_{q,r}}\mathbf{C''_{i,1}}), \cdots , (\mathbf{A'_{i,1}},\mathbf{L_{q,r}}\mathbf{C''_{i,\nu_i}})$ have the same dimension, but none of them might be observable. 
Therefore, we will make at least one of the systems to be observable by truncation. Since $\mathbf{A_{i,1}'}$ is a Jordan block matrix and $\mathbf{L_{q,r}}\mathbf{C''_{i,j}}$ is a row vector, $(\mathbf{A_{i,1}'}, \mathbf{L_{q,r}}\mathbf{C''_{i,j}})$ is observable if and only if the first element of $\mathbf{L_{q,r}}\mathbf{C''_{i,j}}$ is not zero. Let $m_i'$ be the smallest number such that at least one of  the $m_i'$th elements of $\mathbf{L_{q,r}}\mathbf{C''_{i,1}}, \cdots, \mathbf{L_{q,r}}\mathbf{C''_{i,\nu_i}}$ becomes nonzero, and let $\mathbf{L_{q,r}}\mathbf{C''_{i,\nu_i^\star}}$ be the vector that achieves the minimum.

Then, we will reduce the dimension of $(\mathbf{A'_{i,1}},\mathbf{L_{q,r}}\mathbf{C''_{i,\nu_i}})$ by truncating the first $(m_i'-1)$ vectors. Define $\mathbf{C'''_{i,j}}$ as the matrix obtained by truncating first $(m'_i-1)$ columns of $\mathbf{C''_{i,j}}$, $\mathbf{A''_{i,j}}$ as the matrix obtained by truncating first $(m'_i-1)$ rows and columns of $\mathbf{A'_{i,j}}$, and $\mathbf{x'''_{i,j}}$ as the column vector obtained by truncating first $(m'_i-1)$ elements of $\mathbf{x''_{i,j}}$.

In the claim, we assumed that $\mathbf{L_{q,r}}\mathbf{C_1}$ is not $\mathbf{0}$. Remind that the elements of $\mathbf{L_{q,r}}\mathbf{C_1}$ correspond to the first elements of $\mathbf{L_{q,r}} \mathbf{C_{1,1}}, \cdots, \mathbf{L_{q,r}} \mathbf{C_{1,\nu_1}}$, which are again equal to the first elements of $\mathbf{L_{q,r}} \mathbf{C_{1,1}'}, \cdots, \mathbf{L_{q,r}} \mathbf{C_{1,\nu_1}'}$. 
Since the first column of $\mathbf{C''_{i,j}}$ is the first column of $\mathbf{C'_{i,j}}$ times $\alpha_{i,j}^{-(m_{i,1}-1)}$, at least one of the systems $(\mathbf{A'_{1,1}},\mathbf{L_{q,r}}\mathbf{C''_{1,1}}),\cdots,(\mathbf{A'_{1,1}},\mathbf{L_{q,r}}\mathbf{C''_{1,\nu_1}})$ has to be observable. 

Therefore, we can see $m_1'=1$ and 
\begin{align}
\mathbf{C'''_{1,i}}=\mathbf{C''_{1,i}}, \mathbf{A''_{1,i}}=\mathbf{A'_{1,i}}, \mathbf{x'''_{1,i}}=\mathbf{x''_{1,i}}. \label{eqn:evidence1}
\end{align}

Now, \eqref{eqn:dis:thm33} becomes
\begin{align}
&\mathbf{L_{q,r}}\mathbf{C}\mathbf{A^{-(pk+q)}}\mathbf{x}=\mathbf{L_{q,r}}\mathbf{C'''_{1,1}}\mathbf{A''_{1,1}}^{-(pk+q)}\mathbf{x'''_{1,1}}
+\mathbf{L_{q,r}}\mathbf{C'''_{1,2}}\mathbf{A''_{1,1}}^{-(pk+q)}\mathbf{x'''_{1,2}}+\cdots
+\mathbf{L_{q,r}}\mathbf{C'''_{\mu,\nu_\mu}}\mathbf{A''_{\mu,1}}^{-(pk+q)}\mathbf{x'''_{\mu,\nu_\mu}}. \nonumber
\end{align}
Let $c_{i,j,1}'''$ be the first element of $\mathbf{L_{q,r}}\mathbf{C'''_{i,j}}$. By Lemma~\ref{lem:dis:jordan2}, we can find upper triangular matrices $\mathbf{T_{i,j}}$ such that their diagonal elements are $\frac{c_{i,j,1}'''}{c_{i,\nu_i^\star,1}'''}$ and
\begin{align}
\mathbf{L_{q,r}}\mathbf{C}\mathbf{A^{-(pk+q)}}\mathbf{x}&=\mathbf{L_{q,r}}\mathbf{C'''_{1,\nu_1^\star}}\mathbf{A''_{1,1}}^{-(pk+q)}\left(
\mathbf{T_{1,1}}\mathbf{x'''_{1,1}}+\mathbf{T_{1,2}}\mathbf{x'''_{1,2}}+\cdots+
\mathbf{T_{1,\nu_1}}\mathbf{x'''_{1,\nu_1}} \right)+
\cdots \nonumber\\
&+\mathbf{L_{q,r}}\mathbf{C'''_{\mu,\nu_\mu^\star}}\mathbf{A''_{\mu,1}}^{-(pk+q)}\left(
\mathbf{T_{\mu,1}}\mathbf{x'''_{\mu,1}}+\mathbf{T_{\mu,2}}\mathbf{x'''_{\mu,2}}+\cdots+\mathbf{T_{\mu,\nu_\mu}}\mathbf{x'''_{\mu,\nu_\mu}}
\right)\label{eqn:dis:thm4}
\end{align}
where $c_{i,\nu_i^\star,1}'''$ is guaranteed to be nonzero by the construction.

Define $\mathbf{x''''_i}$ as
\begin{align}
\left(\mathbf{T_{i,1}}\mathbf{x'''_{i,1}}+\mathbf{T_{i,2}}\mathbf{x'''_{i,2}}+\cdots+\mathbf{T_{i,\nu_i}}\mathbf{x'''_{i,\nu_i}}\right).  \label{eqn:decoding:4}
\end{align}

Here, $\mathbf{A''_{i,1}}^{-(pk+q)}$ is not in a Jordan block. However, since $\mathbf{A''_{i,1}}$ is a Jordan block, by Lemma~\ref{lem:dis:jordan3} the Jordan decomposition of $\mathbf{A''_{i,1}}^{p}$ is $\mathbf{U_i}\mathbf{\Lambda_i}\mathbf{U_i}^{-1}$ where $\mathbf{\Lambda_i}$ is a Jordan block whose eigenvalue is $p$th power of the eigenvalue of $\mathbf{A''_{i,1}}$ and $\mathbf{U_i}$ is an upper triangular matrix whose diagonal entries are non-zero. Thus, \eqref{eqn:dis:thm4} can be written as
\begin{align}
&\mathbf{L_{q,r}}\mathbf{C}\mathbf{A^{-(pk+q)}}\mathbf{x}=\mathbf{L_{q,r}} \mathbf{C'''_{1,\nu_1^\star}} \mathbf{U_{1}} \mathbf{\Lambda_{1}}^{-k} \mathbf{U_{1}}^{-1} \mathbf{A''_{1,1}}^{-q} \mathbf{x''''_{1}} + \cdots
+\mathbf{L_{q,r}} \mathbf{C'''_{\mu,\nu_\mu^\star}} \mathbf{U_{\mu}} \mathbf{\Lambda_{\mu}}^{-k} \mathbf{U_{\mu}}^{-1} \mathbf{A''_{\mu,1}}^{-q} \mathbf{x''''_{\mu}} \nonumber \\
&=
\begin{bmatrix}
\mathbf{L_{q,r}}\mathbf{C'''_{1,\nu_1^\star}}\mathbf{U_{1}}
& \mathbf{L_{q,r}}\mathbf{C'''_{2,\nu_2^\star}}\mathbf{U_{2}} & \cdots
& \mathbf{L_{q,r}}\mathbf{C'''_{\mu,\nu_\mu^\star}}\mathbf{U_{\mu}}
\end{bmatrix}
\begin{bmatrix}
\mathbf{\Lambda_{1}} & 0 & \cdots & 0 \\
0 & \mathbf{\Lambda_{2}} & \cdots & 0 \\
\vdots & \vdots & \ddots & \vdots \\
0 & 0 & \cdots & \mathbf{\Lambda_{\mu}}
\end{bmatrix}^{-k}\nonumber \\
&\cdot\begin{bmatrix}
\mathbf{U_{1}}^{-1} \mathbf{A''_{1,1}}^{-q} & 0 & \cdots & 0 \\
0 & \mathbf{U_{2}}^{-1} \mathbf{A''_{2,1}}^{-q} & \cdots & 0 \\
\vdots & \vdots & \ddots & \vdots \\
0 & 0 & \cdots & \mathbf{U_{\mu}}^{-1} \mathbf{A''_{\mu,1}}^{-q} \\
\end{bmatrix}
\begin{bmatrix}
\mathbf{x_{1}''''} \\
\mathbf{x_{2}''''} \\
\vdots \\
\mathbf{x_{\mu}''''}
\end{bmatrix}
.\nonumber
\end{align}
Let's define $\mathbf{\bar{C}_{q,r}}$  as $
\begin{bmatrix}
\mathbf{L_{q,r}}\mathbf{C'''_{1,\nu_1^\star}}\mathbf{U_{1}}
& \mathbf{L_{q,r}}\mathbf{C'''_{2,\nu_2^\star}}\mathbf{U_{2}} & \cdots
& \mathbf{L_{q,r}}\mathbf{C'''_{\mu,\nu_\mu^\star}}\mathbf{U_{\mu}}
\end{bmatrix}
$, $\mathbf{\bar{A}_{q,r}}$ as $
\begin{bmatrix}
\mathbf{\Lambda_{1}} & 0 & \cdots & 0 \\
0 & \mathbf{\Lambda_{2}} & \cdots & 0 \\
\vdots & \vdots & \ddots & \vdots \\
0 & 0 & \cdots & \mathbf{\Lambda_{\mu}}
\end{bmatrix}
$, $\mathbf{\bar{U}_{q,r}}$ as \\$
\begin{bmatrix}
\mathbf{U_{1}}^{-1} \mathbf{A''_{1,1}}^{-q} & 0 & \cdots & 0 \\
0 & \mathbf{U_{2}}^{-1} \mathbf{A''_{2,1}}^{-q} & \cdots & 0 \\
\vdots & \vdots & \ddots & \vdots \\
0 & 0 & \cdots & \mathbf{U_{\mu}}^{-1} \mathbf{A''_{\mu,1}}^{-q} \\
\end{bmatrix}$, $\mathbf{\bar{x}_{q,r}}$ as $\begin{bmatrix} \mathbf{x_1''''} \\ \mathbf{x_2''''} \\ \vdots \\ \mathbf{x_\mu''''} \end{bmatrix}$ and $\bar{m}_{q,r}$ as the dimension of $\mathbf{\bar{A}_{q,r}}$.

Here, we can see that $\mathbf{\bar{A}_{q,r}}$ has no eigenvalue cycles and satisfies the condition (i) of the claim. Furthermore, since $\mathbf{U_i}$ is an upper triangular matrix whose diagonal elements are non-zero, the first elements of $\mathbf{L_{q,r}}\mathbf{C'''_{i,\nu_i^\star}}\mathbf{U_{i}}$ are still non-zeros. Thus, the system $(\mathbf{\Lambda_{i}}, \mathbf{L_{q,r}}\mathbf{C'''_{i,\nu_i^\star}}\mathbf{U_{i}})$ is observable and $(\mathbf{\bar{A}_{q,r}},\mathbf{\bar{C}_{q,r}})$ is also observable, which satisfies the condition (ii) of the claim. We also have
\begin{align}
&\mathbf{L_{q,r}}\mathbf{C}\mathbf{A}^{-(pk+q)}\mathbf{x}=\mathbf{\bar{C}_{q,r}}\mathbf{\bar{A}_{q,r}}^{-k}\mathbf{\bar{U}_{q,r}}\mathbf{\bar{x}_{q,r}}
\end{align}
which is the condition (v) of the claim. 

Let $c_{1,j,1}$ be the first element of $\mathbf{L_{q,r}}\mathbf{C_{1,j}}$. Then, we have
\begin{align}
(\mathbf{\bar{x}_{q,r}})_{m_{1,1}}&=(\mathbf{x''''_1})_{m_{1,1}}
=\left(\frac{c_{1,1,1}'''}{c_{1,\nu_1^\star,1}'''} (\mathbf{x'''_{1,1}})_{m_{1,1}}+
\cdots+
\frac{c_{1,\nu_1,1}'''}{c_{1,\nu_1^\star,1}'''} (\mathbf{x'''_{1,\nu_1}})_{m_{1,1}}
\right)\label{eqn:decoding:1} \\
&=\left(\frac{c_{1,1,1}'''}{c_{1,\nu_1^\star,1}'''}\alpha_{1,1}^{-q} (\mathbf{x'_{1,1}})_{m_{1,1}}+
\cdots+
\frac{c_{1,\nu_1,1}'''}{c_{1,\nu_1^\star,1}'''}\alpha_{1,\nu_1}^{-q} (\mathbf{x'_{1,\nu_1}})_{m_{1,1}}
\right)\label{eqn:decoding:2} \\
&=\frac{1}{c_{1,\nu_1^\star,1}'''}
\left(
c_{1,1,1} \alpha_{1,1}^{-q-(m_{1,1}-1)} (\mathbf{x'_{1,1}})_{m_{1,1}}+ \cdots +
c_{1,\nu_1,1} \alpha_{1,\nu_1}^{-q-(m_{1,1}-1)} (\mathbf{x'_{1,\nu_1}})_{m_{1,1}}
\right)\label{eqn:decoding:3} \\
&=\frac{1}{c_{1,\nu_1^\star,1}'''}
\left(
\mathbf{L_{q,r}}\mathbf{C_1}
diag\{
\alpha_{1,1},\cdots, \alpha_{1,\nu_1}
\}^{-(q+(m_{1,1}-1))}
\right)
\begin{bmatrix}
(\mathbf{x'_{1,1}})_{m_{1,1}}\\
\vdots \\
(\mathbf{x'_{1,\nu_1}})_{m_{1,1}}
\end{bmatrix} \label{eqn:dis:thm5}
\end{align}
\eqref{eqn:decoding:1} follows from \eqref{eqn:decoding:4}. \eqref{eqn:decoding:2} follows from \eqref{eqn:evidence2}, \eqref{eqn:evidence1}. \eqref{eqn:decoding:3} follows from \eqref{eqn:evidence3}, \eqref{eqn:evidence1} and that the first column of $\mathbf{C'_{i,j}}$ is the same as the first column of $\mathbf{C_{i,j}}$ as we mentioned above. Furthermore, as we mentioned above, $(\mathbf{x'_{1,1}})_{m_{1,1}}=(\mathbf{x_{1,1}})_{m_{1,1}}$. Therefore, the condition (iv) of the claim is also satisfied, and this finishes the proof.
\end{proof}

$\bullet$ Estimating $(\mathbf{x})_{m_{1,1}}$: Now, we have systems without eigenvalue cycles and with scalar observations. Thus, by applying Lemma~\ref{lem:dis:geodet}, we will estimate the state $(\mathbf{x})_{m_{1,1}}$.

\begin{claim}
We can find a polynomial $\bar{p}(k)$, $\bar{m} \in \mathbb{N}$ and a family of stopping time $\{ \bar{S}(\epsilon, k) : k \in \mathbb{Z}^+, \epsilon > 0\}$ such that for all $\epsilon > 0$, $k \in \mathbb{Z}^+$ there exist $k \leq \bar{k}_1 < \bar{k}_2 < \cdots <\bar{k}_{\bar{m}} \leq \bar{S}(\epsilon, k)$ and $\mathbf{\bar{M}}$ satisfying:\\
(i) $\beta[\bar{k}_i]=1$ for $1 \leq i \leq \bar{m}$\\
(ii) $\mathbf{\bar{M}}
\begin{bmatrix}
\mathbf{C}\mathbf{A}^{-\bar{k}_1} \\
\vdots \\
\mathbf{C}\mathbf{A}^{-\bar{k}_{\bar{m}}} \\
\end{bmatrix}
\mathbf{x}=(\mathbf{x})_{m_{1,1}}$\\ 
(iii) $\left| \mathbf{\bar{M}} \right|_{max} \leq \frac{\bar{p}(\bar{S}(\epsilon,k))}{\epsilon}|\lambda_{1,1}|^{\bar{S}(\epsilon,k)}$\\
(iv) $\lim_{\epsilon \downarrow 0} \exp \limsup_{s \rightarrow \infty} \sup_{k \in \mathbb{Z}^+} \frac{1}{s} \log \mathbb{P} \{
\bar{S}(\epsilon,k) - k = s
\} \leq  p_e^{\frac{l_1}{p_1}}$
\label{claim:donknow00}
\end{claim}

This claim tells that there exists an estimator $\mathbf{\bar{M}}$ for $(\mathbf{x})_{m_{1,1}}$ which use observations at time $\bar{k}_1, \cdots, \bar{k}_{\bar{m}}$.

\begin{proof}
For each $q \in \{0, \cdots, p-1 \}$, we have the down-sampled systems $(\mathbf{\bar{A}_{q,1}},\mathbf{\bar{C}_{q,1}}),\cdots,(\mathbf{\bar{A}_{q,R}},\mathbf{\bar{C}_{q,R}})$
such that all systems are observable, $\mathbf{\bar{A}_{q,i}}$ have no eigenvalue cycles, and $\mathbf{\bar{C}_{q,i}}$ are row vectors. By Lemma~\ref{lem:dis:geodet}, we can find a polynomial $p_q(k)$ and a family of random variable $\{\bar{S}_{q}(\epsilon,k): k \in \mathbb{Z^+}, \epsilon > 0 \}$ such that for all $\epsilon > 0$, $k \in \mathbb{Z}^+$ and $1 \leq i \leq R$ there exist $\bar{m}_{q,i}$ and $ \lceil \frac{k-q}{p}  \rceil \leq k_{i,1} < k_{i,2} < \cdots < k_{i,\bar{m}_{q,i}} \leq \bar{S}_q(\epsilon,k)$ and $\mathbf{M_i}$ satisfying:\\
(i) $\beta[pk_{i,j}+q]=1$ for $1 \leq j \leq \bar{m}_{q,i}$\\
(ii) $\mathbf{M_i}
\begin{bmatrix}
\mathbf{\bar{C}_{q,i}} \mathbf{\bar{A}_{q,i}}^{-k_{i,1}} \\
\mathbf{\bar{C}_{q,i}} \mathbf{\bar{A}_{q,i}}^{-k_{i,2}} \\
\vdots \\
\mathbf{\bar{C}_{q,i}} \mathbf{\bar{A}_{q,i}}^{-k_{i,\bar{m}_{q,i}}}
\end{bmatrix}=\mathbf{I_{\bar{m}_{q,i}\times \bar{m}_{q,i}}}
$\\
(iii) $
\left| \mathbf{M_i} \right|_{max} \leq \frac{ p_q\left( \bar{S}_q(\epsilon,k) \right) }{\epsilon} (|\lambda_{1,1}|^p)^{\bar{S}_q(\epsilon,k)}
$\\
(iv) $
\lim_{\epsilon \downarrow 0} \exp \limsup_{s \rightarrow \infty} \sup_{k \in \mathbb{Z}^+} \frac{1}{s} \log \mathbb{P}
\{ \bar{S}_q(\epsilon,k)- \lceil \frac{k-q}{p} \rceil = s \} = p_e.
$

By the property (iv) of $\bar{S}_q(\epsilon,k)$, we get
\begin{align}
\lim_{\epsilon \downarrow 0} \exp \limsup_{s \rightarrow \infty} \sup_{k \in \mathbb{Z}^+} \frac{1}{s} \log \mathbb{P} \{
p \bar{S}_q(\epsilon,k) - p \lceil \frac{k-q}{p} \rceil = s \} = p_e^{\frac{1}{p}}
\nonumber
\end{align}
which implies
\begin{align}
\lim_{\epsilon \downarrow 0} \exp \limsup_{s \rightarrow \infty} \sup_{k \in \mathbb{Z}^+} \frac{1}{s} \log \mathbb{P} \{
(p \bar{S}_q(\epsilon,k)+q)-k=s \} = p_e^{\frac{1}{p}}.\nonumber
\end{align}
Moreover, $\bar{S}_q(\epsilon,k)$ depends on only $\beta[q],\beta[p+q],\beta[2p+q],\cdots$. Thus, $\bar{S}_0(\epsilon,k),\cdots,\bar{S}_{p-1}(\epsilon,k)$ are independent.

Now, we can estimate the state of each sub-sampled system. We will leverage these estimations to the estimation of the state $(\mathbf{x})_{m_{1,1}}$.

First, notice that the down-sampling rate $p$ is much larger than $p_1$. Therefore, we make the corresponding definition to \eqref{eqn:dis:geofinal:0} for the longer period $p$. Let $T'_1,\cdots,T'_{R'}$ be all the sets $T'$ such that $T' :=\{ t'_1, \cdots, t'_{|T'|} \} \subseteq \{0,1,\cdots,p-1 \}$ and
\begin{align}
\begin{bmatrix}
&\mathbf{C_1} \mathbf{A_1}^{-t'_1} \\
&\mathbf{C_1} \mathbf{A_1}^{-t'_2} \\
&\vdots \\
&\mathbf{C_1} \mathbf{A_1}^{-t'_{|T'|}}
\end{bmatrix} \mbox{ is full rank.}
\end{align}

Here, we can ask how many observations have to be erased to make the observability Gramian of $(\mathbf{A_1}, \mathbf{C_1})$ rank deficient during the period $p$. Obviously, the answer is $l_1 \prod_{ 2 \leq j \leq \mu}p_j$ where the definition of $l_1$ is shown in \eqref{eqn:def:lprime}. The reason for this is that we have to erase at least $l_1$ observations for each period $p_1$ to make the observability Gramian rank deficient. Formally, it can be written as follows:
\begin{align}
min\{|T|: T=\{t_1,\cdots,t_{|T|} \} \subseteq \{ 0,1,\cdots,p-1 \}, T'_i \not\subseteq T \mbox{ for all }1 \leq  i \leq R'  \}=l_1 \prod_{2 \leq j \leq \mu} p_j. \nonumber
\end{align}

Denote a stopping time $\bar{S}(\epsilon,k)$ as the minimum time until we have enough observations to make the observability Gramian of $(\mathbf{A_1}, \mathbf{C_1})$ full rank. Formally,
\begin{align}
\bar{S}(\epsilon,k)-k&:=
\inf \{
s:\exists i  \in \{1, \cdots, R' \}\mbox{ s.t. } T'_i=\{t'_1,t'_2,\cdots t'_{|T'_i|} \} \mbox{ and }\nonumber\\
 &(p \bar{S}_{t'_1}(\epsilon,k) + t'_1)-k  \leq s,
(p \bar{S}_{t'_2}(\epsilon,k) + t'_2)-k  \leq s, \cdots, (p \bar{S}_{t'_{|T'_i|}}(\epsilon,k) + t'_{|T'_i|})-k  \leq s
\}. \nonumber
\end{align}
Then, by Lemma~\ref{lem:dis:geo0} we have
\begin{align}
\lim_{\epsilon \downarrow 0} \exp \limsup_{s \rightarrow \infty} \sup_{k \in \mathbb{Z}^+} \frac{1}{s} \log \mathbb{P} \{
\bar{S}(\epsilon,k) - k = s
\} \leq p_e^{\frac{l_1 \prod_{j \neq 1 } p_j }{p}} = p_e^{\frac{l_1}{p_1}}.
\end{align}

Without loss of generality, let $T_1'$ be the set that satisfies the definition of $\bar{S}(\epsilon,k)$.
Then, by the definition of $T_1'$ and $T_i$, there must exist $T_i$ such that $T_1'$ contains $T_i$ in mod $p_1$. Let $T_1$ be such a set without loss of generality. Then, we can find $\{t'_1, \cdots, t'_{|T_1|}\}$ which is included in $T_1'$ and includes $T_1$ in mod $p_1$. Formally, $\{t'_1, \cdots, t'_{|T_1|} \} \subseteq T_1'$ and $\{t'_1 (mod~p_1), \cdots, t'_{|T_1|}(mod~p_1) \}=T_1$.

Then, from the definition of $\bar{S}(\epsilon,k)$ and $\bar{S}_q(\epsilon,k)$, for each $q\in \{ t'_1,\cdots,t'_{|T_1|} \}$ we can find $\lceil \frac{k-q}{p} \rceil \leq k_{q,1} < k_{q,2} < \cdots < k_{q,\bar{m}_{q,1}} \leq \bar{S}_q(\epsilon,k)$ and $\mathbf{M_q}$ satisfying the following conditions:\\
(i') $\beta[ p k_{q,j} + q ]=1$ for $1 \leq j \leq \bar{m}_{q,1}$\\
(ii') $\mathbf{M_q}
\begin{bmatrix}
\mathbf{\bar{C}_{q,1}} \mathbf{\bar{A}_{q,1}}^{-k_{q,1}} \\
\mathbf{\bar{C}_{q,1}} \mathbf{\bar{A}_{q,1}}^{-k_{q,2}} \\
\vdots \\
\mathbf{\bar{C}_{q,1}} \mathbf{\bar{A}_{q,1}}^{-k_{q,\bar{m}_{q,1}}}
\end{bmatrix}= \mathbf{I_{\bar{m}_{q,1} \times \bar{m}_{q,1}}} $ \\
(iii') $\left| \mathbf{M_q} \right|_{max}  \leq \frac{p_q(\bar{S}_q(\epsilon,k))}{\epsilon}(|\lambda_{1,1}|^p)^{\bar{S}_q(\epsilon,k)}$.\\
(iv') $ p \bar{S}_q(\epsilon,k) + q \leq \bar{S}(\epsilon, k)$

Then, we have
\begin{align}
&diag\{ \mathbf{\bar{U}_{t'_1,1}}^{-1} \mathbf{M_{t'_1}}, \mathbf{\bar{U}_{t'_2,1}}^{-1} \mathbf{M_{t'_2}}, \cdots, \mathbf{\bar{U}_{t'_{|T_1|},1}}^{-1} \mathbf{M_{t'_{|T_1|} }} \}
diag\{ \mathbf{L_{t'_1,1}},\mathbf{L_{t'_1,1}}, \cdots, \mathbf{L_{t'_{|T_1|},1}} \} \nonumber \\
&\cdot
\begin{bmatrix}
\mathbf{C}\mathbf{A}^{-(pk_{t'_1,1}+t'_1)} \\
\mathbf{C}\mathbf{A}^{-(pk_{t'_1,2}+t'_1)} \\
\vdots \\
\mathbf{C}\mathbf{A}^{-(pk_{t'_1,\bar{m}_{t'_1,1}}+t'_1)} \\
\mathbf{C}\mathbf{A}^{-(pk_{t'_2,1}+t'_2)} \\
\vdots \\
\mathbf{C}\mathbf{A}^{-(pk_{t'_{|T_1|},\bar{m}_{t'_{|T_1|},1}}+t'_{|T_1|})}
\end{bmatrix}
\mathbf{x} \nonumber \\
&=
diag\{ \mathbf{\bar{U}_{t'_1,1}}^{-1} \mathbf{M_{t'_1}}, \mathbf{\bar{U}_{t'_2,1}}^{-1} \mathbf{M_{t'_2}}, \cdots, \mathbf{\bar{U}_{t'_{|T_1|},1}}^{-1} \mathbf{M_{t'_{|T_1|}}} \}
\begin{bmatrix}
\mathbf{L_{t'_1,1}} \mathbf{C}\mathbf{A}^{-(pk_{t'_1,1}+t'_1)} \mathbf{x} \\
\mathbf{L_{t'_1,1}} \mathbf{C}\mathbf{A}^{-(pk_{t'_1,2}+t'_1)} \mathbf{x} \\
\vdots \\
\mathbf{L_{t'_1,1}} \mathbf{C}\mathbf{A}^{-(pk_{t'_1,\bar{m}_{t'_1,1}}+t'_1)} \mathbf{x} \\
\mathbf{L_{t'_2,1}} \mathbf{C}\mathbf{A}^{-(pk_{t'_2,1}+t'_2)} \mathbf{x} \\
\vdots \\
\mathbf{L_{t'_{|T_1|},1}} \mathbf{C}\mathbf{A}^{-(pk_{t'_{|T_1|},\bar{m}_{t'_{|T_1|},1}}+t'_{|T_1|})} \mathbf{x}
\end{bmatrix}
\nonumber \\
&=
diag\{ \mathbf{\bar{U}_{t'_1,1}}^{-1} \mathbf{M_{t'_1}}, \mathbf{\bar{U}_{t'_2,1}}^{-1} \mathbf{M_{t'_2}}, \cdots, \mathbf{\bar{U}_{t'_{|T_1|},1}}^{-1} \mathbf{M_{t'_{|T_1|}}} \}
\begin{bmatrix}
\mathbf{\bar{C}_{t'_1,1}} \mathbf{\bar{A}_{t'_1,1}}^{-k_{t'_1,1}} \mathbf{\bar{U}_{t'_1,1}}\mathbf{\bar{x}_{t'_1,1}} \\
\mathbf{\bar{C}_{t'_1,1}} \mathbf{\bar{A}_{t'_1,1}}^{-k_{t'_1,2}} \mathbf{\bar{U}_{t'_1,1}}\mathbf{\bar{x}_{t'_1,1}} \\
\vdots \\
\mathbf{\bar{C}_{t'_1,1}}\mathbf{\bar{A}_{t'_1,1}}^{-k_{t'_1,\bar{m}_{t'_1,1}}}\mathbf{\bar{U}_{t'_1,1}}\mathbf{\bar{x}_{t'_1,1}} \\
\mathbf{\bar{C}_{t'_2,1}}\mathbf{\bar{A}_{t'_2,1}}^{-k_{t'_2,1}} \mathbf{\bar{U}_{t'_2,1}}\mathbf{\bar{x}_{t'_2,1}} \\
\vdots \\
\mathbf{\bar{C}_{t'_{|T_1|},1}}\mathbf{\bar{A}_{t'_{|T_1|},1}}^{-k_{t'_{|T_1|},\bar{m}_{t'_{|T_1|},1}}}\mathbf{\bar{U}_{t'_{|T_1|},1}}\mathbf{\bar{x}_{t'_{|T_1|},1}}
\end{bmatrix} \label{eqn:dis:thm:21} \\
&=
\begin{bmatrix}
\mathbf{\bar{U}_{t'_1,1}}^{-1} \mathbf{M_{t'_1}}
\begin{bmatrix}
\mathbf{\bar{C}_{t'_1,1}}\mathbf{\bar{A}_{t'_1,1}}^{-k_{t'_1,1}} \\
\vdots\\
\mathbf{\bar{C}_{t'_1,1}}\mathbf{\bar{A}_{t'_1,1}}^{-k_{t'_1,\bar{m}_{t'_1,1}}} \\
\end{bmatrix}
\mathbf{\bar{U}_{t'_1,1}} \mathbf{\bar{x}_{t'_1,1}}\\
\vdots \\
\mathbf{\bar{U}_{t'_{|T_1|},1}}^{-1} \mathbf{M_{t'_{|T_1|}}}
\begin{bmatrix}
\mathbf{\bar{C}_{t'_{|T_1|},1}}\mathbf{\bar{A}_{t'_1,1}}^{-k_{t'_{|T_1|},1}} \\
\vdots\\
\mathbf{\bar{C}_{t'_{|T_1|},1}}\mathbf{\bar{A}_{t'_1,1}}^{-k_{t'_{|T_1|},\bar{m}_{t'_{|T_1|},1}}}
\end{bmatrix} \mathbf{\bar{U}_{t'_{|T_1|},1}} \mathbf{\bar{x}_{t'_{|T_1|},1}}
\end{bmatrix} \label{eqn:dis:thm:22} \\
&=\begin{bmatrix}
\mathbf{\bar{x}_{t'_1,1}} \\
\vdots\\
\mathbf{\bar{x}_{t'_{|T_1|},1}}\\
\end{bmatrix}.\label{eqn:dis:thm9}
\end{align}
Here, \eqref{eqn:dis:thm:21} comes from the condition (v) of Claim~\ref{claim:donknow2}. \eqref{eqn:dis:thm9} comes from the definition of $\mathbf{M_q}$.

Now, we will estimate $(\mathbf{x})_{m_{1,1}}$ based on $\mathbf{\bar{x}_{t'_1,1}}, \cdots, \mathbf{\bar{x}_{t'_{|T_1|},1}}$. Let $\mathbf{e_{m_{1,1}}^{\bar{\bar{m}}_{q,r}}}$ be a $1 \times \bar{\bar{m}}_{q,r}$ row vector whose elements are all zeros except $m_{1,1}$th element which is $1$. Then, we have the following equation:
\begin{align}
&\begin{bmatrix}
\frac{1}{g_{t_1',1}} \mathbf{e_{m_{1,1}}^{\bar{\bar{m}}_{t_1',1}}} &  \cdots & \frac{1}{g_{t_{|T_1|}',1}} \mathbf{e_{m_{1,1}}^{\bar{\bar{m}}_{|t'_{|T_1|}|,1}}}
\end{bmatrix}
\begin{bmatrix}
\mathbf{\bar{x}_{t'_1,1}} \\
\vdots\\
\mathbf{\bar{x}_{t'_{|T_1|},1}}\\
\end{bmatrix} \nonumber \\
&=\frac{1}{g_{t_1',1}} (\mathbf{\bar{x}_{t'_1,1}})_{m_{1,1}}+ \cdots + \frac{1}{g_{t_{|T_1|}',1}} (\mathbf{\bar{x}_{t'_{|T_1|},1}})_{m_{1,1}} \nonumber \\
&=
\left(
\mathbf{L_{t'_1,1}}\mathbf{C_1}
diag\{
\alpha_{1,1},\cdots, \alpha_{1,\nu_1}
\}^{-(t'_1+(m_{1,1}-1))}
\right)
\begin{bmatrix}
(\mathbf{x_{1,1}})_{m_{1,1}}\\
\vdots \\
(\mathbf{x'_{1,\nu_1}})_{m_{1,1}}
\end{bmatrix}
+
\cdots \nonumber \\
&+
\left(
\mathbf{L_{t'_{|T_1|},1}}\mathbf{C_1}
diag\{
\alpha_{1,1},\cdots, \alpha_{1,\nu_1}
\}^{-(t'_{|T_1|}+(m_{1,1}-1))}
\right)
\begin{bmatrix}
(\mathbf{x_{1,1}})_{m_{1,1}}\\
\vdots \\
(\mathbf{x'_{1,\nu_1}})_{m_{1,1}}
\end{bmatrix}
\label{eqn:dis:thm:24} \\
&=
\begin{bmatrix}
\mathbf{L_{t'_1,1}} & \cdots & \mathbf{L_{t'_{|T_1|},1}}
\end{bmatrix}
\begin{bmatrix}
\mathbf{C_1}diag\{\alpha_{1,1},\cdots, \alpha_{1,\nu_1}\}^{-t'_1} \\
\vdots \\
\mathbf{C_1}diag\{\alpha_{1,1},\cdots, \alpha_{1,\nu_1}\}^{-t'_{|T_1|}}
\end{bmatrix}
\begin{bmatrix}
\alpha_{1,1}^{-m_{1,1}+1}(\mathbf{x_{1,1}})_{m_{1,1}}\\
\vdots \\
\alpha_{1,\nu_1}^{-m_{1,1}+1}(\mathbf{x'_{1,\nu_1}})_{m_{1,1}}
\end{bmatrix} \nonumber \\
&= \alpha_{1,1}^{-m_{1,1}+1}(\mathbf{x_{1,1}})_{m_{1,1}}= \alpha_{1,1}^{-m_{1,1}+1}(\mathbf{x})_{m_{1,1}}. \label{eqn:dis:thm7}
\end{align}
Here, \eqref{eqn:dis:thm:24} follows from the condition (iv) of Claimi~\ref{claim:donknow2}. 
\eqref{eqn:dis:thm7} follows from \eqref{eqn:dis:thm6} and $\{ t'_1 (mod~p_1), \cdots,  t'_{|T_1|} (mod~ p_1) \}=T_1$.

Now, we merge the results from \eqref{eqn:dis:thm9} and \eqref{eqn:dis:thm7} to make an estimator for $(\mathbf{x})_{m_{1,1}}$. Define
\begin{align}
\mathbf{\bar{M}}:=&
\alpha_{1,1}^{m_{1,1}-1}
\begin{bmatrix}
\frac{1}{g_{t_1',1}} \mathbf{e_{m_{1,1}}^{\bar{\bar{m}}_{|t'_1|,1}}} &  \cdots & \frac{1}{g_{t_{|T_1|}',1}} \mathbf{e_{m_{1,1}}^{\bar{\bar{m}}_{|t'_{|T_1|}|,1}}}
\end{bmatrix} \nonumber \\
&\cdot diag\{ \mathbf{\bar{U}_{t'_1,1}}^{-1} \mathbf{M_{t'_1}}, \mathbf{\bar{U}_{t'_2,1}}^{-1} \mathbf{M_{t'_2}}, \cdots, \mathbf{\bar{U}_{t'_{|T_1|},1}}^{-1} \mathbf{M_{t'_{|T_1|} }} \}
diag\{ \mathbf{L_{t'_1,1}},\mathbf{L_{t'_1,1}}, \cdots, \mathbf{L_{t'_{|T_1|},1}} \} \nonumber
\end{align}
and
\begin{align}
\begin{bmatrix}
\mathbf{C}\mathbf{A}^{-\bar{k}_1} \\
\vdots \\
\mathbf{C}\mathbf{A}^{-\bar{k}_{\bar{m}}} \\
\end{bmatrix}
:=
\begin{bmatrix}
\mathbf{C}\mathbf{A}^{-(pk_{t'_1,1}+t'_1)} \\
\mathbf{C}\mathbf{A}^{-(pk_{t'_1,2}+t'_1)} \\
\vdots \\
\mathbf{C}\mathbf{A}^{-(pk_{t'_1,\bar{m}_{t'_1,1}}+t'_1)} \\
\mathbf{C}\mathbf{A}^{-(pk_{t'_2,1}+t'_2)} \\
\vdots \\
\mathbf{C}\mathbf{A}^{-(pk_{t'_{|T_1|},\bar{m}_{t'_{|T_1|},1}}+t'_{|T_1|})}
\end{bmatrix}. \nonumber
\end{align}
Then, by (iii') and (iv') we can find a positive polynomial $\bar{p}(k)$ such that
\begin{align}
\left| \mathbf{\bar{M}} \right|_{max} \lesssim \max_{1 \leq  i \leq |T_1|} \{ |\mathbf{M_{t'_i}}|_{max} \} \leq \frac{\bar{p}(\bar{S}(\epsilon,k))}{\epsilon}|\lambda_{1,1}|^{\bar{S}(\epsilon,k)}. \label{eqn:dis:thm11}
\end{align}
Moreover, by \eqref{eqn:dis:thm9} and \eqref{eqn:dis:thm7} we have
\begin{align}
\mathbf{\bar{M}}
\begin{bmatrix}
\mathbf{C}\mathbf{A}^{-\bar{k}_1} \\
\vdots \\
\mathbf{C}\mathbf{A}^{-\bar{k}_{\bar{m}}} \\
\end{bmatrix}
\mathbf{x}=(\mathbf{x})_{m_{1,1}}. \label{eqn:dis:thm10}
\end{align}
This finishes the proof of the claim
\end{proof}

$\bullet$ Subtracting $(\mathbf{x})_{m_{1,1}}$ from the observations:  Now, we have an estimation for $(\mathbf{x})_{m_{1,1}}$. We will remove it from the system.

$\mathbf{\widetilde{A}}$, $\mathbf{\widetilde{C}}$ and $\mathbf{\widetilde{x}}$ are the system matrices after the removal. Formally, $\mathbf{\widetilde{A}}$ is obtained by removing $m_{1,1}$th row and column from $\mathbf{A}$, $\mathbf{\widetilde{C}}$ is obtained by removing $m_{1,1}$th row from $\mathbf{C}$, and $\mathbf{\widetilde{x}}$ is obtained by removing $m_{1,1}$th component from $\mathbf{x}$ respectively.

Denote $m_{1,1}$th column of $\mathbf{C}{\mathbf{A}}^{-k}$ as $\mathbf{R}(k)$. Then, we have the following relation between the original system $(\mathbf{A},\mathbf{C})$ and the new system $(\mathbf{\widetilde{A}}$, $\mathbf{\widetilde{C}})$:
\begin{align}
\mathbf{C}\mathbf{A}^{-k}\mathbf{x}-\mathbf{R}(k)(\mathbf{x})_{m_{1,1}}=\mathbf{\widetilde{C}}\mathbf{\widetilde{A}}^{-k}\mathbf{\widetilde{x}} \label{eqn:dis:thm8}
\end{align}
which can be easily proved from the block diagonal structure of $\mathbf{A}$.
From the definition of $\mathbf{R}(k)$, we can further see that there exists a polynomial $\widetilde{p}(k)$ such that $|\mathbf{R}(k)|_{max} \leq \widetilde{p}(k) |\lambda_{1,1}|^{-k}$. Thus, when $|\lambda_{1,1}| > 1$ we can find a threshold $k_{th} \geq 0$ such that all $k \geq k_{th}$, $\widetilde{p}(k) |\lambda_{1,1}|^{-k}$ is a decreasing function. When $|\lambda_{1,1}|=1$, we simply put $k_{th}=0$.

$\bullet$ Decoding the remaining element of $\mathbf{x}$: We decoded and subtracted the state $(\mathbf{x})_{m_{1,1}}$ from the system. After subtracting, the remaining system matrices $\mathbf{\widetilde{A}} \in \mathbb{C}^{(m-1) \times (m-1)}$ and $\mathbf{\widetilde{C}} \in \mathbb{C}^{l \times  (m-1)}$ become one-dimension smaller. Therefore, we can apply the induction hypothesis to estimate $\mathbf{\widetilde{x}}$.

We can also write $\mathbf{\widetilde{A}}$ and $\mathbf{\widetilde{C}}$ in the same way that we write $\mathbf{A}$ and $\mathbf{C}$ as \eqref{eqn:ac:jordan}, \eqref{eqn:ac2:jordan} and \eqref{eqn:def:lprime}, and define the corresponding parameters shown in \eqref{eqn:ac:jordan}, \eqref{eqn:ac2:jordan} and \eqref{eqn:def:lprime}.
To distinguish the parameters for $\mathbf{\widetilde{A}}$ and $\mathbf{\widetilde{C}}$ from the parameters for $\mathbf{A}$ and $\mathbf{C}$, we use tilde. For example, the dimension of $\mathbf{A}$ was $m \times m$, and we define the dimension of $\mathbf{\widetilde{A}}$ as $\widetilde{m} \times \widetilde{m}$. Likewise, the parameters $\widetilde{\mu}$, $\widetilde{\nu}_i$, $\widetilde{\lambda}_{i,j}$, $\widetilde{m}_{i,j}$, $\widetilde{p}_i$, $\widetilde{l}_i$ are defined for the system matrices $\mathbf{\widetilde{A}}$ and $\mathbf{\widetilde{C}}$ in the same ways as \eqref{eqn:ac:jordan}, \eqref{eqn:ac2:jordan} and \eqref{eqn:def:lprime}.

By the induction hypothesis, for $1 \leq i \leq \widetilde{\mu}$ we can find $\widetilde{m}_1',\cdots,\widetilde{m}_{\widetilde{\mu}}' \in \mathbb{N}$, positive polynomials $\widetilde{p}_1(k),\cdots,\widetilde{p}_{\widetilde{\mu}}(k)$ and families of stopping times
$\{ \widetilde{S}_1(\epsilon,k): k \in \mathbb{Z}^+ , 0 < \epsilon < 1 \},\cdots,\{ \widetilde{S}_{\widetilde{\mu}}(\epsilon,k): k \in \mathbb{Z}^+, 0 < \epsilon < 1 \}$
such that for all $0 < \epsilon < 1$ there exist
$\max \{ \bar{S}(\epsilon,k),k_{th}\} \leq \widetilde{k}_1 < \cdots < \widetilde{k}_{\widetilde{m}_1'} \leq \widetilde{S}_1(\epsilon,k) < \widetilde{k}_{\widetilde{m}_1'+1} < \cdots <
\widetilde{k}_{\sum_{1 \leq i \leq \widetilde{\mu}} \widetilde{m}_i'}
\leq  \widetilde{S}_{\widetilde{\mu}}(\epsilon,k)$ and a $\widetilde{m} \times (\sum_{1 \leq i \leq \widetilde{\mu} } \widetilde{m}_i')l$ matrix $\mathbf{\widetilde{M}}$ satisfying the following conditions:\\
(i'') $\beta[\widetilde{k}_i]=1$ for $1 \leq i \leq \sum_{1 \leq i \leq \widetilde{\mu}} \widetilde{m}_i$ \\
(ii'') $\mathbf{\widetilde{M}}
\begin{bmatrix}
\mathbf{\widetilde{C}} \mathbf{\widetilde{A}}^{-\widetilde{k}_1} \\
\mathbf{\widetilde{C}} \mathbf{\widetilde{A}}^{-\widetilde{k}_2} \\
\vdots \\
\mathbf{\widetilde{C}} \mathbf{\widetilde{A}}^{-\widetilde{k}_{\sum_{1 \leq i \leq \widetilde{\mu}} \widetilde{m}_i' }}
\end{bmatrix}=\mathbf{I}
$ \\
(iii'') $
|\mathbf{\widetilde{M}}|_{max} \leq \max_{1 \leq i \leq \widetilde{\mu}} \left\{
\frac{\widetilde{p}_i( \tilde{S}_i(\epsilon,k) )}{\epsilon} |\widetilde{\lambda}_{i,1}|^{\widetilde{S}_i(\epsilon,k)}
\right\}
$\\
(iv'') $\lim_{\epsilon \downarrow 0} \exp \limsup_{s \rightarrow \infty} \esssup
\frac{1}{s} \log \mathbb{P}\{ \widetilde{S}_i(\epsilon,k) - \max\{\bar{S}(\epsilon,k),k_{th} \} = s | \mathcal{F}_{\bar{S}(\epsilon,k)}  \} = \max_{1 \leq j \leq i} \left\{ p_e^{\frac{\widetilde{l}_j}{\widetilde{p}_j}} \right\}
 $ for $1\leq i \leq \widetilde{\mu}$\\
(v'') $\lim_{\epsilon \downarrow 0} \exp \limsup_{s \rightarrow \infty} \esssup \frac{1}{s} \log \mathbb{P} \{
\widetilde{S}_a(\epsilon,k)-\widetilde{S}_b(\epsilon,k)=s | \mathcal{F}_{\widetilde{S}_b(\epsilon,k)}
\} \leq \max_{b < i \leq a} \left\{ p_e^{\frac{\widetilde{l}_i}{\widetilde{p}_i}} \right\} $ for $1 \leq b < a \leq \widetilde{\mu}$.
Compared to Lemma~\ref{lem:dis:achv}, we can notice that the condition (iv'') is slightly different from the condition (iv) of Lemma~\ref{lem:dis:achv}. The $\sup$ over $k$ of (iv) in Lemma~\ref{lem:dis:achv} is replaced by the $\esssup$. However, if we remind that $\max\{\bar{S}(\epsilon,k),k_{th} \}$ is a constant conditioned on\footnote{More proper notations for $\widetilde{S}_1(\epsilon,k)$, $\cdots$, $\widetilde{S}_{\mu}(\epsilon,k)$ are $\widetilde{S}_1(\epsilon,\max\{  \bar{S}(\epsilon,k),k_{th}\})$, $\cdots$, $\widetilde{S}_{\mu}(\epsilon,\max\{  \bar{S}(\epsilon,k),k_{th}\})$ since $\max\{  \bar{S}(\epsilon,k),k_{th}\}$ plays the role of $k$ of Lemma~\ref{lem:dis:achv} after conditioning. However, we use the notations of the paper for simplicity.} $\mathcal{F}_{\bar{S}(\epsilon,k)}$, we just replaced $k$ of Lemma~\ref{lem:dis:achv} with $\max\{\bar{S}(\epsilon,k),k_{th} \}$.

Here, we have
\begin{align}
\mathbf{\widetilde{x}}&=\mathbf{\widetilde{M}}
\begin{bmatrix}
\mathbf{\widetilde{C}} \mathbf{\widetilde{A}}^{-\widetilde{k}_1} \\
\mathbf{\widetilde{C}} \mathbf{\widetilde{A}}^{-\widetilde{k}_2} \\
\vdots \\
\mathbf{\widetilde{C}} \mathbf{\widetilde{A}}^{-\widetilde{k}_{\sum_{1 \leq i \leq \widetilde{\mu}} \widetilde{m}_i' }}
\end{bmatrix}\mathbf{\widetilde{x}} \nonumber \\
&=
\mathbf{\widetilde{M}}
\begin{bmatrix}
\mathbf{\widetilde{C}} \mathbf{\widetilde{A}}^{-\widetilde{k}_1}\mathbf{\widetilde{x}} \\
\mathbf{\widetilde{C}} \mathbf{\widetilde{A}}^{-\widetilde{k}_2}\mathbf{\widetilde{x}} \\
\vdots \\
\mathbf{\widetilde{C}} \mathbf{\widetilde{A}}^{-\widetilde{k}_{\sum_{1 \leq i \leq \widetilde{\mu}} \widetilde{m}_i' }}\mathbf{\widetilde{x}}
\end{bmatrix} \nonumber \\
&=
\mathbf{\widetilde{M}}
\begin{bmatrix}
\mathbf{C} \mathbf{A}^{-\widetilde{k}_1} \mathbf{x} - \mathbf{R}(\widetilde{k}_1) (\mathbf{x})_{m_{1,1}} \\
\mathbf{C} \mathbf{A}^{-\widetilde{k}_2} \mathbf{x} - \mathbf{R}(\widetilde{k}_2) (\mathbf{x})_{m_{1,1}} \\
\vdots \\
\mathbf{C} \mathbf{A}^{-\widetilde{k}_{\sum_{1 \leq i \leq \widetilde{\mu}}\widetilde{m}_i'}} \mathbf{x} - \mathbf{R}(\widetilde{k}_{\sum_{1 \leq i \leq \widetilde{\mu}}\widetilde{m}_i' }) (\mathbf{x})_{m_{1,1}}
\end{bmatrix} (\because \eqref{eqn:dis:thm8})\nonumber \\
&=
\mathbf{\widetilde{M}}
\left(
\begin{bmatrix}
\mathbf{C} \mathbf{A}^{-\widetilde{k}_1} \\
\mathbf{C} \mathbf{A}^{-\widetilde{k}_2} \\
\vdots \\
\mathbf{C} \mathbf{A}^{-\widetilde{k}_{\sum_{1 \leq i \leq \widetilde{\mu}}\widetilde{m}_i'}}
\end{bmatrix}
\mathbf{x}
-
\begin{bmatrix}
\mathbf{R}(\widetilde{k}_1) \\
\mathbf{R}(\widetilde{k}_2) \\
\vdots \\
\mathbf{R}(\widetilde{k}_{\sum_{1 \leq i \leq \widetilde{\mu}}\widetilde{m}_i'})
\end{bmatrix}
(\mathbf{x})_{m_{1,1}}
\right) \nonumber \\
&=
\mathbf{\widetilde{M}}
\left(
\begin{bmatrix}
\mathbf{C} \mathbf{A}^{-\widetilde{k}_1} \\
\mathbf{C} \mathbf{A}^{-\widetilde{k}_2} \\
\vdots \\
\mathbf{C} \mathbf{A}^{-\widetilde{k}_{\sum_{1 \leq i \leq \widetilde{\mu}}\widetilde{m}_i' }}
\end{bmatrix}
\mathbf{x}
-
\begin{bmatrix}
\mathbf{R}(\widetilde{k}_1) \\
\mathbf{R}(\widetilde{k}_2) \\
\vdots \\
\mathbf{R}(\widetilde{k}_{\sum_{1 \leq i \leq \widetilde{\mu}}\widetilde{m}_i' })
\end{bmatrix}
\mathbf{{\bar{M}}}
\begin{bmatrix}
\mathbf{C}\mathbf{A}^{-\bar{k}_1} \\
\mathbf{C}\mathbf{A}^{-\bar{k}_2} \\
\vdots \\
\mathbf{C}\mathbf{A}^{-\bar{k}_{\bar{m}}}
\end{bmatrix}
\mathbf{x}
\right) (\because \mbox{the condition (ii) of Claim~\ref{claim:donknow00}})
\nonumber \\
&=
\mathbf{\widetilde{M}}
\begin{bmatrix}
-
\begin{bmatrix}
\mathbf{R}(\widetilde{k}_1) \\
\mathbf{R}(\widetilde{k}_2) \\
\vdots \\
\mathbf{R}(\widetilde{k}_{\sum_{1 \leq i \leq \widetilde{\mu}}\widetilde{m}_i' })
\end{bmatrix}
\mathbf{{\bar{M}}} & \mathbf{I}
\end{bmatrix}
\begin{bmatrix}
\mathbf{C}\mathbf{A}^{-\bar{k}_1} \\
\vdots \\
\mathbf{C}\mathbf{A}^{-\bar{k}_{\bar{m}}}\\
\mathbf{C} \mathbf{A}^{-\widetilde{k}_1} \\
\vdots \\
\mathbf{C} \mathbf{A}^{-\widetilde{k}_{\sum_{1 \leq i \leq \widetilde{\mu}}\widetilde{m}_i'} }
\end{bmatrix}
\mathbf{x}. \label{eqn:successive:102}
\end{align}

When $|\lambda_{1,1}|>1$, we have
\begin{align}
&
\left|
\mathbf{\widetilde{M}}
\begin{bmatrix}
-
\begin{bmatrix}
\mathbf{R}(\widetilde{k}_1) \\
\mathbf{R}(\widetilde{k}_2) \\
\vdots \\
\mathbf{R}(\widetilde{k}_{\sum_{1 \leq i \leq \widetilde{\mu}}\widetilde{m}_i' })
\end{bmatrix}
\mathbf{{\bar{M}}} & \mathbf{I}
\end{bmatrix}
\right|_{max}
\nonumber \\
&\lesssim
| \mathbf{\widetilde{M}} |_{max} \cdot
\max
\left\{ \left|
\begin{bmatrix}
\mathbf{R}(\widetilde{k}_1) \\
\mathbf{R}(\widetilde{k}_2) \\
\vdots \\
\mathbf{R}(\widetilde{k}_{\sum_{1 \leq i \leq \widetilde{\mu}}\widetilde{m}_i' })
\end{bmatrix}
\right|_{max}  \left| \mathbf{\bar{M}} \right|_{max} , 1 \right\} \nonumber \\
& \lesssim
\max_{1 \leq i \leq \widetilde{\mu}} \left\{
\frac{\widetilde{p}_i( \widetilde{S}_i(\epsilon,k) )}{\epsilon} |\widetilde{\lambda}_{i,1}|^{\widetilde{S}_i(\epsilon,k)}
\right\} \cdot \max \left\{ \widetilde{p}(\widetilde{k}_1)|\lambda_{1,1}|^{-\widetilde{k}_1} \frac{\bar{p}\left(\bar{S}(\epsilon,k)\right)}{\epsilon} |\lambda_{1,1}|^{\bar{S}(\epsilon,k)} , 1 \right\} \label{eqn:successive:101}
\end{align}
where the last inequality follows from (iii''), $|\mathbf{R}(k)| \leq \widetilde{p}(k) |\lambda_{1,1}|^{-k}$, $k_{th} \leq  \widetilde{k}_i$, and the condition (iii) of Claim~\ref{claim:donknow00}.
 Moreover, since $\bar{S}(\epsilon,k) \leq \widetilde{k}_1 \leq \widetilde{S}_i(\epsilon,k)$, there exists some positive polynomials $p'_i(k)$ such that
\begin{align}
&\eqref{eqn:successive:101}\lesssim
\max_{1 \leq i \leq \widetilde{\mu}}
\left\{
\frac{
p'_i(\widetilde{S}_i(\epsilon,k))
}{\epsilon^2}
|\widetilde{\lambda}_{i,1}|^{\widetilde{S}_i(\epsilon,k)}
\right\}\label{eqn:successive:200}
\end{align}

When $|\lambda_{1,1}|=1$, $|\widetilde{\lambda}_{1,1}|$ is also $1$. Thus, we have
\begin{align}
&
\left|
\mathbf{\widetilde{M}}
\begin{bmatrix}
-
\begin{bmatrix}
\mathbf{R}(\widetilde{k}_1) \\
\mathbf{R}(\widetilde{k}_2) \\
\vdots \\
\mathbf{R}(\widetilde{k}_{\sum_{1 \leq i \leq \widetilde{\mu}}\widetilde{m}_i' })
\end{bmatrix}
\mathbf{{\bar{M}}} & \mathbf{I}
\end{bmatrix}
\right|_{max}
\nonumber \\
&\lesssim
| \mathbf{\widetilde{M}} |_{max} \cdot
\max
\left\{ \left|
\begin{bmatrix}
\mathbf{R}(\widetilde{k}_1) \\
\mathbf{R}(\widetilde{k}_2) \\
\vdots \\
\mathbf{R}(\widetilde{k}_{\sum_{1 \leq i \leq \widetilde{\mu}}\widetilde{m}_i' })
\end{bmatrix}
\right|_{max}  \left| \mathbf{\bar{M}} \right|_{max} , 1 \right\} \nonumber \\
& \lesssim
\max_{1 \leq i \leq \widetilde{\mu}} \left\{ \frac{\widetilde{p}_1(\widetilde{S}_1(\epsilon,k))}{\epsilon} \right\}
\cdot \max\left\{ \widetilde{p}(\widetilde{k}_{\sum_{1 \leq i \leq \widetilde{\mu}}\widetilde{m}_i' } \frac{\bar{p}\left(\bar{S}(\epsilon,k)\right)}{\epsilon} , 1\right\} \nonumber \\
& \lesssim \frac{p'(\widetilde{S}_{\widetilde{\mu}}(\epsilon,k))}{\epsilon^2} \label{eqn:successive:201}
\end{align}
for some polynomial $p'_{\widetilde{\mu}}(k)$.

Since we can reconstruct $\mathbf{x}$ from $\mathbf{\widetilde{x}}$ and $(\mathbf{x})_{m_{1,1}}$ , we can say there exists $\mathbf{M}$ such that
\begin{align}
\mathbf{M}\begin{bmatrix}
\mathbf{C}\mathbf{A}^{-\bar{k}_1} \\
\vdots \\
\mathbf{C}\mathbf{A}^{-\bar{k}_{\bar{m}}} \\
\mathbf{C}\mathbf{A}^{-\widetilde{k}_1} \\
\vdots \\
\mathbf{C}\mathbf{A}^{-\widetilde{k}_{\sum_{1 \leq i \leq \widetilde{\mu}} \widetilde{m}_i}} \\
\end{bmatrix}= \mathbf{I}. \nonumber
\end{align}
By the condition (ii) of Claim~\ref{claim:donknow00} and \eqref{eqn:successive:102}, such $\mathbf{M}$ satisfies the following:
\begin{align}
| \mathbf{M} |_{max} &\leq \max\left\{\left|\mathbf{\bar{M}}\right|_{max},\left|
\mathbf{\widetilde{M}}
\begin{bmatrix}
-
\begin{bmatrix}
\mathbf{R}(\widetilde{k}_1) \\
\mathbf{R}(\widetilde{k}_2) \\
\vdots \\
\mathbf{R}(\widetilde{k}_{\sum_{1 \leq i \leq \widetilde{\mu}}\widetilde{m}_i' })
\end{bmatrix}
\mathbf{{\bar{M}}} & \mathbf{I}
\end{bmatrix}
\right|_{max}\right\} \nonumber \\
&\lesssim
\max\left\{
\frac{{\bar{p}(\bar{S}(\epsilon,k))}}{\epsilon} |\lambda_{1,1}|^{\bar{S}(\epsilon,k)},
\max_{1 \leq i \leq \widetilde{\mu}}
\left\{
\frac{ p'_i(\widetilde{S}_i(\epsilon,k) )  }{\epsilon^2}
|\widetilde{\lambda}_{i,1}|^{\widetilde{S}_i(\epsilon,k)}
\right\}
\right\}\label{eqn:successive:202} \\
& \leq
\frac{1}{\epsilon^2}
\max \left\{
{\bar{p}(\bar{S}(\epsilon,k))} |\lambda_{1,1}|^{\bar{S}(\epsilon,k)},
\max_{1 \leq i \leq \widetilde{\mu}}
\left\{
p'_i(\widetilde{S}_i(\epsilon,k) )
|\widetilde{\lambda}_{i,1}|^{\widetilde{S}_i(\epsilon,k)}
\right\}
\right\}. \label{eqn:dis:geofinal:5}
\end{align}
Here, \eqref{eqn:successive:202} follows from the condition (iii) of Claim~\ref{claim:donknow00}, \eqref{eqn:successive:200}, \eqref{eqn:successive:201}.

Moreover, since $k_{th}$ is a constant, the condition (iv) of Claim~\ref{claim:donknow00} implies 
\begin{align}
\lim_{\epsilon \downarrow 0} \exp \limsup_{s \rightarrow \infty} \sup_{k \in \mathbb{Z}^+} \frac{1}{s} \log \mathbb{P} \left\{ \max\left\{\bar{S}(\epsilon,k),k_{th}\right\} -k = s \right\} = p_e^{\frac{l_1}{p_1}}. \label{eqn:successive:204}
\end{align}

Therefore, by applying Lemma~\ref{lem:app:geo} together with \eqref{eqn:successive:204} and (iv'')
we get
\begin{align}
\lim_{\epsilon \downarrow 0} \exp \limsup_{s \rightarrow \infty} \sup_{k \in \mathbb{Z}^+} \frac{1}{s} \log \mathbb{P} \{ \widetilde{S}_i(\epsilon,k) - k =s \}= \max\left\{ p_e^{\frac{l_1}{p_1}}, \max_{1 \leq j \leq i}\left\{ p_e^{\frac{\widetilde{l}_j}{\widetilde{p}_j}} \right\} \right\} \label{eqn:dis:geofinal:6}.
\end{align}

We finish the proof by dividing into two cases depending on $\widetilde{\mu}$. Since $\mathbf{\widetilde{A}}$ is obtained by erasing just one row and column of $\mathbf{A}$, the relation between $\widetilde{\mu}$ and $\mu$ is either $\widetilde{\mu}=\mu$ or $\widetilde{\mu}=\mu-1$.

(1) When $\widetilde{\mu}=\mu$.

In this case, the number of the eigenvalue cycles remains the same. We can see that $|\widetilde{\lambda}_{i,1}|=|\lambda_{i,1}|$. $\mathbf{A_1}$ and $\mathbf{\widetilde{A}_1}$ may be the same or $\mathbf{\widetilde{A}_1}$ has smaller dimension than $\mathbf{A_1}$. Thus, the new system $\mathbf{\widetilde{A}_1}$ becomes easier to estimate, and $\frac{\widetilde{l}_1}{\widetilde{p}_1} \geq \frac{l_1}{p_1}$, i.e. $p_e^{\frac{\widetilde{l}_1}{\widetilde{p}_1}} \leq p_e^{\frac{l_1}{p_1}}$. $\mathbf{A_i}$ and $\mathbf{\widetilde{A}_i}$ are the same for all $2 \leq i \leq \mu$, so $\frac{\widetilde{l}_i}{\widetilde{p}_i} = \frac{l_i}{p_i}$ for $2 \leq j \leq \mu$.
Define $S_i(\epsilon^2,k):=\widetilde{S}_i(\epsilon,k)$,
$p_1(k):=\bar{p}(k)+p_1'(k)$, and
$p_i(k):=p_i'(k)$ for $2 \leq i \leq \mu$. Then, \eqref{eqn:dis:geofinal:5}, \eqref{eqn:dis:geofinal:6} and (v'') reduces as follows:
\begin{align}
| \mathbf{M} |_{max} \leq \max_{1 \leq i \leq \mu} \left\{ \frac{p_i(S_i(\epsilon,k))}{\epsilon} |\lambda_{i,1}|^{S_i(\epsilon,k)}  \right\}, \nonumber
\end{align}
\begin{align}
\lim_{\epsilon \downarrow 0} \exp \limsup_{s \rightarrow \infty} \sup_{k \in \mathbb{Z}^+} \frac{1}{s} \log \mathbb{P} \{ S_i(\epsilon,k) - k =s \} \leq \max_{1 \leq j \leq i}\left\{ p_e^{\frac{l_j}{p_j}} \right\}, \nonumber
\end{align}
\begin{align}
\lim_{\epsilon \downarrow 0} \exp \limsup_{s \rightarrow \infty} \esssup \frac{1}{s} \log \mathbb{P} \{
{S}_a(\epsilon,k)-{S}_b(\epsilon,k)=s | \mathcal{F}_{{S}_b(\epsilon,k)}
\} \leq \max_{b < i \leq a} \left\{ p_e^{\frac{{l}_i}{{p}_i}} \right\}. \nonumber
\end{align}
Here, we reparametrized $\epsilon^2$ to $\epsilon$. Therefore, the lemma is true for this case.

(2) When $\widetilde{\mu}=\mu-1$.

Since one eigenvalue cycle is disappeared, we can see that $|\widetilde{\lambda}_{1,1}|=|\lambda_{2,1}|,|\widetilde{\lambda}_{2,1}|=|\lambda_{3,1}|,\cdots, |\widetilde{\lambda}_{\widetilde{\mu},1}|=|\lambda_{\mu,1}|$.
Moreover, $\mathbf{\widetilde{A}_i}=\mathbf{A_{i+1}}$ for $1 \leq i \leq \widetilde{\mu}$ and $\frac{\widetilde{l}_{i}}{\widetilde{p}_{i}}=\frac{l_{i+1}}{p_{i+1}}$ for $1 \leq i \leq \widetilde{\mu}$.
Define $S_1(\epsilon^2,k):=\bar{S}(\epsilon,k)$, $p_1(k):=\bar{p}(k)$, $S_i(\epsilon^2,k):=\widetilde{S}_{i-1}(\epsilon,k)$ and $p_i(k):=p_{i-1}'(k )$ for $2 \leq i \leq \mu$.
We will also reparametrize $\epsilon^2$ to $\epsilon$. Then, \eqref{eqn:dis:geofinal:5} reduces to
\begin{align}
| \mathbf{M} |_{max} \leq \max_{1 \leq i \leq \mu} \left\{ \frac{p_i(S_i(\epsilon,k))}{\epsilon} |\lambda_{i,1}|^{S_i(\epsilon,k)}  \right\}. \nonumber
\end{align}
By the definition of $S_1(\epsilon,k)$, the condition (iv) of Claim~\ref{claim:donknow00} reduces to 
\begin{align}
\lim_{\epsilon \downarrow 0} \exp \limsup_{s \rightarrow \infty} \sup_{k \in \mathbb{Z}^+} \frac{1}{s} \log \mathbb{P} \{ S_1(\epsilon,k) - k=s \} \leq p_e^{\frac{l_1}{p_1}}. \nonumber
\end{align}
By \eqref{eqn:dis:geofinal:6} and the definition of $S_i(\epsilon,k)$, we have for all $2 \leq i \leq \mu$,
\begin{align}
\lim_{\epsilon \downarrow 0} \exp \limsup_{s \rightarrow \infty} \sup_{k \in \mathbb{Z}^+} \frac{1}{s} \log \mathbb{P} \{ S_i(\epsilon,k) - k=s \}
\leq \max\left\{p_e^{\frac{l_1}{p_1}} , \max_{1 \leq j \leq i-1}\left\{ p_e^{\frac{\widetilde{l}_j}{\widetilde{p}_j}} \right\} \right\}=
\max_{1 \leq j \leq i}\left\{ p_e^{\frac{l_i}{p_i}} \right\}. \nonumber
\end{align}
By (iv''), (v'') and the definition of $S_i(\epsilon,k)$, we have for all $1 \leq b < a \leq \mu$,
\begin{align}
\lim_{\epsilon \downarrow 0} \exp \limsup_{s \rightarrow \infty} \esssup \frac{1}{s} \log \mathbb{P} \{
{S}_a(\epsilon,k)-{S}_b(\epsilon,k)=s | \mathcal{F}_{{S}_b(\epsilon,k)}
\} \leq \max_{b < i \leq a} \left\{ p_e^{\frac{{l}_i}{{p}_i}} \right\}. \nonumber
\end{align}
Therefore, the lemma is also true for this case.

Thus, the proof is finished.
\end{proof}

\bibliographystyle{plain}
\bibliography{seyongbib}
\end{document}